\documentclass[11pt]{article}

\usepackage[all]{xy}
\usepackage{amssymb,amsthm,amsmath}
\usepackage{graphics}
\usepackage{epsfig}
\usepackage{color}
\usepackage{url}
\usepackage{hyperref}

\DeclareFontFamily{OT1}{pzc}{}
\DeclareFontShape{OT1}{pzc}{m}{it}{<-> s * [1.10] pzcmi7t}{}
\DeclareMathAlphabet{\mathscr}{OT1}{pzc}{m}{it}

\renewcommand{\subsection}[1]{\hspace{-\parindent}\refstepcounter{subsection}{\bf
(\arabic{section}\alph{subsection}) #1.}}

\numberwithin{equation}{section}

\newcommand{\bZ}{\mathbb{Z}}
\newcommand{\bR}{\mathbb{R}}
\newcommand{\bQ}{\mathbb{Q}}
\newcommand{\bC}{\mathbb{C}}

\newcommand{\htp}{\simeq}
\newcommand{\iso}{\cong}
\newcommand{\half}{{\textstyle\frac{1}{2}}}

\newcommand{\pabla}{\nabla\mkern-12mu/\mkern2mu}

\renewcommand{\AA}{\mathscr{A}}
\newcommand{\BB}{\mathscr{B}}
\newcommand{\CC}{\mathscr{C}}
\newcommand{\DD}{\mathscr{D}}

\newcommand{\FF}{\mathscr{F}}
\newcommand{\Fuk}{\mathscr{Fuk}}
\newcommand{\GG}{\mathscr{G}}
\newcommand{\HH}{\mathscr{H}}

\newcommand{\JJ}{\mathscr{J}}
\newcommand{\KK}{\mathscr{K}}

\newcommand{\MM}{\mathscr{M}}
\newcommand{\OO}{\mathscr{O}}
\newcommand{\PP}{\mathscr{P}}
\newcommand{\QQ}{\mathscr{Q}}
\renewcommand{\SS}{\mathscr{S}}

\newcommand{\NN}{\mathscr{N}}
\newcommand{\RR}{\mathscr{R}}

\newcommand{\XX}{\mathscr{X}}
\newcommand{\YY}{\mathscr{Y}}

\renewcommand{\aa}{\mathscr{a}}
\newcommand{\bb}{\mathscr{b}}
\newcommand{\cc}{\mathscr{c}}

\newcommand{\ff}{\mathscr{f}}
\newcommand{\mm}{\mathscr{m}}
\newcommand{\pp}{\mathscr{p}}

\theoremstyle{plain}
\newtheorem{lemma}{Lemma}[section]
\newtheorem{definition}[lemma]{Definition}
\newtheorem{proposition}[lemma]{Proposition}
\newtheorem{theorem}[lemma]{Theorem}
\newtheorem{remark}[lemma]{Remark}
\newtheorem{addendum}[lemma]{Addendum}
\newtheorem{example}[lemma]{Example}
\newtheorem{corollary}[lemma]{Corollary}
\newtheorem*{theorem*}{Theorem}
\newtheorem*{question*}{Question}
\newtheorem{assumption}[lemma]{Assumption}
\newtheorem{workingdefinition}[lemma]{Working Definition}
\newtheorem*{claim*}{Claim}

\begin{document}

\title{\bf Abstract analogues of flux\\ as symplectic invariants}
\author{Paul Seidel}
\date{}
\maketitle
\begin{abstract}\noindent
We study families of objects in Fukaya categories, specifically ones whose deformation behaviour is prescribed by the choice of an odd degree cohomology class. This leads to invariants of symplectic manifolds, which we apply to blowups along symplectic mapping tori.
\end{abstract}

\section*{Introduction}

{\bf Motivation.}
An interesting invariant of a closed symplectic manifold $M$ is its flux group, a subgroup of $H^1(M;\bR)$ obtained from the topology of loops of symplectic automorphisms \cite[Section 10.2]{salamon-mcduff-intro}. This can be effectively studied using Floer cohomology, one of the notable insights being that the flux group is always discrete \cite{ono06}. Now consider the following question:
\begin{quote} \it
Are there flux-type subgroups in $H^{2k-1}(M;\bR)$, for $k>1$, which can be nontrivial for manifolds with $H^1(M;\bR) = 0$?
\end{quote}
The last clause excludes one obvious direction, which is to take the subgroup formed by the image of $[\omega_M^k]$ under all the maps $H^{2k}(M;\bR) \rightarrow H^{2k-1}(M;\bR)$ induced by loops of symplectic automorphisms (this reproduces the flux group for $k = 1$, but it vanishes if $H^1(M;\bR) = 0$, by the rigidity theorem \cite{lalonde-mcduff-polterovich97}). Really, what the question is aiming for is a formalism in which higher degree differential forms replace the closed one-forms in their usual relation to symplectic vector fields, so anything related to symplectic automorphism groups can't really be the answer. This clarification may make the whole endeavour seem quixotic. Still, if one looks at it from the point of view of quantum cohomology $\mathit{QH}^*(M)$, the situation is less clear-cut. Passage to quantum cohomology generally reduces the grading to $\bZ/2$, putting all odd degree cohomology formally on equal footing (but degree one classes retain a more direct connection to geometry, because the quantum product with such a class remains equal to the classical cup product; this is by a version of the divisor axiom). In that vein, it turns out that one can give a partially positive answer to the question above, at least if one is willing to settle for an invariant which is somewhat more obscure, lacking the simplicity and geometric elegance of the flux group.

{\bf The examples.} As an application, we consider a particular pair of $28$-dimensional simply-connected symplectic manifolds (the following is only an outline of the construction, omitting many details and assumptions). Let $K$ be a $K3$ surface, and $T \subset K$ a symplectically embedded two-torus. Take $K^7$, the product of seven copies of $K$, and blow up the $12$-dimensional submanifold $T^2 \times K^2$. Denote the outcome by $B^{\mathit{triv}}$. This has a more interesting cousin $B$, defined in the same way but where the blowup locus is a product of $T$ and the symplectic mapping torus of a certain automorphism of $K \times K$ (embedded into $K^7$ by using the $h$-principle). It is known that the symplectic automorphism group of $K$ has many connected components which are not detected by classical topological means (see for instance \cite{seidel04b}, or \cite{seidel-thomas99} for a mirror symmetry viewpoint). Based on that, one can ensure that $B^{\mathit{triv}}$ is diffeomorphic to $B$, and that their symplectic structures are deformation equivalent. Nevertheless, for a specific choice of automorphism, we will show:
\begin{quote} \it
$B$ and $B^{\mathit{triv}}$ are not symplectically isomorphic.
\end{quote}
The construction of these manifolds is similar to that of the first known examples of distinct but deformation equivalent symplectic structures \cite{mcduff87}, which were also based on blowing up. That paper used (roughly speaking) a bordism-valued refinement of Gromov-Witten theory as an invariant. Because such refinements are hard to define and compute, we can't say how they would behave in our situation. In any case, the approach taken in this paper is quite different.

{\bf The invariant.} Let's temporarily go back to the simpler case of symplectic mapping tori. The symplectic mapping torus of an automorphism $f$ is a symplectic fibration over $T$ which has trivial monodromy in one direction, and monodromy $f$ in the other direction. Let's say for concreteness that $T = \bR^2/\bZ^2$ has coordinates $(p,q)$, and that the monodromy is trivial in $q$-direction, and $f$ in $p$-direction. The symplectic mapping torus contains plenty of Lagrangian submanifolds fibered over trivial circles $\{p\} \times S^1$. If one then moves such a Lagrangian submanifold by the time-one map of the symplectic vector field $\partial_p$, the effect is the same as applying $f$ fibrewise. For suitable examples of $f$, this allows one to show that $[dq]$ does not lie in the flux group, which distinguishes the mapping torus from the trivial one (a similar approach was used in \cite{seidel02b}). To make a more abstract version of the argument, we consider families of objects in the Fukaya category whose deformation is driven by the class $[dq]$. The idea of introducing families of Lagrangian submanifolds into Floer cohomology theory is due to Fukaya \cite{fukaya02c,fukaya09}. Generally speaking, it shows much potential for leading to fundamental insights as well as applications, but also encounters considerable foundational difficulties. Here, we bypass these issues by choosing a more constrained version of the notion of family, which has a straightforward basis in algebraic geometry, but is somewhat harder to connect to symplectic geometry.

The outcome is an invariant of a symplectic manifold $M$, which also depends on the following auxiliary data. Take an element $\lambda$ of the (universal single-variable) Novikov field, and an idempotent $z$ in $\mathit{QH}^0(M)$. To these we associate an appropriate version of the Fukaya category $\mathit{Fuk}(M)_{\lambda,z}$ and its completion (split-closed triangulated envelope) $\mathit{Fuk}(M)_{\lambda,z}^{\mathit{perf}}$. Additionally, choose an elliptic curve $\bar{\SS}$ over the Novikov field, together with a nonzero algebraic one-form on it. Given a class in odd degree quantum cohomology, $x \in \mathit{QH}^1(M)$, one can then ask whether it is {\em periodic}, which means whether objects of the Fukaya category can be extended to families over $\bar{\SS}$ with deformation behaviour prescribed by $x$. In particular, we can apply this idea to symplectic mapping tori and recover some of the results ordinarily proved using flux. More interestingly, we can exploit existing ideas about the behaviour of Fukaya categories under blowups \cite{smith10}, and thereby arrive at the result stated above. The main point is that the relevant part of the Fukaya category of $B^{\mathit{triv}}$ is well-understood, allowing us to prove that certain classes $x$ are periodic. The Fukaya category of $B$ is not known to the same extent, but partial computations are enough to determine that certain classes $x$ are not periodic, since that only requires finding a specific Lagrangian submanifold which can serve as a counterexample.

To conclude this discussion, we should mention that the basic idea is by no means new in homological algebra. From that elevated vantage point, what we are doing (in a rather ad hoc way) is to study algebraic one-parameter subgroups (the elliptic curves $\bar{\SS}$ which appeared above) inside the derived Picard group \cite{yekutieli04,keller04} of the Fukaya category. The idea is to think of elements of Hochschild cohomology as vector fields on an abstract ``moduli space of objects'', and that we are asking which vector fields integrate to ``periodic flows''. On an informal level, it is clear that this provides an algebraic counterpart to the geometric ideas underlying flux.

{\bf Structure of the paper.}
Section \ref{sec:theory} sets up the algebraic theory of families of objects, much of it straightforward. The key to uniqueness results for families is the discussion surrounding Lemma \ref{th:its-projective}, which is then further developed for our intended applications in Section \ref{subsec:unique}. An abstract version of our invariant is introduced in Definition \ref{th:per}. Section \ref{sec:elliptic} discusses the simplest example of the two-torus $T$. Of course, its Fukaya category has already been studied exhaustively, starting with \cite{polishchuk-zaslow98}. Still, we devote some energy to it in order to prepare for the case of symplectic mapping tori, which is the topic of Section \ref{sec:mapping-tori} (following some preliminaries on Floer cohomology and Fukaya categories, in Section \ref{sec:automorphisms}). At first, it will seem that our computations lead us further from the intended goal, since we choose Lagrangian submanifolds whose Floer cohomology is largely independent of the choice of automorphism used to construct the mapping torus. However, we eventually do manage to recover some information about that automorphism, by a double covering trick which appears in Section \ref{subsec:covering-trick}. Finally, most of Section \ref{sec:blowup} is general discussion of blowups. The detailed construction of the manifolds $B$ and $B^{\mathit{triv}}$ is carried out in Section \ref{subsec:final}.

{\bf Acknowledgments.}
Section \ref{sec:blowup} of this paper relies crucially on results of \cite{abouzaid10, abouzaid-fukaya-oh-ohta-ono11, smith10}. I am particularly indebted to Mohammed Abouzaid and Ivan Smith for explaining their insights to me. Conversations with Denis Auroux, Ludmil Katzarkov, Davesh Maulik, Tim Perutz, and Nick Sheridan were also helpful. I could not have carried out this project without the generous support of the Simons Center for Geometry and Physics. Further support was received through NSF grant DMS-1005288, and through a Simons Investigator grant from the Simons Foundation. Finally, I would like to thank the referees for doing a remarkably thorough job on the manuscript.

\section{Families of objects\label{sec:theory}}

Suppose that we are given an $A_\infty$-category $A$, and a class in its degree $1$ Hochschild cohomology  $\mathit{HH}^1(A,A)$. This class determines one in $H^1(\mathit{hom}_A(X,X))$ for every object $X$, hence an infinitesimal first order deformation of $X$. Deformations coming from Hochschild cohomology have additional properties, for instance (in characteristic $0$) they can be extended to arbitrarily high orders in a formal parameter. However, instead of looking at the infinitesimal theory, we want to consider global deformations. For the sake of illustration, take the parameter space to be the affine line. One then looks for families $\XX = \{\XX_s\}$ depending on one algebraic variable $s$, whose fibre at the origin is fixed, $\XX_0 \iso X$, and whose first order deformation behaviour at any value of $s$ is the element of $H^1(\mathit{hom}_A(\XX_s,\XX_s))$ induced by our Hochschild class. We are mainly concerned with the uniqueness of such families. This, while not totally straightforward, turns out to be much easier than existence issues. An elementary parallel would be the question of integrating a vector field. Indeed, one might think of the original Hochschild cohomology class as determining a vector field on the ``moduli spaces of objects in $A$'' (making this rigorous requires machinery far beyond that deployed here; interested readers are referred to \cite{toen-vaquie07}).

%
%

We work over an algebraically closed field $R$ of characteristic $0$. While the condition of algebraic closedness is perhaps mostly for the sake of familiarity, the restriction on the characteristic is crucial, since we will be using differentiation (and in particular Lemma \ref{th:locally-free}). Sign conventions for $A_\infty$-algebras and associated structures usually follow \cite{seidel04}. Those for twisted complexes are specifically as in \cite[Remark 3.26]{seidel04}. All categories are assumed to be small.

\subsection{$A_\infty$-categories}
Fix an $A_\infty$-category $A$ over $R$. This is assumed to be $\bZ$-graded, strictly unital, and proper. The units (identity endomorphisms) are denoted by $e_X \in \mathit{hom}_A^0(X,X)$. Write $H(A)$ for the associated cohomology level category, and $H^0(A)$ for the version where only morphisms of degree $0$ are allowed. Properness means that the morphisms in $H(A)$ are graded vector spaces of finite total dimension. Two objects of $A$ are called quasi-isomorphic if they become isomorphic in $H^0(A)$.

There are various canonical formal enlargements of $A$. Possibly the simplest one is the $A_\infty$-category of twisted complexes $A^{\mathit{tw}}$, introduced in \cite{bondal-kapranov91, kontsevich94}. One can carry out this enlargement in two steps. First, consider the additive envelope $A^{\oplus}$, whose objects are formal expressions
\begin{equation} \label{eq:formal-sum}
X = \bigoplus_{i \in I} F^i \otimes X^i[-\sigma^i]
\end{equation}
where $I$ is a finite set, the $X^i$ are objects of $A$, formally shifted by degrees $\sigma^i \in \bZ$, and the $F^i$ are finite-dimensional vector spaces. Morphisms in $A^{\oplus}$ can be thought of as matrices, whose entries are morphisms in $A$ tensored with maps of vector spaces. Correspondingly, the $A_\infty$-products combine those of $A$ with composition of linear maps and matrix multiplication (with auxiliary signs due to the shifts). In the second step, one defines a twisted complex to be an object $X \in \mathit{Ob}\,A^{\oplus}$ equipped with an additional differential. This differential $\delta_X \in \mathit{hom}_{A^{\oplus}}^1(X,X)$ is an endomorphism which is strictly decreasing with respect to some filtration of $X$ by sub-objects, and which satisfies the generalized Maurer-Cartan equation
\begin{equation} \label{eq:generalized-maurer-cartan}
\mu^1_{A^{\oplus}}(\delta_X) + \mu^2_{A^{\oplus}}(\delta_X,\delta_X) + \cdots = 0.
\end{equation}
$A^{\mathit{tw}}$ is an $A_\infty$-category with the same general properties as $A$, and which contains $A$ as a full subcategory. It is closed under shifts and mapping cones, and is characterized up to quasi-equivalence as the minimal enlargement with this property. In our formulation, it also admits a canonical operation of tensoring a given object with a finite-dimensional vector space.

\begin{remark}
Suppose that we allow only trivial one-dimensional spaces $F^i = R$. The resulting objects, which can be written more concisely as $X = \bigoplus_{i \in I} X^i[-\sigma^i]$, form a full $A_\infty$-subcategory, which is quasi-equivalent to all of $A^\oplus$ (and if one equips them with a differential, the same holds for $A^{\mathit{tw}}$). We prefer the form \eqref{eq:formal-sum} since it is better suited to later generalizations.
\end{remark}

There is a different approach to formal enlargements, through $A_\infty$-modules \cite{keller99}. Write $C$ for the differential graded category of complexes of $R$-vector spaces whose cohomology is of finite total dimension. A (right) $A$-module with finite cohomology is an $A_\infty$-functor $A^{\mathit{opp}} \rightarrow C$. Concretely, such a module $M$ assigns to each $X \in \mathit{Ob}\,A$ a graded vector space $M(X)$, together with structure maps
\begin{equation} \label{eq:module}
\begin{aligned}
& \mu^1_M: M(X_0) \longrightarrow M(X_0)[1], \\
& \mu^2_M: M(X_1) \otimes \mathit{hom}_A(X_0,X_1) \longrightarrow M(X_0), \\
& \mu^3_M: M(X_2) \otimes \mathit{hom}_A(X_1,X_2) \otimes \mathit{hom}_A(X_0,X_1) \longrightarrow M(X_0)[-1], \\
& \dots
\end{aligned}
\end{equation}
satisfying the $A_\infty$-module equations (\cite{keller99} or \cite[Section 1j]{seidel04}), and such that $(M(X),\mu^1_M) \in \mathit{Ob}\,C$ for all $X$. We require $M$ to be strictly unital, which means that $\mu^2_M(m,e_{X_0}) = m$, and $\mu^d_M(m,a_{d-1},\dots,a_1) = 0$ whenever $d \geq 3$ and one of the $a_i$ is a unit. Such modules form an $A_\infty$-category $A^{\mathit{mod}}$ (this is in fact a differential graded category, but we prefer to view it as an $A_\infty$-category with trivial higher order products, which entails slightly different sign conventions). A morphism $b \in \mathit{hom}_{A^{\mathit{mod}}}(M_0,M_1)$ consists of
\begin{equation} \label{eq:module-map}
\begin{aligned}
& b^1: M_0(X_0) \longrightarrow M_1(X_0)[|b|], \\
& b^2: M_0(X_1) \otimes \mathit{hom}_A(X_0,X_1) \longrightarrow M_1(X_0)[|b|-1], \\
& \dots
\end{aligned}
\end{equation}
Again, we require strict unitality, which means that $b^d(m,a_{d-1},\dots,a_1) = 0$ whenever one of the $a_i$ is a unit.

The Yoneda embedding (\cite{fukaya01b} or \cite[Section 1l]{seidel04}) is a canonical $A_\infty$-functor $A \rightarrow A^{\mathit{mod}}$. On objects, it maps $Y$ to the module $Y^{\mathit{yon}}$ with $Y^{\mathit{yon}}(X) = \mathit{hom}_A(X,Y)$ and $\mu^d_{Y^{\mathit{yon}}} = \mu^d_A$. The first level map on morphisms is
\begin{equation} \label{eq:yoneda1}
\begin{aligned}
& \mathit{hom}_A(Y_0,Y_1) \longrightarrow \mathit{hom}_{A^{\mathit{mod}}}(Y_0^{\mathit{yon}},Y_1^{\mathit{yon}}), \\
& a \longmapsto a^{\mathit{yon}}, \quad a^{\mathit{yon},d}(a_d,\dots,a_1) = \mu^{d+1}_A(a,a_d,\dots,a_1).
\end{aligned}
\end{equation}

\begin{lemma} \label{th:yoneda}
The map \eqref{eq:yoneda1} is a quasi-isomorphism.
\end{lemma}

\begin{proof}
Consider the map in inverse direction, taking a module homomorphism $b$ to the element $a = b^1(e_{Y_0})$. Composing the two in one way yields the identity map on $\mathit{hom}_A(Y_0,Y_1)$. The other composition is chain homotopic to the identity: an explicit homotopy is
\begin{equation} \label{eq:h-homotopy}
\begin{aligned}
& h: \mathit{hom}_{A^{\mathit{mod}}}(Y_0^{\mathit{yon}},Y_1^{\mathit{yon}}) \longrightarrow \mathit{hom}_{A^{\mathit{mod}}}(Y_0^{\mathit{yon}},Y_1^{\mathit{yon}})[-1], \\
& h(b)^d(a_d,\dots,a_1) = b^{d+1}(e_{Y_0},a_d,\dots,a_1).
\end{aligned}
\end{equation}
\end{proof}

In other words, the Yoneda embedding is cohomologically full and faithful. Moreover, it canonically extends to a cohomologically full and faithful $A_\infty$-functor $A^{\mathit{tw}} \rightarrow A^{\mathit{mod}}$. The easiest way to see this is to think of it as the composition
\begin{equation} \label{eq:yoneda-restriction}
A^{\mathit{tw}} \longrightarrow (A^{\mathit{tw}})^{\mathit{mod}} \longrightarrow A^{\mathit{mod}}.
\end{equation}
where the first arrow is the Yoneda embedding for $A^{\mathit{tw}}$, and the second is restriction of modules from $A^{\mathit{tw}}$ to $A$. Since objects of $A$ generate $A^{\mathit{tw}}$ by definition, the restriction functor is a quasi-equivalence, so \eqref{eq:yoneda-restriction} is again cohomologically full and faithful.

\subsection{Idempotent splittings\label{subsec:splittings}}
Suppose that we have an object $Y$ of $A$ together with an endomorphism in the category $H^0(A)$ which is idempotent. One can always lift it to a {\em homotopy idempotent}, which is a sequence $p = \{p^d\}$ of elements $p^d \in \mathit{hom}_A(Y,Y)^{1-d}$ ($d \geq 1$) satisfying the equations
\begin{equation} \label{eq:homotopy-idempotent}
\sum_r \sum_{k_1 + \cdots + k_r = d} \mu_A^r(p^{k_r},\dots,p^{k_1}) =
\begin{cases} p^{d-1} & \text{$d$ even,} \\ 0 & \text{$d$ odd} \end{cases}
\end{equation}
for any $d \geq 1$, and such that $p^1$ represents our original idempotent. This is proved in \cite[Section 4]{seidel04}, but it is maybe useful to summarize the argument in more elementary language. The choice of $p^1,p^2$ is straightforward, and the remaining process is inductive. Suppose that $p^1,\dots,p^{d-1}$ have been chosen satisfying the respective equations. Take the sum of all the terms on left hand side of \eqref{eq:homotopy-idempotent} which have $r \geq 2$. These give rise to a cocycle of bidegree $(d,2-d)$ in the following periodic complex of graded vector spaces:
\begin{multline} \label{eq:trivial-hh}
\cdots \rightarrow \mathit{Hom}_{H(A)}(Y,Y) \xrightarrow{[a] \mapsto [(-1)^{|a|}\mu^2_A(p^1,a) - \mu^2_A(a,p^1)] } \mathit{Hom}_{H(A)}(Y,Y) \longrightarrow \\ \xrightarrow{[a] \mapsto [(-1)^{|a|+1}\mu^2_A(p^1,a) - \mu^2_A(a,p^1) + a]} \mathit{Hom}_{H(A)}(Y,Y) \rightarrow \cdots \qquad
\end{multline}
Since that complex is acyclic, we can modify $p^{d-1}$ by adding a $\mu^1_A$-cocycle, so that the same sum of terms represents the zero class in $\mathit{Hom}_{H(A)}^{2-d}(Y,Y)$, and then choose $p^d$ so that \eqref{eq:homotopy-idempotent} holds.

The homotopy idempotent can then be used to define an $A_\infty$-module $M = (Y,p)^{\mathit{yon}}$, which in the category $H^0(A^{\mathit{mod}})$ is the direct summand of $Y^{\mathit{yon}}$ associated to the Yoneda image of $[p^1]$ (hence, independent of the choice of the homotopy idempotent up to quasi-isomorphism). It consists of the spaces
\begin{equation} \label{eq:q-splitting}
M(X) = \mathit{hom}_A(X,Y)[q]
\end{equation}
where $q$ is a formal variable of degree $-1$, and has differential
\begin{equation}
\mu^1_M(aq^j) = \sum_{r \geq 0} \sum_{k_1 + \cdots + k_r \leq j} \mu^{r+1}_A(p^{k_r},\dots,p^{k_1},a)\, q^{j-k_1-\cdots-k_r}
+ \begin{cases} 0 & \text{$j$ even,} \\ aq^{j-1} & \text{$j$ odd.}
\end{cases}
\end{equation}
We refer to \cite[Section 4]{seidel04} for full details, including the definition of the remaining maps $\mu^d_M$. By the same kind of argument as in \eqref{eq:yoneda-restriction}, idempotent summands of objects in $A^{\mathit{tw}}$ (or indeed in $A^{\mathit{mod}}$) can also be represented in $A^{\mathit{mod}}$. One defines $A^{\mathit{perf}}$, the category of {\em perfect modules}, to be the full subcategory of $A^{\mathit{mod}}$ consisting of all objects that are quasi-isomorphic to an idempotent summand of an object in $A^{\mathit{tw}}$. It is easy to see that $A^{\mathit{perf}}$ is again proper.

\subsection{Hochschild cohomology}
Let $\mathit{CC}(A,A)$ be the (reduced) Hochschild complex of $A$, and $\mathit{HH}(A,A)$ its cohomology. A Hochschild cochain $g$ is a sequence of multilinear maps
\begin{equation}
g^d: \mathit{hom}_A(X_{d-1},X_d) \otimes \cdots \otimes \mathit{hom}_A(X_0,X_1) \longrightarrow
\mathit{hom}_A(X_0,X_d)[|g|-d],
\end{equation}
$d \geq 0$, which vanish if one of the inputs is an identity morphism. The Hochschild differential is
\begin{equation}
\begin{aligned}
(\partial g)^d(a_d,\dots,a_1) = \sum_{i,j} (-1)^{(|g|-1)(|a_1|+\cdots+|a_i|-i)}
\mu^{d-j+1}_A(a_d,\dots,g^j(a_{i+j},\dots,a_{i+1}),\dots,a_1) \\
+ \sum_{i,j} (-1)^{|g|+|a_1|+\cdots+|a_i|-i} g^{d-j+1}(a_d,\dots,\mu_A^j(a_{i+j},\dots,a_{i+1}),\dots,a_1).
\end{aligned}
\end{equation}
We need to remind the reader briefly of the (partial) functoriality properties of Hochschild cohomology. Let $G: A \rightarrow \tilde{A}$ be a (strictly unital) $A_\infty$-functor. Then there are canonical chain maps
\begin{equation} \label{eq:functoriality}
\mathit{CC}(A,A) \xrightarrow{G_*} \mathit{CC}(A,\tilde{A}) \xleftarrow{G^*} \mathit{CC}(\tilde{A},\tilde{A}),
\end{equation}
where the middle term is the Hochschild cochain complex of $A$ with coefficients in the $A$-bimodule $\tilde{A}$. If $G$ is cohomologically full and faithful, the left hand map is a quasi-isomorphism. Less obviously (it is a form of Morita invariance), if $G$ is a quasi-equivalence, the right hand map is also a quasi-isomorphism. Informally one can think of these two maps as follows. If we deform the $A_\infty$-structure on either $A$ or $\tilde{A}$ infinitesimally, there will be a term measuring the failure of $G$ to be an $A_\infty$-functor for the deformed structure, which is the image of the corresponding deformation classes in $\mathit{HH}(A,\tilde{A})$.

We also need to know about the behaviour of Hochschild cohomology under formal enlargement. Thinking again in terms of deformation theory, one expects deformations of $A$ to induce ones of $A^{\mathit{tw}}$. This can indeed be made rigorous, leading to a canonical map
\begin{equation} \label{eq:twisted-gamma}
\Gamma^{\mathit{tw}}: \mathit{CC}(A,A) \longrightarrow \mathit{CC}(A^{\mathit{tw}},A^{\mathit{tw}}).
\end{equation}
To make this (partially) more explicit, take a Hochschild cochain $g$, and extend it in the obvious way to a cochain $g^{\oplus}$ on $A^{\oplus}$. Then the image $g^{\mathit{tw}} = \Gamma^{\mathit{tw}}(g)$ has leading term
\begin{equation} \label{eq:twisted-gamma-0}
g^{\mathit{tw},0} = g^{\oplus,0} + g^{\oplus,1}(\delta_X) + g^{\oplus,2}(\delta_X,\delta_X) + \cdots \in \mathit{hom}_{A^{\mathit{tw}}}^{|g|}(X,X).
\end{equation}
Maybe more obviously, restriction to the full subcategory $A \subset A^{\mathit{tw}}$ yields a map in reverse direction to \eqref{eq:twisted-gamma}. These two maps are inverse quasi-isomorphisms (strict inverses in one order, and inverses up to homotopy in the other), which is one form of the derived invariance of Hochschild cohomology.

Let's look at the analogous question for $A_\infty$-modules. As before, a deformation of $A$ induces a deformation of $A^{\mathit{mod}}$, but one which remains within the class of dg categories with curvature. In fact, the product $\mu^2_{A^{\mathit{mod}}}$ does not actually change, since its definition does not involve $\mu^*_A$. Concretely, this means that we have a map
\begin{equation} \label{eq:gamma-for-modules}
\Gamma^{\mathit{mod}}: \mathit{CC}(A,A) \longrightarrow \mathit{CC}(A^{\mathit{mod}},A^{\mathit{mod}})
\end{equation}
such that $g^{\mathit{mod}} = \Gamma^{\mathit{mod}}(g)$ has only two nontrivial components
\begin{equation} \label{eq:gamma-mod}
\begin{aligned}
& g^{\mathit{mod},0} \in \mathit{hom}_{A^{\mathit{mod}}}^{|g|}(M_0,M_0), \\
& g^{\mathit{mod},1}: \mathit{hom}_{A^{\mathit{mod}}}(M_0,M_1) \longrightarrow \mathit{hom}_{A^{\mathit{mod}}}(M_0,M_1)[|g|-1].
\end{aligned}
\end{equation}
These are given by
\begin{equation} \label{eq:gamma-formula}
\begin{aligned}
(g^{\mathit{mod},0})^d(m,a_{d-1},\dots,a_1) & = \sum_{ij} (-1)^\ast \, \mu_{M_0}^{d-j+1}(m,\dots, g^j(a_{i+j},\dots,a_{i+1}),\dots,a_1), \\
(g^{\mathit{mod},1})(b)^d(m,a_{d-1},\dots,a_1) & = \sum_{ij} (-1)^\ast \, b^{d-j+1}(m,\dots,g^j(a_{i+j},\dots,a_{i+1}),\dots,a_1), \\
& \ast = (|g|-1)(|a_{i+1}|+\cdots+|a_{d-1}|+|m|+d-i-1)+|g|.
\end{aligned}
\end{equation}


There is a restriction map from the Hochschild complex of $A^{\mathit{mod}}$ to that of its full subcategory $A^{\mathit{perf}}$. One can accordingly restrict \eqref{eq:gamma-for-modules} and get a map $\Gamma^{\mathit{perf}}$. Another manifestation of derived invariance of Hochschild cohomology says that $\Gamma^{\mathit{perf}}$ is a quasi-isomorphism. Next, $A^{\mathit{perf}}$ contains a full subcategory quasi-isomorphic to $A^{\mathit{tw}}$, and by restriction and the argument from \eqref{eq:functoriality}, one gets a further map from the Hochschild cohomology of $A^{\mathit{perf}}$ to that of $A^{\mathit{tw}}$. The situation can be summarized in the commutative diagram
\begin{equation}
\xymatrix{
\mathit{HH}(A,A)
\ar@/_1pc/[drr]^-{\Gamma^{\mathit{tw}}}_{\iso} \ar@/_3pc/[ddrr]^-{\Gamma^{\mathit{perf}}}_{\iso} \ar@/_6pc/[dddrr]^-{\Gamma^{\mathit{mod}}} \ar[rr]^-{\mathrm{id}} && \mathit{HH}(A,A) \\
&& \mathit{HH}(A^{\mathit{tw}},A^{\mathit{tw}}) \ar[u]_{\iso} \\
&& \mathit{HH}(A^{\mathit{perf}},A^{\mathit{perf}}) \ar[u]_{\iso} \\
&& \mathit{HH}(A^{\mathit{mod}},A^{\mathit{mod}}) \ar[u]
}
\end{equation}
where the vertical arrows are restriction maps. Other than derived invariance (which we will not discuss further, but see \cite[Section 5.4]{keller06}), the only nontrivial point in this diagram is that the map $\Gamma^{\mathit{mod}}$ is compatible with restriction. To understand that, we have to look at \eqref{eq:functoriality} for the Yoneda embedding:
\begin{equation} \label{eq:yoneda-functoriality}
\mathit{CC}(A,A) \longrightarrow \mathit{CC}(A,A^{\mathit{mod}}) \longleftarrow \mathit{CC}(A^{\mathit{mod}},A^{\mathit{mod}}).
\end{equation}
Take $g \in \mathit{CC}(A,A)$, and consider its image under the first map in \eqref{eq:yoneda-functoriality}. For simplicity let's look only at the constant term of this, which consists of an endomorphism $b \in \mathit{hom}_{A^{\mathit{mod}}}(Y^{\mathit{yon}},Y^{\mathit{yon}})$ for each $Y$, given by
\begin{equation}
b^d(a_d,\dots,a_1) = \mu^{d+1}_A(g^0,a_d,\dots,a_1). \label{eq:1-cc}
\end{equation}
On the other hand, we can take $g^{\mathit{mod}} \in \mathit{CC}(A^{\mathit{mod}},A^{\mathit{mod}})$ and pull it back to $\mathit{CC}(A,A^{\mathit{mod}})$ as in the second map in \eqref{eq:yoneda-functoriality}, which leads to another cochain $\mathit{CC}(A,A^{\mathit{mod}})$ with constant term
\begin{equation} \label{eq:2-cc}
\begin{aligned}
\tilde{b}^d(a_d,\dots,a_1) = & -\sum_{i+j<d} (-1)^\ast \mu^{d-j+1}_A(a_d,\dots,g^j(a_{i+j},\dots,a_{i+1}),\dots,a_1), \\ & \ast = (|g|-1)(|a_{i+1}|+\cdots+|a_d|+d-i).
\end{aligned}
\end{equation}
Assuming that $g$ is a Hochschild cocycle, we can write
\begin{equation}
\begin{aligned} &
\tilde{b}^d(a,a_{d-1},\dots,a_1) - b^d(a,a_{d-1},\dots,a_1) = \\ &
\qquad \sum_{i<d} (-1)^{(|g|-1)(|a_{i+1}| + \cdots + |a_d| + d - i)} \mu^{i+1}_A(g^{d-i}(a_d,\dots,a_{i+1}),\dots,a_1) \\  & \qquad + \sum_{i,j} (-1)^{|g| + |a_1| + \cdots + |a_i| - i + (|g|-1)(|a_1|+\cdots+|a_d|-d)} g^{d-j+1}(a_d,\dots,\mu_A^j(a_{i+j},\dots,a_{i+1}),\dots,a_1).
\end{aligned}
\end{equation}
This difference is the coboundary of another endomorphism of $Y^{\mathit{yon}}$ of degree $|g|-1$, given by $(a_d,\dots,a_1) \mapsto (-1)^{|g|(|a_1|+ \cdots + |a_d|-d)+1}g^d(a_d,\dots,a_1)$. The general computation is similar.

\subsection{Bimodules\label{subsec:bimodules}}
$A_\infty$-bimodules have already made a brief appearance before, but we will now consider them in a little more detail. Let $A$ and $\tilde{A}$ be $A_\infty$-categories over $R$, with the usual unitality and properness assumptions. An {\em $(A,\tilde{A})$-bimodule with finite cohomology} $P$ assigns to any pair of objects $(X,\tilde{X}) \in \mathit{Ob}\, A \times \mathit{Ob}\,\tilde{A}$ a graded vector space $P(\tilde{X},X)$, which comes with structure maps
\begin{equation} \label{eq:bimodule-structure}
\begin{aligned}
& \mu_P^{s|1|t}: \mathit{hom}_A(X_{s-1},X_s) \otimes \cdots \otimes \mathit{hom}_A(X_0,X_1) \otimes P(\tilde{X}_t,X_0) \\ & \qquad \qquad \qquad \qquad \otimes \mathit{hom}_{\tilde{A}}(\tilde{X}_{t-1},\tilde{X}_t) \otimes \cdots \otimes \mathit{hom}_{\tilde{A}}(\tilde{X}_0,\tilde{X}_1) \longrightarrow P(\tilde{X}_0,X_s)[1-s-t]
\end{aligned}
\end{equation}
for all $s,t \geq 0$, satisfying analogues of the $A_\infty$-module equation. We assume that the cohomology groups $H(P(\tilde{X},X),\mu_P^{0|1|0})$ are of finite total dimension, and also impose strict unitality properties. $A_\infty$-bimodules form an $A_\infty$-category $(A,\tilde{A})^{\mathit{mod}}$ (just like in the case of modules, this has vanishing higher order compositions, hence could be considered a dg category). In the special case $A = \tilde{A}$, we use the terminology $A$-bimodule instead of $(A,A)$-bimodule.

\begin{example}
The standard example is the diagonal $A$-bimodule, which has $P(X,Y) = \mathit{hom}_A(X,Y)$ with $\mu^{s|1|t}_P = \mu^{s+1+t}_A$ (this also indirectly illustrates our sign conventions). We usually just write $P = A$. Another example is the one which appeared in \eqref{eq:functoriality}: if $G: A \rightarrow \tilde{A}$ is an $A_\infty$-functor, one can define an $A$-bimodule $P(X,Y) = \mathit{hom}_{\tilde{A}}(GX,GY)$, with structure maps similar to the diagonal one but plugging in multiple copies of $G$. We again denote this by $P = \tilde{A}$, but always make sure to mention that the pullback to an $A$-bimodule is intended.
\end{example}

One traditional use of bimodules is as ``kernels'' defining ``convolution functors'' between categories of modules. The tensor product of an $A$-module $M$ and an $(A,\tilde{A})$-bimodule $P$ is an $\tilde{A}$-module $\tilde{M} = M \otimes_A P$, given by a bar construction
\begin{equation} \label{eq:ainfty-tensor}
\tilde{M}(\tilde{X}) = \bigoplus M(X_r) \otimes \mathit{hom}_A(X_{r-1},X_r)[1] \otimes \cdots \mathit{hom}_A(X_0,X_1)[1] \otimes P(\tilde{X},X_0),
\end{equation}
where the sum is over all $r \geq 0$ and objects $(X_0,\dots,X_r)$. The induced differential is
\begin{equation}
\begin{aligned}
& \mu^1_{\tilde{M}}(m \otimes a_r \otimes \cdots \otimes a_1 \otimes p) = \\
& \qquad \sum_i (-1)^{|p|+|a_1|+\cdots+|a_i|-i} \mu^{r-i+1}_M(m,a_r,\dots,a_{i+1}) \otimes a_i \otimes \cdots \otimes a_1 \otimes p \\ &
\qquad + \sum_{i,j} (-1)^{|p|+|a_1|+\cdots+|a_i|-i} m \otimes \cdots \otimes \mu_A^j(a_{i+j},\dots,a_{i+1}) \otimes a_i \otimes \cdots \otimes a_1 \otimes p \\ &
\qquad + \sum_i m \otimes \cdots \otimes a_{i+1} \otimes \mu_P^{i|1|0}(a_i,\dots,a_1,p).
\end{aligned}
\end{equation}
There are similar formulae for $\mu^d_{\tilde{M}}$, $d > 1$, which are in fact simpler since they involve only $\mu_P^{*|1|d-1}$. Naturality of the tensor product with a fixed $P$ is expressed by a map $\mathit{hom}_{A^{\mathit{mod}}}(M_0,M_1) \rightarrow \mathit{hom}_{\tilde{A}^{\mathit{mod}}}(\tilde{M}_0,\tilde{M}_1)$ for $\tilde{M}_k = M_k \otimes_A P$. This takes $b$ to $b \otimes_A e_P$, given by
\begin{equation}
\begin{aligned}
& (b \otimes_A e_P)^1(m_0 \otimes a_r \cdots \otimes a_1 \otimes p) = \\
& \qquad \qquad \qquad \sum_i (-1)^{|p|+|a_1|+\cdots+|a_i|-i} b^{r-i+1}(m_0,a_r,\dots,a_{i+1}) \otimes a_i \otimes \cdots \otimes p,
\end{aligned}
\end{equation}
with vanishing higher order terms. The resulting convolution $A_\infty$-functor (in fact a dg functor, since it has no higher order terms) between module categories will be denoted by $K_P$. So far, we have skirted the issue of whether $M \otimes_A P$ is really an object of $A^{\mathit{mod}}$ as defined, meaning whether it satisfies the cohomological finiteness condition. This fails in general, but it will hold if $M$ is perfect (based on the fact that $X^{\mathit{yon}} \otimes_A P$ is quasi-isomorphic to the $\tilde{A}$-module $P(\cdot,X)$, which has finite cohomology by assumption on $P$). Hence, we always get a functor $K_P: A^{\mathit{perf}} \rightarrow \tilde{A}^{\mathit{mod}}$. If $P$ is itself {\em right perfect}, which means that $P(\cdot,X)$ is itself a perfect $\tilde{A}$-module for any $X$, then $K_P$ takes $A^{\mathit{perf}}$ to $\tilde{A}^{\mathit{perf}}$.

Let's suppose, for the sake of concreteness, that $P$ is right perfect. There are natural chain maps
\begin{equation} \label{eq:one-sided-deformation}
\mathit{CC}(A,A) \longrightarrow \mathit{hom}_{(A,\tilde{A})^{\mathit{mod}}}(P,P) \longleftarrow
\mathit{CC}(\tilde{A},\tilde{A}).
\end{equation}
Informally, one can think of these as follows. Given an infinitesimal deformation of either $A$ or $\tilde{A}$, they measure the failure of $P$ to remain a bimodule with respect to the deformed structure. This may remind the reader of \eqref{eq:functoriality}, and indeed one can define a map (represented by the dashed arrow below) which fits into a homotopy commutative diagram
\begin{equation} \label{eq:grrr}
\xymatrix{
\mathit{CC}(A,A) \ar[d]^{\Gamma^{\mathit{perf}}} \ar[r] & \mathit{hom}_{(A,\tilde{A})^{\mathit{mod}}}(P,P) \ar@{-->}[d] & \ar[l] \ar[d]^{\Gamma^{\mathit{perf}}} \mathit{CC}(\tilde{A},\tilde{A}) \\
\mathit{CC}(A^{\mathit{perf}},A^{\mathit{perf}}) \ar[r]^-{(K_P)_*} &
\mathit{CC}(A^{\mathit{perf}},\tilde{A}^{\mathit{perf}}) & \ar[l]_-{(K_P)^*}
\mathit{CC}(\tilde{A}^{\mathit{perf}},\tilde{A}^{\mathit{perf}}).
}
\end{equation}

\subsection{Connections}
Fix a smooth affine algebraic curve $\SS$ over $R$, with $\RR = R[\SS]$ its ring of functions. Recall that $\RR$-modules correspond to quasi-coherent sheaves on $\SS$; finitely generated modules to coherent sheaves; projective modules to vector bundles; and rank $1$ projective modules to line bundles. We will mostly use the algebraic language as it is more elementary, but the reader is encouraged to keep the geometric viewpoint in mind.

Denote by $\Omega^1_\RR$ the module of K{\"a}hler differentials, which is a rank $1$ projective module, and comes with its canonical derivation $d: \RR \rightarrow \Omega^1_\RR$. A connection on an $\RR$-module $\FF$ is a map $\nabla_\FF: \FF \rightarrow \Omega^1_\RR \otimes \FF$ satisfying the Leibniz identity with respect to $d$. An equivalent viewpoint is as follows. For any $\FF$ there is a canonical short exact sequence of modules
\begin{equation} \label{eq:pullback-sequence}
0 \rightarrow \Omega^1_\RR \otimes \FF \longrightarrow J^1(\FF) \longrightarrow \FF \rightarrow 0,
\end{equation}
where $\JJ^1(\FF)$ is the one-jet module \cite{atiyah57}. Connections correspond to splittings of this sequence. $\FF$ admits a connection if and only if its Atiyah class $\mathit{At}(\FF) \in \mathit{Ext}^1_\RR(\FF, \Omega^1_\RR \otimes \FF)$, which is the extension class of \eqref{eq:pullback-sequence}, vanishes. In particular, projective modules always admit connections (this can also be proved directly).
In the other direction, one has \cite[Lecture 2]{bernstein-dmodules}:

\begin{lemma} \label{th:locally-free}
Let $\FF$ be a finitely generated $\RR$-module which admits a connection. Then it is necessarily projective. \qed
\end{lemma}

Let $R(\SS)$ be the field of rational functions on our curve, which is the quotient field of $\RR$. Tautologically, all projective modules of the same rank become isomorphic after tensoring with $R(\SS)$. However, it is intuitively clear that the situation for modules with connections is quite different, which is indeed confirmed by:

\begin{lemma} \label{th:rational-section}
Let $\FF$ be a finitely generated $\RR$-module with a connection. Then, if $f \in \FF \otimes_\RR R(\SS)$ is a rational solution of $\nabla_\FF(f) = 0$, it automatically lies in $\FF$ itself.
\end{lemma}

\begin{proof}
Suppose that $f \neq 0$, and choose a point on our curve. Take a function $r \in \RR$ which vanishes (to order $1$) at this point. There is a unique $m \in \bZ$ such that $r^m f$ takes a nonzero finite value at our point. By assumption
\begin{equation}
\nabla_\FF( r^m f) = m (r^{-1} dr) r^m f.
\end{equation}
Since connections are local operations, the left hand side is regular locally around our point; but the right hand one has a pole, unless $m = 0$.
\end{proof}

One can extend the notion of connection to the derived category (see for instance \cite{markarian09,huybrechts-thomas10}), as follows. Any complex of modules $\FF$ sits in a short exact sequence generalizing \eqref{eq:pullback-sequence}. Define a {\em homotopy connection} to be a splitting $\FF \rightarrow J^1(\FF)$ in the derived category. If this exists, it induces connections on all the cohomology modules $\HH^i(\FF)$. The obstruction to the existence of a homotopy connection is the morphism completing the short exact sequence to an exact triangle in the derived category, which we again call the Atiyah class $\mathit{At}(\FF) \in \mathit{Hom}_{D(\RR)}(\FF, \Omega^1_\RR \otimes \FF[1])$. To get a more hands-on description, assume that each of the modules $\FF^i$ forming our complex already comes equipped with a connection. We call this a {\em pre-connection} on $\FF$, and denote it by $\pabla_\FF$. Its failure to commute with the differential $d_\FF$ gives rise to a chain map
\begin{equation} \label{eq:cochain-atiyah}
\begin{aligned}
& \mathit{at}(\pabla_\FF): \FF \longrightarrow \Omega^1_\RR \otimes \FF[1], \\
& \mathit{at}(\pabla_\FF)(\ff) = (\mathit{id}_{\Omega^1_\RR} \otimes d_\FF)(\pabla_{\FF} \ff) - \pabla_{\FF}(d_\FF \ff),
\end{aligned}
\end{equation}
which is the boundary homomorphism for \eqref{eq:pullback-sequence}, hence represents $\mathit{At}(\FF)$. We call $\mathit{at}(\pabla_\FF)$ the {\em Atiyah cocycle} of $\pabla_\FF$. Suppose that $\FF_0,\FF_1$ are complexes of modules, with $\FF_0$ consisting of projective modules. Then we have a short exact sequence \cite[Theorem VI.10.11]{dold}
\begin{equation} \label{eq:dold}
0 \rightarrow \mathit{Ext}^1(H(\FF_0),H(\FF_1))[-1] \longrightarrow H(\mathit{hom}(\FF_0,\FF_1)) \longrightarrow \mathit{Hom}(H(\FF_0),H(\FF_1)) \rightarrow 0.
\end{equation}
In particular:

\begin{lemma} \label{th:proj}
If $\FF_0$ is a complex of projective $\RR$-modules, and $\FF_1$ an acyclic complex of $\RR$-modules, then $\mathit{hom}(\FF_0,\FF_1)$ is again acyclic. \qed
\end{lemma}


In the terminology of \cite{spaltenstein88}, this means that (unbounded) complexes of projective $\RR$-modules are $K$-projective.

\begin{lemma} \label{th:homotopy-idempotent}
Let $\FF$ be a complex of projective $\RR$-modules, and $c: \FF \rightarrow \FF$ a chain map which is chain homotopic to its square, and which induces an isomorphism on cohomology. Then $c$ is chain homotopic the identity.
\end{lemma}

\begin{proof}
Consider \eqref{eq:dold} with $\FF_0 = \FF_1 = \FF$. Left composition with $c$ induces isomorphisms on the left and right terms of that sequence, hence also on the middle one. Since the identity and $c$ become homotopic after composition with $c$, they must have been homotopic in the first place.
\end{proof}

Lemma \ref{th:proj} has implications for homotopy connections, as follows. Take a complex of projective modules $\FF$, and assume that $\mathit{At}(\FF) = 0$. Then the map $\mathit{at}(\pabla_\FF)$ must be nullhomotopic, which means that one can modify the given pre-connection $\pabla_\FF$ so that it becomes compatible with the differential. The result should then be properly called a connection on the complex $\FF$, and we reserve the notation $\nabla_{\FF}$ for those.

\begin{remark} \label{th:geometric}
We want to briefly consider the extension of the theory to non-affine bases. Let $\SS$ be a smooth quasi-projective curve over $R$, and $\Omega^1_\SS$ the line bundle of differentials. For any quasi-coherent sheaf $\FF$ we have an analogue of \eqref{eq:pullback-sequence}, which can be used to define connections and Atiyah classes $\mathit{At}(\FF) \in \mathit{Ext}^1_{\SS}(\FF, \Omega^1_\SS \otimes \FF)$. The same holds for complexes, except that projective resolutions do not exist, and need to be replaced by injective quasi-coherent ones, which do. The analogue of Lemma \ref{th:proj} says the following: if $\FF_1$ is a complex of injective quasi-coherent sheaves, and $\FF_0$ an acyclic complex of quasi-coherent sheaves, the complex of vector spaces $\mathit{hom}(\FF_0,\FF_1)$ is again acyclic.

Let $\FF_0,\FF_1$ be complexes of injective quasi-coherent sheaves, and $\mathscr{hom}(\FF_0,\FF_1)$ the complex of hom sheaves. By choosing pre-connections as in \eqref{eq:cochain-atiyah}, one sees that $\mathit{At}(\mathscr{hom}(\FF_0, \FF_1))$ is the difference between left multiplication with $\mathit{At}(\FF_1)$ and right multiplication with $\mathit{At}(\FF_0)$ (compare \cite{markarian09}, which gives a similar formula for the Atiyah class of a tensor product). Now fix some $\gamma \in H^1(\SS, \Omega^1_\SS)$, and restrict attention to complexes $\FF$ with bounded coherent cohomology, and such that
\begin{equation} \label{eq:central}
\mathit{At}(\FF) = \gamma \otimes \mathit{id}_\FF.
\end{equation}
For any two such complexes, $\mathit{At}(\mathscr{hom}(\FF_0,\FF_1)) = 0$. By applying Lemma \ref{th:locally-free}, it then follows that the cohomology sheaves $\HH(\mathscr{hom}(\FF_0,\FF_1))$ are vector bundles. This is a simple illustration of the ideas that will play an important role later on (starting with Lemma \ref{th:its-projective}).
\end{remark}

\begin{remark} \label{th:higher-dimensions}
All we have said so far generalizes to higher-dimensional smooth varieties. The higher-dimensional analogue of Lemma \ref{th:locally-free} can be derived from the case of curves, which is indeed what happens in \cite{bernstein-dmodules}. The generalization of Lemma \ref{th:proj} to higher-dimensional affine varieties is \cite[Satz 3.1]{dold60} (there is a spectral sequence which replaces \eqref{eq:dold}, and which can be used to generalize Lemma \ref{th:homotopy-idempotent}). For injective quasi-coherent sheaves on affine quasi-projective varieties, one has \cite[Example 3.10]{krause05}. However, we have no real use for higher-dimensional bases in the present paper.
\end{remark}


\subsection{Families of objects\label{subsec:families}}
Take an $A_\infty$-category $A$ as before, and denote by $\AA$ the constant family of $A_\infty$-categories over $\SS$ with fibre $A$. This has the same objects as $A$, and its morphisms and $A_\infty$-structure are obtained by extending constants to $\RR$ in the obvious way:
\begin{equation}
\mathit{hom}_\AA(X_0,X_1) = \RR \otimes_R hom_A(X_0,X_1).
\end{equation}

Objects of $\AA$ can be thought of as constant families. To get more interesting ones, we again have to introduce formal enlargements. First, there is an additive enlargement $\AA^{\oplus}$, whose objects are finite formal sums
\begin{equation} \label{eq:xx}
\XX = \bigoplus_{i \in I} \FF^i \otimes X^i[-\sigma^i]
\end{equation}
where the $\FF^i$ are finitely generated projective $\RR$-modules, the $X^i$ are objects of $A$, and the $\sigma^i$ integers. The $A_\infty$-structure is extended to such sums exactly as for $A^{\oplus}$. One then defines a {\em family of twisted complexes} to be a pair $(\XX,\delta_\XX)$ where $\XX \in \mathit{Ob}\, \AA^{\oplus}$, and $\delta_\XX \in \mathit{hom}_{\AA^{\oplus}}^1(\XX,\XX)$ is strictly decreasing with respect to some filtration of \eqref{eq:formal-sum}, and satisfies the analogue of \eqref{eq:generalized-maurer-cartan}. The $A_\infty$-category $\AA^{\mathit{tw}}$ of twisted complexes obtained in this way allows the operations of shifts, mapping cones, and tensoring with a finitely generated projective $\RR$-module.

Let $\CC$ be the dg category of complexes of projective $\RR$-modules with bounded finitely generated cohomology. A {\em family of $A$-modules with finite cohomology} is an $A_\infty$-functor $A^{\mathit{opp}} \rightarrow \CC$. Concretely, such a family is given by a graded projective $\RR$-module $\MM(X)$ for each $X \in \mathit{Ob}(A)$, with structure maps as in \eqref{eq:module} but which are $\RR$-linear, hence extend to
\begin{equation}
\mu^d_\MM: \MM(X_d) \otimes \mathit{hom}_\AA(X_{d-1},X) \otimes \cdots \otimes \mathit{hom}_\AA(X_0,X_1) \longrightarrow \MM(X_0)[2-d].
\end{equation}
We impose the same strict unitality conditions as before. Such modules form an $A_\infty$-category over $\RR$, denoted by $\AA^{\mathit{mod}}$. The following statement is well-known in the case of $A_\infty$-categories over a field, see for instance \cite[Section 4]{keller99}, but slightly less so in the current framework:

\begin{lemma} \label{th:acyclic-modules}
If the chain complexes $(\MM(X),\mu^1_\MM)$ are acyclic for all $X$, $\MM$ is quasi-isomorphic to zero in $\AA^{\mathit{mod}}$.
\end{lemma}

\begin{proof}
The length filtration of $\mathit{hom}_{\AA^{\mathit{mod}}}(\MM,\MM)$ gives rise to a spectral sequence, whose starting page is
\begin{equation} \label{eq:length-sequence}
E_1^{p\bullet} = \prod_{X_0,\dots,X_p} H(\mathit{Hom}(\MM(X_p) \otimes \mathit{hom}_A(X_{p-1},X_p) \otimes \cdots \otimes \mathit{hom}_A(X_0,X_1),\MM(X_0))).
\end{equation}
Even though that spectral sequence does not converge in general, one can apply comparison and vanishing arguments to it. Since $\MM(X_p) \otimes \mathit{hom}_A(X_{p-1},X_p) \otimes \cdots \otimes \mathit{hom}_A(X_0,X_1)$ is a complex of projective $\RR$-modules, and $\MM(X_0)$ is acyclic, it follows from Lemma \ref{th:proj} that the $E_1$ page vanishes.
\end{proof}

One has a Yoneda functor $\AA \rightarrow \AA^{\mathit{mod}}$ as well as its extension $\AA^{\mathit{tw}} \rightarrow \AA^{\mathit{mod}}$, which are cohomologically full and faithful for the same reason as before. Moreover, given $\YY \in \mathit{Ob}\,\AA^{\mathit{tw}}$ and an idempotent endomorphism on the cohomology level, one can find an object of $\AA^{\mathit{mod}}$ representing the associated direct summand of the Yoneda image $\YY^{\mathit{yon}}$. To see that, one goes through the construction in Section \ref{subsec:splittings}, which defines a homotopy idempotent $\pp$ over $\RR$ as well as an associated family of modules $\MM = (\YY,\pp)^{\mathit{yon}}$. A noteworthy technical point is that $\MM(X) = \mathit{hom}_{\AA^{\mathit{tw}}}(X,\YY)[q]$ is still a complex of projective $\RR$-modules for any $X$, and has finitely generated cohomology since $H(\MM(X))$ is a direct summand of $H(\mathit{hom}_{\AA^{\mathit{tw}}}(X,\YY))$. One then defines the full $A_\infty$-subcategory of {\em perfect families of modules}, $\AA^{\mathit{perf}} \subset \AA^{\mathit{mod}}$, to consist of all objects quasi-isomorphic to direct summands of families of twisted complexes. This category is proper, in the sense that $\mathit{hom}_{\AA^{\mathit{perf}}}(\MM_0,\MM_1)$ is a chain homotopy retract of a bounded complex of finitely generated projective $\RR$-modules. In particular, the cohomology $H(\mathit{hom}_{\AA^{\mathit{perf}}}(\MM_0,\MM_1))$ is itself bounded and finitely generated over $\RR$ in each degree.

For any point of $\SS$, with associated map $\RR \rightarrow R$, we can define restriction $A_\infty$-functors
\begin{equation} \label{eq:3-functors}
\begin{aligned}
& \AA^{\mathit{tw}} \longrightarrow A^{\mathit{tw}}, \\
& \AA^{\mathit{mod}} \longrightarrow A^{\mathit{mod}}, \\
& \AA^{\mathit{perf}} \longrightarrow A^{\mathit{perf}}.
\end{aligned}
\end{equation}
The first one takes an $\XX$ as in \eqref{eq:xx} and passes to the fibres $F^i = R \otimes_\RR \FF^i$ to get an ordinary twisted complex. The second one is similarly given by $M(X) = R \otimes_\RR \MM(X)$. In either case, the morphism spaces again get specialized to the given point, which is unproblematic from a homological algebra viewpoint since they consist of projective $\RR$-modules for $\AA^{\mathit{tw}}$, and at least of flat $\RR$-modules for $\AA^{\mathit{mod}}$ (see for instance \cite[p.\ 122, Exercise 4]{cartan-eilenberg}). Clearly, the first two functors in \eqref{eq:3-functors} are compatible with Yoneda embeddings, which ensures that the third one is well-defined.

\subsection{Twisted complexes with connections\label{subsec:twisted-connections}}
Let $\XX$ be an object of $\AA^{\mathit{tw}}$, written as in \eqref{eq:xx}. A {\em pre-connection on $\XX$} is a pair $\pabla_\XX = (\{\nabla_{\FF^i}\},\alpha_\XX)$ consisting of an ordinary connection $\nabla_{\FF^i}$ on each $\RR$-module $\FF^i$, together with an element $\alpha_\XX \in \mathit{hom}_{\AA^{\mathit{tw}}}^0(\XX,\Omega^1_\RR \otimes \XX) = \Omega^1_\RR \otimes \mathit{hom}_{\AA^{\mathit{tw}}}^0(\XX,\XX)$. This becomes more meaningful if one writes it as a formal expression
\begin{equation} \label{eq:pre-connection-on-tw}
\pabla_\XX = \bigoplus_i \nabla_{\FF^i} \otimes e_{X^i[-\sigma^i]} + \alpha_\XX
\end{equation}
(recall that due to sign conventions, the identity for the shifted object $X^i[-\sigma^i]$ is $e_{X^i[-\sigma^i]} = (-1)^{\sigma^i} e_{X^i}$). There is some redundancy in this description: given elements $\ff_i \in \mathit{Hom}(\FF_i,\Omega^1_\RR \otimes \FF_i)$, one can change $\nabla_{\FF^i} \rightarrow \nabla_{\FF^i} + \ff_i$, and simultaneously $\alpha_\XX \rightarrow \alpha_{\XX} - \bigoplus_i \ff_i \otimes e_{X^i[-\sigma^i]}$, and the result is still considered to be the same pre-connection. With that in mind, pre-connections form an affine space over $\mathit{hom}_{\AA^{\mathit{tw}}}^0(\XX,\Omega^1_\RR \otimes \XX)$.

Now we include the differential $\delta_\XX$ in our discussion. Its compatibility with a pre-connection is measured by the {\em deformation cocycle}
\begin{equation} \label{eq:deformation-cocycle}
\begin{aligned}
& \mathit{def}(\pabla_\XX) \in \mathit{hom}_{\AA^{\mathit{tw}}}^1(\XX,\Omega^1_\RR \otimes \XX), \\
& \mathit{def}(\pabla_\XX) = -\Big( \textstyle\bigoplus_{i,j} \nabla_{\mathit{Hom}(\FF^i,\FF^j)} \otimes \mathit{id}_{\mathit{hom}_A(X^i[-\sigma^i],X^j[-\sigma^j])} \Big) (\delta_\XX) + \mu^1_{\AA^{\mathit{tw}}}(\alpha_\XX).
\end{aligned}
\end{equation}
Here, $\nabla_{\mathit{Hom}(\FF^i,\FF^j)}$ is the connection induced by $\nabla_{\FF^i}$ and $\nabla_{\FF^j}$. The formula becomes clearer if one thinks in terms of components
\begin{equation}
\begin{aligned}
& \alpha_{\XX}^{ji} \in \mathit{Hom}(\FF^i,\Omega^1_\RR \otimes \FF^j) \otimes \mathit{hom}_A(X^i,X^j)[\sigma^i-\sigma^j], \\
& \delta_{\XX}^{ji} \in \mathit{Hom}(\FF^i,\FF^j) \otimes \mathit{hom}_A(X^i,X^j)[\sigma^i-\sigma^j]
\end{aligned}
\end{equation}
Then, the components of the deformation cocycle are
\begin{equation} \label{eq:def-components}
\begin{aligned}
\mathit{def}(\pabla_\XX)^{ji} = & -(\nabla_{\mathit{Hom}(\FF^i,\FF^j)} \otimes \mathit{id})(\delta_\XX^{ji}) + (-1)^{\sigma^i} \mu^1_{\AA}(\alpha_\XX^{ji}) \\
& + (-1)^{\sigma^i} \sum_k \mu^2_{\AA}(\alpha_\XX^{jk},\delta_\XX^{ki}) + \mu^2_{\AA}(\delta_\XX^{jk},\alpha_\XX^{ki}) + \cdots
\end{aligned}
\end{equation}

\begin{lemma} \label{th:def-cocycle}
$\mathit{def}(\pabla_\XX)$ is a cocycle, whose cohomology class $\mathit{Def}(\XX)$ (the {\em deformation class} of $\XX$) is independent of the choice of pre-connection.
\end{lemma}

\begin{proof}
There is a simple trick which facilitates these computations, namely to temporarily forget that morphisms are supposed to be $\RR$-linear. In this weakened sense, one can consider $\pabla_\XX$ itself as a morphism from $\XX$ to $\Omega^1_\RR \otimes \XX$, and then the formula \eqref{eq:deformation-cocycle} actually represents
\begin{equation}
\mu^2_{\AA^{\oplus}}(\delta_{\Omega^1_\RR \otimes \XX},\pabla_\XX-\alpha_\XX) + \mu^2_{\AA^{\oplus}}(\pabla_\XX-\alpha_\XX,\delta_\XX) + \mu^1_{\AA^{\mathit{tw}}}(\alpha_\XX) = \mu^1_{\AA^{\mathit{tw}}}(\pabla_\XX)
\end{equation}
(this takes into accounts cancellations which arise from the fact that the components of $\pabla_\XX - \alpha_\XX$ are multiples of the identity $e_{X^i[-\sigma^i]}$). The desired statements now follow directly.
\end{proof}

Take two families of twisted complexes $\XX_k = \bigoplus_{i \in I_k} \FF^i_k \otimes X^i_k[-\sigma^i_k]$ ($k = 0,1$) equipped with pre-connections $\pabla_{\XX_k}$, written in the analogous way. This induces a pre-connection on the chain complex $\mathit{hom}_{\AA^{\mathit{tw}}}(\XX_0,\XX_1)$, namely
\begin{equation} \label{eq:hom-pre-connection}
\begin{aligned}
\pabla_{\mathit{hom}_{\AA^{\mathit{tw}}}(\XX_0,\XX_1)}(\aa) = & \bigoplus_{i,j} \Big(\nabla_{\mathit{Hom}(\FF^i_0,\FF^j_1)} \otimes \mathit{id}_{\mathit{hom}_A(X^i_0,X^j_1)} \Big)(\aa) \\ & + (-1)^{|\aa|} \mu^2_{\AA^{\mathit{tw}}}(\alpha_{\XX^1},\aa) - \mu^2_{\AA^{\mathit{tw}}}(\aa,\alpha_{\XX^0}).
\end{aligned}
\end{equation}
By a computation similar to the one in Lemma \ref{th:def-cocycle}, this gives a formula for the Atiyah cocycle \eqref{eq:cochain-atiyah} of the chain complex $\mathit{hom}_{\AA^{\mathit{tw}}}(\XX_0,\XX_1)$:
\begin{equation} \label{eq:atiyah-def}
\begin{aligned}
& \mathit{at}(\pabla_{\mathit{hom}_{\AA^{\mathit{tw}}}(\XX_0,\XX_1)})(\aa) =
\mu^1_{\AA^{\mathit{tw}}}(\pabla_{\mathit{hom}_{\AA^{\mathit{tw}}}(\XX_0,\XX_1)}(\aa)) -  \pabla_{\mathit{hom}_{\AA^{\mathit{tw}}}(\XX_0,\XX_1)}(\mu_{\AA^{\mathit{tw}}}^1(\aa)) \\ & \qquad \qquad
= \mu^2_{\AA^{\mathit{tw}}}(\mathit{def}(\pabla_{\XX^1}),\aa) + \mu^2_{\AA^{\mathit{tw}}}(\aa,\mathit{def}(\pabla_{\XX^0})).
\end{aligned}
\end{equation}
%
%
From here on, the obvious development would be the following one. Define a {\em connection} on $\XX$ to be a pre-connection for which \eqref{eq:deformation-cocycle} vanishes. If $\XX_0$ and $\XX_1$ carry connections, \eqref{eq:atiyah-def} implies that $\mathit{hom}_{\AA^{\mathit{tw}}}(\XX_0,\XX_1)$ carries a connection. However, families with vanishing deformation class are close to constant ones and therefore not terribly interesting.

Instead, we want to introduce a relative version, as follows. Let $\mathit{CC}(\AA, \AA)$ be the (reduced) Hochschild complex of the constant family $\AA$, and $HH(\AA,\AA)$ its cohomology. More relevant for us is a twisted version, with coefficients in the bimodule $\Omega^1_\RR \otimes \AA$. The effect of twisting is fairly trivial, both on the cochain and cohomology level:
\begin{equation}
\begin{aligned}
& \mathit{CC}(\AA,\Omega^1_\RR \otimes \AA) \iso \Omega^1_\RR \otimes \mathit{CC}(\AA,\AA), \\
& \mathit{HH}(\AA,\Omega^1_\RR \otimes \AA) \iso \Omega^1_\RR \otimes \mathit{HH}(\AA,\AA).
\end{aligned}
\end{equation}
As before, there is a chain map $\Gamma^{\mathit{tw}}$ between the twisted Hochschild chain complex and its analogue for $\AA^{\mathit{tw}}$.

\begin{definition} \label{th:deformation-field}
A cohomology class $[\gamma] \in \mathit{HH}^1(\AA,\Omega^1_\RR \otimes \AA)$ is called a {\em deformation field}. Let $\XX$ be a family of twisted complexes. We say that $\XX$ {\em follows the deformation field} if the image $\gamma^{\mathit{tw}} = \Gamma^{\mathit{tw}}(\gamma)$ satisfies
\begin{equation}
\mathit{Def}(\XX) = [\gamma^{\mathit{tw},0}] \in \mathit{H}^1(\mathit{hom}_{\AA^{\mathit{tw}}}(\XX,\Omega^1_\RR \otimes \XX)).
\end{equation}
\end{definition}

In almost every situation where we use deformation fields, a choice of cocycle representative $\gamma$ is assumed to have been made (the resulting theory is always independent of that choice up to quasi-isomorphism). If $\XX$ follows $[\gamma]$, we can choose a pre-connection whose deformation cocycle is exactly $\gamma^{\mathit{tw,0}}$. Call these {\em relative connections}, and denote them by $\nabla_\XX$. Given two objects $\XX_k$ with relative connections, we can introduce a modified version of \eqref{eq:hom-pre-connection}, namely
\begin{equation} \label{eq:hom-connection}
\begin{aligned}
\nabla_{\mathit{hom}_{\AA^{\mathit{tw}}}(\XX_0,\XX_1)}(\aa) = \pabla_{\mathit{hom}_{\AA^{\mathit{tw}}}(\XX_0,\XX_1)}(\aa) + \gamma^{\mathit{tw},1}(\aa).
\end{aligned}
\end{equation}
This is an actual connection, since the cocycle equation for $\gamma^{\mathit{tw}}$ says that
\begin{equation}
\mu^1_{\AA^{\mathit{tw}}}(\gamma^{\mathit{tw},1}(\aa)) -
\gamma^{\mathit{tw},1}(\mu^1_{\AA^{\mathit{tw}}}(\aa)) =
- \mu^2_{\AA^{\mathit{tw}}}(\gamma^{\mathit{tw},0},\aa) -
\mu^2_{\AA^{\mathit{tw}}}(\aa,\gamma^{\mathit{tw},0}),
\end{equation}
which one can add to \eqref{eq:atiyah-def} to get the desired property.

We would like to study the behaviour of \eqref{eq:hom-connection} under composition of morphisms in $\AA^{\mathit{tw}}$. Let's temporarily forget about Hochschild cohomology and just assume that we have families of twisted complexes $\XX_k$ ($k = 0,1,2$) equipped with pre-connections. Along the same lines as in \eqref{eq:atiyah-def} one finds that for any $\mu^1_{\AA^{\mathit{tw}}}$-cocycles $\aa_k \in \mathit{hom}_{\AA^{\mathit{tw}}}(\XX_{k-1},\XX_k)$ ($k=1,2$),
\begin{equation}
\begin{aligned}
& - \pabla_{\mathit{hom}_{\AA^{\mathit{tw}}}(\XX_0,\XX_2)}(\mu^2_{\AA^{\mathit{tw}}}(\aa_2,\aa_1)) \\ & \qquad \qquad +
 \mu^2_{\AA^{\mathit{tw}}}(\pabla_{\mathit{hom}_{\AA^{\mathit{tw}}}(\XX_1,\XX_2)}(\aa_2),\aa_1)
+ \mu^2_{\AA^{\mathit{tw}}}(\aa_2,\pabla_{\mathit{hom}_{\AA^{\mathit{tw}}}(\XX_0,\XX_1)}(\aa_1)) \\ &
= \mu^3_{\AA^{\mathit{tw}}}(\mathit{def}(\pabla_{\XX_2}),\aa_2,\aa_1) +
\mu^3_{\AA^{\mathit{tw}}}(\aa_2,\mathit{def}(\pabla_{\XX_1}),\aa_1) +
\mu^3_{\AA^{\mathit{tw}}}(\aa_2,\aa_1,\mathit{def}(\pabla_{\XX_0})) \\ & \qquad \qquad + \text{\it (coboundary)}.
\end{aligned}
\end{equation}
Assuming that the $\XX_k$ follow $[\gamma]$ and come with relative connections, one adds the correction terms from \eqref{eq:hom-connection} and gets
\begin{equation} \label{eq:added-product}
\begin{aligned}
& -\nabla_{\mathit{hom}_{\AA^{\mathit{tw}}}(\XX_0,\XX_2)}(\mu^2_{\AA^{\mathit{tw}}}(\aa_2,\aa_1)) \\ & \qquad \qquad +
\mu^2_{\AA^{\mathit{tw}}}(\nabla_{\mathit{hom}_{\AA^{\mathit{tw}}}(\XX_1,\XX_2)}(\aa_2),\aa_1)
+ \mu^2_{\AA^{\mathit{tw}}}(\aa_2,\nabla_{\mathit{hom}_{\AA^{\mathit{tw}}}(\XX_0,\XX_1)}(\aa_1)) \\ &
= -\mu^1_{\AA^{\mathit{tw}}}(\gamma^{\mathit{tw},2}(\aa_2,\aa_1)) + \text{\it (same coboundary as before)},
\end{aligned}
\end{equation}
which means that the cohomology level connections act as derivations with respect to the product.

We have to confess that the framework introduced above, even though natural and accessible, is not satisfactory, for two reasons. The first (technical) reason is that the definition of pre-connection \eqref{eq:pre-connection-on-tw} makes sense only under the assumption of strict unitality, which we have imposed so far but will want to relax eventually. The second (conceptual) reason is that families of twisted complexes are far too restrictive to be useful in general -- as one can see by observing that if $\XX$ is such a family, its fibre at any point of $\SS$ represents the same class in the $K$-theory of $A^{\mathit{tw}}$. With this in mind, we will now switch to $A_\infty$-modules and carry out the corresponding developments there.

\subsection{Modules with connections}
Let $\MM$ be a family of $A$-modules. A {\em pre-connection} $\pabla_\MM$ is a sequence of maps
\begin{equation} \label{eq:pre-connection-components}
\begin{aligned}
& \pabla_\MM^1: \MM(X_0) \longrightarrow \Omega^1_\RR \otimes \MM(X_0), \\
& \pabla_\MM^2: \MM(X_1) \otimes \mathit{hom}_A(X_0,X_1) \longrightarrow \Omega^1_\RR \otimes \MM(X_0)[-1], \\
& \pabla_\MM^3: \MM(X_2) \otimes \mathit{hom}_A(X_1,X_2) \otimes \mathit{hom}_A(X_0,X_1) \longrightarrow \Omega^1_\RR \otimes \MM(X_0)[-2], \\
& \dots
\end{aligned}
\end{equation}
where: the maps $\mm \mapsto (-1)^{|\mm|}\pabla_\MM^1(\mm)$ are connections in the standard sense; and the higher order terms $\pabla_\MM^d$, $d>1$, are $\RR$-linear. Clearly, pre-connections form an affine space over $\mathit{hom}_{\AA^{\mathit{mod}}}^0(\MM,\Omega^1_\RR \otimes \MM) \iso \Omega^1_\RR \otimes \mathit{hom}_{\AA^{\mathit{mod}}}^0(\MM,\MM)$. The {\em deformation cocycle}
\begin{equation} \label{eq:module-cocycle}
\mathit{def}(\pabla_\MM) \in \mathit{hom}_{\AA^{\mathit{mod}}}^1(\MM, \Omega^1_\RR \otimes \MM)
\end{equation}
of a pre-connection is obtained by applying $\mu^1_{\mathit{\AA}^{\mathit{mod}}}$ to \eqref{eq:pre-connection-components}. This makes sense even though $\pabla_\MM^1$ is not $\RR$-linear. The same observation shows that:

\begin{lemma}
$\mathit{def}(\pabla_\MM)$ is closed, and its cohomology class $\mathit{Def}(\MM)$ is independent of the choice of pre-connection. \qed
\end{lemma}

\begin{example}
Suppose that our pre-connection has vanishing higher order terms $\pabla_\MM^d = 0$, $d>1$, hence is just given by a family of connections $\nabla_{\MM(X)}(\mm) = (-1)^{|\mm|} \pabla_\MM^1(\mm)$ on the graded $\RR$-modules $\MM(X)$. Then the deformation cocycle is
\begin{equation}
\mathit{def}(\pabla_\MM)^d(\mm,a_{d-1},\dots,a_1) = (\mathit{id}_{\Omega^1_\RR} \otimes \mu^d_\MM)(\pabla_{\MM(X_d)}(\mm),\dots,a_1) - \pabla_{\MM(X_0)}(\mu^d_\MM(\mm,a_d,\dots,a_1))
\end{equation}
for $\mm \in \MM(X_d)$, $a_k \in \mathit{hom}_A(X_{k-1},X_k)$. This is just the covariant derivative of the module structure of $\MM$, measuring its failure to be compatible with the connections.
\end{example}

Given two families of modules $\MM_k$ ($k = 0,1$) equipped with pre-connections $\pabla_{\MM_k}$, consider the map \begin{equation} \label{eq:hom-pre-connection-2}
\begin{aligned}
& \pabla_{\mathit{hom}_{\AA^{\mathit{mod}}}(\MM_0,\MM_1)}: \mathit{hom}_{\AA^{\mathit{mod}}}(\MM_0,\MM_1) \longrightarrow \Omega^1_\RR \otimes \mathit{hom}_{\AA^{\mathit{mod}}}(\MM_0,\MM_1), \\
& \pabla_{\mathit{hom}_{\AA^{\mathit{mod}}}(\MM_0,\MM_1)}(\bb) = (-1)^{|\bb|} \mu^2_{\mathit{\AA}^{\mathit{mod}}}(\pabla_{\MM_1},\bb) - \mu^2_{\mathit{\AA}^{\mathit{mod}}}(\mathit{id}_{\Omega^1_\RR} \otimes \bb,\pabla_{\MM_0}).
\end{aligned}
\end{equation}
Spelled out, this means that $\cc = \pabla_{\mathit{hom}_{\AA^{\mathit{mod}}}(\MM_0,\MM_1)}(\bb)$ is given by
\begin{equation}
\begin{aligned}
& \cc^1(\mm) = (-1)^{|\bb^1(\mm)|} \pabla_{\MM_1}^1(\bb^1(\mm)) - (-1)^{|\mm|} (\mathit{id}_{\Omega^1_\RR} \otimes \bb^1)(\pabla_{\MM_0}^1(\mm)), \\
& \cc^2(\mm,a) = (-1)^{|\bb^2(\mm,a)|} \pabla_{\MM_1}^1(\bb^2(\mm,a)) + (-1)^{|\bb^1(\mm)|} \pabla_{\MM_1}^2(\bb^1(\mm),a) \\ & \qquad \qquad - (-1)^{|\mm|+|a|-1}
(\mathit{id}_{\Omega^1_\RR} \otimes \bb^1)(\pabla_{\MM_0}^2(\mm,a)) - (-1)^{|\mm|}
(\mathit{id}_{\Omega^1_\RR} \otimes \bb^2)(\pabla_{\MM_0}^1(\mm),a), \\
& \dots \\
& \cc^d(\mm,a_{d-1},\dots,a_1) = (-1)^{|\bb^d(\mm,a_{d-1},\dots,a_1)|} \pabla_{\MM_1}^1(\bb^d(\mm,a_{d-1},\dots,a_1)) \\ & \qquad \qquad - (-1)^{|\mm|} (\mathit{id}_{\Omega^1_\RR} \otimes \bb^d(\pabla_{\MM_0}^1(\mm),a_{d-1},\dots,a_1) \\ & \qquad \qquad + \text{\it (terms involving the higher order parts of the pre-connections on $\MM_0,\MM_1$)}.
\end{aligned}
\end{equation}
This shows that \eqref{eq:hom-pre-connection-2} is a pre-connection on the chain complex of $\mathit{hom}$'s. It follows from the definition and the $A_\infty$-equations on $\AA^{\mathit{mod}}$ that
\begin{equation} \label{eq:atiyah-def-2}
\begin{aligned}
& \mu^1_{\AA^{\mathit{mod}}}(\pabla_{\mathit{hom}_{\AA^{\mathit{mod}}}(\MM_0,\MM_1)}(\bb)) - \pabla_{\mathit{hom}_{\AA^{\mathit{mod}}}(\MM_0,\MM_1)}(\mu_{\AA^{\mathit{mod}}}^1(\bb)) = \\ & \qquad \qquad
\mu^2_{\AA^{\mathit{mod}}}(\mathit{def}(\pabla_{\MM^1}),\bb) + \mu^2_{\AA^{\mathit{mod}}}(\bb,\mathit{def}(\pabla_{\MM^0})).
\end{aligned}
\end{equation}

Given a deformation field represented by $\gamma \in \mathit{CC}^1(\AA,\Omega^1_{\RR} \otimes \AA)$, write $\gamma^{\mathit{mod}} = \Gamma^{\mathit{mod}}(\gamma)$. In parallel with Definition \ref{th:deformation-field}, we say that a family of modules $\MM$ {\em follows} $[\gamma]$ if
\begin{equation}
\mathit{Def}(\MM) = [\gamma^{\mathit{mod},0}] \in H^1(\mathit{hom}_{\AA^{\mathit{mod}}}(\MM,\Omega^1_\RR \otimes \MM)).
\end{equation}
If this holds, one can equip $\MM$ with a {\em relative connection} $\nabla_\MM$, by which we again mean a pre-connection whose deformation cocycle is exactly $\gamma^{\mathit{mod},0}$. Given two modules equipped with relative connections, one can modify \eqref{eq:hom-pre-connection-2} to get an actual connection on that chain complex, as in \eqref{eq:hom-connection}:
\begin{equation} \label{eq:hom-connection-2}
\nabla_{\mathit{hom}_{\AA^{\mathit{mod}}}(\MM_0,\MM_1)} =
\pabla_{\mathit{hom}_{\AA^{\mathit{mod}}}(\MM_0,\MM_1)} + \gamma^{\mathit{mod},1}.
\end{equation}
Moreover, by essentially the same computation as in \eqref{eq:added-product}, these connections satisfy
\begin{equation} \label{eq:multiplicative-connections}
\begin{aligned}
& -\nabla_{\mathit{hom}_{\AA^{\mathit{mod}}}(\XX_0,\XX_2)}(\mu^2_{\AA^{\mathit{tw}}}(\bb_2,\bb_1)) \\ & \qquad \qquad +
\mu^2_{\AA^{\mathit{mod}}}(\nabla_{\mathit{hom}_{\AA^{\mathit{mod}}}(\XX_1,\XX_2)}(\bb_2),\bb_1)
+ \mu^2_{\AA^{\mathit{mod}}}(\bb_2,\nabla_{\mathit{hom}_{\AA^{\mathit{mod}}}(\XX_0,\XX_1)}(\bb_1)) \\ &
\qquad \qquad = \text{\it (coboundary)}
\end{aligned}
\end{equation}
for any cocycles $\bb_k \in \mathit{hom}_{\AA^{\mathit{mod}}}(\MM_{k-1},\MM_k)$.

\begin{addendum}
Unsurprisingly, all these notions are compatible with their counterparts from Section \ref{subsec:twisted-connections} via the Yoneda embedding. If $\YY = \bigoplus_i \FF^i \otimes Y^i[-\sigma^i]$ is a family of twisted complexes with a pre-connection $\pabla_\YY$ as in \eqref{eq:pre-connection-on-tw}, then the family of modules $\YY^{\mathit{yon}}$ inherits a pre-connection:
\begin{equation}
\begin{aligned}
& \pabla_{\YY^{\mathit{yon}}}^1(\aa) =  (-1)^{|\aa|} (\textstyle\bigoplus_i \nabla_{\FF^i} \otimes \mathit{id})(\aa) + \mu^2_\AA(\alpha_{\YY},\aa)  \\
& \qquad \qquad \text{for $\textstyle\aa \in \YY^{\mathit{yon}}(X) = \mathit{hom}_{\AA^{\mathit{tw}}}(X,\YY) = \bigoplus_i \FF^i \otimes \mathit{hom}_A(X,Y^i)[-\sigma^i]$, and } \\
& \pabla_{\YY^{\mathit{yon}}}^d(\aa_d,a_{d-1},\dots,a_1) = \mu^{d+1}_\AA(\alpha_{\YY},\aa_d,a_{d-1},\dots,a_1) \\
& \qquad \qquad \text{where $\aa_d \in \mathit{hom}_{\AA^{\mathit{tw}}}(X_d,\YY)$, and $a_k \in \mathit{hom}_A(X_{k-1},X_k)$ for $k =1,\dots,d-1$.}
\end{aligned}
\end{equation}
The deformation cocycle of $\pabla_{\YY^{\mathit{yon}}}$ is the image of that of $\pabla_\YY$ under the Yoneda functor. One can take this comparison further to relative connections, but we will not need that.
\end{addendum}

We want to highlight one simple consequence:

\begin{lemma} \label{th:its-projective}
Suppose that $\MM_0,\MM_1$ are families of modules (as always, with finite cohomology) following $[\gamma]$, and where $\MM_0$ is perfect. Then $H(\mathit{hom}_{\AA^{\mathit{mod}}}(\MM_0,\MM_1))$ is a finitely generated graded projective $\RR$-module.
\end{lemma}

\begin{proof}
Finite generation follows from the fact that $\MM_0$ is perfect. On the other hand, the module carries a connection, hence Lemma \ref{th:locally-free} applies.
\end{proof}

By construction, relative connections on a given family $\MM$ form an affine space over the space of cocycles inside $\mathit{hom}^0_{\mathit{\AA}^{\mathit{perf}}}(\MM,\Omega^1_\RR \otimes \MM)$. If we change the relative connections on $\MM_k$ ($k =0,1$) to $\nabla_{\MM_k}' = \nabla_{\MM_k} + \cc_k$, the induced connection on the morphism spaces changes to
\begin{equation} \label{eq:modify-connection}
\nabla_{\mathit{hom}_{\AA^{\mathit{mod}}}(\MM_0,\MM_1)}'(\bb) =
\nabla_{\mathit{hom}_{\AA^{\mathit{mod}}}(\MM_0,\MM_1)}(\bb) + (-1)^{|\bb|} \mu^2_{\mathit{\AA}^{\mathit{mod}}}(\cc_1,\bb) - \mu^2_{\mathit{\AA}^{\mathit{mod}}}(\mathit{id}_{\Omega^1_\RR} \otimes \bb,\cc_0).
\end{equation}
In particular, if we are only interested in the connection on the cohomology level, relative connections which differ by coboundaries yield the same result, so the space of relevant choices is an affine space over $\mathit{Hom}_{H^0(\mathit{\AA}^{\mathit{perf}})}(\MM,\Omega^1_\RR \otimes \MM)$.


If $\MM$ is a family of modules with a pre-connection $\pabla_\MM$, and $\FF$ a finitely generated projective $\RR$-module with its own connection $\nabla_\FF$, the tensor product $\FF \otimes \MM$ inherits a pre-connection, defined by
\begin{equation}
\begin{aligned}
& \pabla_{\FF \otimes \MM}^1(\ff \otimes \mm) = (-1)^{|\mm|} (\nabla_\FF \ff) \otimes \mm + \ff \otimes \pabla_\MM^1(\mm), \\
& \pabla_{\FF \otimes \MM}^d(\ff \otimes \mm,a_{d-1},\dots,a_1) = \ff \otimes \pabla_\MM^d(\mm,a_{d-1},\dots,a_1) \quad \text{for $d>1$.}
\end{aligned}
\end{equation}
The associated deformation cocycle is
\begin{equation} \label{eq:tensor-deformation}
\mathit{def}(\pabla_{\FF \otimes \MM}) = \mathit{id}_\FF \otimes \mathit{def}(\pabla_\MM).
\end{equation}
In particular, if $\MM$ follows a given deformation field $[\gamma]$, then so does $\FF \otimes \MM$. One could generalize this slightly by allowing $\FF$ to be a complex of projective $\RR$-modules with a pre-connection, in which case the associated Atiyah cocycle would appear in an additional summand in \eqref{eq:tensor-deformation}.

\subsection{Functoriality\label{subsec:functor}}
Let $G: A \rightarrow \tilde{A}$ be a (strictly unital) $A_\infty$-functor. We want to study the action of $G$ on families of objects, over a fixed base space $\SS$. For expository reasons, we temporarily return to the framework of twisted complexes. It is well-known (to the man on the street) that $G$ induces an $A_\infty$-functor $G^{\mathit{tw}}: A^{\mathit{tw}} \rightarrow \tilde{A}^{\mathit{tw}}$. The same formulae applied to families define an $A_\infty$-functor $\GG^{\mathit{tw}}: \AA^{\mathit{tw}} \rightarrow \tilde{\AA}^{\mathit{tw}}$.

A pre-connection \eqref{eq:pre-connection-on-tw} on $\XX \in \mathit{Ob}\,\AA^{\mathit{tw}}$ induces one on its image $\tilde{\XX} = \GG^{\mathit{tw}}(\XX)$:
\begin{equation} \label{eq:pre-connection-image}
\begin{aligned}
& \pabla_{\tilde\XX} = \bigoplus_i \nabla_{\FF^i} \otimes e_{G(X^i)[-\sigma^i]} + \GG^{\mathit{tw},1}(\alpha_\XX), \\
& \mathit{def}(\pabla_{\tilde\XX}) = \GG^{\mathit{tw},1}(\mathit{def}(\pabla_\XX)).
\end{aligned}
\end{equation}
If $\XX_k$ ($k = 0,1$) are families with pre-connections, and $\tilde{\XX}_k$ their images under $\GG^{\mathit{tw}}$ equipped with the induced pre-connections, then for any cocycle $\aa$ we have
\begin{equation} \label{eq:functoriality-fails}
\begin{aligned}
& \pabla_{\mathit{hom}_{\tilde{\AA}^{\mathit{tw}}}(\tilde{\XX}_0,\tilde{\XX}_1)}(\GG^{\mathit{tw},1}(\aa)) =
\GG^{\mathit{tw},1}(\pabla_{\mathit{hom}_{\AA^{\mathit{tw}}}(\XX_0,\XX_1)}(\aa)) \\ & \qquad \qquad
- \GG^{\mathit{tw},2}(\mathit{def}(\pabla_{\XX_1}),\aa)
- \GG^{\mathit{tw},2}(\aa,\mathit{def}(\pabla_{\XX_0})) + \text{\it (coboundary)}.
\end{aligned}
\end{equation}

\begin{assumption} \label{th:matching-deformations}
Let $[\gamma]$ and $[\tilde\gamma]$ be deformation fields for $\AA$ and $\tilde{\AA}$, respectively. Suppose that there is a $\beta \in \mathit{CC}^0(\AA,\Omega^1_\RR \otimes \tilde\AA)$ such that
\begin{equation} \label{eq:beta-diff}
\partial \beta= \GG^*(\tilde\gamma) - \GG_*(\gamma).
\end{equation}
\end{assumption}

In the definition of $\beta$, we consider $\tilde\AA$ as an $\AA$-bimodule by pullback through $\GG$. As in \eqref{eq:twisted-gamma}, $\beta$ induces a cochain $\beta^{\mathit{tw}} \in \mathit{CC}^0(\AA^{\mathit{tw}},\Omega^1_\RR \otimes \tilde{\AA}^{\mathit{tw}})$, which satisfies the analogue of \eqref{eq:beta-diff}. Concretely, this means that
\begin{equation} \label{eq:relative-field}
\begin{aligned}
& \mu^1_{\tilde{\AA}^{\mathit{tw}}}(\beta^{\mathit{tw},0}) = \tilde\gamma^{\mathit{tw},0} - \GG^{\mathit{tw},1}(\gamma^{\mathit{tw},0}) \in \mathit{hom}_{\tilde{\AA}^{\mathit{tw}}}(\GG^{\mathit{tw}}(\XX_0),\GG^{\mathit{tw}}(\XX_0)), \\
& \mu^1_{\tilde{\AA}^{\mathit{tw}}}(\beta^\mathit{tw,1}(\aa)) + \mu^2_{\tilde{\AA}^{\mathit{tw}}}(\GG^{\mathit{tw},1}(\aa),\beta^{\mathit{tw},0}) +
(-1)^{|\aa|-1} \mu^2_{\tilde{\AA}^{\mathit{tw}}}(\beta^{\mathit{tw},0},\GG^{\mathit{tw},1}(\aa)) \\ & \qquad \qquad +
\beta^{\mathit{tw},1}(\mu^1_{\AA^{\mathit{tw}}}(\aa)) = \tilde\gamma^{\mathit{tw},1}(\GG^{\mathit{tw},1}(\aa)) - \GG^{\mathit{tw},1}(\gamma^{\mathit{tw},1}(\aa)) \\ & \qquad \qquad \qquad \qquad
- \GG^{\mathit{tw},2}(\gamma^{\mathit{tw},0},\aa) - \GG^{\mathit{tw},2}(\aa,\gamma^{\mathit{tw},0}), \\
& \dots
\end{aligned}
\end{equation}

A first consequence of \eqref{eq:relative-field} is that if $\XX$ follows $[\gamma]$, then $\tilde\XX = \GG^{\mathit{tw}}(\XX)$ follows $[\tilde{\gamma}]$. Indeed, if $\nabla_\XX = \pabla_\XX$ is a relative connection with respect to $\gamma$, then
\begin{equation} \label{eq:image-relative-connection}
\nabla_{\tilde\XX} = \pabla_{\tilde\XX} + \beta^{\mathit{tw},0}
\end{equation}
is a relative connection for $\tilde\gamma$. Suppose that $\XX_k$ ($k = 0,1$) are families with relative connections, and we equip their images $\tilde{\XX}_k$ with the induced relative connections as in \eqref{eq:image-relative-connection}. From \eqref{eq:functoriality-fails} and \eqref{eq:relative-field} it then follows that for any cocycle $\aa$,
\begin{equation} \label{eq:functoriality-of-connections}
\begin{aligned}
& \nabla_{\mathit{hom}_{\tilde{\AA}^{\mathit{tw}}}(\tilde{\XX}_0,\tilde{\XX}_1)}(\GG^{\mathit{tw},1}(\aa)) =
\GG^{\mathit{tw},1}(\nabla_{\mathit{hom}_{\AA^{\mathit{tw}}}(\XX_0,\XX_1)}(\aa)) + \text{\it (coboundary)}.
\end{aligned}
\end{equation}
This explains the sense in which, under Assumption \ref{th:matching-deformations}, the cohomology level connections on $\mathit{hom}$ spaces are functorial.

Let's turn to the corresponding question for $A_\infty$-modules. Take $\tilde{A}$, considered as an $(A,\tilde{A})$-bimodule by $G$-pullback on the left side only. This is right perfect, since $\tilde{A}(\cdot,X) = G(X)^{\mathit{yon}}$, hence gives rise to a convolution functor
\begin{equation}
G^{\mathit{perf}} = K_{\tilde{A}}: A^{\mathit{perf}} \longrightarrow \tilde{A}^{\mathit{perf}}.
\end{equation}
In the same way, one can define an analogue $\GG^{\mathit{perf}} = \KK_{\tilde{A}}$ acting on families. Suppose that $\MM$ is a perfect family of modules carrying a pre-connection $\pabla_\MM$. Then there is an induced pre-connection on $\tilde{\MM} = \GG^{\mathit{perf}}(\MM) = \MM \otimes_A \tilde{A}$:
\begin{equation}
\begin{aligned}
& \pabla_{\tilde{\MM}}^0(\mm \otimes a_r \otimes \cdots \otimes a_1 \otimes \tilde{a}) =
\textstyle\sum_i \pabla_{\tilde{\MM}}^{r-i+1}(\mm,a_r,\dots,a_{i+1}) \otimes a_i \otimes \cdots \otimes \tilde{a}, \\
& \pabla_{\tilde{\MM}}^d = 0 \quad \text{for all $d>0$,} \\
& \mathit{def}(\pabla_{\tilde{\MM}}) = \GG^{\mathit{perf},1}(\mathit{def}(\pabla_\MM)).
\end{aligned}
\end{equation}
If $\MM_k$ ($k = 0,1$) are families with pre-connections, and $\tilde{\MM}_k$ their images under $\GG^{\mathit{perf}}$ equipped with the induced pre-connections, the following simpler analogue of \eqref{eq:functoriality-fails} holds:
\begin{equation}
\pabla_{\mathit{hom}_{\tilde{\AA}^{\mathit{perf}}}(\tilde{\MM}_0,\tilde{\MM}_1)}(\GG^{\mathit{perf},1}(\bb)) = \GG^{\mathit{perf},1}(\pabla_{\mathit{hom}_{\AA^{\mathit{perf}}}(\MM_0,\MM_1)}(\bb)).
\end{equation}

Now suppose again that Assumption \ref{th:matching-deformations} holds. Write $\gamma^{\mathit{perf}} \in \mathit{CC}^1(\AA^{\mathit{perf}}, \Omega^1_\RR \otimes \AA^{\mathit{perf}})$ for the element induced by $\gamma$ as in \eqref{eq:gamma-for-modules}, and similarly for $\tilde\gamma^{\mathit{perf}}$. Then, there is a corresponding element $\beta^{\mathit{perf}} \in \mathit{CC}^0(\AA^{\mathit{perf}},\Omega^1_\RR \otimes \tilde{\AA}^{\mathit{perf}})$ which satisfies
\begin{equation} \label{eq:beta-perf}
\partial \beta^{\mathit{perf}} = (\GG^{\mathit{perf}})^*(\tilde\gamma^{\mathit{perf}}) - \GG^{\mathit{perf}}_*(\gamma^{\mathit{perf}}).
\end{equation}
Instead of attempting to define $\beta^{\mathit{perf}}$ by a direct formula, it seems more reasonable to argue by restriction to the images of the Yoneda embeddings $\AA^{\mathit{tw}} \rightarrow \AA^{\mathit{perf}}$, $\tilde{\AA}^{\mathit{tw}} \rightarrow \tilde{\AA}^{\mathit{perf}}$. This restriction induces quasi-isomorphisms on the relevant Hochschild complexes, and it essentially reduces this situation to the previously discussed case of twisted complexes. \eqref{eq:beta-perf} also has similar consequences as before: if $\nabla_\MM = \pabla_\MM$ is a relative connection for $\gamma$, then
\begin{equation} \label{eq:image-relative-connection-for-modules}
\nabla_{\tilde\MM} = \pabla_{\tilde\MM} + \beta^{\mathit{perf},0}
\end{equation}
is a relative connection for $\tilde\gamma$, and moreover these relative connections satisfy a simplified version of \eqref{eq:functoriality-of-connections}:
\begin{equation} \label{eq:strict-functoriality-of-connections}
\nabla_{\mathit{hom}_{\tilde{\AA}^{\mathit{perf}}}(\tilde{\MM}_0,\tilde{\MM}_1)}(\GG^{\mathit{perf},1}(\bb)) =
\GG^{\mathit{perf},1}(\nabla_{\mathit{hom}_{\AA^{\mathit{perf}}}(\MM_0,\MM_1)}(\bb)).
\end{equation}

Thinking in terms of modules naturally accomodates a generalization, in which we do not start with a functor $G$, but instead with a general right perfect $(A,\tilde{A})$-bimodule $P$, and its convolution functor $\KK_P$ for families.
As in \eqref{eq:grrr}, we then have a homotopy commutative diagram
\begin{equation} \label{eq:grrr-2}
\xymatrix{
\mathit{CC}(\AA,\Omega^1_\RR \otimes \AA) \ar[d] \ar[r] & \mathit{hom}_{(\AA,\tilde{\AA})^{\mathit{mod}}}(P,\Omega^1_\RR \otimes P) \ar[d] & \ar[l] \ar[d] \mathit{CC}(\tilde{\AA},\Omega^1_\RR \otimes \tilde{\AA}) \\
\mathit{CC}(\AA^{\mathit{perf}},\Omega^1_\RR \otimes \AA^{\mathit{perf}}) \ar[r] &
\mathit{CC}(\AA^{\mathit{perf}},\Omega^1_\RR \otimes \tilde{\AA}^{\mathit{perf}}) & \ar[l]
\mathit{CC}(\tilde{\AA}^{\mathit{perf}},\Omega^1_\RR \otimes \tilde{\AA}^{\mathit{perf}}),
}
\end{equation}
where $P$ is considered as a constant family of bimodules over $\RR$ (the general theory of such families will be our next topic of discussion, but it is easy to see what we mean in this special case). The natural analogue of Assumption \ref{th:matching-deformations} in this context is therefore:

\begin{assumption} \label{th:matching-deformations-2}
Suppose that we have deformation fields $[\gamma]$ and $[\tilde\gamma]$ for $\AA$ and $\tilde{\AA}$, respectively, whose images in $H^1(\mathit{hom}_{(\AA,\tilde{\AA})^{\mathit{mod}}}(P,\Omega^1_\RR \otimes P))$ agree.
\end{assumption}

If this is the case, one can apply the same argument as before to $\KK_P$, meaning that relative connections on perfect families of modules can be pushed forward, and the analogue of \eqref{eq:strict-functoriality-of-connections} will hold.

\subsection{Existence\label{subsec:existence}}
In our discussion of functoriality, we have used the tensor product of a family of modules and a fixed bimodule. The other combination, where the module is fixed but the bimodule varies, is also useful. A {\em family of bimodules with finite cohomology} over $\SS$ associates to any $(X,\tilde{X}) \in \mathit{Ob}\, A \times \mathit{Ob} \, \tilde{A}$ a complex $\PP(\tilde{X},X)$ of projective $\RR$-modules, which comes with structure maps as in \eqref{eq:bimodule-structure}, and such that the cohomology of $(\PP(\tilde{X},X),\mu^{0|1|0}_{\PP})$ is bounded and finitely generated in each degree. Such bimodules form an $A_\infty$-category over $\RR$, denoted by $(\AA,\tilde\AA)^{\mathit{mod}}$. The elementary theory of families of modules, as developed in Section \ref{subsec:families}, carries over to this situation without any complications.

Let $\PP$ be a family of $(A,\tilde{A})$-bimodules. A pre-connection $\pabla_\PP$ is a sequence of maps
\begin{equation} \label{eq:bimod-pre-connection-components}
\begin{aligned}
& \pabla_\PP^{s|1|t}: \mathit{hom}_A(X_{s-1},X_s) \otimes \cdots \otimes \mathit{hom}_A(X_0,X_1) \otimes \PP(\tilde{X}_t,X_0) \\ & \qquad \qquad \qquad \qquad \otimes \mathit{hom}_{\tilde{A}}(\tilde{X}_{t-1},\tilde{X}_t) \otimes \cdots \otimes \mathit{hom}_{\tilde{A}}(\tilde{X}_0,\tilde{X}_1) \longrightarrow \Omega^1_\RR \otimes \PP(\tilde{X}_0,X_s)[1-s-t],
\end{aligned}
\end{equation}
where the maps $\pp \mapsto (-1)^{|\pp|}\pabla_\PP^{0|1|0}(\pp)$ are connections in the standard sense, while all the other terms are $\RR$-linear. Each pre-connection has a deformation cocycle
\begin{equation} \label{eq:bimodule-cocycle}
\mathit{def}(\pabla_\PP) \in \mathit{hom}_{(\AA,\tilde{\AA})^{\mathit{mod}}}^1(\PP, \Omega^1_\RR \otimes \PP),
\end{equation}
obtained by applying $\mu^1_{(\AA,\tilde{\AA})^{\mathit{mod}}}$ to \eqref{eq:bimod-pre-connection-components}. As usual, the cohomology class $\mathit{Def}(\PP)$ represented by \eqref{eq:bimodule-cocycle} is independent of the choice of pre-connection.

Take a perfect $A$-module $M$. Then $\tilde{\MM} = M \otimes_A \PP$, defined as in \eqref{eq:ainfty-tensor}, is a family of $\tilde{A}$-modules with finite cohomology. If we assume in addition that $\PP$ is right perfect (in the appropriate sense for families), then $\tilde{\MM}$ is again a perfect family. Moreover, a pre-connection on $\PP$ defines one on $\tilde{\MM}$, formally defined by taking the identity on $M$ and tensoring it with $\pabla_\PP$. We have an obvious correspondence between deformation cocycles:
\begin{equation} \label{eq:def-of-a-tensor}
\mathit{def}(\pabla_{\tilde{\MM}}) = e_M \otimes_A \mathit{def}(\pabla_\PP),
\end{equation}
where $e_M$ is the identity endomorphism. Now suppose that our target category $\tilde{A}$ comes with a deformation field represented by $\tilde{\gamma} \in \mathit{CC}^1(\tilde\AA,\Omega^1_{\RR} \tilde\otimes \AA)$. In a slight generalization of \eqref{eq:grrr-2}, we have a canonical chain map
\begin{equation} \label{eq:bimodule-induced}
\mathit{CC}(\tilde{\AA},\Omega^1_\RR \otimes\tilde{\AA}) \longrightarrow
\mathit{hom}_{(\AA,\tilde{\AA})^{\mathit{mod}}}(\PP,\Omega^1_\RR \otimes \PP).
\end{equation}
As usual, we say that $\PP$ follows $\tilde{\gamma}$ if its deformation class is the image of $[\tilde{\gamma}]$ under \eqref{eq:bimodule-induced}, and define the notion of relative connection by requiring equality on the cocycle level. It follows from \eqref{eq:def-of-a-tensor} and the explicit formula for \eqref{eq:bimodule-induced} that if $\PP$ follows $[\tilde{\gamma}]$, then so do all the families $\tilde{\MM} = M \otimes_\AA \PP$. The induced connection on the space of morphisms between two such families is given by \eqref{eq:hom-connection-2}, which one can write as
\begin{equation}
\nabla_{\mathit{hom}_{\tilde\AA^{\mathit{mod}}}(\tilde\MM_0,\tilde\MM_1)}(\tilde\bb) =
(-1)^{|\tilde{\bb}|} \mu^2_{\tilde\AA^{\mathit{mod}}}(e_{M_1} \otimes_A \pabla_\PP, \tilde\bb) - \mu^2_{\tilde\AA^{\mathit{mod}}}(\mathit{id}_{\Omega^1_\RR} \otimes_A \tilde\bb, e_{M_0} \otimes_A \pabla_\PP) + \gamma^{\mathit{mod},1}(\tilde{\bb}).
\end{equation}
In particular, if $\tilde{\bb} = b \otimes_A e_\PP$, then the first two terms cancel, while the last one vanishes by inspection of \eqref{eq:gamma-formula}. The application we are aiming for is this:

\begin{corollary} \label{th:universal-family}
Take an $A_\infty$-category $A$ with a deformation field $[\gamma]$. Suppose that there is a family of $A$-bimodules $\PP$ which is right perfect, follows $[\gamma]$, and whose fibre at some base point $s \in \SS$ is quasi-isomorphic to the diagonal bimodule. Then, for any $M \in \mathit{Ob}\,A^{\mathit{perf}}$, there is a perfect family of modules $\MM$ which follows $[\gamma]$, and with $\MM_s$ quasi-isomorphic to $M$. Moreover, any two such families satisfy
\begin{equation}
H^0(\mathit{hom}_{\AA^{\mathit{perf}}}(\MM_0,\MM_1)) \iso \RR \otimes H^0(\mathit{hom}_{A^{\mathit{perf}}}(M_0,M_1)),
\end{equation}
and (for a suitable choice of relative connection) the induced connections on these morphism spaces are trivial.
\end{corollary}

\begin{proof}
Define $\MM = M \otimes_A \PP$. By our previous discussion, this follows $[\gamma]$ and has the required behaviour at the fibre over $s$. Moreover, if we make the obvious choice of relative connections, for each morphism $[b] \in H^0(\mathit{hom}_{A^{\mathit{perf}}}(M_0,M_1))$ we have a covariantly constant section $[b \otimes e_{\PP}] \in H^0(\mathit{hom}_{\AA^{\mathit{perf}}}(\MM_0,\MM_1))$, which specializes to $[b]$ at the point $s$. This establishes the remaining part of the statement.
\end{proof}

\subsection{Uniqueness\label{subsec:unique}}
As before, we work with a fixed deformation field $[\gamma]$. As an aid to intuition, we will increasingly use geometric language. Take two points $s,s' \in \SS$. One can envisage a process of {\em moving objects of $A^{\mathit{perf}}$ along the deformation field} from $s$ to $s'$. Namely, start with some object $M$, and suppose that there is a perfect family of modules $\MM$ following $[\gamma]$, whose fibre at $s$ is quasi-isomorphic to $M$. Then, take the fibre $M' = \MM_{s'}$. Generally speaking, no such family may exist, making it impossible to carry out the process at all. However, assuming existence, there is a good uniqueness statement at least for a certain class of objects $M$. Suppose from now on that the following holds:

\begin{assumption} \label{th:augmented}
$H^0(\mathit{hom}_{A^{\mathit{perf}}}(M,M))$ is a commutative ring. Moreover, the ideal of nilpotent elements in that ring has codimension $1$.
\end{assumption}

\begin{proposition} \label{th:uniqueness-1}
If $M$ satisfies Assumption \ref{th:augmented} and can be moved along the deformation field from $s$ to $s'$, the outcome $M'$ is unique up to quasi-isomorphism (which means independent of the family $\MM$).
\end{proposition}

The proof is based on a number of elementary observations. For the sake of brevity, let's write $\HH = H^0(\mathit{hom}_{\AA^{\mathit{perf}}}(\MM,\MM))$.  Choose a relative connection $\nabla_{\MM}$ on our family, and write $\nabla_{\HH}$ for the induced connection on $\HH$.

\begin{claim*}
$\HH$ is a commutative $\RR$-algebra.
\end{claim*}

\begin{proof}
By \eqref{eq:multiplicative-connections}, the image of the commutator map
\begin{equation}
\HH \otimes_\RR \HH \longrightarrow \HH, \;\; x \otimes y \longmapsto xy-yx
\end{equation}
is a subsheaf of $\HH$ invariant under $\nabla_\HH$, which therefore must be locally free. By looking at the point $s$, one sees that this sheaf must be zero.
\end{proof}

Note that, because of the commutativity and \eqref{eq:modify-connection}, $\nabla_{\HH}$ is actually independent of the choice of relative connection on $\MM$.

\begin{claim*}
Consider the ideal $\HH^{\mathit{nil}} \subset \HH$ of nilpotent elements. Then $\HH^{\mathit{nil}}$ is preserved by $\nabla_\HH$, and $\HH/\mathit{\HH}^{\mathit{nil}}$ is the trivial line bundle.
\end{claim*}

\begin{proof}
Choose a tangent vector field $\xi$ on $\SS$. For any element $h \in \HH$ and any $m>0$, we have
\begin{equation}
\nabla_{\HH,\xi}^m (h^m) \in m! (\nabla_{\HH,\xi} h)^m + h\HH.
\end{equation}
Choosing $h \in \HH^{\mathit{nil}}$ and $m$ large, one sees that $\HH^{\mathit{nil}}$ is closed under $\nabla_{\HH}$. By the same reasoning as before, $\HH/\HH^{\mathit{nil}}$ must be a line bundle. But the identity endomorphism yields a nowhere vanishing section, which provides a trivialization.
\end{proof}

Now suppose that we have two perfect families $\MM_+$ and $\MM_-$ both following $[\gamma]$ and whose fibres at $s$ are quasi-isomorphic to the same object $M$, which still satisfies Assumption \ref{th:augmented}. Working on the cohomology level as before, we denote by $\HH_{\pm}$ the endomorphism rings of these objects, by $\HH_{\pm}^{\mathit{nil}}$ the ideals of nilpotent endomorphisms, and by $\HH_{-+}$, respectively $\HH_{+-}$, the space of morphisms from $\MM_+$ to $\MM_-$, and vice versa. We choose relative connections on $\MM_{\pm}$, equipping all these morphism spaces with the induced connections.

\begin{claim*}
The multiplication maps
\begin{align}
\label{eq:compose-1} & \HH_{-+} \otimes \HH_{+}^{\mathit{nil}} \longrightarrow \HH_{-+}, \\
\label{eq:compose-2} & \HH_{-}^{\mathit{nil}} \otimes \HH_{-+} \longrightarrow \HH_{-+}
\end{align}
both have the same image, which we denote by $\HH_{-+}^{\mathit{nil}}$. This is preserved by the connection, and $\HH_{-+}/\HH_{-+}^{\mathit{nil}}$ is a line bundle.
\end{claim*}

\begin{proof}
Let's first define $\HH_{-+}^{\mathit{nil}}$ to be the image of the first map \eqref{eq:compose-1}. By compatibility with connections, $\HH_{-+}/\HH^{\mathit{nil}}_{-+}$ must be locally free, hence in view of the behaviour at the point $s$ a line bundle. Now take \eqref{eq:compose-2} and compose it with projection to $\HH_{-+}/\HH_{-+}^{\mathit{nil}}$. Again by the same argument, the composition vanishes identically, hence the image of \eqref{eq:compose-2} is contained in that of \eqref{eq:compose-1}. Running the argument the other way yields the required equality.
\end{proof}

Let's define $\HH_{+-}^{\mathit{nil}} \subset \HH_{+-}$ in the same way. It follows directly from the definition that the composition $\HH_{+-} \otimes \HH_{-+} \rightarrow \HH_+$ takes $\HH_{+-} \otimes \HH_{-+}^{\mathit{nil}}$ and $\HH_{+-}^{\mathit{nil}} \otimes \HH_{-+}$ to $\HH_+^{\mathit{nil}}$, and the same is true in the other order.

\begin{claim*}
Multiplication induces isomorphisms
\begin{equation} \label{eq:multiplication-mod-nilpotents}
\begin{aligned}
& (\HH_{+-}/\HH_{+-}^{\mathit{nil}}) \otimes (\HH_{-+}/\HH_{-+}^{\mathit{nil}}) \longrightarrow \HH_{+}/\HH_{+}^{\mathit{nil}}, \\
& (\HH_{-+}/\HH_{-+}^{\mathit{nil}}) \otimes (\HH_{+-}/\HH_{+-}^{\mathit{nil}}) \longrightarrow \HH_{-}/\HH_{-}^{\mathit{nil}}.
\end{aligned}
\end{equation}
\end{claim*}

\begin{proof}
We already established the well-definedness of these maps. Both sides are line bundles and carry connections, which are compatible with the maps, and at the fibre at $s$ we get isomorphisms.
\end{proof}

\begin{claim*}
There is a line bundle $\FF$ such that $\MM_-$ is quasi-isomorphic to $\FF \otimes \MM_+$.
\end{claim*}

\begin{proof}
Since we are free to tensor $\MM_+$ with a line bundle, we may assume without loss of generality that $(\HH_{-+}/\HH_{-+}^{\mathit{nil}})$ is the trivial line bundle. By \eqref{eq:multiplication-mod-nilpotents}, the same must then be true for $(\HH_{+-}/\HH_{+-}^{\mathit{nil}})$. Choose trivializations and lift them to sections of $\HH_{-+}$ and $\HH_{+-}$, respectively (recall that we are working over an affine curve, so there is no problem in doing this). The product of these in either order yields invertible elements of $\HH_+$ and $\HH_-$.
\end{proof}

The last-mentioned claim clearly establishes Proposition \ref{th:uniqueness-1}.

\begin{remark}
Here is a slightly weaker uniqueness statement, which does not require Assumption \ref{th:augmented}. Suppose that $\MM_+$ and $\MM_-$ are perfect families following $[\gamma]$, and whose fibres at $s$ are quasi-isomorphic to $M$. The composition maps $\HH_{+-} \otimes \HH_{-+} \rightarrow \HH_+$ and $\HH_{-+} \otimes \HH_{+-} \rightarrow \HH_-$ are onto. Specializing to any other fibre $s'$, one finds that $\MM_{+,s'}$ is quasi-isomorphic to a direct summand of a finite direct sum of copies of $\MM_{-,s'}$, and vice versa.
\end{remark}

We want to take a similar approach to morphisms. Let $M_k$ ($k = 0,1$) be objects of $A^{\mathit{perf}}$, and $B \subset \mathit{Hom}_{H^0(A^{\mathit{perf}})}(M_0,M_1)$ a one-dimensional subspace. Suppose that $\MM_k$ ($k = 0,1$) are perfect families which follow $[\gamma]$, and with $\MM_{k,s} \iso M_k$. Choose relative connections on them. Suppose also that there is a line bundle $\BB \subset \mathit{Hom}_{H^0(\AA^{\mathit{perf}})}(\MM_0, \MM_1)$ whose fibre at $s$ equals $B$, and which is invariant under the induced connection (if such a $\BB$ exists, it is unique). We can then restrict the given data to the fibre $s'$, yielding $B' \subset \mathit{Hom}_{H^0(A^{\mathit{perf}})}(M_0',M_1')$. We say that $B'$ is obtained from $B$ {\em by parallel transport}. Obviously, this can't be unique unless Proposition \ref{th:uniqueness-1} applies to both objects $M_k$, but in fact we will need more than that:

\begin{assumption} \label{th:augmented-plus}
Both $M_k$ satisfy Assumption \ref{th:augmented}, and the two multiplication maps
\begin{equation} \label{eq:swap-sides}
\begin{aligned}
& B \otimes \mathit{Hom}_{H^0(A^{\mathit{perf}})}(M_0,M_0) \longrightarrow \mathit{Hom}_{H^0(A^{\mathit{perf}})}(M_0,M_1), \\
& \mathit{Hom}_{H^0(A^{\mathit{perf}})}(M_1,M_1) \otimes B \longrightarrow \mathit{Hom}_{H^0(A^{\mathit{perf}})}(M_0,M_1)
\end{aligned}
\end{equation}
have the same image.
\end{assumption}

\begin{proposition} \label{th:uniqueness-2}
Suppose that $(M_0,M_1,B)$ satisfy Assumption \ref{th:augmented-plus}, and that parallel transport to $s'$ yields $(M_0',M_1',B')$. Then, this is unique up to quasi-isomorphism (independent of the choice of families, and of the relative connections).
\end{proposition}

We can apply some preliminary simplifications. One can change any of the families $\MM_k$ by tensoring it with a line bundle $\FF_k$ (equipped with a connection), and then take the corresponding line bundle $\FF_1 \otimes \FF_0^\vee \otimes \BB \subset \mathit{Hom}_{H^0(\AA^{\mathit{perf}})}(\FF_0 \otimes \MM_0, \FF_1 \otimes \MM_1)$. This does not affect the outcome of parallel transport. With this and the results of the previous proof in mind, the choice of families is indeed irrelevant, so we can consider some fixed choices $\MM_k$.

\begin{claim*}
The two multiplication maps
\begin{equation} \label{eq:swap-sides-2}
\begin{aligned}
& \BB \otimes \mathit{Hom}_{H^0(\AA^{\mathit{perf}})}(\MM_0,\MM_0) \longrightarrow \mathit{Hom}_{H^0(\AA^{\mathit{perf}})}(\MM_0,\MM_1), \\
& \mathit{Hom}_{H^0(\AA^{\mathit{perf}})}(\MM_1,\MM_1) \otimes \BB \longrightarrow \mathit{Hom}_{H^0(\AA^{\mathit{perf}})}(\MM_0,\MM_1)
\end{aligned}
\end{equation}
have the same image. \qed
\end{claim*}

The proof is routine. We denote the image by $\JJ$.

\begin{claim*}
$\JJ$ is preserved by the connection, and this remains true if we change the relative connections on $\MM_k$.
\end{claim*}

\begin{proof}
The first statement is obvious from the definition and the corresponding property of $\BB$. The second one follows from this and \eqref{eq:modify-connection}, because left and right multiplication with endomorphisms of $\MM_k$ preserves $\JJ$.
\end{proof}

Now, suppose that we have made different choices of relative connections $\nabla_{\MM_0,\pm}$ and $\nabla_{\MM_1,\pm}$, leading to two different line bundles $\BB_\pm$, and associated subbundles $\JJ_{\pm}$.

\begin{claim*}
In fact, $\JJ_+ = \JJ_-$.
\end{claim*}

\begin{proof}
Choose one of the two connections on $\mathit{Hom}_{H^0(\AA^{\mathit{perf}})}(\MM_0,\MM_1)$ arising from our choices. Both $\JJ_+$ and $\JJ_-$ are invariant under this connection, and they agree at one point.
\end{proof}

Specializing to the fibres at $s'$, this means that $B_-'$ is contained in the image of the multiplication map $\mathit{Hom}_{H^0(A^{\mathit{perf}})}(M_1',M_1') \otimes B_+' \rightarrow \mathit{Hom}_{H^0(A^{\mathit{perf}})}(M_0',M_1')$, and vice versa. Hence, one can write $B_-' = x_{-+}' B_+'$ and $B_+' = x_{+-}' B_-'$ for some $x_{-+}',x_{+-}' \in \mathit{Hom}_{H^0(A^{\mathit{perf}})}(M_1',M_1')$. Since the endomorphism ring is commutative and $B_+' = x_{+-}'x_{-+}' B_+'$, none of the $x_{-+}',x_{+-}'$ can be nilpotent. Arguing as in the proof of Proposition \ref{th:uniqueness-1}, they must then be invertible. What we have shown is that $B_-'$ can be obtained from $B_+'$ by applying automorphisms of the $M_k'$, which is indeed what was claimed in Proposition \ref{th:uniqueness-2}.

\subsection{Periods\label{subsec:flux}}
Fix a smooth elliptic curve $\bar{\SS}$ over $R$, together with a differential $\bar\theta \in H^0(\bar{\SS},\Omega^1_{\bar{\SS}})$. Given any nonempty affine open subset $\SS \subset \bar{\SS}$, we consider the one-form $\theta = \bar\theta|\SS$.

\begin{workingdefinition} \label{th:per}
A class $[g] \in \mathit{HH}^1(A,A)$ is called {\em periodic} if the following holds. For every $X \in \mathit{Ob}\,A^{\mathit{perf}}$ there is a subset $\SS$ as before, as well as a perfect family $\MM$ over $\SS$ which follows $[\gamma] = \theta \otimes [g]$, and whose fibre at some point $s \in \SS$ is quasi-isomorphic to $X$. Moreover, given any two objects $(X_0,X_1)$, one can choose families $(\MM_0,\MM_1)$ as before with quasi-isomorphisms $\MM_{k,s} \iso X_k$, so that the bundle of cohomology level morphisms is trivial:
\begin{equation} \label{eq:triv-h0}
H^0(\mathit{hom}_{\AA^{\mathit{perf}}}(\MM_0,\MM_1)) \iso \RR \otimes_R H^0(\mathit{hom}_A(X_0,X_1)).
\end{equation}
In addition, one should be able to choose relative connections on the families so that the induced connection is trivial, which means compatible with a trivialization \eqref{eq:triv-h0}. The subset of periodic classes is denoted by
\begin{equation} \label{eq:periodic-set}
\mathit{Per}(A,\bar{\SS},\bar\theta) \subset \mathit{HH}^1(A,A).
\end{equation}
\end{workingdefinition}

Note that this is really an invariant of $A^{\mathit{perf}}$ up to quasi-isomorphism. By pulling back a given family $\XX$ by the $n$-th power map $\bar{\SS} \rightarrow \bar{\SS}$ (defined by taking $s$ to be the origin) for some $n \in \bZ$, one gets a family which follows the restriction of $n[(\bar\theta|\SS) \otimes g]$. This proves that \eqref{eq:periodic-set} is closed under multiplication by integers.

\begin{remark} \label{th:downsides}
The notion introduced above is called ``working definition'' in view of its numerous shortcomings. These deserve some discussion, even though they do not stand in the way of our immediate application.

The first and most obvious point is the object-by-object approach we've taken. This violates categorical common sense and manners, and is likely to be the reason why we can't prove that the set of periodic classes is an abelian group. However, it is not difficult to envisage a more universal approach, based on Corollary \ref{th:universal-family}; it is maybe more appropriate to think of it as an $A_\infty$-version of the {\em derived Picard group} \cite{yekutieli99,yekutieli04,keller04}.

The next point is the use of a priori undetermined open affine subsets $\SS \subset \bar{\SS}$, which essentially means that we are working (Zariski) locally around the base point $s$. This reflects the insufficient technical sophistication of our definitions of families, which are not local over the base. This is a major inconvenience, but does not lead to a decisive loss of information (see Lemma \ref{th:rational-section}).

There is one more issue which is conceptually by far the most important one. Asking for families parametrized by an elliptic curve amounts to a double periodicity requirement, but single periodicity seems a more fundamental notion. Instead of a ``torus'' one would then want a ``thin annulus'' as a parameter space. This makes sense in analytic geometry, either over $\bC$ or over a non-archimedean field. Such a theory would require extensive reworking of the foundations. On the other hand, if successful, it might allow substantial simplifications and extensions of the main arguments in this paper, bringing them in line with existing ideas about convergence in Floer cohomology \cite{fukaya02c,fukaya09}.
\end{remark}

The techniques from Section \ref{subsec:existence} can be used to show that a given Hochschild class is periodic. On the other hand, if one wants to show that $[g]$ is not periodic, the arguments from Section \ref{subsec:unique}, in combination with the following trick, can be useful.

\begin{assumption}
Suppose that for some $\SS \subset \bar{\SS}$ we have a smooth affine curve $\tilde{\SS}$ together with a morphism $\tilde\SS \rightarrow \SS$, and two points $\tilde{s}_\pm \in \tilde{\SS}$ whose image is the same $s \in \SS$. Let $\tilde\theta$ be the pullback of $\theta$. Suppose also that we have perfect families $\tilde{\MM}_0,\tilde{\MM}_1$ on $\tilde{\SS}$ following $\tilde\theta \otimes [g]$. Equip them with relative connections, and suppose further that $\tilde\BB \subset \mathit{Hom}_{H^0(\AA^{\mathit{perf}})}(\tilde{\MM}_0, \tilde{\MM}_1)$ is a line bundle preserved by the induced connection. By restriction to the fibres at $\tilde{s}_{\pm}$, we get objects $M_{0,\pm}$ and $M_{1,\pm}$, as well as morphism subspaces $B_{\pm}$. We require that $(M_{0,+},M_{1,+},B_+)$ should satisfy Assumption \ref{th:augmented-plus}.
\end{assumption}

\begin{lemma} \label{th:contradiction}
In the situation set up above, assume additionally that $(M_{0,+},M_{1,+},B_+)$ is not isomorphic to $(M_{0,-},M_{1,-},B_-)$. Then $[g]$ is not periodic (for the original $\bar{\SS}$ and $\bar{\theta}$).
\end{lemma}

\begin{proof}
Assume that $[g]$ is in fact periodic. Then we can find families $\MM_0$ and $\MM_1$ over some open subset $\SS$, whose fibre at some point $s$ is $M_{0,+}$ and $M_{1,+}$, respectively. A priori, the open subset and the point do not  have to coincide with those that appeared in the statement of the Lemma. However, that discrepancy can be removed by using the group structure of the elliptic curve $\bar{\SS}$ (which yields translations acting transitively on points, and preserving $\bar\theta$), and by making the open subsets smaller if necessary. Having resolved that issue, we continue the discussion: by definition, the families can be chosen so that $\mathit{Hom}_{H^0(\AA^{\mathit{perf}})}(\MM_0,\MM_1)$ is the trivial bundle and carries the trivial connection. This allows one to find a line bundle $\BB$ inside that morphism space, which is compatible with the connection and has fibre $B_+$ at $s$. Now pull back those families to $\tilde{\SS}$. Comparison with $(\tilde\MM_0,\tilde\MM_1,\tilde\BB)$ shows that Proposition \ref{th:uniqueness-2} is violated, since the two choices of families are isomorphic at $s_+$ but not at $s_-$.
\end{proof}

\subsection{Relaxing the assumptions}
To conclude the abstract part of our discussion, we want to enlarge the existing framework in two minor ways. The first one is to pass from $\bZ$-gradings to $(\bZ/2)$-gradings. We then need a version of Lemma \ref{th:proj} for $\bZ/2$-graded complexes $\FF_0$ and $\FF_1$, but that is unproblematic: one passes to the associated $\bZ$-graded periodic complexes and applies the original form of the Lemma to those, thereby deriving the desired result.

The other generalization is to allow $A_\infty$-categories $A$ which are only cohomologically unital (but still proper). The first effect of this is on twisted complexes, where we have to prove:

\begin{lemma}
$A^{\mathit{tw}}$ and $\AA^{\mathit{tw}}$ are cohomologically unital.
\end{lemma}

\begin{proof}
Let $X$ be a twisted complex written as in \eqref{eq:formal-sum}. There is a spectral sequence (convergent after finitely many steps) which leads to $H(\mathit{hom}_{A^{\mathit{tw}}}(X,X))$ and starts with
\begin{equation}
E_1 = \bigoplus_{ij} \mathit{Hom}(F^i,F^j) \otimes H(\mathit{hom}_A(X^i,X^j))[\sigma^i-\sigma^j].
\end{equation}
Moreover, this spectral sequence is multiplicative, which implies that the identity element in the $E_1$ page always survives. 
This yields a degree zero endomorphism $[u_X]$, with the property (by a spectral sequence comparison theorem) that the maps
\begin{equation}
\begin{aligned}
& [a] \longmapsto [u_X] \cdot [a]: H(\mathit{hom}_{A^{\mathit{tw}}}(Y,X)) \longrightarrow H(\mathit{hom}_{A^{\mathit{tw}}}(Y,X)), \\
& [a] \longmapsto [a] \cdot [u_X]: H(\mathit{hom}_{A^{\mathit{tw}}}(X,Y)) \longrightarrow H(\mathit{hom}_{A^{\mathit{tw}}}(X,Y))
\end{aligned}
\end{equation}
are isomorphisms for any $Y$. In particular, there is an $[e_X]$ which satisfies $[e_X] \cdot [u_X] = [u_X]$, and then automatically also $[u_X] \cdot [e_X] = [u_X]$. One easily checks that this is the required cohomological identity. This argument extends to families without any problems.
\end{proof}

\begin{remark} \label{th:unit-h-unit}
In the case of families of twisted complexes, $\mu^2_{\AA^{\mathit{tw}}}(\cdot,e_\XX): \mathit{hom}_{\AA^{\mathit{tw}}}(\XX,\YY) \rightarrow \mathit{hom}_{\AA^{\mathit{tw}}}(\XX,\YY)$ induces the identity on cohomology, and is chain homotopic to its own square (by the $A_\infty$-equations). Therefore, it is actually chain homotopic to the identity, by Lemma \ref{th:homotopy-idempotent}, and the same holds on the other side. This is a slightly stronger property than cohomological unitality, and generally more appropriate for $A_\infty$-categories defined over a ring.
\end{remark}

Unfortunately, the theory of pre-connections on twisted complexes does not generalize to the cohomologically unital context in an obvious way, so we'll have to be careful to use that only for strictly unital $A_\infty$-categories.

The situation for modules is slightly different, since unitality requirements enter into the definition of the objects and morphisms themselves. Given a cohomologically unital $A$, one defines $A^{\mathit{mod}}$ by taking cohomologically unital modules with finite cohomology as objects, and arbitrary module homomorphisms as morphisms. If $A$ was strictly unital, this would yield a category quasi-equivalent to the previously considered version using strictly unital modules and homomorphisms \cite[Section 3.3]{lefevre}. The same construction for families defines $\AA^{\mathit{mod}}$, and Lemma \ref{th:acyclic-modules} still holds.

\begin{remark}
The cohomological unitality condition on a family of modules $\MM$ says that if $e_X \in \mathit{hom}_A^0(X,X)$ is a representative for the cohomology unit in $A$, then $\mu^2_\MM(\cdot,e_X): \MM(X) \longrightarrow \MM(X)$ induces the identity on cohomology. Arguing as in Remark \ref{th:unit-h-unit}, one sees that this map is in fact chain homotopic to the identity, which would again be the more natural condition in general (but as we've seen, turns out to be equivalent in our context).
\end{remark}

\begin{lemma}
The Yoneda embeddings $A \rightarrow A^{\mathit{mod}}$ and $\AA \rightarrow \AA^{\mathit{mod}}$ are quasi-isomorphisms.
\end{lemma}

\begin{proof}
Take the same maps
\begin{equation}
\mathit{hom}_A(Y_0,Y_1) \longrightarrow \mathit{hom}_{A^{\mathit{mod}}}(Y^{\mathit{yon}}_0,Y_1^{\mathit{yon}}) \longrightarrow \mathit{hom}_A(Y_0,Y_1)
\end{equation}
as in Lemma \ref{th:yoneda}, using any representative $e_{Y_0}$ for the cohomological unit. Composition in the given order is the endomorphism $a \mapsto \mu^2_A(a,e_{Y_0})$, which by definition acts as the identity on cohomology. Take the composition in the opposite order and add the coboundary of the homotopy \eqref{eq:h-homotopy}. The outcome is the map
\begin{equation} \label{eq:k-homotopy}
\begin{aligned}
& k: \mathit{hom}_{A^{\mathit{mod}}}(Y_0^{\mathit{yon}},Y_1^{\mathit{yon}}) \longrightarrow \mathit{hom}_{A^{\mathit{mod}}}(Y_0^{\mathit{yon}},Y_1^{\mathit{yon}}), \\
& k(b)^d(a_d,\dots,a_1) = \sum_i (-1)^{|a_{i+1}| + \cdots + |a_d| + d - i + 1} b^{i+1}(\mu_A^{d-i+1}(e_{Y_0},a_d,\dots, a_{i+1}),a_i,\dots,a_1).
\end{aligned}
\end{equation}
By looking at the length filtration \eqref{eq:length-sequence}, one sees that this is a quasi-isomorphism. The same thing applies to constant families.
\end{proof}

\begin{remark}
As before, for a strictly unital $A$ we now have a two version of $\AA^{\mathit{mod}}$, one defined in a strictly unital context, and the other by treating $A$ as cohomology unital. Unfortunately, the previously quoted result from \cite{lefevre} does not immediately extend to this context, since it relies on minimal models for modules, which only exist over a field (and the alternative approach from \cite[Section 2]{seidel04} has not been extended to modules so far). It seems highly plausible that the two versions are still quasi-equivalent, but we will allow ourselves to sweep the issue under the rug. In fact, in all our applications what counts are the subcategories of perfect families $\AA^{\mathit{perf}}$, where this problem does not arise (because in both contexts they are split-generated by constant families).
\end{remark}

Finally, in the cohomologically unital context, one similarly wants to adjust the notion of bimodule, and use the full Hochschild complex instead of the reduced one.


\section{The two-torus\label{sec:elliptic}}

In this section we consider a specific finite-dimensional algebra, together with its $A_\infty$-deforma\-tions. The algebra occurs geometrically in connection with degree $2$ line bundles on elliptic curves, and its $A_\infty$-deformations yield one possible model for the derived category of coherent sheaves on such curves. In view of homological mirror symmetry (see \cite{kontsevich94} for the general statement, and \cite{polishchuk-zaslow98, kontsevich-soibelman00, fukaya02b, abouzaid-smith09, lekili-perutz11} for the case relevant here), the same structure describes the Fukaya category of the two-torus. Our aim is to construct a particular family of objects, which in terms of the elliptic curve is the tautological family of structure sheaves of its points, and in terms of the Fukaya category is a family of parallel lines on the torus (the connection between the two could be made directly via SYZ transformations, as in \cite{kontsevich-soibelman00, fukaya02c}). The point of the exercise is to see how this fits in with the technical notions of family given in the previous section. This is not entirely straightforward, since the $K$-theory class varies, which precludes a description as family of twisted complexes.

Initially, we will work over an arbitrary field $R$ of characteristic $0$. Later on, when considerations become more geometric, we will re-introduce the added assumption that $R$ be algebraically closed. In the last parts, we will specialize this further to the simplest (one-variable) Novikov field from Floer theory, namely
\begin{equation} \label{eq:novikov-field}
R = \left\{
\begin{aligned} & u = c_0 \hbar^{m_0} + c_1 \hbar^{m_1} + \cdots, \\
& \text{where } c_k \in \bC, \;\; m_k \in \bR, \text{ and } \lim_{k \rightarrow \infty} m_k = +\infty \end{aligned}\right\}.
\end{equation}
Note that, since the coefficient field $\bC$ is algebraically closed, so is $R$ \cite[Appendix]{fukaya-oh-ohta-ono10}. The sign conventions used in constructing the Fukaya category of the two-torus follow \cite[Section 13]{seidel04}.

\subsection{Koszul algebras\label{subsec:koszul}}
Let $W$ be a finite-dimensional graded $R$-vector space. A {\em quadratic algebra} is an associative unital graded $R$-algebra of the form $A = TW/J$, where $TW$ is the tensor algebra, and $J \subset W \otimes W$ is a graded linear subspace. Let $J^\perp \subset W^\vee \otimes W^\vee$ be the orthogonal complement of $J$ with respect to the canonical pairing
\begin{equation}
\begin{aligned}
& W^\vee \otimes W^\vee \otimes W \otimes W \longrightarrow R, \\
& w_2^\vee \otimes w_1^\vee \otimes w_1 \otimes w_2 \longmapsto (-1)^{|w_2|} w_2^\vee(w_2)w_1^\vee(w_1).
\end{aligned}
\end{equation}
The {\em quadratic dual} of $A$ is defined as $A^! = T(W^\vee[-1])/J^\perp[-2]$. The Koszul complex is the graded vector space $A^! \otimes A$ with differential
\begin{equation}
x^! \otimes x \longmapsto \sum_r (-1)^{|x|} x^! w_r^\vee \otimes w_r x,
\end{equation}
where $\{w_r\}$ is a basis of $W$, and $\{w_r^\vee\}$ the dual basis. One says that $A$ is a {\em Koszul algebra} if the Koszul complex is acyclic.

There is also a more abstract formulation. Consider the abelian category of graded left $A$-modules, and in it the simple module $R$. We then have a bigraded group $\mathit{Ext}^i_A(R,R[j])$, where $i$ is the cohomological grading and $j$ the internal one, inherited from the grading of $A$ itself. For instance,
\begin{equation} \label{eq:low-ext}
\begin{aligned}
& \mathit{Ext}^0_A(R,R[j]) = \begin{cases} R & j = 0, \\ 0 & \text{otherwise,} \end{cases} \\
& \mathit{Ext}^1_A(R,R[j]) \iso (W^\vee)^j, \\
& \mathit{Ext}^2_A(R,R[j]) \iso (J^\vee)^j.
\end{aligned}
\end{equation}
In our case $A$ has an extra grading by pathlength, and this induces another grading on each $\mathit{Ext}^i_A(R,R[j])$. For low values of $i$, one sees from \eqref{eq:low-ext} that $\mathit{Ext}^i$ is concentrated in path length $i$. Then, $A$ is Koszul if and only if the same holds for all $i$ 
(the original reference is \cite{priddy70}; for more recent expositions in slightly varying degrees of generality, see \cite{beilinson-ginsburg-schechtman88, froberg89, beilinson-ginzburg-soergel96}).

\begin{addendum} \label{th:double-grading}
Even though we have kept track of some signs arising from the grading of $W$, these are actually irrelevant for the purpose of determining whether $A$ is Koszul or not, as the following trick shows. Let $\tilde{A}$ be the algebra obtained from $A$ by multiplying the given grading of $W$ by 2. The category of graded $A$-modules embeds fully and faithfully into that of graded $\tilde{A}$-modules in the same way, which we denote by $M \mapsto \tilde{M}$. This doubles the amount of shift,
\begin{equation} \label{eq:double-shift}
\widetilde{M[j]} \iso \widetilde{M}[2j].
\end{equation}
Using that and any projective resolution, one sees that
\begin{equation}
\mathit{Ext}^i_{\tilde{A}}(\tilde{M},\tilde{N}[j]) = \begin{cases} \mathit{Ext}^i_{A}(M,N[j/2]) & \text{if $j$ is even}, \\
0 & \text{otherwise.}
\end{cases}
\end{equation}
Moreover, the isomorphism is compatible with path length. Hence, $A$ is Koszul if and only if $\tilde{A}$ is. Once one has done this change, it is clear that the same will hold even if one changes the grading to be trivial (concentrated in degree zero), since the Koszul complex is not affected by even changes in the grading.
\end{addendum}

The Hochschild cohomology of a Koszul algebra $A$ can be computed \cite{green-hartman-marcos-solberg05} as the cohomology of $A^! \otimes A$ with a modified differential
\begin{equation} \label{eq:koszul-hochschild}
x^! \otimes x \longmapsto \sum_r (-1)^{|x|} x^! w_r^\vee \otimes w_r x - (-1)^{(|w_r|+1)(|x|+|x^!|)} w_r^\vee x^! \otimes x w_r.
\end{equation}
More precisely, the Hochschild cohomology of any graded algebra $A$ is a bigraded vector space $\mathit{HH}^i(A,A[j])$, where again $i$ is the cohomological grading, and $j$ the internal one. For Koszul algebras, $i+j$ corresponds to the natural grading of $A^! \otimes A$, whereas $i$ measures path lengths in the $A^!$ factor of \eqref{eq:koszul-hochschild}.

\begin{example} \label{th:truncated}
Take a free algebra $A = R\langle w \rangle$, where $w$ has odd degree $d$. The Koszul dual is a truncated polynomial algebra $A^! = R[w^\vee]/(w^\vee)^2$, where $w^\vee$ has even degree $1-d$. The Hochschild cohomology has a basis consisting of $1 \otimes w^i$ with $i$ even, and $w^\vee \otimes w^i$ with $i$ odd. This contrasts with the case of even $d$, where the differential \eqref{eq:koszul-hochschild} vanishes.

It is useful to consider this example in the context of Addendum \ref{th:double-grading}. In the category of graded $A$-bimodules, we have
\begin{equation} \label{eq:hh-ext}
\mathit{HH}^i(A,A[j]) = \mathit{Ext}_{A \otimes A^{\mathit{opp}}}^i(A,A[j])
\end{equation}
where the shifted space $A[j]$ has the bimodule structure which is given by ordinary multiplication on the left side, and twisted multiplication $(-1)^{j|x|} xy$ on the right side. If we take $A$ and double its given grading to $\tilde{A}$, the category of graded $A$-bimodules embeds into that of $\tilde{A}$-bimodules in the obvious way. However, this embedding fails to be compatible with shifts, so that the analogue of \eqref{eq:double-shift} for the groups \eqref{eq:hh-ext} holds only if $j$ is even. 
\end{example}

\subsection{A Hochschild cohomology computation\label{subsec:hh}}
Consider the graded path algebra associated to the quiver
\begin{equation} \label{eq:quiver}
\xymatrix{\stackrel{1}{\bullet} \ar@<.5em>@/^2pc/[rr]^-{w_1} \ar@/^2pc/[rr]_-{w_2} && \ar@/^2pc/[ll]_-{w_3} \ar@<.5em>@/^2pc/[ll]^-{w_4} \stackrel{2}{\bullet}}
\end{equation}
where $w_1,w_2$ have degree $0$, and $w_3,w_4$ degree $1$. Composition of paths will be written from right to left, so the path $w_3w_1$ means going first along $w_1$ and then $w_3$.

\begin{definition} \label{th:q-algebra}
The graded algebra $Q$ is the quotient of the path algebra of \eqref{eq:quiver} obtained by imposing the relations
\begin{equation} \label{eq:quadratic-relations}
w_3w_2 + w_4w_1 = 0, \quad w_1w_4 + w_2w_3 = 0, \quad w_3w_1 = w_4w_2 = 0.
\end{equation}
\end{definition}

Here's an alternative description. Let $e_1,e_2 \in Q$ be the idempotents associated to the length $0$ paths. Take $V = R^2$. One can identify
\begin{equation}\label{eq:identifications} \left\{
\begin{aligned}
& e_2 Q e_1 = V && \text{using $w_1,w_2$ as a basis,} \\
& e_1 Q e_2 = V && \text{using $w_3,w_4$ as a basis,} \\
& e_1 Q e_1 = \Lambda^{\mathit{even}}V = \Lambda^0 V \oplus \Lambda^2 V && \text{using $e_1,q_1 = w_3w_2$ as a basis,} \\
& e_2 Q e_2 = \Lambda^{\mathit{even}}V = \Lambda^0 V \oplus \Lambda^2 V && \text{using $e_2,q_2 = w_1w_4$ as a basis.}
\end{aligned} \right.
\end{equation}
where of course $\Lambda^0 V = R$. With respect to these identifications, both nontrivial multiplications
\begin{equation} \label{eq:wedge}
\begin{aligned}
& e_1 Q e_2 \otimes e_2 Q e_1 \longrightarrow e_1 Q e_1, \\
& e_2 Q e_1 \otimes e_1 Q e_2 \longrightarrow e_2 Q e_2
\end{aligned}
\end{equation}
equal the ordinary wedge product $V \otimes V \rightarrow \Lambda^2 V$.

The entire theory of Koszul algebras can be carried out over any semisimple base algebra, such as $R_2 = R e_1 \oplus R e_2$ (as already noticed in \cite{beilinson-ginsburg-schechtman88}). With respect to the exposition in Section \ref{subsec:koszul}, the main change needed is that all tensor products should be taken over the base algebra. Our original description \eqref{eq:quadratic-relations} shows that $Q$ is quadratic in these terms, and in fact:

\begin{lemma}
$Q$ is Koszul.
\end{lemma}

\begin{proof}
Consider the $\bZ/2$ action on $V$, where the nontrivial element acts by $-\mathit{Id}$. It follows from \eqref{eq:wedge} that
\begin{equation} \label{eq:lambda-q}
Q \iso \Lambda(V) \rtimes \bZ/2
\end{equation}
as an algebra over $R[\bZ/2] \iso R_2$. Of course, this isomorphism is not compatible with the grading of $Q$ and the natural grading of the exterior algebra. However, by Addendum \ref{th:double-grading} the discrepancy is irrelevant for deciding whether $Q$ is Koszul or not. On the other hand, the Koszulness of $\Lambda(V) \rtimes \bZ/2$ is well-known, being a minor variation on the basic case of the exterior algebra itself (for the dual case of polynomial algebras twisted by finite groups, see \cite[Section 2.7]{gordon08}).
\end{proof}

The quadratic dual $Q^!$ of $Q$ (again, taken over $R_2$) is based on the quiver
\begin{equation}
\label{eq:dual-quiver}
\xymatrix{\stackrel{1}{\bullet} \ar@<.5em>@/^2pc/[rr]^-{w_3^\vee} \ar@/^2pc/[rr]_-{w_4^\vee} && \ar@/^2pc/[ll]_-{w_1^\vee} \ar@<.5em>@/^2pc/[ll]^-{w_2^\vee} \stackrel{2}{\bullet}}
\end{equation}
where $w_1^\vee,w_2^\vee$ have degree $0$, and $w_3^\vee,w_4^\vee$ degree $1$. This is actually isomorphic to \eqref{eq:quiver} but we avoid making the identification, and in any case the relations defining $Q^!$ are different:
\begin{equation}
w_3^\vee w_2^\vee - w_4^\vee w_1^\vee = 0, \quad w_1^\vee w_4^\vee - w_2^\vee w_3^\vee = 0.
\end{equation}
As in \eqref{eq:lambda-q}, if one forgets the grading then $Q^! \iso \mathit{Sym}(V^\vee) \rtimes \bZ/2$.

\begin{lemma} \label{th:hh-computation-1}
$\mathit{HH}^i(Q,Q[2-i])$ vanishes for $i \geq 5$, and $\mathit{HH}^i(Q,Q[3-i])$ vanishes for $i \geq 7$.
\end{lemma}

\begin{proof}
We need to adapt the previous discussion slightly to the framework over $R_2$. The relevant complex computing the Hochschild cohomology is now
\begin{equation}
(Q^! \otimes_{R_2} Q)_{diag} = \bigoplus_{i,j=1,2} e_i Q^! e_j \otimes_R e_j Q e_i,
\end{equation}
where the differential \eqref{eq:koszul-hochschild} remains the same as before. {\em If $i$ is even}, any path in $Q^!$ of length $i$ has degree $i/2$, which implies that $\mathit{HH}^i(Q,Q[j]) = 0$ for $i+j < i/2$. {\em If $i$ is odd}, the minimal degree of a path of length $i$ in $Q^!$ is $(i-1)/2$, but the paths having this minimal degree all lie in $e_2 Q^! e_1$, whereas $e_1 Q e_2$ is concentrated in degree $1$. Therefore, $\mathit{HH}^i(Q,Q[j]) = 0$ for $i+j < (i+1)/2$.
\end{proof}

\begin{lemma} \label{th:hh-computation-2}
$\mathit{HH}^3(Q,Q[-1]) = 0$, and $\mathit{HH}^4(Q,Q[-2]) \iso \mathit{Sym}^4(V^\vee)$.
\end{lemma}

\begin{proof}
The $\mathit{HH}^3$ case can be carried out by an explicit calculation, which we omit (software for doing such calculations is available from the author's homepage). For $\mathit{HH}^4$ we can follow a more conceptual path. The argument from Example \ref{th:truncated} shows that
\begin{equation} \label{eq:hh-lambda}
\mathit{HH}^4(Q,Q[-2]) \iso \mathit{HH}^4(\Lambda(V) \rtimes \bZ/2, \Lambda(V) \rtimes \bZ/2[-4]).
\end{equation}
The Hochschild cohomology of $\Lambda(V) \rtimes \bZ/2$ is isomorphic (with a suitable adjustment in the bigrading) to that of its Koszul dual $Sym(V^\vee) \rtimes \bZ/2$. The particular group \eqref{eq:hh-lambda} corresponds to the length $4$ piece of the center of the Koszul dual, which is just $Sym^4(V^\vee)$.
\end{proof}

\subsection{Deformations}
An $A_\infty$-deformation of $Q$ is an $A_\infty$-structure $\{\mu^*\}$ which respects the grading and $R_2$-bimodule structure, and whose starting terms are
\begin{equation}
\mu^1(x) = 0, \;\; \mu^2(x_2,x_1) = (-1)^{|x_1|} x_2 x_1.
\end{equation}
In particular, $\{\mu^*\}$ is necessarily cohomologically unital. As part of the higher order product structure, we then have maps
\begin{equation} \label{eq:mu4}
\begin{aligned}
& \mu^4_{[12121]}: e_1 Q e_2 \otimes e_2 Q e_1 \otimes e_1 Q e_2 \otimes e_2 Q e_1 \iso V^{\otimes 4} \longrightarrow R e_1 \iso R, \\
& \mu^4_{[21212]}: e_2 Q e_1 \otimes e_1 Q e_2 \otimes e_2 Q e_1 \otimes e_1 Q e_2 \iso V^{\otimes 4} \longrightarrow R e_2 \iso R.
\end{aligned}
\end{equation}
where the identifications \eqref{eq:identifications} have been applied. Suppose that $\mu^3 = 0$. In that case, because of the $A_\infty$-equations, the two order $4$ expressions in \eqref{eq:mu4} must have the same symmetric part, which we denote by
\begin{equation} \label{eq:p-poly}
p(v) = \mu^4_{[12121]}(v,v,v,v) = \mu^4_{[21212]}(v,v,v,v) \in Sym^4(V^\vee).
\end{equation}

\begin{proposition} \label{th:classify-deformations}
(i) $A_\infty$-deformations of $Q$ satisfying $\mu^3 = 0$ are classified up to isomorphism by \eqref{eq:p-poly}: any polynomial can occur, and it determines the isomorphism class of the deformation. (ii) Any $A_\infty$-deformation of $Q$ is isomorphic to one which is strictly unital, and has $\mu^3 = 0$, $\mu^5 = 0$.
\end{proposition}

\begin{proof}
Most of this follows from Lemmas \ref{th:hh-computation-1} and \ref{th:hh-computation-2}, together with the general classification theory of $A_\infty$-deformations \cite{kadeishvili88} (compare also the discussion in \cite[Section 3]{seidel03b}). The vanishing of $\mathit{HH}^3(Q,Q[-1])$ tells us that any $A_\infty$-deformation is equivalent to one with $\mu^3 = 0$. For general reasons, $\mu^4$ then defines a class in $\mathit{HH}^4(Q,Q[-2])$. We know that this group is isomorphic to $\mathit{Sym}^4(V^\vee)$, and one can check (by explicitly comparing the standard Hochschild complex with the one coming from Koszul duality) that this isomorphism takes $[\mu^4]$ to the polynomial $p$ defined above. Since $\mathit{HH}^i(Q,Q[2-i])$ vanishes for all $i>4$, $[\mu^4]$ determines the isomorphism class of the $A_\infty$-deformation completely. The obstructions to existence lie in $\mathit{HH}^i(Q,Q[3-i])$ for $i \geq 7$, which vanish in our case. The last statement follows by inspection of the inductive procedure in which the previously mentioned obstruction groups appear: the first component beyond $\mu^4$ which appears is $\mu^6$, which is introduced to solve the $A_\infty$-associativity equation
\begin{multline}
\mu^2(x_7,\mu^6(x_6,\dots,x_1)) + (-1)^{|x_1|-1} \mu^2(\mu^5(x_7,\dots,x_2),x_1) +
\mu^6(x_7,\dots,x_3,\mu^2(x_2,x_1)) \\ + \cdots + (-1)^{|x_1|+\cdots+|x_5|-5} \mu^6(\mu^2(x_7,x_6),\dots,x_1) \\ =
-\mu^4(x_7,\dots,x_5,\mu^4(x_4,\dots,x_1)) \cdots - (-1)^{|x_1|+|x_2|+|x_3|-3}
\mu^4(\mu^4(x_7,\dots,x_4),x_3,\dots,x_1)).
\end{multline}
\end{proof}

\begin{definition}
For a given $p$, we denote by $Q_p$ the $A_\infty$-structure obtained by equipping $Q$ with the higher order products from Proposition \ref{th:classify-deformations}.
\end{definition}

$Q_p$ is an $A_\infty$-algebra over $R_2$, or equivalently an $A_\infty$-category with two objects $X_1,X_2$. By construction, it is strictly unital and has $\mu^3_{Q_p} = \mu^5_{Q_p} = 0$. Of course, it is unique only up to $A_\infty$-isomorphism.

\begin{addendum}
There is also an $A_\infty$-isomorphism $Q_p \iso Q_{\gamma^2 p}$ for any $\gamma \in R^\times$, obtained (for suitable choices on both sides) by multiplying the degree $k$ part of the algebra with $\gamma^k$. In particular, if $R$ is algebraically closed (or at least contains square roots), knowing $p$ up to nonzero multiples is sufficient to determine the $A_\infty$-algebra.
\end{addendum}

Let's briefly consider the Hochschild cohomology $\mathit{HH}^*(Q_p,Q_p)$ (unlike that of a graded algebra like $Q$, this carries a single grading). The length filtration of the Hoch\-schild complex yields a spectral sequence, starting with $E_2^{ij} = \mathit{HH}^i(Q,Q[j])$. Here is a picture of all the nonzero entries in the lines $i+j = 1$ and $i+j = 2$ of the $E_2$ page, obtained by the same techniques as Lemmas \ref{th:hh-computation-1} and \ref{th:hh-computation-2}:
\begin{equation} \label{eq:hh-table}
\begin{array}{r|ccccc}
j = 2 & 0 \\
j = 1 & \Lambda^2(V) \oplus \Lambda^2(V) & 0 \\
j = 0 && \mathit{End}(V) & \mathit{Sym}^2(V^\vee) \\
j = -1 &&& 0 & 0 \\
j = -2 &&&& 0 & \mathit{Sym}^4(V^\vee) \\
\hline
& i = 0 & i = 1 & i = 2 & i = 3 & i = 4
\end{array}
\end{equation}
Since $\mu^3_{Q_p}$ vanishes, the first potentially nontrivial differential is $d^3: E_1^{ij} \rightarrow E_1^{i+3,j-2}$, which is the Gerstenhaber bracket with $\mu^4_{Q_p}$. In \eqref{eq:hh-table}, this occurs as $d^3: \mathit{End}(V) \rightarrow \mathit{Sym}^4(V^\vee)$, which can be interpreted as the action of linear vector fields on the polynomial $p$, $d^3(Z) = L_Zp$. Moreover, since $\mu^5_{Q_p} = 0$, the next nontrivial differential is $d^5$, which vanishes in \eqref{eq:hh-table} (this requires a bit of thought, since the higher differentials are related to the $\mu^k_{Q_p}$ by a nonlinear ``zigzagging'' procedure: for instance, $\mu^4$ itself could yield a nontrivial contribution to $d^5$). One concludes in particular:

\begin{lemma} \label{th:generic}
Suppose that there is no nontrivial linear vector field which acts trivially on $p$. Then $\mathit{HH}^1(Q_p,Q_p) \iso \Lambda^2(V) \oplus \Lambda^2(V)$ is two-dimensional. \qed
\end{lemma}

\begin{addendum} \label{th:quasi-generators}
Inspection of the argument above allows one to approximately determine the form of two generators $[g_1]$, $[g_2]$ of $\mathit{HH}^1(Q_p,Q_p)$. Each of them is represented by a Hochschild cochain whose leading order term $g_k^0$ is a nontrivial element of $(e_k Q e_k)^1 \iso \Lambda^2(V)$, let's say the standard elements $q_1 = w_3w_2 = -w_4w_1$ (for $k = 1$) and $q_2 = w_1w_4 = -w_2w_3$ (for $k = 2$). Moreover, the next order term $g_k^1$ can be chosen to be zero.
\end{addendum}

\subsection{Some twisted complexes\label{subsec:twisted}}
$Q$ can be thought of as a linear graded category with two objects $X_i$ corresponding to the vertices of the quiver, so that for instance $\mathit{hom}_Q(X_1,X_2) = e_2 Q e_1 \iso V$. The $A_\infty$-deformation $Q_p$ can the be viewed as an $A_\infty$-category with the same objects. The aim of the following discussion is to understand how the choice of $p$ affects the structure of the formal enlargement $Q_p^{\mathit{tw}}$ (and $Q_p^{\mathit{perf}}$). For that, it is useful to require $R$ to be algebraically closed, which we will do from now on.

For any nonzero $v \in V$, consider the twisted complex $C_v = \mathit{Cone}(v: X_1 \rightarrow X_2)$.
We have
\begin{equation} \label{eq:endomorphism-of-c}
\mathit{hom}_{Q_p^{\mathit{tw}}}(C_v,C_v) = \begin{pmatrix} \Lambda^0(V) \oplus \Lambda^2(V)[-1] & V \\
V[-1] & \Lambda^0(V) \oplus \Lambda^2(V)[-1] \end{pmatrix}.
%
\end{equation}
The matrix notation here stands for taking the direct sum of the four graded vector spaces involved. Taking into account the fact that $\mu^1_{Q_p}$ and $\mu^3_{Q_p}$ vanish, the differential on \eqref{eq:endomorphism-of-c} is
\begin{equation}
\begin{pmatrix} x_{11} & x_{12} \\ x_{21} & x_{22}
\end{pmatrix} \stackrel{\mu^1_{Q_p^{\mathit{tw}}}}{\longmapsto}
\begin{pmatrix}
 x_{12} \wedge v & 0 \\
 -v \wedge x_{11} - x_{22} \wedge v & v \wedge x_{12}
\end{pmatrix}
\end{equation}
and the product is
\begin{equation}
\begin{aligned} &
\begin{pmatrix} y_{11} & y_{12} \\ y_{21} & y_{22}
\end{pmatrix} \otimes
\begin{pmatrix} x_{11} & x_{12} \\ x_{21} & x_{22}
\end{pmatrix} \stackrel{\mu^2_{Q_p^{\mathit{tw}}}}{\longmapsto}
\begin{pmatrix}
y_{11} \wedge x_{11} + y_{12} \wedge x_{21} & y_{11} \wedge x_{12} + y_{12} \wedge x_{22} \\
y_{21} \wedge x_{11} + y_{22} \wedge x_{21} & y_{22} \wedge x_{22} + y_{21} \wedge x_{12}
\end{pmatrix} \\
& + \begin{pmatrix} \mu^4_{Q_p}(y_{12},v,x_{12},v) & \mu^4_{Q_p}(v,y_{12},v,x_{11}) \\
\mu^4_{Q_p}(y_{22},v,x_{12},v) + \mu^4_{Q_p}(v,y_{12},x_{22},v) + \mu^4_{Q_p}(v,y_{11},x_{12},v) & \mu^4_{Q_p}(v,y_{12},v,x_{12})
\end{pmatrix}.
\end{aligned}
\end{equation}
Representative cochains for a basis in the cohomology of $\mu^1_{Q_p^{\mathit{tw}}}$ are
\begin{equation} \label{eq:endomorphism}
e = \begin{pmatrix} -1 & 0 \\ 0 & 1 \end{pmatrix}, \;\;
t = \begin{pmatrix} 0 & v \\ 0 & 0 \end{pmatrix}, \;\;
q = \begin{pmatrix} v \wedge v^* & 0 \\ 0 & v \wedge v^* \end{pmatrix}, \;\;
u = \begin{pmatrix} 0 & 0 \\ v^* & 0 \end{pmatrix},
\end{equation}
where $v^* \in V$ satisfies $v^* \wedge v \neq 0$. The first generator is the identity element, with the sign due to convention. Some explicit products of the other generators are
\begin{equation} \label{eq:qp-products}
\begin{aligned}
& \mu^2_{Q_p^{\mathit{tw}}}(u,t) = \begin{pmatrix} -v \wedge v^* & 0 \\ 0 & 0 \end{pmatrix} =
 \half \mu^1_{Q_p^{\mathit{tw}}}\begin{pmatrix} 0 & v^* \\ 0 & 0 \end{pmatrix} - \half q,
\\ &
\mu^2_{Q_p^{\mathit{tw}}}(t,u) = \begin{pmatrix} 0 & 0 \\ 0 & v \wedge v^* \end{pmatrix}
 = \half \mu^1_{Q_p^{\mathit{tw}}}\begin{pmatrix} 0 & v^* \\ 0 & 0 \end{pmatrix} + \half q,
\\
& \mu^2_{Q_p^{\mathit{tw}}}(u,u) = 0, \\ &
\mu^2_{Q_p^{\mathit{tw}}}(t,t) = p(v)e.
\end{aligned}
\end{equation}
On the cohomology level, this implies that
\begin{equation} \label{eq:cv-endomorphism}
H^0(\mathit{hom}_{Q_p^{\mathit{tw}}}(C_v,C_v)) \iso R[t]/(t^2-p(v)).
\end{equation}
We have therefore shown:

\begin{lemma} \label{th:cone-splitting}
$C_v$ splits into two summands in $Q_p^{\mathit{perf}}$ if and only if $p(v) \neq 0$. \qed
\end{lemma}

This allows one to reconstruct $p$ up to a scalar multiple from categorical data.

\begin{addendum} \label{th:hochschild-cohomology-for-cones}
Just like the $A_\infty$-structure itself, a Hochschild cochain $g \in \mathit{CC}(Q_p,Q_p)$ has components $g^d_{[i_d\dots i_0]}$ for $d \geq 0$ and $i_0,\dots,i_d \in \{1,2\}$. The induced cochain $g^{\mathit{tw}}$ as in \eqref{eq:twisted-gamma} has in particular
\begin{equation} \label{eq:explicit-gamma}
\begin{aligned}
& g^{\mathit{tw},0} \in \mathit{hom}_{Q_p^{\mathit{tw}}}(C_v,C_v), \\
& g^{\mathit{tw},0} = \begin{pmatrix} -g^0_{[1]} & 0 \\ -g^1_{[21]}(v) & g^0_{[2]} \end{pmatrix}.
\end{aligned}
\end{equation}
Suppose that $p$ satisfies the assumptions of Lemma \ref{th:generic}, and consider the generators $g_1,g_2$ from Addendum \ref{th:quasi-generators}. The notation here is potentially confusing: $g_k$ is a whole Hochschild cocycle, whose components would be written as $(g_k)^d_{[i_d\dots i_0]}$. Suppose that $v^*$ is chosen in such a way that $v \wedge v^* = (g_k)^0_{[k]}$. By comparing \eqref{eq:explicit-gamma} with \eqref{eq:qp-products}, one sees that
\begin{equation} \label{eq:anti-sign}
(g_1)^{\mathit{tw},0} = -(g_2)^{\mathit{tw},0} = \mu^2_{Q_p^{\mathit{tw}}}(u,t).
\end{equation}
\end{addendum}

By carefully inspecting the argument leading to Lemma \ref{th:cone-splitting}, we can sharpen it to a criterion that determines $p$ on the nose, and also works in slightly more general circumstances. Suppose that $\tilde{Q}$ is an $A_\infty$-category with objects $\tilde{X}_1,\tilde{X}_2$, together with a fixed isomorphism $H(\tilde{Q}) \iso Q$ on the cohomology level. We will use the triangulated structure of $H^0(\tilde{Q}^{\mathit{tw}})$, following \cite[Section 3]{seidel04} for the sign conventions used in establishing exact triangles.

\begin{lemma} \label{th:exactly-complete-triangle}
Given $v \in V \iso \mathit{Hom}_{H^0(\tilde{Q}^{\mathit{tw}})}(\tilde{X}_1,\tilde{X}_2)$, complete it to an exact triangle
\begin{equation} \label{eq:tilde-exact-triangle}
\tilde{X}_1 \stackrel{v}{\longrightarrow} \tilde{X}_2 \longrightarrow \tilde{C}_v \longrightarrow \tilde{X}_1[1].
\end{equation}
There is a unique (degree $0$) endomorphism $\tilde{t}$ of $\tilde{C}_v$ with the following two properties. The composition
\begin{equation} \label{eq:recover-v}
\tilde{X}_2 \longrightarrow \tilde{C}_v \stackrel{\tilde{t}}{\longrightarrow} \tilde{C}_v \longrightarrow \tilde{X}_1[1]
\end{equation}
(where the first and last map are taken from the exact triangle) equals $v$; and $\tilde{t}^2$ is a multiple of the identity endomorphism. Moreover, that multiple is necessarily given by $\tilde{t}^2 = p(v)$, where $Q_p$ is the $A_\infty$-deformation of $Q$ quasi-isomorphic to $\tilde{Q}$.
\end{lemma}

\begin{proof}
The object $\tilde{C}_v$ is unique up to (non-canonical) isomorphisms which commute with the maps from $\tilde{X}_2$ and to $\tilde{X}_1[1]$. Hence, the statement is independent of the specific choice of $\tilde{C}_v$. Without loss of generality, we can assume that $\tilde{Q} = Q_p$ for some $p$. By definition,
\begin{equation} \label{eq:archetypal-triangle}
X_1 \stackrel{v}{\longrightarrow} X_2 \xrightarrow{(0,e_{X_2})} C_v \xrightarrow{(-e_{X_1},0)} X_1[1]
\end{equation}
is an exact triangle. With this and \eqref{eq:cv-endomorphism}, one checks easily that $t$ is the unique endomorphism satisfying \eqref{eq:recover-v}.
\end{proof}

%
%
%

\subsection{A perfect family\label{subsec:tautological}}
It is a natural next step to let the parameter $v$ vary. For simplicity, we will use the affine line (rather than the projective line) as a parameter space, setting $v = (1,s_2)$ with $s_2 \in R$. Take the double cover of the affine line ramified at the zero-locus of $p(1,s_2)$, and then remove the branch points. The outcome is a smooth curve $\SS$ whose ring of functions is
\begin{equation} \label{eq:punctured-rr}
\RR = R[s_1,s_1^{-1},s_2]/(s_1^2 - p(1,s_2)).
\end{equation}
We equip this with the nowhere vanishing one-form
\begin{equation} \label{eq:punctured-rr-form}
\theta = -\half s_1^{-1} ds_2 \in \Omega^1_\RR.
\end{equation}

Let $\QQ_p$ be the constant family of $A_\infty$-structures over $\SS$ associated to $Q_p$. Consider the object of $\QQ_p^{\mathit{tw}}$ given by $\CC = \mathit{Cone}((1,s_2): X_1 \rightarrow X_2)$. By the same computation as in \eqref{eq:cv-endomorphism}, taking the natural choice $v^* = (0,1)$ of generator linearly independent of $v$, we have
\begin{equation}
H(\mathit{hom}_{Q_p^{\mathit{tw}}}(\CC,\CC)) \iso \RR[t,u]/(t^2-s_1^2).
\end{equation}
After lifting the idempotent endomorphism
\begin{equation} \label{eq:projection}
\half (1 + s_1^{-1} t) \in H^0(\mathit{hom}_{Q_p^{\mathit{tw}}}(\CC,\CC))
\end{equation}
to a homotopy idempotent, one associates to this a family $\MM$ of perfect modules, which is a direct summand of the Yoneda image of $\CC$. Use the generator $g_2$ from Addendum \ref{th:quasi-generators} to define a deformation field
\begin{equation} \label{eq:elliptic-deformation-field}
\gamma = -2\theta \otimes g_2 \in \mathit{HH}^1(\QQ_p, \Omega^1_\RR \otimes \QQ_p).
\end{equation}

\begin{lemma} \label{th:example-family}
$\MM$ follows $[\gamma]$ (in the sense of Definition \ref{th:deformation-field}).
\end{lemma}

\begin{proof}
The deformation cocycle of $\CC$ can be determined by applying \eqref{eq:def-components} to the trivial pre-connection. On the other hand, $\gamma^{\mathit{tw},0}$ can be computed as in Addendum \ref{th:hochschild-cohomology-for-cones}. The result is
\begin{equation}
\begin{aligned}
& \mathit{def}(\pabla_\CC) = -\partial_{s_2}(\delta_\CC)  ds_2 = 2\theta \otimes s_1 \partial_{s_2}(\delta_\CC) = 2\theta \otimes s_1 u, \\
& \gamma^{\mathit{tw},0} = 2\theta \otimes \mu^2_{Q_p^{\mathit{tw}}}(u,t).
\end{aligned}
\end{equation}
The corresponding elements for $\MM$, at least on the cohomology level, can be computed by applying the projection \eqref{eq:projection} (it doesn't matter on which side, since there are no $\mathit{Hom}$s from one summand of $\CC$ to the other), which indeed yields the same result in both cases:
\begin{equation}
\half(1 + s_1^{-1} t) ut = \half (s_1^{-1} t^2 + t) u = \half (s_1 + t) u = \half (1 + s_1^{-1}t) s_1 u.
\end{equation}
\end{proof}

We could add any multiple of $g_1 + g_2$ to our deformation field and still obtain the same result, in view of \eqref{eq:anti-sign}.


\subsection{Elliptic curves\label{subsec:elliptic}}
The algebra $Q$ from Definition \ref{th:q-algebra} arises in the following algebro-geometric context. Take some $p \in \mathit{Sym}^4(V^\vee)$ which is {\em simple}, meaning that it has four distinct zeros. This gives rise to a double branched cover $\pi: Y_p \rightarrow \mathbb{P}(V)$, which is a smooth elliptic curve, embedded into the total space of the bundle $\OO_{\mathbb{P}(V)}(2)$. The sheaf $\pi_*\OO_{Y_p}$ decomposes into $\pm 1$ eigenspaces for the action of the covering transformation. These can be identified with
\begin{align}
& (\pi_* \OO_{Y_p})_{+1} \iso \OO_{\mathbb{P}(V)}, \\
& (\pi_* \OO_{Y_p})_{-1} \iso \OO_{\mathbb{P}(V)}(-2) \iso \Omega^1_{\mathbb{P}(V)} \otimes \Lambda^2(V), \label{eq:anti-invariant}
\end{align}
where the second part is obtained by taking functions linear on the fibres of $\OO_{\mathbb{P}(V)}(2)$ and restricting them to $Y_p$. If $E_1,E_2$ are locally free sheaves on $\mathbb{P}(V)$, we have canonical isomorphisms
\begin{equation} \label{eq:pushforward}
\begin{aligned}
\mathit{Ext}^*_{Y_p}(\pi^*E_1,& \pi^*E_2) \iso H^*(Y_p,\pi^*E_1^\vee \otimes \pi^*E_2) \\
& \iso H^*(Y_p,\pi^*(E_1^\vee \otimes E_2)) \\
& \iso H^*(\mathbb{P}(V),(E_1^\vee \otimes E_2) \otimes \pi_*\OO_{Y_p}) \\ &
\iso H^*(\mathbb{P}(V),E_1^\vee \otimes E_2) \oplus H^*(\mathbb{P}(V), E_1^\vee \otimes E_2 \otimes \Omega_{\mathbb{P}(V)}^1) \otimes \Lambda^2(V) \\ &
\iso \mathit{Ext}^*_{\mathbb{P}(V)}(E_1,E_2) \oplus \mathit{Ext}^{1-*}_{\mathbb{P}(V)}(E_2,E_1)^\vee \otimes \Lambda^2(V),
\end{aligned}
\end{equation}
where the last isomorphism uses Serre duality on $\mathbb{P}(V)$. Consider in particular $E_1 = \OO_{\mathbb{P}(V)}$, $E_2 = \OO_{\mathbb{P}(V)}(1) \otimes \Lambda^2(V)$, which has $\mathit{Hom}_{\mathbb{P}(V)}(E_1,E_2) = V^\vee \otimes \Lambda^2(V) \iso V$ by definition. The computation above (with $E_1$ and $E_2$ exchanged) shows that $\mathit{Ext}^1_{Y_p}(\pi^*E_2,\pi^*E_1) \iso V^\vee \otimes \Lambda^2(V) \iso V$ as well. Using this and similar arguments (compare \cite[Section 3c]{seidel-thomas99}) one sees that:

\begin{lemma} \label{th:q-geometric}
We have an isomorphism of graded algebras, $\mathit{Ext}^*_{Y_p}(\pi^*E_1 \oplus \pi^*E_2,\pi^*E_1 \oplus \pi^*E_2) \iso Q$ (if one thinks of $Q$ as defined in \eqref{eq:identifications}, the isomorphism is canonical). \qed
\end{lemma}

Let $D^b\mathit{Coh}(Y_p)$ be a suitable dg enhancement of the standard bounded derived category of coherent sheaves on $Y_p$. This category is closed under shifts, mapping cones and direct summands (the last-mentioned fact follows from the characterization of its objects as compact objects in a larger category \cite[Theorem 3.1.1]{bondal-vandenbergh02}). Lemma \ref{th:q-geometric} says that the subcategory of $D^b\mathit{Coh}(Y_p)$ with objects $\pi^*E_1,\pi^*E_2$ is quasi-isomorphic to an $A_\infty$-deformation of $Q$, which by Lemma \ref{th:classify-deformations} can be chosen to be $Q_{\tilde{p}}$ for some polynomial $\tilde{p}$. We then have a cohomologically full and faithful $A_\infty$-functor from $Q_{\tilde{p}}$ to $D^b\mathit{Coh}(Y_p)$, taking $X_i$ to $\pi^*E_i$. Moreover, since the $\pi^*E_i$ are split-generators of $D^b\mathit{Coh}(Y_p)$, this functor extends to a quasi-equivalence
\begin{equation} \label{eq:quiver-embedding}
Q_{\tilde{p}}^{\mathit{perf}} \stackrel{\htp}{\longrightarrow} D^b\mathit{Coh}(Y_p).
\end{equation}
Unsurprisingly,

\begin{lemma} \label{th:p-p}
The polynomial $\tilde{p}$ in \eqref{eq:quiver-embedding} is a nonzero constant multiple of $p$.
\end{lemma}

\begin{proof}
Any $A_\infty$-functor take cones to cones, up to quasi-isomorphism. In particular, $C_v$ maps to the cone of the morphism $\pi^*E_1 \rightarrow \pi^*E_2$ corresponding to $v$. That cone, which we denote by $\tilde{C}_v$, is isomorphic to the structure sheaf of the scheme-theoretic fibre $\pi^{-1}([v])$; more canonically, it can be written as that structure sheaf tensored with the one-dimensional vector space $V/Rv$. If $p(v) \neq 0$, the fibre consists of two closed points, hence has a nontrivial idempotent endomorphism. On the other hand, there are four points for which $p(v) = 0$, and where the scheme-theoretic fibre is a single fat point. Lemma \ref{th:cone-splitting} then yields the desired result.
\end{proof}

We can refine this observation slightly. One can compute geometrically that
\begin{align}
& \mathit{Hom}_{Y_p}(\tilde{C}_v,\tilde{C}_v) \iso R \oplus (Rv)^{\otimes 2}\label{eq:geom-cone-end}, \\
& \mathit{Hom}_{Y_p}(\pi^*E_2,\tilde{C}_v) \iso R \oplus (Rv)^{\otimes 2} \label{eq:in-cone}, \\
& \mathit{Ext}^1_{Y_p}(\tilde{C}_v,\pi^*E_1) \iso (Rv)^{\vee} \oplus Rv \label{eq:out-cone}.
\end{align}
In the ring structure of \eqref{eq:geom-cone-end}, the first summand is generated by the identity endomorphism, and the square of  $v \otimes v$ in the second summand is exactly $p(v)$ times the identity. The action of $v \otimes v$ on \eqref{eq:in-cone} by left multiplication is given by $(1,0) \mapsto (0,v \otimes v)$ (and correspondingly $(0,v \otimes v) \mapsto (p(v),0)$, so as to satisfy the given relation). Both groups \eqref{eq:in-cone} and \eqref{eq:out-cone} contain canonical elements, which are parts of the obvious exact triangle involving $\tilde{C}_v$, and those are just the generators of the first summands (in the case of \eqref{eq:out-cone}, this is the generator of $(Rv)^\vee$ dual to $v \in Rv$). Finally, the composition
\begin{equation}
\mathit{Ext}^1_{Y_p}(\tilde{C}_v,\pi^*E_1) \otimes \mathit{Hom}_{Y_p}(\pi^*E_2,\tilde{C}_v) \longrightarrow \mathit{Ext}^1_{Y_p}(\pi^*E_2,\pi^*E_1)
\end{equation}
is given by the obvious maps $(Rv)^{\vee} \otimes (Rv)^{\otimes 2} \rightarrow Rv \subset V$, $Rv \otimes R \rightarrow Rv \subset V$. By putting together those facts, one sees that taking $\tilde{t} = v \otimes v$ exactly satisfies the assumptions of Lemma \ref{th:exactly-complete-triangle} (modulo tedious sign verifications, which we have omitted), and therefore that:

\begin{lemma}
The constant from Lemma \ref{th:p-p} is trivial, meaning that $p = \tilde{p}$. \qed
\end{lemma}

\begin{remark}
This geometric interpretation also throws some light on Lemma \ref{th:generic}. In view of the derived invariance of Hochschild cohomology, one has
\begin{equation} \label{eq:hh-geometric-1}
\mathit{HH}^d(Q_p,Q_p) \iso \mathit{HH}^d(D^b\mathit{Coh}(Y_p)) \iso \mathit{HH}^d(Y_p) \iso \bigoplus_{i+j=d} H^i(Y_p,\Lambda^j TY_p),
\end{equation}
where $TY_p$ is the tangent bundle. In particular, $\mathit{HH}^1(Q_p,Q_p) \iso H^1(Y_p,\OO_{Y_p}) \oplus H^0(Y_p,T_{Y_p})$ is indeed two-dimensional.

We make a slight digression, whose aim is to explain our original computations of $\mathit{HH}^*(Q,Q)$ geometrically. One can associate to an arbitrary $p \in \mathit{Sym}^4(V^\vee)$ a subscheme $Y_p$ of the total space of $\OO_{\mathbb{P}(V)}(2)$, and Lemma \ref{th:q-geometric} still holds. So does \eqref{eq:quiver-embedding}, once one replaces the derived category of coherent sheaves with its subcategory of perfect complexes. In particular, we can set $p = 0$, in which case the ``double branched cover'' $Y_0$ is the first order infinitesimal neighbourhood of the zero-section in $\OO_{\mathbb{P}(V)}(2)$. Using the action of $R^\times$ by fibrewise rescaling, one can show that the resulting $A_\infty$-structure on $Q$ is formal (this is a well-known idea, in a sense going back to \cite{deligne-griffiths-morgan-sullivan75}, see \cite[Remark 7.6]{seidel-solomon10} and \cite{lekili-perutz11} for recent occurrences). Hence,
\begin{equation} \label{eq:hh-geometric}
\bigoplus_{i+j = d} \mathit{HH}^i(Q,Q[j]) \iso \mathit{HH}^d(Y_0),
\end{equation}
where the right hand side can be written as the hypercohomology of a complex of sheaves on $\mathbb{P}(V)$ (locally quasi-isomorphic to the Hochschild complex of the ring $R[t_1,t_2]/t_2^2$), making it easily amenable to computation. Moreover, the equivariant version of the same Hochschild cohomology recovers the bigrading on the left hand side of \eqref{eq:hh-geometric}.
\end{remark}

Identify $V = R^2$ with coordinates $(v_1,v_2)$. Consider the affine chart for the total space of $\OO_{\mathbb{P}(V)}(2)$ with coordinates $(s_1,s_2)$, which is such that $s_2 = v_2/v_1$ for the underlying point $[v_1:v_2] \in \mathbb{P}(V)$, and the section $s_1 = 1$ corresponds to the quadratic polynomial $v_1^2$. In this chart, $Y_p$ has equation $s_1^2 = p(1,s_2)$. The object $\tilde{C}_v$ constructed above, for $v = (1,s_2)$, is the structure sheaf of the ideal obtained by additionally setting $s_2$ to a specific value, and its endomorphism $\tilde{t}$ is multiplication by $s_1$. If $p(1,s_2) \neq 0$, the idempotent endomorphism $\frac{1}{2}(1 + s_1^{-1} \tilde{t})$ of $\tilde{C}_v$ singles out a direct summand, which is the structure sheaf of the point $(s_1,s_2)$. Applying Lemma \ref{th:example-family} to this, we get a perfect family of sheaves on $Y_p$ parametrized by the curve $\SS$ from \eqref{eq:punctured-rr}, which is itself an affine open part of $Y_p$, and such that the fibre of the family at $(s_1,s_2)$ is isomorphic to the structure sheaf of that point. This justifies calling it a ``tautological family''.

\begin{remark}
The canonical bundle of the total space of $\OO_{\mathbb{P}(V)}(2)$ is the pullback of $\OO_{\mathbb{P}(V)}(-2) \otimes \Omega^1_{\mathbb{P}(V)} \iso \OO_{\mathbb{P}(V)}(-4) \otimes \Lambda^2(V)^\vee$. Hence, fixing a symplectic form on the vector space $V$ singles out a two-form with poles exactly along $Y_p$, whose residue is then a nowhere vanishing one-form on $Y_p$. Returning to the identification $V = R^2$ and taking the symplectic form to be the standard form $dv_1 \wedge dv_2$, one finds that the restriction of the associated one-form to $\SS \subset Y_p$ is precisely \eqref{eq:punctured-rr-form}, since that satisfies $\theta \wedge d(s_1^2 - p(1,s_2)) = ds_1 \wedge ds_2$.
\end{remark}

\subsection{A universal construction\label{subsec:universal-bimodule}}
We will now give an alternative construction of the tautological family (and related ones). Any $N \in \mathit{Ob}\, D^b\mathit{Coh}(Y_p \times Y_p)$ defines a Fourier-Mukai functor $K_N$ from $D^b\mathit{Coh}(Y_p)$ to itself. Its action on objects is $K_N(E) = (q_2)_*(q_1^*E \otimes N)$, where $q_k: Y_p \times Y_p \rightarrow Y_p$ are the projections. To explain the interaction of this with the description \eqref{eq:quiver-embedding} of the derived category, we find it convenient to reverse directions of the arrows, which means to consider the pullback functor
\begin{equation} \label{eq:pullback-sheaf-1}
D^b\mathit{Coh}(Y_p) \longrightarrow Q_p^{\mathit{mod}}.
\end{equation}
The image of an object $F$ is a module $M$ with $H(M(X_i)) = \mathit{Hom}^*_{Y_p}(E_i,F)$. One shows easily that \eqref{eq:pullback-sheaf-1} is cohomologically full and faithful, and in fact a quasi-equivalence to the subcategory $Q_p^{\mathit{perf}}$ of perfect modules, which is inverse to \eqref{eq:quiver-embedding}. There is a similar pullback functor
\begin{equation}
D^b\mathit{Coh}(Y_p \times Y_p) \longrightarrow (Q_p,Q_p)^{\mathit{mod}}.
\end{equation}
This maps $N$ to a bimodule $P$ whose cohomology is $H(P(X_i,X_j)) = \mathit{Hom}^*_{Y_p \times Y_p}(E_j \boxtimes E_i^\vee,N) \iso \mathit{Hom}^*_{Y_p}(E_i,(q_2)_*(q_1^*E_j \otimes N))$. For instance, the structure sheaf of the diagonal maps to the diagonal bimodule. Note also that if we consider $P(\cdot,X_j)$ just as a right $Q_p$-module, it is quasi-isomorphic to the image of the sheaf $(q_2)_*(q_1^*E_j \otimes N)$ under \eqref{eq:pullback-sheaf-1}. This implies that $P$ is always right perfect. Finally, the following diagram is commutative up to homotopy:
\begin{equation} \label{eq:fm-versus-convolution}
\xymatrix{
\ar[d] D^b\mathit{Coh}(Y_p) \ar[rr]^-{K_N} && D^b\mathit{Coh}(Y_p) \ar[d] \\
Q_p^{\mathit{mod}} \ar[rr]^-{K_P} && Q_p^{\mathit{mod}}.
}
\end{equation}
In the top row of this, $K_N$ is the Fourier-Mukai functor, whereas on the bottom row we have the tensor product functor $K_P$. We will in fact only need to know commutativity of this on the level of quasi-isomorphism classes of objects, which is somewhat easier than the full statement.

The same observations hold for families. Let $\QQ_p$ be the constant family of $A_\infty$-structures over $\SS$ with fibre $Q_p$. There are functors
\begin{align}
\label{eq:family-modules} & D^b\mathit{Coh}(\SS \times Y_p) \longrightarrow \QQ_p^{\mathit{mod}}, \\
& D^b\mathit{Coh}(\SS \times Y_p \times Y_p) \longrightarrow (\QQ_p,\QQ_p)^{\mathit{mod}}, \label{eq:kernels-to-bimodules-2}
\end{align}
of which the first is an equivalence to the subcategory of perfect families, and the second one at least lands in the subcategory of right perfect families of bimodules.
Any object $\NN$ of $D^b\mathit{Coh}(\SS \times Y_p \times Y_p)$ defines a functor $K_\NN: D^b\mathit{Coh}(Y_p) \longrightarrow D^b\mathit{Coh}(\SS \times Y_p)$, which can be thought of as a family of Fourier-Mukai functors parametrized by $\SS$. If $\PP$ is the image of $\NN$ under \eqref{eq:kernels-to-bimodules-2}, we have a tensor product functor $K_\PP: Q_p^{\mathit{mod}} \longrightarrow \QQ_p^{\mathit{mod}}$, already considered (except for the notation) in Section \ref{subsec:existence}. The analogue of \eqref{eq:fm-versus-convolution} is
\begin{equation} \label{eq:fm-versus-convolution-2}
\xymatrix{
\ar[d] D^b\mathit{Coh}(Y_p) \ar[rr]^-{K_\NN} && D^b\mathit{Coh}(\SS \times Y_p) \ar[d] \\
Q_p^{\mathit{mod}} \ar[rr]^-{K_\PP} && \QQ_p^{\mathit{mod}},
}
\end{equation}
where the vertical arrows are \eqref{eq:family-modules}. So far the discussion has been essentially limited to abstract nonsense, but now we want to draw some consequences more specific to the case of elliptic curves.

Fix a point $s \in \SS \subset Y_p$, and give $Y_p$ its unique structure of an elliptic curve with $s$ as the neutral element. The graph of the addition morphism $\Sigma: Y_p \times Y_p \rightarrow Y_p$ restricts to a smooth subvariety $\{(s,y_1,y_2) \in \SS \times Y_p \times Y_p \, : \, y_2 = \Sigma(s,y_1)\}$. Let $\NN$ be the structure sheaf of that subvariety, and $\PP$ its image under \eqref{eq:kernels-to-bimodules-2}. This is a right perfect family of $Q_p$-bimodules parametrized by $\SS$, whose fibre at $s$ is quasi-isomorphic to the diagonal bimodule.

\begin{lemma} \label{th:fm-follows}
$\PP$ follows the deformation field $[\gamma]$ from \eqref{eq:elliptic-deformation-field}.
\end{lemma}

\begin{proof}
At any point $s \in \SS$, $\PP_s$ is the graph of an autoequivalence of $Q_p^{\mathit{perf}}$. This implies that the maps $\mathit{HH}^*(Q_p,Q_p) \longrightarrow H^*(\mathit{hom}_{(Q_p,Q_p)^{\mathit{mod}}}(\PP_s,\PP_s))$ considered in \eqref{eq:one-sided-deformation} are both isomorphisms. The same then holds in the entire family, which means that \eqref{eq:bimodule-induced} is an isomorphism. By definition, we have therefore shown that $\PP$ follows {\em some} deformation field $[\tilde{\gamma}] \in \mathit{HH}^1(\QQ_p, \Omega^1_\RR \otimes \QQ_p)$.

We will now argue indirectly based on Corollary \ref{th:universal-family} and the Hochschild cohomology computation in Lemma \ref{th:generic}. The sheaf $E_1 = \OO_{Y_p}$ is invariant under translations, which means that its Fourier-Mukai convolution with $\NN$ yields a constant family. By \eqref{eq:fm-versus-convolution-2}, this implies that $X_1 \otimes_{\QQ_p} \PP$ is a constant family, hence that the component $\tilde{\gamma}^0_{[1]} \in e_1\QQ_p e_1$ must be zero. If we now take the skyscraper sheaf at the point $s$ as our starting object, the outcome of $K_\NN$ is again a  ``tautological'' family, meaning that the fibre at any point of $\SS$ is isomorphic to the skyscraper sheaf at that point. Any two such families are isomorphic up to tensoring with line bundles over $\SS$. By comparing this with Lemma \ref{th:example-family}, one sees that $\tilde{\gamma}^0_{[2]} \in e_2\QQ_p e_2$ must be equal to $-2\theta$.
\end{proof}

\begin{corollary}
The category $Q_p$ has a nonempty set of periodic elements, in the sense of Definition \ref{th:per}. More specifically, if $[g_2] \in \mathit{HH}^1(Q_p,Q_p)$ is as in Addendum \ref{th:quasi-generators}, then $[-2g_2] \in \mathit{Per}(Q_p,\bar{\SS},\bar{\theta})$. Here, $\bar{\SS}$ is the smooth closure of $\SS$ (obviously isomorphic to $Y_p$), and $\bar\theta$ the extension of $\theta$ to that closure.
\end{corollary}

This is a direct consequence of Lemma \ref{th:fm-follows} and Corollary \ref{th:universal-family}. Now recall that the category $Q_p^{\mathit{perf}} \htp D^b\mathit{Coh}(Y_p)$ carries an action of $\mathit{SL}_2(\bZ)$ (in a weak sense, and ignoring even shifts) \cite[Remark 3.15]{mukai81}. This acts on Hochschild cohomology and also maps families to families. From this, one gets the following:

\begin{corollary} \label{th:periodic-lattice}
The set $\mathit{Per}(Q_p,\bar{\SS},\bar{\theta}) \subset \mathit{HH}^1(Q_p,Q_p)$ contains $m_1[g_1] + m_2[g_2]$ for all $m_1 \in \bZ$, $m_2 \in m_1 + 2 \bZ$. \qed
\end{corollary}

\subsection{Theta functions}
For the remainder of this section, we work exclusively over the field $R$ from \eqref{eq:novikov-field}. Elliptic curves over fields like $R$ can be studied using methods of non-archimedean analytic geometry. The comparison with ordinary algebraic geometry over $R$ is provided by a suitable GAGA theorem (see \cite{papikian05,conrad08} for introductory accounts). We will borrow the intuition from there, but otherwise proceed by direct computation, avoiding abstract tools as much as possible.

Let $F$ be the ring of Laurent series over $R$ in one variable $t$, and which have ``infinite convergence radius'' (the non-archimedean analogue of holomorphic functions on $\bC^*$). Explicitly, this means that
\begin{equation} \label{eq:tate}
F = \left\{
\begin{aligned}
& f(t) = c_0 \hbar^{m_0} t^{n_0} + c_1 \hbar^{m_1} t^{n_1} + \cdots, \\
& c_k \in \bC, \;\; m_k \in \bR, \;\; n_k \in \bZ, \;\; \lim_{k \rightarrow \infty} m_k + A n_k = +\infty \text{ for any $A \in \bR$}
\end{aligned}
\right\}.
\end{equation}
We will be particularly interested in the theta-functions (a standard reference is \cite{mumford85}, but our notation is a little different)
\begin{equation} \label{eq:theta-functions}
\vartheta_{n,k}(t) = \sum_{i \in n\bZ+k} \hbar^{i^2/2n} t^i,
\end{equation}
where $n$ is a positive integer, and $k \in \bZ/n \bZ$. These functions satisfy
\begin{align}
\label{eq:periodicity}
& \vartheta_{n,k}(\hbar t) = \hbar^{-n/2} t^{-n} \vartheta_{n,k}(t) && \text{(periodicity)}, \\
\label{eq:fractional-period} & \vartheta_{n,k}(\hbar^{1/n} t) = \hbar^{-1/2n} t^{-1} \vartheta_{n,k+1}(t) && \text{(fractional periodicity)}, \\
& \vartheta_{n,k}(t^{-1}) = \vartheta_{n,n-k}(t) && \text{(symmetry)}. \label{eq:theta-symmetry}
\end{align}
The simplest example, $\vartheta_{1,1}(t) = \vartheta(t)$, is the common or garden Jacobi theta function. On the next level $n = 2$, one has two functions $\vartheta_{2,1}(t)$ and $\vartheta_{2,2}(t)$, which besides \eqref{eq:fractional-period} are related by the addition formula
\begin{equation} \label{eq:theta-factorization}
\vartheta_{2,1}(u) \vartheta_{2,1}(t) + \vartheta_{2,2}(u) \vartheta_{2,2}(t)  = \vartheta(u t)\vartheta(u^{-1} t).
\end{equation}
As a consequence of that formula, we have
\begin{equation} \label{eq:special-doubling}
\begin{aligned}
& \vartheta_{2,2}(\hbar^{1/2})\vartheta_{2,1}(t) - \vartheta_{2,1}(\hbar^{1/2})\vartheta_{2,2}(t) = -\hbar^{-1/4}\vartheta(-t)^2, \\
& \vartheta_{2,2}(-\hbar^{1/2})\vartheta_{2,1}(t) - \vartheta_{2,1}(-\hbar^{1/2})\vartheta_{2,2}(t) = \hbar^{-1/4}\vartheta(t)^2, \\
& \vartheta_{2,2}(1)\vartheta_{2,1}(t) - \vartheta_{2,1}(1)\vartheta_{2,2}(t) =
\hbar^{1/4} t \,\vartheta(-\hbar^{1/2}t)^2, \\
& \vartheta_{2,2}(-1)\vartheta_{2,1}(t) - \vartheta_{2,1}(-1)\vartheta_{2,2}(t) =
\hbar^{1/4}t \, \vartheta(\hbar^{1/2}t)^2.
\end{aligned}
\end{equation}
We will need the duplication formula \cite[p.\ 488]{whittaker-watson}
\begin{equation} \label{eq:duplication}
\begin{aligned}
\vartheta(t)\vartheta(-t)\vartheta(\hbar^{1/2}t)\vartheta(-\hbar^{1/2}t) & = \half \vartheta(1) \vartheta(-1) \vartheta(\hbar^{1/2}) \vartheta(-\hbar^{1/2}t^2) \\
& = \half \hbar^{-1/8}t^{-1}\, \vartheta(1)\vartheta(-1) \vartheta(\hbar^{1/2}) (\vartheta_{4,1}(t) - \vartheta_{4,3}(t)).
\end{aligned}
\end{equation}
We will also need an identity for the derivatives (in $t$-direction),
\begin{equation} \label{eq:derivative-identity}
\begin{aligned}
t(\vartheta_{2,2}'(t)\vartheta_{2,1}(t) - \vartheta_{2,1}'(t)\vartheta_{2,2}(t)) & = \vartheta_{4,1}(t)\vartheta_{4,3}'(1) + \vartheta_{4,3}(t)\vartheta_{4,1}'(1) \\
& = \vartheta_{4,3}'(1)(\vartheta_{4,1}(t) - \vartheta_{4,3}(t)).
\end{aligned}
\end{equation}

\begin{definition} \label{th:unit-torus-polynomial}
The {\em unit torus polynomial} $p \in R[v_1,v_2]$ is
\begin{align} \label{eq:torus-polynomial}
& \begin{aligned}
p(v_1,v_2) & = c\,(\vartheta_{2,2}(\hbar^{1/2})v_2 - \vartheta_{2,1}(\hbar^{1/2})v_1)(\vartheta_{2,2}(-\hbar^{1/2})v_2 - \vartheta_{2,1}(-\hbar^{1/2})v_1)\\ & \qquad \qquad \qquad \cdot (\vartheta_{2,2}(1)v_2 - \vartheta_{2,1}(1)v_1)(\vartheta_{2,2}(-1)v_2 - \vartheta_{2,1}(-1)v_1) \\
& = c\,(\vartheta_{2,2}(\hbar^{1/2})^2 v_2^2 - \vartheta_{2,1}(\hbar^{1/2})^2 v_1^2)(
\vartheta_{2,2}(1)^2 v_2^2 - \vartheta_{2,1}(1)^2 v_1^2),
\end{aligned}
\intertext{where}
& \begin{aligned}
c & = -\hbar^{1/4} \vartheta(1)^{-2}\vartheta(-1)^{-2}\vartheta(\hbar^{1/2})^{-2} \vartheta_{4,3}'(1)^2.
\end{aligned}
\end{align}
\end{definition}

One can associate to this polynomial a smooth elliptic curve over $R$, by the double branched cover method from Section \ref{subsec:elliptic}. For the most part, we will only look at the affine part $\SS = \mathit{Spec}(\RR)$ of that curve as in \eqref{eq:punctured-rr}, and equip that with the one-form $\theta$ from \eqref{eq:punctured-rr-form} (this extends to a regular one-form on the entire elliptic curve). We call this the {\em unit torus curve}. Points of $\SS$ have a transcendental parametrization by $u \in R^\times$, called the {\em theta parametrization}:
\begin{equation} \label{eq:theta-parametrization}
\begin{aligned}
s_2 & = \vartheta_{2,1}(u)\,\vartheta_{2,2}(u)^{-1}, \\
s_1 & = -\half u s_2'(u) = \half\, u \vartheta_{2,2}(u)^{-2}(\vartheta_{2,2}'(u)\vartheta_{2,1}(u) - \vartheta_{2,1}'(u)\vartheta_{2,2}(u)), \\
& = \half \vartheta_{2,2}(u)^{-2} \vartheta_{4,3}'(1) (\vartheta_{4,1}(u) - \vartheta_{4,3}(u)).
\end{aligned}
\end{equation}
The fact that these points satisfy the equation for $\SS$ follows from the relations between theta-functions listed above. Note that, due to the periodicity property, $u$ and $\hbar^k u$ ($k \in \bZ$) describe the same point (if we had set up the general machinery properly, this would yield the Tate uniformization of the closure of $\SS$, which of course is the original $Y_p$, as a quotient $R^\times/\hbar^\bZ$). To be precise, we have to exclude the values $u = \pm i \hbar^{k+1/2}$, which is where $\vartheta_{2,2}$ vanishes (these would be mapped to points at infinity) as well as $u = \pm \hbar^{k/2}$, which is where $\vartheta_{4,1}-\vartheta_{4,3}$ vanishes, see \eqref{eq:duplication} (those points would be mapped to the branch points which do not lie in $\SS$ by definition). The involution $u \mapsto \hbar^{1/2}u$ corresponds to $(s_1,s_2) \mapsto (-s_1s_2^{-2},s_2^{-1})$, whereas $u \mapsto u^{-1}$ corresponds to $(s_1,s_2) \mapsto (-s_1,s_2)$. Note also that in this parametrization, the one-form is $\theta = u^{-1} du$.

\begin{addendum} \label{th:double-cover}
Suppose that we change variables from $\hbar$ to $\hbar^2$ (an automorphism of the field $R$). Apply this to the coefficients of $p$ and denote the resulting polynomial by $\tilde{p}$, with its associated curve $\tilde{\SS}$ and one-form $\tilde{\theta}$. Our claim is that then, the projective closure of $\tilde{\SS}$ is an \'etale double cover of that of $\SS$, and that the given one-forms are compatible with the covering map. In fact, the covering transformation is the abovementioned involution $(\tilde{s}_1,\tilde{s}_2) \mapsto (-\tilde{s}_1\tilde{s}_2^{-2},\tilde{s}_2^{-1})$. This is obvious set-theoretically by comparing the theta-parametrizations, and one can derive from the abstract theory (or at least in principle, check by hand) that the resulting map is indeed algebraic.
\end{addendum}

\subsection{A nonarchimedean model\label{subsec:nonarchimedean}}
Nonarchimedean analytic geometry appears naturally in the context of mirror symmetry, as shown for torus fibrations in \cite{kontsevich-soibelman00}. Here, we will spell out a version of their construction for the case of elliptic curves. The ring $F$ from \eqref{eq:tate} comes with an action of $\bZ$, generated by the translation
\begin{equation} \label{eq:translation-action}
(Tf)(t) = f(\hbar t).
\end{equation}
We want to consider equivariant $T$-modules, or equivalently (right) modules over the semidirect product algebra $F \rtimes \bZ$ (whose generators are $f(t)$ and $\tau$, with relations $\tau f = T(f) \tau$). For any $d \in \bZ$ one has the equivariant module $F(d)$, which is $F$ itself with the twisted $\bZ$-action generated by
\begin{equation}
(T(d)f)(t) = t^d f(\hbar t).
\end{equation}
$\mathit{Hom}_{F \rtimes \bZ}(F(d_0),F(d_1))$ is isomorphic to the subspace of those elements of $F$ which are invariant under $T(d_1-d_0)$. Similarly, using the projective resolution (the map is left multiplication)
\begin{equation} \label{eq:f-resolution}
\tilde{F}(d) = \Big\{F \rtimes \bZ \xrightarrow{\mathit{id} - \tau t^{-d}} F \rtimes \bZ \Big\}
\end{equation}
of $F(d)$, one sees that $\mathit{Ext}^1_{F \rtimes \bZ}(F(d_0),F(d_1))$ is isomorphic to the space of coinvariants for $T(d_1-d_0)$. In fact, we have a subcomplex of the chain complex $\mathit{Hom}_{F \rtimes \bZ}(\tilde{F}(d_0),\tilde{F}(d_1))$ which is quasi-isomorphic to the whole thing, and that subcomplex is of the form
\begin{equation} \label{eq:t-subcomplex}
\Big\{F \xrightarrow{\mathit{id} - T(d_1-d_0)} F\Big\}
\end{equation}
(concentrated in degrees $0$ and $1$).

Consider $Z_1 = F$ and $Z_2 = F(2)$. Using the computational ideas which we have just explained, one easily shows that
\begin{equation} \label{eq:q-algebra-again}
\mathit{Ext}^*_{F \rtimes \bZ}(Z_1 \oplus Z_2, Z_1 \oplus Z_2) \iso Q
\end{equation}
is the graded algebra from Definition \ref{th:q-algebra}. Hence, the underlying cochain level algebraic structure is quasi-isomorphic to $Q_p$ for some $p$. One can use the subcomplexes \eqref{eq:t-subcomplex} to write down a relatively simple (if still infinite-dimensional) dga model. In principle, the Homological Perturbation Lemma could be applied to that, as in \cite{polishchuk09}, which would give a way of determining $p$ directly. However, there is also a more categorical approach, in parallel with that from Section \ref{subsec:elliptic}. By \eqref{eq:periodicity}, an explicit basis of $\mathit{Ext}^*_{F \rtimes \bZ}(Z_1,Z_2)$ is given by the functions
\begin{equation}
f_k(t) = \vartheta_{2,k}(\hbar^{-1/2}t) = t^{-1}\vartheta_{2,k+1}(t),
\end{equation}
($k = 1,2$). Given $u \in R^\times$, consider the linear combination
\begin{equation}
\vartheta_{2,2}(u)f_1 + \vartheta_{2,1}(u)f_2 : Z_1 \longrightarrow Z_2.
\end{equation}
This is always injective. If $u \notin \{\pm \hbar^{k/2}\, : \, k \in \bZ\}$, its cokernel splits into a direct sum of two nontrivial objects. This is a consequence of \eqref{eq:theta-factorization}, which shows that the map can be written as the product of two elements which generate distinct (and $\bZ$-invariant) principal ideals. One can also check that this fails for the remaining values of $u$, where the cokernel is indecomposable. By applying Lemma \ref{th:cone-splitting}, it then follows that the homogeneous polynomial $p$ relevant to our situation must vanish at the points $(\vartheta_{2,2}(\pm \hbar^{1/2}),\vartheta_{2,1}(\pm \hbar^{1/2}))$ and $(\vartheta_{2,2}(\pm 1),\vartheta_{2,1}(\pm 1))$. Hence, it must be of the form \eqref{eq:torus-polynomial} for some nonzero constant $c$. One can adjust the isomorphism \eqref{eq:q-algebra-again} to make that constant equal to $1$, which shows:

\begin{lemma} \label{th:nonarch}
The full subcategory of the (chain level) derived category of modules over $F \rtimes \bZ$ with objects $F(0)$, $F(2)$ is quasi-isomorphic to $Q_p$, where $p$ is the unit torus polynomial. \qed
\end{lemma}

%
%

\subsection{The two-torus\label{subsec:twotorus}}
Take the unit area (symplectic) two-torus
\begin{equation}
\begin{aligned}
& T = \bR^2/\bZ^2, \\
& \omega_T = dp \wedge dq
\end{aligned}
\end{equation}
in coordinates $(p,q)$; and equip it with the standard complex structure, as well as the nonzero holomorphic one-form $dz = dp + i\, dq$. The Fukaya category $\mathit{Fuk}(T)$ is a $\bZ$-graded, cohomological unitally $A_\infty$-category over the field $R$ from \eqref{eq:novikov-field}. This is related by mirror symmetry to the derived category of an elliptic curve over $R$, see \cite{polishchuk-zaslow98,abouzaid-smith09}. More directly, \cite{kontsevich-soibelman00} relates it to the ring $F \rtimes \bZ$ described above. One could use either of these theorems to derive the desired results, but we prefer to argue by direct geometric computation. This means that our exposition is a little ad hoc, which is hopefully forgivable in view of the relative simplicity of the target geometry (for more details of the approach to Fukaya categories used here, see Section \ref{subsec:fukaya} and the references therein).

Objects of $\mathit{Fuk}(T)$ are simple closed curves $L \subset T$, equipped with a grading with respect to the complex one-form $dz$ (this excludes contractible curves), and with a local coefficient system $\xi \rightarrow L$ whose fibre is $R^r$ for some $r$, and whose holonomy lies in the subgroup
\begin{multline} \label{eq:gl-subgroup}
\mathit{GL}_0(r,R) \stackrel{\mathrm{def}}{=} \{A = A_0 + A_1 \hbar^{m_1} + \cdots, \; \\ A_0 \in \mathit{GL}(r,\bC), \text{and $A_i \in \mathit{Mat}(r,\bC)$, $m_i > 0$ for $i \geq 1$}\} \subset \mathit{GL}(r,R).
\end{multline}

\begin{remark}
The grading of $L$ induces an orientation. In the general definition of the Fukaya category, one requires an additional choice of {\it Spin} structure on $L$. However, changing the {\it Spin} structure is the same as tensoring $\xi$ with a line bundle having holonomy $\{\pm 1\}$. Hence, it is enough to consider curves $L$ with the trivial {\it Spin} structure (that which is compatible with the trivialization of $\mathit{TL}$; equivalently, it is the one which is nontrivial in {\it Spin} bordism), which is why that structure does not appear explicitly in our formulation.
\end{remark}

The space of morphisms between two objects $(L_0,\xi_0)$ and $(L_1,\xi_1)$ is the Floer cochain space $\mathit{CF}^*(L_0,L_1)$. In the case where $L_0$ intersects $L_1$ transversally, one can set it up so that generators correspond bijectively to points $x \in L_0 \cap L_1$. More precisely, each such point has an absolute index $\mathrm{deg}(x) \in \bZ$, which depends on the gradings. This has the property that $(-1)^{\mathrm{deg}(x)+1}$ is the local intersection number. Then, $x$ contributes a copy of $\mathit{Hom}(\xi_{0,x},\xi_{1,x})$ to $\mathit{CF}^*(L_0,L_1)$ in degree $\mathrm{deg}(x)$.

\begin{remark}
The general theory dictates that the contribution of $x$ is $\mathit{Hom}(\xi_{0,x},\xi_{1,x}) \otimes o_x$, where $o_x$ is the orientation space, a one-dimensional $R$-vector space which can be identified with $R$ uniquely up to sign. However, in the specific case where the symplectic manifold is a surface, we know \cite[Equation (13.6)]{seidel04} that $o_x \iso R$ canonically if $\mathrm{deg}(x)$ is even, and $o_x \iso (TL_1)_x \otimes_\bR R$ otherwise. In the second case, we use the orientation of $L_1$ given by the grading to pick a preferred generator for $o_x$. Again, this is why orientation spaces do not appear explicitly here.
\end{remark}

\begin{figure}
\begin{center}
\begin{picture}(0,0)%
\includegraphics{abcd1.pstex}%
\end{picture}%
\setlength{\unitlength}{5920sp}%
\begingroup\makeatletter\ifx\SetFigFont\undefined%
\gdef\SetFigFont#1#2#3#4#5{%
  \reset@font\fontsize{#1}{#2pt}%
  \fontfamily{#3}\fontseries{#4}\fontshape{#5}%
  \selectfont}%
\fi\endgroup%
\begin{picture}(1227,1408)(1189,-4607)
\put(2301,-4561){\makebox(0,0)[lb]{\smash{{\SetFigFont{10}{12}{\rmdefault}{\mddefault}{\updefault}{\color[rgb]{0,0,0}$w_2,w_3$}%
}}}}
\put(1681,-4561){\makebox(0,0)[lb]{\smash{{\SetFigFont{10}{12}{\rmdefault}{\mddefault}{\updefault}{\color[rgb]{0,0,0}$w_1,w_4$}%
}}}}
\end{picture}%
\caption{\label{fig:abcd1}}
\end{center}
\end{figure}%
The first two objects relevant for our argument are
\begin{equation}
L_1 = \{q = 0\}, \quad L_2 = \{q = -2 p\},
\end{equation}
We equip them with gradings so that $\mathit{HF}^*(L_1,L_2)$ is concentrated in degree $0$, and so that the induced orientations are as shown in Figure \ref{fig:abcd1}. Both should carry trivial local systems. Write $w_1, -w_2 \in \mathit{CF}^0(L_1,L_2)$ for the generators coming from the two intersection points $(\half,0)$ and $(0,0)$. We also write $w_4, w_3 \in \mathit{CF}^1(L_2,L_1)$ for the generators coming from the same two points. By introducing a perturbed version of one of our two $L_k$ and counting triangles, one shows that each of the products
\begin{equation} \label{eq:w-product}
\begin{aligned}
& [w_4] \cdot [w_1] \in \mathit{HF}^1(L_1,L_1) \iso H^1(L_1;R), \\
& [w_1] \cdot [w_4] \in \mathit{HF}^1(L_2,L_2) \iso H^1(L_2;R), \\
& -[w_2] \cdot [w_3] \in \mathit{HF}^1(L_2,L_2) \iso H^1(L_2;R), \\
& -[w_3] \cdot [w_2] \in \mathit{HF}^1(L_1,L_1) \iso H^1(L_1;R)
\end{aligned}
\end{equation}
equals the generator of $H^1$ given by our choice of orientations. Hence:

\begin{lemma} \label{th:product}
The cohomology level product satisfies the relations \eqref{eq:quadratic-relations}. This means that as a graded algebra, $\bigoplus_{i,j=1}^2 \mathit{HF}^*(L_i,L_j) \iso Q$. \qed
\end{lemma}

Take any $u \in R^\times$, written as $u = \hbar^{m_0}a$ for some $m_0 \in \bR$ and $a \in \mathit{GL}_0(1,R)$. We associate to this another object of the Fukaya category, as follows. The underlying curve is
\begin{equation}
L_{3,u} = \{p = m_0\}.
\end{equation}
It comes equipped with the grading such that $\mathit{HF}^*(L_1,L_{3,u})$, and then also $\mathit{HF}^*(L_2,L_{3,u})$, is concentrated in degree $0$ (the induced orientation is in negative $q$-direction). The rank $1$ local system $\xi_u$ on $L_{3,u}$ should have holonomy $a$ when going around the curve in positive $q$-direction. Both chain level morphism spaces $\mathit{CF}^0(L_1,L_{3,u})$ and $\mathit{CF}^0(L_2,L_{3,u})$ are canonically isomorphic to the fibre of this local system at the unique intersection points, which are $(m_0,0)$ and $(m_0,-2m_0)$, respectively. We identify these with $R$ as follows:
\begin{equation} \label{eq:xi-trivializations}
\parbox{35em}{
Pick an arbitrary isomorphism $(\xi_u)_{(m_0,0)} \iso R$. Then, choose $(\xi_u)_{(m_0,-2m_0)} \iso R$ in such a way that parallel transport along the path $\{(m_0,-2m_0 t) \; : \; 0 \leq t \leq 1\}$ is multiplication with $\hbar^{m_0^2}$.
}
\end{equation}
Denote the resulting generators by $z_{1,u} \in \mathit{CF}^0(L_1,L_{3,u})$, $z_{2,u} \in \mathit{CF}^0(L_2,L_{3,u})$. A triangle count (compare \cite[Section 4]{polishchuk-zaslow98}) determines the product
\begin{equation} \label{eq:theta-product-1}
\begin{aligned}
& \mathit{HF}^0(L_2,L_{3,u}) \otimes \mathit{HF}^0(L_1,L_2) \longrightarrow \mathit{HF}^0(L_1,L_{3,u}), \\
& [z_{2,u}] \cdot [w_1] = \vartheta_{2,1}(u) [z_{1,u}], \\
& [z_{2,u}] \cdot [w_2] = -\vartheta_{2,2}(u) [z_{1,u}],
\end{aligned}
\end{equation}
where the negative sign comes from our choice of $w_2$, rather than from any geometric aspect of the computation.

The intersection point $(m_0,0)$ contributes a copy of $(\xi_u^\vee)_{(m_0,0)}$ to $\mathit{CF}^1(L_{3,u},L_1)$. Take $y_{1,u}$ to be the generator dual to $z_{1,u}$. In the same way, we define a generator $y_{2,u} \in \mathit{CF}^1(L_{3,u},L_2)$ dual to $z_{2,u}$. Then, the cohomology level products
\begin{equation}
\begin{aligned}
& [y_{1,u}] \cdot [z_{1,u}] \in \mathit{HF}^1(L_1,L_1) \iso H^1(L_1;R), \\
& [z_{1,u}] \cdot [y_{1,u}] \in \mathit{HF}^1(L_{3,u},L_{3,u}) \iso H^1(L_{3,u};R), \\
& [y_{2,u}] \cdot [z_{2,u}] \in \mathit{HF}^1(L_2,L_2) \iso H^1(L_2;R), \\
& [z_{2,u}] \cdot [y_{2,u}] \in \mathit{HF}^1(L_{3,u},L_{3,u}) \iso H^1(L_{3,u};R)
\end{aligned}
\end{equation}
each equal the generator of $H^1$ singled out by the given orientations. Using this, \eqref{eq:theta-product-1}, and the associativity of the product (on the cohomology level), one computes the following:
\begin{equation} \label{eq:cyclic-mu2}
\begin{aligned}
& \mathit{HF}^0(L_1,L_2) \otimes \mathit{HF}^1(L_{3,u},L_1) \longrightarrow \mathit{HF}^1(L_{3,u},L_2), \\
& [w_1] \cdot [y_{1,u}] = \vartheta_{2,1}(u) [y_{2,u}], \\
& [w_2] \cdot [y_{1,u}] = -\vartheta_{2,2}(u) [y_{2,u}],
\end{aligned}
\end{equation}
and
\begin{equation} \label{eq:more-products}
\begin{aligned}
& \mathit{HF}^1(L_{3,u},L_1) \otimes \mathit{HF}^0(L_2,L_{3,u}) \longrightarrow \mathit{HF}^1(L_2,L_1), \\
& [y_{1,u}] \cdot [z_{2,u}] = \vartheta_{2,2}(u)[w_3] + \vartheta_{2,1}(u)[w_4].
\end{aligned}
\end{equation}

From now on, assume that $u \notin \{\pm \hbar^{k/2} \, : \, k \in \bZ\}$. Then $L_{3,u}$ and $L_{3,u^{-1}}$ are mutually orthogonal objects (which means that the Floer cohomology from one to the other is zero). The computations above, together with \eqref{eq:theta-symmetry}, show that the composition of any two of the following morphisms vanishes:
\begin{equation} \label{eq:123-triangle}
\xymatrix{L_1 \ar[rrrr]^-{\frac{\vartheta_{2,2}(u)[w_1] + \vartheta_{2,1}(u)[w_2]}{\vartheta_{4,3}'(1)(\vartheta_{4,1}(u) - \vartheta_{4,3}(u))}} &&&& L_2 \ar[dll]^{\;\;\;\;\;\;\; ([z_{2,u}], [z_{2,u^{-1}}])} \\
&& L_{3,u} \oplus L_{3,u^{-1}} \ar[ull]_{[1]}^{([y_{1,u}], -[y_{1,u^{-1}}])\;\;\;\;\;\;\;} &&
}
\end{equation}

\begin{lemma} \label{th:exact-triangle}
The diagram \eqref{eq:123-triangle} is an exact triangle in the category $H^0(\mathit{Fuk}(T)^{tw})$.
\end{lemma}

To prove this, we need one part of the higher order $A_\infty$-structure, namely:
\begin{equation} \label{eq:mu3}
\begin{aligned}
& \mu^3_{\mathit{Fuk}(T)}: \mathit{CF}^0(L_2,L_{3,u}) \otimes \mathit{CF}^0(L_1,L_2) \otimes \mathit{CF}^1(L_{3,u},L_1) \longrightarrow \mathit{CF}^0(L_{3,u},L_{3,u}), \\
& \mu^3_{\mathit{Fuk}(T)}(z_{2,u},w_1,y_{1,u}) = \left( - u \vartheta_{2,1}'(u)\, + b \, \vartheta_{2,1}(u) \right) e_{3,u}, \\
& \mu^3_{\mathit{Fuk}(T)}(z_{2,u},w_2,y_{1,u}) = \left( u \vartheta_{2,2}'(u) - b \, \vartheta_{2,2}(u) \right) \, e_{3,u}. \\
\end{aligned}
\end{equation}
Here, we assume that $\mathit{CF}^*(L_{3,u},L_{3,u})$ is minimal (has vanishing differential).
$e_{3,u} \in \mathit{CF}^0(L_{3,u},L_{3,u})$ represents the unit element. $\vartheta_{n,k}'(t)$ is the derivative of $\vartheta_{n,k}(t)$ with respect to the $t$ variable, and $b \in R$ is a constant depending on exactly how one defines $\mathit{CF}^*(L_{3,u},L_{3,u})$ and the $A_\infty$-multiplications involving it. 

\begin{remark}
It it maybe helpful to explain why, for general reasons, the ambiguity takes on the form described in \eqref{eq:mu3}. Different choices made in the setup of the Fukaya category yield maps related by 
\begin{equation} \label{eq:mu3-setup1}
\begin{aligned}
& \tilde{\mu}^3_{\mathit{Fuk}(T)}(z_{2,u},w_1,y_{1,u}) 
- \mu^3_{\mathit{Fuk}(T)}(z_{2,u},w_1,y_{1,u}) \\ 
& = \phi^2(\mu^2_{\mathit{Fuk}(T)}(z_{2,u},w_1),y_{1,u}) +
\phi^2(z_{2,u},\mu^2_{\mathit{Fuk}(T)}(w_1,y_{1,u})) \\
& = \vartheta_{2,1}(u) \big( \phi^2(z_{1,u},y_{1,u}) - \phi^2(z_{2,u}, y_{2,u}) \big),
\end{aligned}
\end{equation}
respectively
\begin{equation} \label{eq:mu3-setup2}
\begin{aligned}
& \tilde{\mu}^3_{\mathit{Fuk}(T)}(z_{2,u},w_2,y_{1,u}) 
- \mu^3_{\mathit{Fuk}(T)}(z_{2,u},w_2,y_{1,u}) \\ 
& = \phi^2(\mu^2_{\mathit{Fuk}(T)}(z_{2,u},w_2),y_{1,u}) +
\phi^2(z_{2,u},\mu^2_{\mathit{Fuk}(T)}(w_2,y_{1,u})) \\
& = -\vartheta_{2,2}(u) \big( \phi^2(z_{1,u},y_{1,u}) - \phi^2(z_{2,u}, y_{2,u}) \big).
\end{aligned}
\end{equation}
Here, $\phi^2$ are bilinear maps of degree $-1$, which appear as components of the $A_\infty$-isomorphism relating the Fukaya categories for the two choices of construction. These are unknown a priori, but crucially the same expressions involving them appear in \eqref{eq:mu3-setup1} and \eqref{eq:mu3-setup2}. To establish the connection with the notation in \eqref{eq:mu3}, one would write
\begin{equation}
 (\tilde{b} - b)\, e_{3,u} = \phi^2(z_{1,u},y_{1,u}) - \phi^2(z_{2,u}, y_{2,u}).
\end{equation}
\end{remark}

To check \eqref{eq:mu3} concretely, one can adopt a Morse-Bott approach (see Section \ref{subsec:clean} below for more explanations and references), in which generators of $\mathit{CF}^*(L_{3,u},L_{3,u})$ correspond to the unique minimum and maximum of a Morse function on $L_{3,u}$. Then \eqref{eq:mu3} is determined by counting triangles with sides on $(L_{3,u},L_1,L_2)$ and with an additional marked boundary point which goes through the minimum point. Figure \ref{fig:mu3} shows the universal cover $\tilde{T} = \bR^2$, with three triangles. If we choose the minimum to be the {\em white} dot, the resulting coefficient in $\mu^3_{\mathit{Fuk}(T)}(z_{2,u},w_1,y_{1,u})$ is
\begin{equation}
\cdots - \hbar^{1/4} u + \hbar^{1/4} u^{-1} + 3 \hbar^{9/4} u^{-3} + \cdots = -u\vartheta_{2,1}'(u).
\end{equation}
If we move the minimum to the {\em black} dot, the coefficient changes to
\begin{equation}
\cdots -2 \hbar^{1/4} u + 0 \hbar^{1/4} u^{-1} + 2 \hbar^{9/4} u^{-2} + \cdots = -u \vartheta_{2,1}'(u) - \vartheta_{2,1}(u). 
\end{equation}
This ambiguity is of the form described in \eqref{eq:mu3}. One can check that the other $\mu^3$ computation behaves compatibly with that.

\begin{figure}
\begin{center}
\begin{picture}(0,0)%
\includegraphics{mu3.pstex}%
\end{picture}%
\setlength{\unitlength}{2763sp}%
\begingroup\makeatletter\ifx\SetFigFont\undefined%
\gdef\SetFigFont#1#2#3#4#5{%
  \reset@font\fontsize{#1}{#2pt}%
  \fontfamily{#3}\fontseries{#4}\fontshape{#5}%
  \selectfont}%
\fi\endgroup%
\begin{picture}(2724,5274)(1789,-4948)
\put(2206,-2161){\makebox(0,0)[lb]{\smash{{\SetFigFont{10}{12.0}{\rmdefault}{\mddefault}{\updefault}{\color[rgb]{0,0,0}$(0,0)$}%
}}}}
\end{picture}%
\caption{\label{fig:mu3}}
\end{center}
\end{figure}%

\begin{proof}[Proof of Lemma \ref{th:exact-triangle}]
Because the underlying chain complexes have trivial differentials, one actually knows that the composition of any two maps in \eqref{eq:123-triangle} is zero on the chain level. Write $v \in \mathit{CF}^0(L_1,L_2)$ for the morphism appearing in \eqref{eq:123-triangle}. Then the maps
\begin{equation}
\begin{aligned}
& (0,z_{2,u}) \in \mathit{hom}_{\mathit{Fuk}(T)^{\mathit{tw}}}^0(C_v,L_{3,u}) = \mathit{CF}^{-1}(L_1,L_{3,u}) \oplus \mathit{CF}^0(L_2,L_{3,u}), \\
& (y_{1,u},0) \in \mathit{hom}_{\mathit{Fuk}(T)^{\mathit{tw}}}^0(L_{3,u},C_v) = \mathit{CF}^1(L_{3,u},L_1) \oplus \mathit{CF}^0(L_{3,u},L_2)
\end{aligned}
\end{equation}
are cocycles. From \eqref{eq:mu3} and \eqref{eq:derivative-identity} one sees that their composition in one direction is
\begin{equation}
\mu^2_{\mathit{Fuk}(T)^{\mathit{tw}}}((0,z_{2,u}),(y_{1,u},0)) =
\mu^3_{\mathit{Fuk}(T)}(z_{2,u},v,y_{1,u}) = e_{3,u}.
\end{equation}
The analogous properties hold for the maps
\begin{equation}
\begin{aligned}
& (0,z_{2,u^{-1}}) \in \mathit{hom}_{\mathit{Fuk}(T)^{\mathit{tw}}}^0(C_v,L_{3,u^{-1}}), \\
& (-y_{1,u^{-1}},0) \in \mathit{hom}_{\mathit{Fuk}(T)^{\mathit{tw}}}^0(L_{3,u^{-1}},C_v). \\
\end{aligned}
\end{equation}
This shows that $L_{3,u} \oplus L_{3,u^{-1}}$ is a direct summand of $C_v$, but then a comparison of the sizes of the endomorphism rings, with one side computed as in \eqref{eq:cv-endomorphism}, shows that the two must actually be quasi-isomorphic. Moreover, these quasi-isomorphisms fit in with \eqref{eq:123-triangle}.
\end{proof}

\begin{remark} \label{th:twist-sequence}
Suppose that $u = -\hbar^{m_0}$ for some $m_0 \notin \half \bZ$, which means that $L_{3,u^{\pm 1}} = L_{3,-\hbar^{\pm m_0}}$ can be thought of as curves equipped with the nontrivial {\em Spin} structure. From the well-known exact triangle associated to a Dehn twist, and the Hamiltonian isotopy
\begin{equation}
L_2 \htp \tau_{L_{3,-\hbar^{m_0}}}\tau_{L_{3,-\hbar^{-m_0}}}(L_1),
\end{equation}
one can derive the existence of a diagram involving the same objects as in \eqref{eq:123-triangle}. This can be generalized as follows. For $a \in \mathit{GL}_0(1,R)$, let $\theta_a: \mathit{Fuk}(T)^{\mathit{tw}} \rightarrow \mathit{Fuk}(T)^{\mathit{tw}}$ be the functor obtained by tensoring all objects with a flat $R$-line bundle on $M$ which has monodromy $a$ in $q$-direction (and trivial monodromy in $p$-direction). Supposing that $u = -a \hbar^{m_0} \notin \{\pm \hbar^{\bZ/2}\}$, we have quasi-isomorphisms in $\mathit{Fuk}(T)^{\mathit{tw}}$,
\begin{equation}
\begin{aligned}
L_2 
& \htp 
\theta_{a^{-1}}\, 
\tau_{L_{3,-\hbar^{m_0}}}\, 
\theta_{a^2}\,
\tau_{L_{3,-\hbar^{-m_0}}}(L_1), \\
& \htp \theta_{a^{-1}}\, \tau_{L_{3,-\hbar^{m_0}}}\, \theta_{a^2}(\mathit{Cone}(L_{3,-\hbar^{-m_0}}[-1] \longrightarrow L_1)) \\
& \htp \theta_{a^{-1}} \tau_{L_{3,-\hbar^{m_0}}}(\mathit{Cone}((L_{3,-a^{-2} \hbar^{-m_0}})[-1] \longrightarrow L_1)) \\
& \htp \theta_{a^{-1}}(\mathit{Cone}(L_{3,-\hbar^{m_0}} [-1] \oplus L_{3,-a^{-2} \hbar^{-m_0}}[-1] \longrightarrow L_1)) \\ & \htp \mathit{Cone}(L_{3,u}[-1] \oplus L_{3,u^{-1}}[-1] \rightarrow L_1).
\end{aligned}
\end{equation}
However, from this point of view it is not straightforward to write down explicitly the maps involved in the exact triangle, in particular the counterpart of the horizontal one in \eqref{eq:123-triangle} (which would be given by counting holomorphic sections of a Lefschetz fibration).
\end{remark}

\begin{remark}
There is another (more direct) geometric approach to Lemma \ref{th:exact-triangle}, where instead of Dehn twists, one uses the relation between forming cones in $\mathit{Fuk}(T)^{\mathit{tw}}$ and the connected sum of Lagrangian submanifolds intersecting at a point (for proofs of that relationship, see \cite{fooo-extra} and \cite[Section 6]{biran-cornea13}; for a discussion of the specific example relevant here, see \cite[Lecture 23]{auroux-lecture}. The extension to Lagrangian submanifolds with local systems is quite natural in this context).
\end{remark}

By Lemma \ref{th:product}, we know that $\mathit{Fuk}(T)$ induces an $A_\infty$-structure on $Q$, which is necessarily quasi-isomorphic to $Q_{\tilde{p}}$ for some $\tilde{p}$.

\begin{lemma} \label{th:p-p2}
$\tilde{p}$ is a nonzero constant multiple of the unit torus polynomial $p$.
\end{lemma}

\begin{proof}
The considerations above show that the cones $C_v$ split into orthogonal direct summands for all $v = (\vartheta_{2,2}(u),\vartheta_{2,1}(u)) \in V$, as long as $u \notin \pm h^{k/2}$. Hence $\tilde{p}$ is nonzero at all those points, by Lemma \ref{th:cone-splitting}.

We can use a symmetry trick to derive a little bit of additional information from the geometry. Consider translation by $(\half,0)$, which is a free symplectic involution of $T$ preserving $dz$. This maps each $L_i$ to itself, hence induces an action on $\mathit{HF}^*(L_i,L_j)$. The action on $\mathit{HF}^*(L_i,L_i) \iso H^*(L_i;R)$ is trivial, whereas that on $\mathit{HF}^*(L_1,L_2)$ maps $w_1 \rightarrow - w_2$ and vice versa. It is not hard to lift this to an action on $\mathit{Fuk}(T)$, and there is an equivariant analogue of Proposition \ref{th:classify-deformations}, which implies that $\tilde{p}$ must be invariant under $(v_1,v_2) \mapsto (-v_2,-v_1)$. Another automorphism of the Fukaya category is the tensor product operation $\theta_{-1}$ which already appeared in Remark \ref{th:twist-sequence}. This preserves both our $L_i$. The induced action on $\mathit{HF}^*(L_i,L_i)$ is trivial, whereas that on $\mathit{HF}^*(L_1,L_2)$ preserves one of the two generators, and reverses the sign of the other one (exactly which one this is depends on how one trivializes the restriction of the line bundle to our Lagrangian submanifolds). In the same way as before, this implies that $\tilde{p}$ is invariant under $(v_1,v_2) \mapsto (-v_1,v_2)$. As a consequence, its order of vanishing of at the four points $(\vartheta_{2,2}(\pm 1), \vartheta_{2,1}(\pm 1))$, $(\vartheta_{2,2}(\pm \hbar^{1/2}),\vartheta_{2,2}(\pm \hbar^{1/2}))$ must be the same. Since $\tilde{p}$ is nonzero everywhere else, it must have simple zeros at all the four points, which implies the desired result.
\end{proof}

\begin{lemma} \label{th:fukaya-torus}
The constant from Lemma \ref{th:p-p2} is trivial, meaning that $\tilde{p} = p$.
\end{lemma}

\begin{proof}
Consider the degree zero endomorphism of $L_{3,u} \oplus L_{3,u^{-1}}$ given by
\begin{equation} \label{eq:geometric-t-tilde}
\tilde{t} = \frac{[e_{3,u}] \oplus [-e_{3,u^{-1}}]}{2 \vartheta_{4,3}'(1)(\vartheta_{4,1}(u) - \vartheta_{4,3}(u))} \in \mathit{HF}^0(L_{3,u},L_{3,u}) \oplus \mathit{HF}^0(L_{3,u^{-1}}, L_{3,u^{-1}}).
\end{equation}
Its square is clearly a multiple of the identity, and by \eqref{eq:more-products} we have
\begin{equation} \label{eq:dual-triple-product}
([y_{1,u}], [-y_{1,u^{-1}}]) \tilde{t} ([z_{2,u}],[z_{2,u^{-1}}]) = \frac{\vartheta_{2,2}(u)[w_3] + \vartheta_{2,1}(u)[w_4]}{\vartheta_{4,3}'(1)(\vartheta_{4,1}(u) - \vartheta_{4,3}(u))},
\end{equation}
which means that $\tilde{t}$ precisely satisfies the assumptions of Lemma \ref{th:exactly-complete-triangle}. As a result, we get the following information concerning $\tilde{p}$:
\begin{equation} \label{eq:determine-ptilde}
\begin{aligned}
& \tilde{p}\left (1,\frac{\vartheta_{2,1}(u)}{\vartheta_{2,2}(u)}\right) \\
& = \frac{\vartheta_{4,3}'(1)^4(\vartheta_{4,1}(u) - \vartheta_{4,3}(u))^4}{\vartheta_{2,2}(u)^4} \tilde{p}\left(\frac{\vartheta_{2,2}(u)}{\vartheta_{4,3}'(1)(\vartheta_{4,1}(u) - \vartheta_{4,3}(u))}, \frac{\vartheta_{2,1}(u)}{\vartheta_{4,3}'(1)(\vartheta_{4,1}(u) - \vartheta_{4,3}(u))}\right) \\
& = \frac{\vartheta_{4,3}'(1)^4(\vartheta_{4,1}(u) - \vartheta_{4,3}(u))^4}{\vartheta_{2,2}(u)^4} \, \frac{\tilde{t}^2}{[e_{3,u}] \oplus [e_{3,u^{-1}}]}
= \frac{\vartheta_{4,3}'(1)^2(\vartheta_{4,1}(u) - \vartheta_{4,3}(u))^2}{4 \vartheta_{2,2}(u)^4}.
\end{aligned}
\end{equation}
In terms of the parametrization \eqref{eq:theta-parametrization}, this shows that $s_1^2 = \tilde{p}(1,s_2)$, which implies that $\tilde{p} = p$.
\end{proof}

\begin{remark}
What does this say about the actual $A_\infty$-products in the Fukaya category? By definition, $\mu^4_{Q_p}(w_3,w_1,w_3,w_1)$ is given by the coefficient of $v_1^4$ in $p(v_1,v_2)$, which is
\begin{equation}
\begin{aligned}
& -\hbar^{1/4} \vartheta(1)^{-2}\vartheta(-1)^{-2}\vartheta(\hbar^{1/2})^{-2} \vartheta_{4,3}'(1)^2 \vartheta_{2,1}(\hbar^{1/2})^2 \vartheta_{2,1}(1)^2 \\
& = -\textstyle{\frac{1}{4}}\hbar^{1/4} \vartheta(1)^{-2}\vartheta(-1)^{-2}\vartheta(\hbar^{1/2})^2 \vartheta_{4,3}'(1)^2
= -\hbar^{1/2} - 4 \hbar^{3/2} + \cdots
\end{aligned}
\end{equation}
At least on the two leading orders we've written down, this agrees with the result of counting holomorphic squares with vertices on $(w_3,w_1,w_3,w_1)$ and which go through an additional generic marked point of $L_1$. However, a direct attempt to compute all of $\mu^4$ directly in $\mathit{Fuk}(T)$ is tricky, because the moduli spaces of constant holomorphic discs mapping to points of $L_1 \cap L_2$ are not regular. This difficulty is avoided in the approach we've chosen here.
\end{remark}

\begin{corollary} \label{th:whole-fukaya-torus}
$Q_p^{\mathit{perf}}$ is quasi-equivalent to $\mathit{Fuk}(T)^{\mathit{perf}}$.
\end{corollary}

\begin{proof}
From Lemma \ref{th:fukaya-torus} we get a full and faithful functor $Q_p^{\mathit{perf}} \rightarrow \mathit{Fuk}(T)^{\mathit{perf}}$. The only additional fact needed is that the objects $L_1,L_2$ which are in the image of this functor split-generate the Fukaya category. They clearly split-generate the object $L_{3,u}$ for generic $u$, by the previous argument. On the other hand, one can use \cite[Corollary 5.8]{seidel04} (together with suitable tensor product functors) to show that $L_1,L_{3,u}$ split-generate the Fukaya category. For alternative approaches, see \cite{abouzaid-smith09} or the review in Section \ref{subsec:split-generate} below.
\end{proof}

The Fukaya category comes with a canonical {\em open-closed string map} (which has a long history going back to \cite{kontsevich94}, see also Section \ref{subsec:define-open-closed} below)
\begin{equation} \label{eq:open-closed-t2}
H^*(T;R) \longrightarrow \mathit{HH}^*(\mathit{Fuk}(T),\mathit{Fuk}(T)).
\end{equation}

\begin{corollary} \label{th:t2-family}
Let $\SS = \mathit{Spec}(\RR)$ be the affine curve associated to the unit torus polynomial $p$, and $\theta$ its standard one-form. There is a perfect family of modules over $\mathit{Fuk}(T)$ parametrized by $\SS$, which follows the image of
\begin{equation} \label{eq:explicit-current}
\theta \otimes [dq] \in H^0(\SS,\Omega^1_\SS) \otimes H^1(T;R)
\end{equation}
under the open-closed string map. The fibre of this family associated to a point $(s_1,s_2) \in \SS$ is isomorphic to $L_{3,u}$, where $u \in R^\times/\hbar^{\bZ}$ satisfies \eqref{eq:theta-parametrization}.
\end{corollary}

\begin{proof}
The constant term of the open-closed string map, for any object $L$, is the standard map $H^*(T;R) \rightarrow H^*(L;R)$. We restrict this to the subcategory consisting of the objects $L_1,L_2$, so that it lands in $\mathit{HH}^*(Q_p,Q_p)$. With respect to the basis from Addendum \ref{th:quasi-generators}, $[dp] \in H^1(T;\bR)$ maps to $[g_1]+[g_2]$, and $[dq]$ to $-2[g_2]$. We take the family $\MM$ from Section \ref{subsec:tautological} and carry it over to the Fukaya category through the equivalence from Corollary \ref{th:whole-fukaya-torus}. By construction this follows the deformation field $[\gamma]$ from \eqref{eq:elliptic-deformation-field}, which is indeed the image of \eqref{eq:explicit-current}.

By definition, the object of this family associated to a point $(s_1,s_2)$ is the direct summand of $C_{(1,s_2)}$ associated to the projection $\half (e + s_1^{-1} t)$. Reversing the rescaling applied in \eqref{eq:determine-ptilde}, one finds that this is quasi-isomorphic to the direct sumand of $C_v$ associated to the projection
\begin{equation}
\half (e + \vartheta_{2,2}(u)^{-2} \vartheta_{4,3}'(1)^2 (\vartheta_{4,1}(u)-\vartheta_{4,3}(u))^2 s_1^{-1} t) = \half e + \vartheta_{4,3}'(1) (\vartheta_{4,1}(u) - \vartheta_{4,3}(u)) t.
\end{equation}
Under the isomorphism $C_{(\vartheta_{2,2}(u),\vartheta_{2,1}(u))} \iso L_{3,u} \oplus L_{3,u^{-1}}$, $t$ goes to the endomorphism $\tilde{t}$ from \eqref{eq:geometric-t-tilde}, so the corresponding projection is
\begin{equation}
\half ([e_{3,u}] \oplus [e_{3,u^{-1}}]) + \vartheta_{4,3}'(1) (\vartheta_{4,1}(u) - \vartheta_{4,3}(u)) \tilde{t} = [e_{3,u}],
\end{equation}
which indeed picks out the summand $L_{3,u}$.
\end{proof}

\section{Symplectic automorphisms\label{sec:automorphisms}}

The automorphism group of a symplectic manifold of dimension $\geq 4$ often has many connected components which map to the identity component of the diffeomorphism group (see for instance \cite{seidel97,seidel04b}). One way to detect this phenomenon is by using fixed point Floer cohomology. Through the connection with the Lagrangian Floer cohomology of graphs, this also provides interesting examples of Lagrangian submanifolds in products. In this section, we discuss both versions of Floer theory (with emphasis on computational methods that will be useful later in the paper), and then consider some specific examples of automorphisms obtained as compositions of Dehn twists.

We will work in a ``symplectic Calabi-Yau'' context, in which Floer cohomology groups are defined over the Novikov field $R$ \eqref{eq:novikov-field} and carry absolute $\bZ$-gradings. As a side-effect, this makes the definition of fixed point Floer cohomology technically simpler. We will impose additional restrictions on Lagrangian submanifolds, which rule out bubbling of holomorphic discs, hence permit a similar simplification to take place in the construction of Lagrangian Floer cohomology.

\subsection{Fixed point Floer cohomology\label{subsec:fixed-point-floer}}
Let $M^{2n}$ be a (connected) closed symplectic manifold, satisfying

\begin{assumption} \label{th:calabi-yau}
$c_1(M) = 0$. In fact, we want to fix a trivialization of the anticanonical line bundle $K_M^{-1} = \Lambda^n_\bC(TM)$ (for some compatible almost complex structure).
\end{assumption}

The choice of trivialization allows one to define the notion of graded symplectic automorphism \cite{seidel99}. 
%
Fixed point Floer cohomology \cite{dostoglou-salamon94,seidel97,seidel04b} associates to each graded symplectic automorphism $f$ a $\bZ$-graded $R$-vector space $\mathit{HF}^*(f)$, whose Euler characteristic is the Lefschetz number of $f$ (up to a sign which depends on the choice of grading).
This invariant comes with a rich structure of operations, among which we list the basic ones.

{\em The pair-of-pants product.} This is an associative multiplication
$\mathit{HF}^*(f_2) \otimes \mathit{HF}^*(f_1) \rightarrow \mathit{HF}^*(f_2f_1)$, which comes with a two-sided unit element in $\mathit{HF}^*(\mathit{id})$.
The standard example is $f_2 = f_1 = \mathit{id}$, where the canonical isomorphism \cite{piunikhin-salamon-schwarz94} $H^*(M;R) = \mathit{QH}^*(M) \iso \mathit{HF}^*(\mathit{id})$ identifies the pair-of-pants product with the small quantum product (in particular, the unit is the standard one in $H^0(M;R) = R$). As a consequence, any fixed point Floer cohomology group $\mathit{HF}^*(f)$ inherits the structure of a $\mathit{QH}^*(M)$-module \cite{floer88,schwarz96}.

{\em Duality.} There is a distinguished co-unit $\mathit{HF}^{2n}(\mathit{id}) \rightarrow R$ (in terms of the isomorphism with ordinary cohomology, it is the standard integration map). In combination with the pair-of-pants product, the co-unit gives rise to a nondegenerate pairing $\mathit{HF}^{2n-*}(f) \otimes \mathit{HF}^*(f^{-1}) \rightarrow R[-2n]$.

{\em Continuation elements.} Let $\{f_t\}$ be a Hamiltonian isotopy of graded symplectic automorphisms, with $f_0 = f$ and $f_1 = \mathit{id}$. This determines an element $I_{\{f_t\}} \in \mathit{HF}^*(f)$. The pair-of-pants product with such elements is used to prove Hamiltonian isotopy invariance of general fixed point Floer cohomology groups.

{\em Conjugation isomorphisms.} These are canonical isomorphisms $C_{f_2,f_1}: \mathit{HF}^*(f_1) \longrightarrow \mathit{HF}^*(f_2f_1f_2^{-1})$.
Besides their general functoriality properties, which say that $C_{f_3,f_2f_1f_2^{-1}}C_{f_2,f_1} = C_{f_3f_2,f_1}$ and $C_{\mathit{id},f} = \mathit{id}$, we have the self-conjugation identity
\begin{equation} \label{eq:self-conjugation}
C_{f,f} = \mathit{id}.
\end{equation}
This implies that $C_{f,f^m}$ generates an action of $\bZ/m$ on $\mathit{HF}^*(f^m)$.

\begin{remark} \label{th:tcft}
All these operations have chain map realizations on the level of the Floer complexes $\mathit{CF}^*(f)$, and the relations between them are given by appropriate chain homotopies. Here is a more systematic way to approach the formal description of the theory \cite{seidel97,seidel04b}. Write $\mathit{Aut}^{\mathrm{gr}}(M)$ for the group of graded symplectic automorphisms, equipped with the Hamiltonian topology (in which only Hamiltonian isotopies are continuous). Fixed point Floer cohomology can be viewed as a $(1+1)$-dimensional TCFT (topological conformal field theory) with target space $B\mathit{Aut}^{\mathrm{gr}}(M)$, which means a TCFT for surfaces equipped with graded Hamiltonian fibrations. In this framework, we view a symplectic automorphism as giving rise to its mapping torus
\begin{equation} \label{eq:ordinary-mapping-torus}
Z_f = \bR \times M\;\; / \;\; (t,x) \sim (t-1,f(x)),
\end{equation}
which is an $\mathit{Aut}^{\mathrm{gr}}(M)$-fibration over $S^1 = \bR/\bZ$; the TCFT associates to that fibration a chain complex, which is $\mathit{CF}^*(f)$. The fibrewise action of $f_2$ yields an isomorphism $Z_{f_1} \rightarrow Z_{f_2f_1f_2^{-1}}$, for which there is an associated chain map $c_{f_2,f_1}$ inducing the previously introduced conjugation maps. This for instance explains \eqref{eq:self-conjugation}: even though the fibrewise action of $f$ on $Z_f$ itself is nontrivial, it can be deformed continuously to the identity through rotations of the base, and this gives rise to a chain homotopy between $c_{f,f}$ and the identity. Similarly, $c_{f,f^m}$ is chain homotopic to the order $m$ automorphism
\begin{equation}
\begin{aligned}
& Z_{f^m} \longrightarrow Z_{f^m},
& (t,x) \longmapsto (t-\textstyle\frac1m,f(x)).
\end{aligned}
\end{equation}
\end{remark}

The actual definition of fixed point Floer cohomology is a mild generalization of the better-known Hamiltonian Floer cohomology. For simplicity, assume that $f$ has nondegenerate fixed points. The graded $R$-vector space $\mathit{CF}^*(f)$ is the direct sum of one-dimensional spaces $o_x \iso R$ associated to fixed points $x$. Each such point has an absolute Conley-Zehnder index $\mathrm{deg}(x) \in \bZ$, which determines the degree in which $o_x$ is placed. Take a family $J_f = (J_{f,t})$ of almost complex structures, parametrized by $t \in \bR$ and satisfying
\begin{equation} \label{eq:j-periodicity}
J_{f,t-1} = f_*(J_{f,t}).
\end{equation}
The differential $d: \mathit{CF}^*(f) \rightarrow \mathit{CF}^{*+1}(f)$ counts solutions of
\begin{equation} \label{eq:closed-floer}
\left\{
\begin{aligned}
& u: \bR \times \bR \longrightarrow M, \\
& u(s,t-1) = f(u(s,t)), \\
& \partial_s u + J_{f,t}(u)\,\partial_t u = 0
\end{aligned}
\right.
\end{equation}
asymptotic to fixed points as $s \rightarrow \pm\infty$, with powers $\hbar^{E(u)}$ given by their energies $E(u) = \int_{\bR \times [0,1]} u^*\omega_M$. The technical trick is to avoid bubbling off of holomorphic spheres, which can be done by a dimension-counting argument as in \cite{hofer-salamon95}. The outcome is independent of the choice of almost complex structure up to quasi-isomorphism. These quasi-isomorphisms are defined through continuation maps, and are ``essentially canonical'' (which means unique up to chain homotopies, which can be extended to a system of higher homotopies; this is what makes it possible to omit the almost complex structures from a formal description as in Remark \ref{th:tcft}). One special application is loop rotation. Suppose that we have chosen a family $J_{f,+}$ of almost complex structures as in \eqref{eq:j-periodicity}. Fix some constant $t_0 \in \bR$, and set $J_{f,-}(t) = J_{f,+}(t-t_0)$. If $u_+(s,t)$ is a solution of \eqref{eq:closed-floer} for $J_{f,+}$, then $u_-(s,t) = u_+(s,t-t_0)$ is a solution of the corresponding equation for $J_{f,-}$. This implies that the associated differentials $d_{\pm}$ agree. Note that on the other hand, we have a quasi-isomorphism $(\mathit{CF}^*(f),d_+) \rightarrow (\mathit{CF}^*(f),d_-)$ defined through continuation maps, which means solutions of
\begin{equation} \label{eq:cont-floer}
\left\{
\begin{aligned}
& u: \bR \times \bR \longrightarrow M, \\
& u(s,t-1) = f(u(s,t)), \\
& \partial_s u + J_{\mathrm{cont},s,t}(u)\,\partial_t u = 0
\end{aligned}
\right.
\end{equation}
where $J_{\mathrm{cont}} = (J_{\mathrm{cont},s,t})$ is a two-parameter family with the same periodicity \eqref{eq:j-periodicity} in $t$-direction, and such that $J_{\mathrm{cont},s,t} = J_{f,\pm,t}$ for $\pm s \gg 0$. There is a parametrized moduli problem which involves varying $t_0$ in an interval, and this yields a chain homotopy showing that the continuation map is homotopic to the previously defined isomorphism of Floer chain complexes. Suppose for instance that we take $t_0 = 1$, in which case $J_{f,-} = f_*J_{f,+}$. Then, the argument we have just outlined explains \eqref{eq:self-conjugation}.

Another part of the theory for which we'll need an explicit expression is the structure of $\mathit{HF}^*(f)$ as a module over $\mathit{QH}^*(M)$, sometimes called the {\em quantum cap product}. Fix a Morse function $h$ on $M$ (whenever we do that, we also tacitly choose a Riemannian metric, which is used to form $\nabla h)$, and let $\mathit{CM}^*(h)$ be the resulting Morse cochain complex. Suppose that we have fixed the almost complex structure $J_f$ defining the Floer differential. Then, choose a family $J_{\mathrm{cap}}$ in a similar way as for $J_{\mathrm{cont}}$, but where now the behaviour on both ends $s \rightarrow \pm\infty$ is given by $J_f$. Choose also a family $h_{\mathrm{cap},s}$, $s \in [0,\infty)$, of functions (with their associated metrics), such that $h_{\mathrm{cap},s} = h$ for $s \gg 0$. Then, consider pairs $(u_1,u_2)$ as follows:
\begin{equation} \label{eq:cap-floer}
\left\{
\begin{aligned}
& u_1: \bR \times \bR \longrightarrow M, \\
& u_1(s,t-1) = f(u_1(s,t)), \\
& \partial_s u_1 + J_{\mathrm{cap},s,t}(u_1)\,\partial_t u_1 = 0, \\
& u_2: [0,\infty) \longrightarrow M, \\
& du_2/ds + \nabla h_{\mathrm{cap},s}(u_2) = 0, \\
& u_1(0,0) = u_2(0).
\end{aligned}
\right.
\end{equation}
$u_1$ should be asymptotic to fixed points of $f$ at both ends, and $u_2$ is asymptotic to a critical point of $h$. A count of the number of solutions of \eqref{eq:cap-floer} yields a chain map representing the quantum cap product \cite{schwarz96}:
\begin{equation} \label{eq:quantum-cap}
\mathit{CM}^*(h) \otimes \mathit{CF}^*(f) \longrightarrow \mathit{CF}^*(f).
\end{equation}

\begin{remark}
It is in fact possible to choose $J_{\mathrm{cap},s,t} = J_{f,t}$ and $h_{\mathrm{cap},s} = h$, and this leads to the more familiar picture of ``cutting down moduli spaces''. However, the greater freedom allowed above is more natural, and also technically useful.
\end{remark}

\subsection{Clean intersections\label{subsec:clean}}
For $M$ as before (Assumption \ref{th:calabi-yau}), we will consider Lagrangian submanifolds with the following added properties and structure.

\begin{assumption} \label{th:adiscic}
Each Lagrangian submanifold $L$ is equipped with a grading (and hence an orientation), a {\it Spin} structure, as well as a local coefficient system with holonomy in \eqref{eq:gl-subgroup}. Moreover, it comes with a compatible almost complex structure $J_L$ with the following property. The subset of points of $L$ which lie either on a non-constant $J_L$-holomorphic sphere $\bC P^1 \rightarrow M$, or on the boundary of a non-constant $J_L$-holomorphic disc $(D,\partial D) \rightarrow (M,L)$, has dimension $\leq n-3$.
\end{assumption}

Here, by a subset of dimension $\leq k$, we mean one that is contained in the image of a smooth map from a (possibly noncompact and disconnected) manifold of dimension $\leq k$ to $L$. Note that the assumption on $J_L$-holomorphic spheres is a generic one (since the image of all such spheres is generically a subset of dimension $\leq 2n-4$ in $M$), but that on discs is not (assuming regularity, the boundary points of such discs would be of dimension $\leq n-2$ in $L$, while we require one more dimension). Hence, in order to check Assumption \ref{th:adiscic} in applications, one needs to know a $J_L$ for which the pseudo-holomorphic discs can be controlled very specifically. An exception to this is the low-dimensional situation $n \leq 2$, where the moduli spaces of pseudo-holomorphic discs are generically empty.

The Floer cohomology of two submanifolds satisfying Assumption \ref{th:adiscic} is fairly straightforward to define. To make later computations easier, we adopt a Morse-Bott approach \cite{pozniak, bourgeois02, biran-cornea09, biran-cornea09c, johns08, sheridan11}. A small amount of technicalities will be included, but without any attempt at completeness or full justification. Take $(L_0,L_1)$, each satisfying Assumption \ref{th:adiscic}, and which have clean intersection \cite{pozniak}. Choose a Morse function $h_{L_0,L_1}$ on $L_0 \cap L_1$. The Morse-Bott type Floer cochain complex is a modification of the Morse cochain space of that function, more precisely:
\begin{equation} \label{eq:hom-c}
\mathit{CF}^*(L_0,L_1) \stackrel{\mathrm{def}}{=} \bigoplus_C \mathit{CM}^{*-\mathrm{deg}(C)}(h_{L_0,L_1}|C;\mathit{Hom}(\xi_0,\xi_1)|C \otimes o_C),
\end{equation}
where the direct sum is over connected components $C \subset L_0 \cap L_1$; the dimension offset $\mathrm{deg}(C) \in \bZ$ is an absolute Maslov index, which depends on the gradings; $\xi_k$ are the given local systems on $L_k$, which we restrict to $L_0 \cap L_1$; and there is an additional local system $o_C \rightarrow C$ with holonomy $\pm 1$, which depends on the {\em Spin} structures. Choose also a family $J_{L_0,L_1} = (J_{L_0,L_1,t})$ of almost complex structures, parametrized by $t \in [0,1]$, and such that $J_{L_0,L_1,0} = J_{L_0}$, $J_{L_0,L_1,1} = J_{L_1}$. To define the Floer differential, one primarily considers holomorphic strips, which are non-constant solutions of
\begin{equation} \label{eq:strip-equation}
\left\{
\begin{aligned}
& u: \bR \times [0,1] \longrightarrow M, \\
& u(\bR \times \{0\}) \subset L_0, \;\; u(\bR \times \{1\}) \subset L_1, \\
& \partial_s u + J_{L_0,L_1,t}(u) \partial_t u = 0,
\end{aligned}
\right.
\end{equation}
which are asymptotic to points of $L_0 \cap L_1$ as $s \rightarrow \pm\infty$. However, these have to be combined with Morse theory in an appropriate way, which we now set out to describe (see the references above, especially \cite{biran-cornea09c}, for other accounts of this).

\begin{definition} \label{th:pearly-chain}
A {\em pearly chain} $T$ is a decorated graph of the following kind. First, the graph itself has only two-valent vertices, and two ends, one of which is singled out (and called the root, the other being the leaf). This determines an orientation of the graph (from the root to the leaf). Each edge $e$ is decorated with a closed interval $I_e \subset \bR$. This is unbounded below if and only if the edge contains the root, and unbounded above if and only if contains the other end. In the bounded case we allow the length to become zero, meaning that $I_e$ is a point (while still thinking of $e$ combinatorially as an edge). Finally, each vertex $v$ is equipped with the Riemann surface $S_v = \bR \times [0,1]$.
\end{definition}

Note that we allow one slightly exceptional case: namely, that $T$ has a single edge which is infinite in both directions, with $I_e = \bR$, and no vertices. It is convenient to associate to $T$ a topological space $\bar{S}_T$, obtained by compactifying each $S_v$ to $\bar{S}_v = S_v \cup \{s = \pm \infty\}$ (the closed unit disc), then identifying the added points with the endpoints of the intervals $I_e$ (compatibly with the orientations), and finally adding two more points at infinity to the ends of the non-compact intervals. The two points added in the last step will be denoted by $\bar{z}_0$ (corresponding to the root) and $\bar{z}_1$.

\begin{definition} \label{th:chain-perturbation}
A {\em perturbation datum} on a pearly chain is given by a family $h_e = (h_{e,s})$ of functions on $L_0 \cap L_1$ (with their associated choices of metrics), parametrized by $s \in I_e$, for each edge $e$, subject to the following additional conditions. If $e$ is the edge containing the root, then $h_{e,s} = h_{L_0,L_1}$ for $s \ll 0$. If $e$ contains the other end, $h_{e,s} = h_{L_0,L_1}$ for $s \gg 0$. In the exceptional case $I_e = \bR$, we ask that $h_{e,s} = h_{L_0,L_1}$ for all $s$.
\end{definition}

Given such a perturbation datum and a choice of critical points $x_0,x_1$ of $h_{L_0,L_1}$, one considers continuous maps $\bar{u}: \bar{S}_T \rightarrow M$ with $\bar{u}(\bar{z}_k) = x_k$, which satisfy the following equations. For any vertex $v$, the restriction of $\bar{u}$ to $S_v = \bR \times [0,1]$ yields a smooth non-constant map $u_v$ which solves \eqref{eq:strip-equation}. On the other hand, restriction to an interval $I_e$ yields
\begin{equation} \label{eq:u1-equation}
\left\{
\begin{aligned}
& u_e: I_e \longrightarrow L_0 \cap L_1, \\
& du_e/ds + \nabla h_{e,s}(u_e) = 0.
\end{aligned}
\right.
\end{equation}
Two maps which are related by a translation of the $u_v$ components are considered to be the same. Similarly, in the exceptional case $I_e = \bR$, we divide out by translation acting on $u_e$ (which is then also required to be non-constant).

To define the differential on $\mathit{CF}^*(L_0,L_1)$, one has to choose a perturbation datum on every pearly chain, depending smoothly on the lengths. There are additional consistency conditions beyond those in Definition \ref{th:chain-perturbation}, which appear in the limit when the length of some edge goes to infinity. We will not formulate these in detail (but see \cite{seidel04} for the general idea, and \cite{sheridan11} for a case closer to the one discussed here). One then considers the moduli space of solutions $\bar{u}$ of the equations above, varying over all pearly chains. For generic choice of perturbation data, a count of points in the zero-dimensional strata, with appropriate signs and energies, defines the coefficient of $(x_0,x_1)$ in the differential. Assumption \ref{th:adiscic} allows us to avoid bubbling of holomorphic discs.

\begin{remark}
One technical point deserves mention. Given connected components $C_0,C_1 \subset L_0 \cap L_1$, write $\MM(C_0,C_1)$ for the space of solutions of \eqref{eq:strip-equation} with limits in those components. Standard transversality theory shows that for generic choice of $J_{L_0,L_1}$, this is smooth of dimension $\mathrm{deg}(C_0) - \mathrm{deg}(C_1) + \mathrm{dim}(C_0) - 1$. In particular, for $C_0 = C_1$ the dimension is $\mathrm{dim}(C_0) - 1$, and by a further application of transversality theory one can achieve that the asymptotic evaluation map $\MM(C_0,C_0) \rightarrow C_0^2$ avoids the diagonal. This and similar arguments show that for generic choice of almost complex structures, the fibre products $\MM(C_0,C_1) \times_{C_1} \MM(C_1,C_2)$ are smooth of the expected dimension. In particular, for $C_0 = C_2$ that dimension is $\mathrm{dim}(C_0) - 2$, and one can again arrange that the asymptotic evaluation map $\MM(C_0,C_1) \times_{C_1} \MM(C_1,C_0) \rightarrow C_0^2$ avoids the diagonal. One can iterate that idea to higher fibre products. This is important since those products appear in our moduli spaces when the length of the intervals becomes zero (transversality for positive lengths is much simpler, since one can choose the families of functions $h_{e}$ essentially freely). Interested readers may want to consult \cite[Section 3.1.1]{biran-cornea09c}, where an argument in the same spirit is used to address the corresponding problem for monotone Lagrangian submanifolds.
\end{remark}


\begin{example} \label{th:graph}
Take a graded symplectic automorphism $f$ of $M$. Suppose that we have chosen a family $J_f$ as in \eqref{eq:j-periodicity}. Write $M^-$ for the same manifold but with the sign of the symplectic form reversed. The diagonal $\Delta \subset M^- \times M$ is Lagrangian, and admits a distinguished grading. Set $J_\Delta = (-J_{f,1/2}) \times J_{f,1/2}$. Then, holomorphic discs $(D,\partial D) \rightarrow (M^- \times M,\Delta)$ correspond bijectively to holomorphic spheres $\bC P^1 \rightarrow M$. Hence, the space of points of $\Delta$ lying on a non-constant disc is generically of dimension $\leq 2n-4$ (and the same holds for holomorphic spheres, for even more trivial reasons). Similarly, the graph $\Gamma = \{y = f(x)\}$ is a graded Lagrangian submanifold, which we can equip with $J_{\Gamma} = (-J_{f,1}) \times f_*(J_{f,1}) = (-J_{f,1}) \times J_{f,0}$. A holomorphic strip $(u_x,u_y): \bR \times [0,1] \longrightarrow M^- \times M$ which satisfies \eqref{eq:strip-equation} (with $L_0 = \Gamma$, $L_1 = \Delta$) for the family of almost complex structures
\begin{equation}
J_{\Gamma,\Delta,t} = (-J_{f,1-t/2}) \times J_{f,t/2}
\end{equation}
gives rise to a solution of \eqref{eq:closed-floer}. Namely, consider first
\begin{equation} \label{eq:roll-up}
\begin{aligned}
& u: \bR \times [0,1] \longrightarrow M, \\
& u(s,t) = \begin{cases} u_y(2s,2t) & t \leq 1/2, \\ u_x(2s,2-2t) & t \geq 1/2. \end{cases}
\end{aligned}
\end{equation}
This satisfies $\partial_s u + J_{f,t}(u)\partial_t u = 0$, and has the boundary periodicity condition $u(s,0) = f(u(s,1))$. By the removable singularity theorem, it extends to a solution of \eqref{eq:closed-floer}. The same machinery runs (a little more easily) in reverse, producing $(u_x,u_y)$ from $u$. Assuming that $M$ is {\em Spin} so as to make the Lagrangian submanifolds fit into our framework (one can avoid this assumption by being more careful about the role of relative {\em Spin} structures in Fukaya categories, see \cite{wehrheim-woodward11}), and taking $f$ to have nondegenerate fixed points for simplicity, it is then easy to show that $\mathit{HF}^*(\Gamma,\Delta) \iso \mathit{HF}^*(f)$.
\end{example}

\subsection{The $A_\infty$-structure\label{subsec:fukaya}}
We will now carry out the corresponding construction of the $A_\infty$-structure on Lagrangian Floer cochains. Fix, once and for all, a set of Lagrangian submanifolds in $M$. Each of them should satisfy Assumption \ref{th:adiscic}; and any two should have clean intersection. All Lagrangian submanifolds appearing in the following discussion are assumed to be taken from this set. We suppose that for any two $(L_0,L_1)$, the Floer complex $\mathit{CF}^*(L_0,L_1)$ with its differential $\mu^1_{\mathit{Fuk}(M)}$ has already been defined, which in particular means that functions $h_{L_0,L_1}$ and almost complex structures $J_{L_0,L_1}$ have been chosen. Again, we refer to the previously quoted literature, in particular \cite{sheridan11}, and additionally to \cite{fukaya-oh-ohta-ono08,seidel08b}.

\begin{definition}
Fix some $d \geq 1$. A {\em pearly tree} with $d$ leaves is a decorated graph of the following kind. The underlying graph $T$ is a ribbon tree with $(d+1)$ ends, one of which is singled out (and called the root, the others being the leaves). Moreover, it is assumed that all vertices $v$ of $T$ have valence $|v| \geq 2$. We orient the tree from the root to the leaves. Each edge $e$ is decorated with a closed interval $I_e \subset \bR$, with the same properties as in Definition \ref{th:pearly-chain}. Each vertex $|v|$ is decorated with a Riemann surface $S_v = D \setminus \{\bar{z}_{v,0},\dots,\bar{z}_{v,|v|-1}\}$, where $D$ is the closed unit disc and the $\bar{z}_{v,i}$ are cyclically ordered distinct boundary points.
\end{definition}

For $d = 1$ this reduces to a pearly chain, up to the irrelevant issue of choosing identifications between a two-punctured disc and $\bR \times [0,1]$. Note that for each vertex $v$, there is a preferred correspondence between ends of $S_v$ and edges adjacent to $v$, which is compatible with the cyclic ordering and assigns the point at infinity $\bar{z}_{v,0}$ to the edge oriented towards $v$ (we call this end of $S_v$ its negative end, and the others positive ends). Using that, we can construct a compact topological space $\bar{S}_T$, obtained by compactifying each $S_v$ to $\bar{S}_v = D$, identifying the endpoints of $I_e$ with the $\bar{z}_{v,k}$, and then adding points at infinity to the noncompact intervals. We denote by $\bar{z}_0,\dots,\bar{z}_d \in \bar{S}_T$ the points added in the last step, starting with the root and proceeding in the ordering given by a planar embedding of $T$.

Let's clear up few more book-keeping matters. Supposing that $T$ is embedded properly in $\bR^2$, we say that it is a {\em labeled pearly tree} if each component of $\bR^2 \setminus T$ comes with a Lagrangian submanifold (taken from our collection). One can in fact number these components by $\{0,\dots,d\}$, compatibly with the cyclic ordering and in such a way that the root separates the first and last component. Therefore, a labeling of $T$ just corresponds to a choice of Lagrangian submanifolds $(L_0,\dots,L_d)$. Suppose from now on that such a labeling has been fixed. For any edge $e$ we then have a pair $(L_{i_{e,0}},L_{i_{e,1}})$, corresponding to the components of $\bR^2 \setminus T$ lying to the left ($i_{e,1}$) and right ($i_{e,0}$) with respect to the orientation of $e$. By definition, $0 \leq i_{e,0} < i_{e,1} \leq d$. Similarly, for any vertex $v$, there is a canonical correspondence between the boundary components of $S_v$ and connected components of $\bR^2 \setminus T$ adjacent to $v$. If we label the boundary components by $\partial_0 S_v,\dots, \partial_{|v|-1} S_v$, so that the negative end separates the first and last one, then this leads to having associated Lagrangian submanifolds $(L_{i_{v,0}},\dots,L_{i_{v,|v|-1}})$ for $0 \leq i_{v,0} < \cdots < i_{v,|v|-1} \leq d$.

\begin{definition}
A {\em perturbation datum} on a labeled pearly tree consists of the following data. For each edge $e$, we want to have a family of functions $h_e = (h_{e,s})$ on $L_{i_{e,0}} \cap L_{i_{e,1}}$ parametrized by $s \in I_e$. If $I_e$ is noncompact, we ask that outside a compact subset, $h_{e,s} = h_{L_{i_{e,0}},L_{i_{e,1}}}$ is one of the previously chosen functions (and in the exceptional case $I_e = \bR$, we impose the same additional condition as in Definition \ref{th:chain-perturbation}).

Next, take a vertex of valence $|v| \geq 3$. We then want to choose strip-like ends on $S_v$, which means proper holomorphic embeddings
\begin{equation}
\begin{aligned}
& \epsilon_{v,0}: (-\infty,0] \times [0,1] \longrightarrow S_v, \\
& \epsilon_{v,1},\dots,\epsilon_{v,|v|-1}: [0,\infty) \times [0,1] \longrightarrow S_v
\end{aligned}
\end{equation}
giving preferred coordinates on its ends. Given those, we want to have a family $(J_{v,z})$ of compatible almost complex structures parametrized by $z \in S_v$. If $z \in \partial_k S_v$, then $J_{v,z} = J_{L_{i_{v,k}}}$ should be one of the structures that come from Assumption \ref{th:adiscic}. Moreover, on the strip-like ends
\begin{equation}
J_{v,\epsilon_{v,k}(s,t)} = \begin{cases} J_{L_{i_{v,0}},L_{i_{v,k}},t} & k = 0 \text{ and } s \ll 0, \\
J_{L_{i_{v,k-1}},L_{i_{v,k}},t} & k > 0 \text{ and } s \gg 0.
\end{cases}
\end{equation}
Additionally, we want to have a one-form $K_v$ on $S_v$ with values in the space $C^\infty(M,\bR)$ (which means a section of the pullback bundle $T^*S_v \rightarrow S_v \times M$), supported in a compact subset of the interior of $S_v$.
\end{definition}

Given such a perturbation datum and critical points $x_0 \in \mathit{Crit}(h_{L_0,L_d})$, $x_k \in \mathit{Crit}(h_{L_{k-1},L_k})$ ($1 \leq k \leq d$), we can define an associated moduli space. Points are (isomorphism classes of) continuous maps $\bar{u}: \bar{S}_T \rightarrow M$, with $\bar{u}(\bar{z}_k) = x_k$, which satisfy the following equations. Let $S_v$ be the Riemann surface associated to a vertex with $|v| \geq 3$. Restriction of $\bar{u}$ to that surface yields a smooth map
\begin{equation} \label{eq:generalized-u2-equation}
\left\{
\begin{aligned}
& u_v: S_v \longrightarrow M, \\
& u_v(\partial_k S_v) \subset L_{i_{v,k}} \text{ for $k = 0,\dots,|v|-1$}, \\
& (du_v - X_{v,z}(u_v)) \circ i = J_{v,z}(u_v) \circ (du_v - X_{v,z}(u_v)).
\end{aligned}
\right.
\end{equation}
Here, we have $K_{v,z}: TS_z \rightarrow C^\infty(M,\bR)$ and consider the associated Hamiltonian vector field, $X_{v,z}: TS_z \rightarrow C^\infty(M,TM)$, then evaluate that at the point $u(z)$. We extend that to $|v| = 2$ by identifying $S_v \iso \bR \times [0,1]$, and equipping that with $J_{v,s,t} = J_{L_{i_{v,0}},L_{i_{v,1}},t}$ as well as $K_v = 0$, which of course results in \eqref{eq:generalized-u2-equation} being an equation of the form \eqref{eq:strip-equation} (in that case we again exclude constant solutions). Next, let $I_e$ be the interval associated to an edge. Restriction of $\bar{u}$ to it yields a map
\begin{equation} \label{eq:generalized-u1-equation}
\left\{
\begin{aligned}
& u_e: I_e \longrightarrow L_{i_{e,0}} \cap L_{i_{e,1}}, \\
& du_e/ds + \nabla h_{e,s}(u_e) = 0.
\end{aligned}
\right.
\end{equation}
To define $\mu^d_{\mathit{Fuk}(M)}$, one has to choose perturbation data for all decorated pearly trees, depending smoothly on the moduli and lengths, and related by other consistency conditions.

\begin{remark} \label{th:constant-regularity}
The addition of an inhomogeneous term $X_v$ to \eqref{eq:generalized-u2-equation} is necessary to achieve transversality in general. Concretely, the problem with setting $K_v = 0$ is that then, constant maps at points of $L_{v,i_{v,0}} \cap \cdots \cap L_{v,i_{v,|v|-1}}$ would always be solutions, irrespective of the choice of $J_v$. The dimension of the moduli space of constant maps is $\mathrm{dim}(L_{i_{v,0}} \cap \cdots \cap L_{i_{v,|v|-1}}) + (|v|-3)$, whereas its expected dimension is $\mathrm{dim}(L_{i_{v,0}} \cap L_{i_{v,|v|-1}}) + \mathrm{deg}(L_{i_{v,0}} \cap L_{i_{v,|v|-1}}) - \mathrm{deg}(L_{i_{v,0}} \cap L_{i_{v,1}}) - \cdots - \mathrm{deg}(L_{i_{v,|v|-2}} \cap L_{i_{v,|v|-1}}) + (|v|-3)$. Here, the dimensions and degrees really refer to the connected components to which our constant map belongs. One can show that the moduli space is regular if and only if those two numbers agree.
\end{remark}

\begin{example} \label{th:single-l}
Consider a single $L$, and assume that there are no nonconstant $J_L$-holomorphic discs with boundary on $L$, and no non-constant $J_L$-holomorphic spheres intersecting $L$. To define the $A_\infty$-structure on $\mathit{CF}^*(L,L)$, one can take all the almost complex structures to be $J_L$, and all inhomogeneous terms to be zero. The only solutions of \eqref{eq:generalized-u2-equation} are constant maps at points of $L$, which are regular (Remark \ref{th:constant-regularity}). The only contribution to $0$-dimensional moduli spaces comes from trees $T$ with only trivalent vertices. Transversality can be achieved by varying the functions, which recovers a version of the picture in \cite{fukaya-oh98}. In particular, if $\xi$ is the local coefficient system on $L$, we have an isomorphism of rings
\begin{equation} \label{eq:hf-non-deformed}
\mathit{HF}^*(L,L) \iso H^*(L;\mathit{Hom}(\xi,\xi)).
\end{equation}
One can realize this more canonically by an open string analogue of the Piunikhin-Salamon-Schwarz map \cite{albers08}.
\end{example}


\subsection{The open-closed string map\label{subsec:define-open-closed}}
We now want to give a similar description of the open-closed string map
\begin{equation} \label{eq:open-closed}
 \mathit{QH}^*(M) \longrightarrow \mathit{HH}^*(\mathit{Fuk}(M),\mathit{Fuk}(M)),
\end{equation}
or at least the part that lands in the subcategory consisting of Lagrangian submanifolds in our fixed collection. Even though it would be possible (and maybe more in tune with our general developments) to represent $\mathit{QH}^*(M) = H^*(M;R)$ Morse-theoretically, we prefer the simpler picture that comes from thinking of cohomology classes as cycles. More specifically, fix a co-oriented submanifold $G \subset M$, and consider the Poincar\'e dual class $[G] \in H^*(M;\bZ)$.

\begin{definition}
A {\em pointed pearly tree} with $d \geq 0$ leaves is a decorated graph of the following kind. The underlying graph $T$ is a ribbon tree with $(d+1)$ ends, again with a distinguished root. Each edge $e$ is decorated with a closed interval $I_e \subset \bR$, and each vertex $|v|$ with a Riemann surface $S_v = D \setminus \{\bar{z}_{v,0},\dots,\bar{z}_{v,|v|-1}\}$, as before. The new ingredient is that for exactly one vertex $v_*$, the surface $S_{v_*}$ carries an additional interior marked point $z_*$. Moreover, this particular vertex can be univalent, whereas for all others the condition $|v| \geq 2$ still applies.
\end{definition}

We define labelings, and other book-keeping devices, as before.

\begin{definition}
A {\em perturbation datum} on a labeled pointed pearly tree consists of the following data. For each edge $e$, we want to have a family of functions $h_e = (h_{e,s})$ on $L_{i_{e,0}} \cap L_{i_{e,1}}$ as usual. Next, take a vertex, which either satisfies $|v| \geq 3$ or is equal to $v_*$. We then want to choose strip-like ends, a family $J_v$ of almost complex structures, and a one-form $K_v$ as before.
\end{definition}

Given this, we can define an associated moduli space which combines gradient flow lines and perturbed pseudo-holomorphic maps, where the component $u_{v_*}$ additionally satisfies $u_{v_*}(z_*) \in G$. Counting solutions of this moduli problem yields a Hochschild cocycle $g$ which represents the image of $[G]$ under the open-closed string map.

\begin{example} \label{th:take-a-single}
Take a single Lagrangian submanifold $L$ as in Example \ref{th:single-l}. Assume that $G$ is transverse to $L$, and consider only the part of $g$ involving only $L$, which is an element of the Hochschild complex of the $A_\infty$-algebra $\mathit{CF}^*(L,L)$. Again, one can take all almost complex structures equal to $J_L$, and all inhomogeneous terms to be zero, so $g$ can be expressed in purely Morse-theoretic terms. In particular, the linear part $g^0_L \in \mathit{CF}^*(L,L) = \mathit{CM}^*(h_{L,L})$ is just given by counting gradient half-lines which start in $G \cap L$. This is the Morse-theoretic representative for $[G]|L \in H^*(L;R)$. As another consequence of the same observation, if $L \cap G = \emptyset$, all the components $g^d_{L,\dots,L}$ vanish.
\end{example}

\begin{example} \label{th:totally-disjoint}
We can generalize the last-mentioned observation as follows. Suppose that $G$ has (real) codimension $1$. Suppose also that every Lagrangian submanifold $L$ in our collection satisfies the condition from Example \ref{th:single-l}, and additionally is disjoint from $G$. Then, for a suitable choice, $g$ is identically zero. Namely, given a pointed pearly tree with $d>0$ ends, one can forget $z_*$ and then collapse components if necessary, so as to obtain an ordinary pearly tree. This allows one to lift the perturbation data used to define the $A_\infty$-structures to pointed pearly trees, giving a picture whereby $g$ is obtained by cutting down moduli spaces by asking that the holomorphic discs should go through $G$. However, because of the codimension and intersection assumptions, this can never reduce the dimension to zero unless the moduli space is empty.
\end{example}

%

\subsection{Abelian coverings\label{subsec:pushdown}}
The following material is not new (compare \cite{seidel03b}) or difficult, but we will need the statements in a specific form for later reference. The geometric situation is that we have a finite covering of symplectic manifolds
\begin{equation}
z: \tilde{M} \longrightarrow M,
\end{equation}
with abelian covering group $\Gamma$. Both manifolds are supposed to come with trivializations of their anticanonical bundles, related in the obvious way. We consider Lagrangian submanifolds $\tilde{L} \subset \tilde{M}$, equipped with additional structures which turn them into objects of the Fukaya category, and also subject to the following conditions:
\begin{equation} \label{eq:covered-lagrangians}
\parbox{35em}{
$z|\tilde{L}$ is itself a covering (for some subgroup of $\Gamma$) of a Lagrangian submanifold $L = z(\tilde{L}) \subset M$. The grading of $\tilde{L}$ is then automatically lifted from a grading of $L$. We impose the additional requirements that the {\em Spin} structure on $\tilde{L}$ should be the lift of one on $L$, and the same for the almost complex structure $J_{\tilde{L}}$.
}
\end{equation}
As part of the data, $\tilde{L}$ carries a local system $\tilde{\xi}$, but we do not impose any additional conditions on that.
Let $\tilde{F} \subset \mathit{Fuk}(\tilde{M})$ be the full $A_\infty$-subcategory whose objects are Lagrangian submanifolds satisfying \eqref{eq:covered-lagrangians}. When defining the $A_\infty$-structure, one can similarly lift all choices of Morse functions, almost complex structures, and inhomogeneous terms from $M$. The result is that $\tilde{F}$ comes with a strict action of $\Gamma$ by covering transformations, as well as with a pushforward functor $Z: \tilde{F} \rightarrow \mathit{Fuk}(M)$, which takes $\tilde{L}$ to $Z(\tilde{L}) = L$ with the local coefficient system $\xi = z_*(\tilde{\xi})$. The behaviour of this functor can be fully described in terms of the $\Gamma$-action. We have
\begin{equation} \label{eq:pushdown-cf}
\mathit{CF}^*(Z(\tilde{L}_0),Z(\tilde{L}_1)) = \bigoplus_{\gamma \in \Gamma} \mathit{CF}^*(\tilde{L}_0,\gamma(\tilde{L}_1)),
\end{equation}
and the $A_\infty$-structure maps $\mu^d_{\mathit{Fuk}(M)}$ are direct sums of
\begin{equation}
\begin{aligned}
& \mathit{CF}^*(\tilde{L}_{d-1},\gamma_d(\tilde{L}_d)) \otimes \cdots \otimes \mathit{CF}^*(\tilde{L}_1,\gamma_2(\tilde{L}_2)) \otimes \mathit{CF}^*(\tilde{L}_0,\gamma_1(\tilde{L}_1)) \\ & \iso
\mathit{CF}^*(\gamma_1 \dots \gamma_{d-1}(\tilde{L}_{d-1}),\gamma_1 \dots \gamma_d(\tilde{L}_d)) \otimes
\cdots \otimes \mathit{CF}^*(\gamma_1(\tilde{L}_1), \gamma_1 \gamma_2(\tilde{L}_2))
\otimes \mathit{CF}^*(\tilde{L}_0, \gamma_1(\tilde{L}_1)) \\ &
\xrightarrow{\mu^d_{\tilde{F}}} \mathit{CF}^*(\tilde{L}_0, \gamma_1\cdots \gamma_d(\tilde{L}_d)).
\end{aligned}
\end{equation}

\begin{example}
Suppose that $\tilde{L} \rightarrow L$ is a covering with group $G \subset \Gamma$. There is an obvious isomorphism of local systems $\mathit{Hom}(z_*\tilde{\xi},z_*\tilde{\xi}) \iso z_*\tilde{\xi} \otimes_\bZ \bZ[G]$, hence $\mathit{HF}^*(Z(\tilde{L}),Z(\tilde{L})) \iso H^*(L;\mathit{Hom}(z_*\tilde{\xi},z_*\tilde{\xi})) \iso H^*(\tilde{L}; \tilde{\xi}) \otimes_\bZ \bZ[G]$. This describes the decomposition induced by \eqref{eq:pushdown-cf} (the summands for elements of $\Gamma \setminus G$ are zero).
\end{example}

This has implications for the open-closed string map as well. Let $\tilde{F}$ be the diagonal bimodule of $\mathit{Fuk}(\tilde{M})$. We can twist it by applying $\gamma \in \Gamma$ to the left (but not the right) actions, thereby obtaining another bimodule $\tilde{F}^\gamma$, which can be thought of as the graph of $\gamma^{-1}$. On the other hand, take the diagonal bimodule of $\mathit{Fuk}(M)$ and pull it back by $Z$ (on both sides) to get a bimodule over $\mathit{Fuk}(\tilde{M})$. As a consequence of the observations above, we have an isomorphism
\begin{equation}
Z^*\mathit{Fuk}(M) \iso \bigoplus_{\gamma \in \Gamma} \tilde{F}^\gamma,
\end{equation}
and the canonical bimodule map $\tilde{F} \rightarrow Z^*\mathit{Fuk}(M)$ is just the inclusion of the $\gamma = e$ summand. The open-closed string maps for $M$ and $\tilde{M}$ (the latter restricted to the category $\tilde{F}$) and the maps \eqref{eq:functoriality} associated to the functor $Z$ fit into a commutative diagram
\begin{equation} \label{eq:open-closed-pushforward}
\xymatrix{
\mathit{QH}^*(\tilde{M}) \ar[rr] && \mathit{HH}^*(\mathit{Fuk}(\tilde{M}),\mathit{Fuk}(\tilde{M})) \ar[d]^-{Z_*} \\
&& \mathit{HH}^*(\mathit{Fuk}(\tilde{M}), Z^*\mathit{Fuk}(M)) \\
\mathit{QH}^*(M) \ar[rr] \ar[uu]^-{z^*} && \mathit{HH}^*(\mathit{Fuk}(M),\mathit{Fuk}(M)), \ar[u]_-{Z^*}
}
\end{equation}
where the right column is as in \eqref{eq:functoriality} (with an unfortunate reversal of notation).

\begin{addendum} \label{th:pullback-functor}
There is also a functor in opposite direction, which is better-behaved since it is defined on the whole Fukaya category, $\mathit{Fuk}(M) \rightarrow \mathit{Fuk}(\tilde{M})$. It maps any object $L$ to its entire primage. As for morphisms, the Floer cochain complex $\mathit{CF}^*(z^{-1}(L_0),z^{-1}(L_1))$ comes with a natural action of $\Gamma$, and the invariant part is the image of $\mathit{CF}^*(L_0,L_1)$ under pullback.
\end{addendum}

\subsection{Split-generators\label{subsec:split-generate}}
Let $O \subset A^{\mathit{perf}}$ be a full subcategory. One says that the objects of $O$ {\em split-generate} $A^{\mathit{perf}}$ if the following holds: any object of $A^{\mathit{perf}}$ up to quasi-isomorphism can be constructed by starting with objects of $O$ and applying the following operations: shifts; mapping cones; and taking direct summands with respect to idempotent endomorphisms. We will quote two abstract split-generation criteria for Fukaya categories from the literature, the second stronger than the first. Both involve Hochschild cohomology and the open-closed string map. More precisely, given $O \subset \mathit{Fuk}(M)$, we consider
\begin{equation} \label{eq:open-closed-restrict}
\xymatrix{
\mathit{QH}^*(M) \ar[r] & \mathit{HH}^*(\mathit{Fuk}(M),\mathit{Fuk}(M)) \ar[r] & \mathit{HH}^*(O,O).
}
\end{equation}

\begin{theorem}[\protect{\cite[Theorem 7.2]{abouzaid-smith09}}] \label{th:split-generation-1}
Suppose that $O$ is smooth \cite[Definition 8.1.12]{kontsevich-soibelman00}, and that \eqref{eq:open-closed-restrict} is an isomorphism. Then the objects in $O$ split-generate $\mathit{Fuk}(M)^{\mathit{perf}}$.
\end{theorem}

The next result is an analogue of \cite{abouzaid10} for compact manifolds, to appear in \cite{abouzaid-fukaya-oh-ohta-ono11}; see also \cite[Section 13]{sheridan13} for the special case of monotone symplectic manifolds.

\begin{theorem}[Abouzaid-Fukaya-Oh-Ohta-Ono] \label{th:split-generation-2}
Suppose that there is a linear map $\mathit{HH}^{2n}(O,O) \rightarrow R$ whose composition with \eqref{eq:open-closed-restrict} yields the integration map $\mathit{QH}^{2n}(M) \rightarrow R$.
Then the objects in $O$ split-generate $\mathit{Fuk}(M)^{\mathit{perf}}$.
\end{theorem}

Because $\mathit{QH}^{2n}(M)$ is one-dimensional, the condition in Theorem \ref{th:split-generation-2} is just that \eqref{eq:open-closed-restrict} should be nonzero in degree $2n$, but we prefer the formulation above, which extends to non-Calabi-Yau cases and more accurately reflects the core of the argument. Note that for a general $M$, there is no reason to suppose that the conditions of either theorem above would hold even for $O = \mathit{Fuk}(M)$ (no actual counterexamples are known, but there are suggestions coming from mirror symmetry for non-algebraic varieties). However, if they do hold for one set of split-generating objects, then the same is true for any other such set.
%

\begin{example}
Take for instance the two-torus $T$. As discussed in \cite{abouzaid-smith09}, two curves intersecting in a point satisfy the criterion of Theorem \ref{th:split-generation-1}, hence split-generate the Fukaya category. As already pointed out in Remark \ref{th:whole-fukaya-torus}, it then follows that the same holds for the two curves from Figure \ref{fig:abcd1}.
\end{example}

\begin{example} \label{th:quartic}
Let $K \subset \bC P^3$ be a smooth quartic surface, equipped with the restriction of the Fubini-Study form. Classical Picard-Lefschetz theory shows that the orthogonal complement $[\omega_K]^\perp \subset H_2(K;\bQ)$ is spanned by Lagrangian spheres. For completeness, we describe the argument briefly: let $(K_z)_{z \in \bC \cup \{\infty\}}$ be a Lefschetz pencil of quartic surfaces, with $K_\infty = K$. We have
\begin{equation}
H_*(\bC P^3 \setminus K;\bQ) \iso H^{6-*}(\bC P^3,K;\bQ) \iso 
\begin{cases} 
\bQ & \ast = 0, \\
[\omega_K]^\perp & \ast = 3, \\
0 & \text{in other degrees.}
\end{cases}
\end{equation}
The base of the pencil, $C = K_0 \cap K_\infty$, is a smooth Riemann surface of genus $33$ representing a multiple of $[\omega_K]$, and we have
\begin{equation}
H_*(K \setminus C) \iso H^{4-*}(K,C) \iso \begin{cases}
\bQ & \ast = 0, \\
[\omega_K]^\perp \oplus H_1(C;\bQ) & \ast = 2, \\
0 & \text{in other degrees.}
\end{cases}
\end{equation}
In particular, the image of $H_2(K \setminus C;\bQ) \rightarrow H_2(K;\bQ)$ is $[\omega_K]^\perp$. Up to homotopy equivalence, $\bC P^3 \setminus K$ is obtained from $K \setminus C$ by attaching $3$-handles along a collection of Lagrangian spheres in $K \setminus C$ (a basis of vanishing cycles of the pencil: there are $108$ of them, as an Euler characteristic computation shows). The fact that $H_2(\bC P^3 \setminus K;\bQ) = 0$ shows that the homology classes of these spheres span all of $H_2(K \setminus B;\bQ)$. Hence, in $K$ the same spheres span $[\omega_K]^\perp$. Choose Lagrangian spheres $(L_1,\dots,L_{21})$ whose homology classes form a basis for $[\omega_K]^\perp$.

Homological mirror symmetry \cite{seidel03b} says that $\mathit{Fuk}(K)^{\mathit{perf}}$ is quasi-equivalent to $D^b\mathit{Coh}(X)$, where the mirror $X$ is a smooth $K3$ surface over $R$. We have $\mathit{HH}^4(X,X) \iso H^2(X,K_X^{-1}) \iso R$, and the product $\mathit{HH}^2(X,X)^{\otimes 2} \rightarrow \mathit{HH}^4(X,X)$ is a nondegenerate quadratic form (nondegeneracy is a consequence of Serre duality, thinking of $\mathit{HH}^*(X,X) \iso \mathit{Ext}^*_{X \times X}({\mathcal O}_\Delta,{\mathcal O}_\Delta)$ as the endomorphism ring of the diagonal). The Hochschild-Kostant-Rosenberg theorem implies that
\begin{equation}
\mathrm{dim}\, \mathit{HH}^2(X, X) = \mathrm{dim}\, \big(
H^2(X,{\mathcal O}_X) \oplus H^1(X,TX) \oplus H^0(X,K_X^{-1}) \big) = 22 = \mathrm{dim}\, H_2(K).
\end{equation}
By combining these two facts, one sees that any isotropic subspace of $\mathit{HH}^2(X,X)$ is at most of dimension $11$. The same must then hold for the Hochschild cohomology of $\mathit{Fuk}(K)$. 

Take the collection of spheres introduced above, and consider 
\begin{equation} \label{eq:combined-restriction}
H^2(K;R) = \mathit{QH}^2(K) \longrightarrow \mathit{HH}^2(\mathit{Fuk}(K),\mathit{Fuk}(K)) \longrightarrow \bigoplus_i \mathit{HF}^2(L_i,L_i) \iso \bigoplus_i H^2(L_i;R),
\end{equation}
where the first arrow is the open-closed string map, and the second one the standard map from Hochschild cohomology to the endomorphism ring of any object. Because of the absence of holomorphic discs, \eqref{eq:combined-restriction} just consists of the ordinary restriction maps on cohomology, hence is surjective (compare Example \ref{th:take-a-single}). This shows that the open-closed string map in degree $2$ is of rank $21$ or $22$; therefore, its image is not an isotropic subspace. Since the open-closed string map is a ring homomorphism, it follows that $\mathit{QH}^4(K) \rightarrow \mathit{HH}^4(\mathit{Fuk}(K),\mathit{Fuk}(K))$ must be nonzero, hence an isomorphism. This implies that for any set of split-generating objects, the requirement of Theorem \ref{th:split-generation-2} is satisfied. 

In fact, our previous argument shows that the kernel of the open-closed string map in degree $2$ is either zero or else spanned by $[\omega_K]$. But the second case is impossible since $[\omega_K]^2$ is nontrivial, so the open-closed string map must be an isomorphism. Since $D^b\mathit{Coh}(X)$ is smooth, so is $\mathit{Fuk}(K)$, and the assumption of Theorem \ref{th:split-generation-1} holds as well, allowing one to avoid Theorem \ref{th:split-generation-2}.
\end{example}

\begin{example} \label{th:product-mirror}
Again following \cite{abouzaid-smith09}, we point out that this strategy extends well to products. For instance, take $M = T \times K$ to be the product of the two-torus and the quartic surface. Consider the $A_\infty$-subcategory $O \subset \mathit{Fuk}(M)$ of objects which are themselves products. It turns out that $O$ is quasi-isomorphic to the $A_\infty$-tensor product $\mathit{Fuk}(T) \otimes \mathit{Fuk}(K)$. We will not enter into a full discussion of this fact here, but there are several strategies of proof:
\begin{itemize} \itemsep1em
\item A direct proof, which involves deforming the diagonals for the associahedra to the boundary (matching the definition of the tensor product of $A_\infty$-structures, see \cite{loday-diag} and references therein).
\item Using quilted Floer cohomology \cite{ww}, one can define an $A_\infty$-functor
\begin{equation} \label{eq:yoneda-type}
\mathit{Fuk}(M) \longrightarrow (\mathit{Fuk}(T^-), \mathit{Fuk}(K))^{\mathit{mod}}.
\end{equation}
One compares this to the image of the (algebraically defined) Yoneda-type embedding
\begin{equation} \label{eq:yoneda-type-2}
\mathit{Fuk}(T) \otimes \mathit{Fuk}(K) \iso
\mathit{Fuk}(T^-)^{\mathit{opp}} \otimes \mathit{Fuk}(K) \longrightarrow
(\mathit{Fuk}(T^-), \mathit{Fuk}(K))^{\mathit{mod}}.
\end{equation}
The outcome is that the restriction of \eqref{eq:yoneda-type} to $O$ is a cohomologically full and faithful $A_\infty$-functor, whose image is quasi-equivalent to $\mathit{Fuk}(T) \otimes \mathit{Fuk}(K)$.
\item The version closest to \cite{abouzaid-smith09} would replace \eqref{eq:yoneda-type} with
\begin{equation} \label{eq:yoneda-type-3}
\mathit{Fuk}(M) \longrightarrow \mathit{fun}(\mathit{Fuk}(T^-),\mathit{Fuk}^\#(K))
\end{equation}
where $\mathit{fun}(\cdot,\cdot)$ is the $A_\infty$-category of $A_\infty$-functors, and $\mathit{Fuk}^\#(K)$ the extended Fukaya category \cite{mww}. The rest of the argument would be structured as before. If one wishes, one can avoid $A_\infty$-tensor products, and instead work with suitable full $A_\infty$-subcategories of the categories on the right hand sides of \eqref{eq:yoneda-type} or \eqref{eq:yoneda-type-3} (which are quasi-equivalent to the $A_\infty$-tensor product; the difference is purely one of language).
\end{itemize}

As a consequence of this computation, $O$ is again smooth; moreover, the associated open-closed string map is an isomorphism, which shows that $O$ split-generates $\mathit{Fuk}(M)$. This can be used to prove homological mirror symmetry for the product (in the parallel case of $T \times T$, this is the main result of \cite{abouzaid-smith09}).
\end{example}

\subsection{Products of Dehn twists}
We will now concentrate on constructing specific examples of automorphisms $f$ where the $\bZ/m$-action on $\mathit{HF}^*(f^m)$ is nontrivial. The symplectic manifold $M$ should still satisfy Assumption \ref{th:calabi-yau}. For technical simplicity, we assume that it is a four-manifold ($n = 2$; higher-dimensional generalizations would have to use more advanced methods, as in \cite{oh10}). Take Lagrangian spheres $L_1,\dots,L_r \subset M$, equipped with an arbitrary choice of grading, the unique {\em Spin} structure, and the trivial local coefficient system. Because of our dimensional restriction, generically chosen almost complex structures then satisfy Assumption \ref{th:calabi-yau}. Consider the composition of Dehn twists
\begin{equation}
f = \tau_{L_1} \tau_{L_2} \cdots \tau_{L_r},
\end{equation}
which is naturally a graded symplectic automorphism of $M$.

\begin{proposition}[Perutz]
There is a spectral sequence converging to $\mathit{HF}^*(f)$, with
\begin{equation} \label{eq:e1fixed}
E_1^{p*} = \begin{cases}
\mathit{QH}^*(M) \iso H^*(M;R) & p = 0, \\
\bigoplus_i \mathit{HF}^*(L_i,L_i) \iso \bigoplus_i H^*(L_i;R) & p = 1, \\
\bigoplus_{i_1 > \dots > i_p} \mathit{HF}^*(L_{i_p},L_{i_1}) \otimes \mathit{HF}^*(L_{i_{p-1}},L_{i_p}) \otimes \cdots \\ \qquad \qquad \qquad \qquad \qquad \cdots \otimes \mathit{HF}^*(L_{i_1},L_{i_2})[n(p-1)] & 1 < p \leq r, \\
0 & \text{otherwise.}
\end{cases}
\end{equation}
\end{proposition}

In fact, Perutz's work \cite{perutz10} yields an explicit chain complex which computes $\mathit{HF}^*(f)$ in terms of the Fukaya $A_\infty$-structure and open-closed string map. An appropriate filtration of that chain complex then gives rise to \eqref{eq:e1fixed} (for which an alternative approach is due to Ma'u \cite{mau11}).
One can rewrite the nontrivial columns $E_1^{p*}$, $p > 0$, in a way that highlights the cyclic symmetry:
\begin{equation} \label{eq:cyclic-e1}
E_1^{p*} = \Big( \bigoplus_{i_1, \dots, i_p} \mathit{HF}^*(L_{i_p},L_{i_1}) \otimes \mathit{HF}^*(L_{i_{p-1}},L_{i_p}) \otimes \cdots \otimes \mathit{HF}^*(L_{i_1},L_{i_2}) \Big)^{\bZ/p} [n(p-1)],
\end{equation}
where the direct sum is over cyclically decreasingly ordered $p$-tuples, and $\bZ/p$ acts by cyclically permuting these $p$-tuples (with additional signs). This point of view is particularly convenient for considering conjugation invariance. Namely, the spectral sequences computing $\mathit{HF}^*(f)$ and $\mathit{HF}^*(\tau_{L_r} f \tau_{L_r}^{-1})$ are related by an automorphism, which converges to $C_{\tau_{L_r},f}$. On the $E_1$ level, that automorphism is just the obvious relation between the expressions \eqref{eq:cyclic-e1}.

The class of examples of interest to us is where $r = 2m$ and $L_1 = L_3 = \cdots = L_{2m-1}$, $L_2 = L_4 = \cdots = L_{2m}$, so that $f = (\tau_{L_1}\tau_{L_2})^m$. Additionally, we ask that:

\begin{assumption} \label{th:spread}
$\mathit{HF}^*(L_1,L_2)$ is concentrated in degrees $[n/2-k,n/2+k]$ for some $k \geq n/2$. Moreover, both $\mathit{HF}^{n/2-k}(L_1,L_2)$ and $\mathit{HF}^{n/2+k}(L_1,L_2)$ are nonzero, and at least one of those spaces has dimension $>1$.
\end{assumption}

The argument outlined above shows how the conjugation isomorphisms \begin{equation}
\begin{aligned}
& C_{\tau_{L_2},f}: \mathit{HF}^*(f) \longrightarrow \mathit{HF}^*(\tau_{L_2} f \tau_{L_2}^{-1}), \\
& C_{\tau_{L_1},\tau_{L_2} f \tau_{L_2}^{-1}}: \mathit{HF}^*(\tau_{L_2} f \tau_{L_2}^{-1}) \longrightarrow \mathit{HF}^*(\tau_{L_1} \tau_{L_2} f \tau_{L_2}^{-1} \tau_{L_1}^{-1}) = \mathit{HF}^*(f)
\end{aligned}
\end{equation}
act on the $E_1$ pages of the respective spectral sequences. The composition of these two isomorphisms defines the standard $\bZ/m$-action. In particular, this acts on the last column
\begin{equation}
\begin{aligned}
E_1^{2m,*} & = \mathit{HF}^*(L_1,L_{2m}) \otimes \mathit{HF}^*(L_2,L_1) \otimes \cdots \otimes \mathit{HF}^*(L_{2m},L_{2m-1})[n(2m-1)] \\
& \iso \big(\mathit{HF}^*(L_1,L_2) \otimes \mathit{HF}^*(L_2,L_1)\big)^{\otimes m}[n(2m-1)]
\end{aligned}
\end{equation}
by cyclically permuting the $m$ tensor factors, up to (degree-dependent) signs. The signs could in principle be determined by a more careful argument, but they turn out to be irrelevant for our purpose.

\begin{lemma} \label{th:regular-representation}
$\mathit{HF}^{2m(1-k-n/2)+n}((\tau_{L_1}\tau_{L_2})^m)$ contains a copy of the regular representation of $\bZ/m$, for any $m \geq 1$.
\end{lemma}

\begin{proof}
The $E_1^{0*}$ column contributes only in total degrees $\geq 0$, the $E_1^{1*}$ column in total degrees $\geq 1$, and the $E_1^{p*}$ ($1 < p < 2m$) columns in total degrees $\geq p(1-k-n/2) + n$. This means that the following piece, whose total degree is $2m(1-k-n/2) + n$, survives to $E_\infty$:
\begin{equation} \label{eq:m-tensor}
(\mathit{HF}^{n/2-k}(L_1,L_2) \otimes \mathit{HF}^{n/2-k}(L_2,L_1))^{\otimes m} \subset E_1^{2m,2m(-k-n/2) + n}.
\end{equation}
By assumption, $\mathit{HF}^{n/2-k}(L_1,L_2) \otimes \mathit{HF}^{n/2-k}(L_2,L_1)$ is at least two-dimensional. If $a_1,a_2$ are linearly independent elements in it, then $a_1 \otimes a_2^{\otimes m-1}$ and its images under the $\bZ/m$-action are all linearly independent elements of \eqref{eq:m-tensor}, which proves the claim.
\end{proof}

%
%

\subsection{The quartic surface}
Before continuing on to our concrete example, we have to agree on criteria for a symplectic automorphism to be trivial from a topological viewpoint.

\begin{definition} \label{th:interesting}
Let $f$ be a symplectic automorphism of a closed symplectic manifold $M$. We say that $f$ is {\em undistinguishable from the identity by topological means} if there is an isotopy from the identity to $f$ inside the diffeomorphism group, with the following additional properties. First, by starting with $Df$ and deforming it along the isotopy, we get an automorphism of the symplectic vector bundle $TM$ (taking each fibre to itself), and we ask that this should be homotopic to the identity in the group of such automorphisms. Secondly, integrating $\omega_M$ along the isotopy yields a flux-type class in $H^1(M;\bR)$, and we also require that this should vanish.
\end{definition}

For the rest of this discussion, we concentrate on the case of smooth quartic surface, as in Example \ref{th:quartic}. There are quartic surfaces with a rational (Kleinian or du Val) singularity of type $(A_3)$.
By smoothing out such a singularity and using Moser's theorem, we see that $K$ contains an $(A_3)$ chain of Lagrangian spheres, which we denote by $(V_1,V_2,V_3)$. Consider the spheres
\begin{equation}
\begin{aligned}
& L_1 = \tau_{V_1}\tau_{V_3}(V_2) = \tau_{V_3}\tau_{V_1}(V_2), \\
& L_2 = \tau_{V_1}^{-1}\tau_{V_3}^{-1}(V_2) = \tau_{V_3}^{-1}\tau_{V_1}^{-1}(V_2).
\end{aligned}
\end{equation}
Figure \ref{fig:quiver1} shows a schematic picture of these Lagrangian submanifolds, in the style of \cite{khovanov-seidel98}.
\begin{figure}
\begin{centering}
\begin{picture}(0,0)%
\includegraphics{quiver1.pstex}%
\end{picture}%
\setlength{\unitlength}{3947sp}%
\begingroup\makeatletter\ifx\SetFigFont\undefined%
\gdef\SetFigFont#1#2#3#4#5{%
  \reset@font\fontsize{#1}{#2pt}%
  \fontfamily{#3}\fontseries{#4}\fontshape{#5}%
  \selectfont}%
\fi\endgroup%
\begin{picture}(1896,1611)(1753,-4175)
\put(3001,-3286){\makebox(0,0)[lb]{\smash{{\SetFigFont{10}{12.0}{\rmdefault}{\mddefault}{\updefault}{\color[rgb]{0,0,0}$L_2$}%
}}}}
\put(3151,-2761){\makebox(0,0)[lb]{\smash{{\SetFigFont{10}{12.0}{\rmdefault}{\mddefault}{\updefault}{\color[rgb]{0,0,0}$V_3$}%
}}}}
\put(2551,-2761){\makebox(0,0)[lb]{\smash{{\SetFigFont{10}{12.0}{\rmdefault}{\mddefault}{\updefault}{\color[rgb]{0,0,0}$V_2$}%
}}}}
\put(1951,-2761){\makebox(0,0)[lb]{\smash{{\SetFigFont{10}{12.0}{\rmdefault}{\mddefault}{\updefault}{\color[rgb]{0,0,0}$V_1$}%
}}}}
\put(3001,-4111){\makebox(0,0)[lb]{\smash{{\SetFigFont{10}{12.0}{\rmdefault}{\mddefault}{\updefault}{\color[rgb]{0,0,0}$L_1$}%
}}}}
\end{picture}%
\caption{\label{fig:quiver1}}
\end{centering}
\end{figure}

\begin{lemma} \label{th:2-twists}
$\tau_{L_1}\tau_{L_2}$ is indistinguishable from the identity by topological means.
\end{lemma}

\begin{proof}
For any Lagrangian two-sphere, the Dehn twist and its inverse are isotopic as diffeomorphisms, hence so are $\tau_{L_1} = \tau_{V_1}\tau_{V_3}\tau_{V_2}\tau_{V_3}^{-1}\tau_{V_1}^{-1}$ and $\tau_{L_2}^{-1} = \tau_{V_1}^{-1}\tau_{V_3}^{-1}\tau_{V_2}^{-1}\tau_{V_3}\tau_{V_1}$. From an analysis of the simultaneous resolution of the $(A_3)$ singularity as in \cite{khovanov-seidel98, seidel04b}, one obtains the following stronger {\em fragility} statement. Take a closed two-form $\beta$ such that $\int_{V_k} \beta \neq 0$ for $k = 1,2,3,$. Then, there is a family of diffeomorphisms $f_r$, defined for small $r \geq 0$ and starting with $f_0 = f$, such that:
\begin{itemize}
\item $f_r$ preserves $\omega_K + r\beta$;
\item for any $r>0$, $f_r$ is isotopic to the identity in the symplectic automorphism group of $(K,\omega_K + r\beta)$. Moreover, the isotopies can be chosen to depend smoothly on $r$.
\end{itemize}
Using such an isotopy from $f_r$ to the identity, one immediately shows that the second part of Definition \ref{th:interesting} holds, and the third part is trivial since $H^1(K;\bR) = 0$.
\end{proof}

A straightforward computation using the relation between Dehn twists and algebraic twists \cite[Corollary 17.17]{seidel04} shows that (for suitable choices of gradings)
\begin{equation}
\mathit{HF}^*(L_1,L_2) = \begin{cases} R^2 & * = 0, \\ R & * = 1,2, \\ 0 & \text{otherwise.} \end{cases}
\end{equation}
This yields a concrete example where Lemma \ref{th:regular-representation} applies. In particular, one sees that $\tau_{L_1}\tau_{L_2}$ has infinite order up to symplectic isotopy (a known result, compare \cite{seidel-thomas99}).

\section{Symplectic mapping tori\label{sec:mapping-tori}}

The symplectic mapping torus construction provides a way of obtaining interesting examples of symplectic manifolds from automorphisms. A complete description of the Fukaya categories of symplectic mapping tori is beyond the aim of this paper (but see Section \ref{subsec:algebraic-model} for some conjectural discussion). Instead, we focus on a particular class of mapping tori, and consider only the most obvious Lagrangian submanifolds, which are fibered over circles in the (two-torus) base. For those submanifolds, ad hoc methods parallel to those in Section \ref{subsec:elliptic} are sufficient to carry out the necessary Floer cohomology computations. Under suitable additional assumptions, this will allow us to show that the Lagrangian isotopy obtained by moving the circle around the base can be encoded into a perfect family.

Concretely, the starting point for our considerations will always be a symplectic $K3$ surface $K$, by which we mean a closed symplectic four-manifold diffeomorphic to a $K3$ surface. Recall that this is simply connected, admits a perfect Morse function (one without critical points of index $1$ or $3$) \cite{harer-kas-kirby}, and is {\em Spin}. The symplectic structure necessarily has $c_1(K) = 0$ \cite{taubes94}, and we choose a trivialization of the anticanonical line bundle in the unique homotopy class. We also suppose that a symplectic automorphism $f \in \mathit{Aut}(K)$ is given which has nondegenerate fixed points as well as nondegenerate $2$-periodic points, and which is indistinguishable from the identity by topological means (Definition \ref{th:interesting}). As part of the last-mentioned condition, $f$ is isotopic to the identity in $\mathit{Diff}(K)$, and we fix such an isotopy, as well as a grading of $f$. 


\subsection{Basic geometry\label{subsec:basic-mapping-torus}}
Consider $K^- \times K$, where the sign of the symplectic form is reversed on the first factor, as in Example \ref{th:graph}. The symplectic mapping torus of $f \times f \in \mathit{Aut}(K^- \times K)$, which we denote by $E = E_f$, is
\begin{equation} \label{eq:symplectic-mapping-torus}
\begin{aligned}
& E = \bR \times \bR \times K \times K\;\; / \;\; (p,q,x,y) \sim (p,q-1,x,y) \sim (p-1,q,f(x),f(y)), \\
& \omega_E = dp \wedge dq - \omega_K(x) + \omega_K(y).
\end{aligned}
\end{equation}
By definition, projection $\pi: E \rightarrow T = \bR^2/\bZ^2$ is a locally trivial Hamiltonian fibration, with monodromy $f \times f$ in $p$-direction, and trivial monodromy in $q$-direction. Our assumptions on $f$ ensure that $E$ is diffeomorphic to $T \times K \times K$, in a way which is compatible with the homotopy classes of almost complex structures, and which maps $[\omega_E]$ to $[dp \wedge dq] \times 1 + 1 \times [\omega_{K^- \times K}] \in H^2(T \times K \times K;\bR)$. Moreover, the grading of $f$ yields a trivialization of the anticanonical line bundle of $E$.

\begin{remark}
As a symplectic fibration over a surface, $E$ is an object of the TCFT with target space $K^- \times K$ discussed in Remark \ref{th:tcft}. Hence, there is an associated numerical invariant (a priori an element of $R$, but which will actually turn out to be an integer) counting its pseudo-holomorphic sections. This can be computed in two different ways. On one hand, in terms of \eqref{eq:ordinary-mapping-torus}, $E$ is obtained by gluing together the two boundary components of $Z_{f \times f} \times [0,1]$ in the trivial way, which means that the numerical invariant is the Euler characteristic of $\mathit{HF}^*(f \times f)$. On the other hand, one can think it as $[0,1] \times S^1 \times K^- \times K$ with both ends glued together using a twist by $f \times f$, in which case the numerical invariant is the supertrace of the action of $f \times f$ on $H^*(K^- \times K;R)$. Both ways of course yield the same result, namely the square of the Lefschetz fixed point number of $f$.
\end{remark}

Consider the following Lagrangian submanifolds in $E$:
\begin{equation} \label{eq:lagrangians}
\begin{aligned}
& \Delta_1 = \{q = 0, \;\; y = x\}, \\
& \Delta_2 = \{q = -2p, \;\; y = x\}, \\
& \Delta_{3,u} = \{p = m_0, \;\; y = x \}. \\
\end{aligned}
\end{equation}
In the last line, the parameter $u \in R^\times$ is written as $u = \hbar^{m_0}a$ with $a \in \mathit{GL}_0(1,R)$. All Lagrangian submanifolds in \eqref{eq:lagrangians} fibre over loops in $T$, with fibre $K$. These fibrations are actually trivial (tautologically so for $\Delta_{3,u}$, and by using the isotopy $f \htp \mathit{id}$ in the other cases). In particular, we can choose the product of the trivial {\em Spin} structures on the underlying loop and the unique {\em Spin} structure on $K$. Moreover, our Lagrangian submanifolds also admit gradings (we make a particular choice of gradings, which will become clear in the Floer cohomology formulas below). On $\Delta_{3,u}$ we use the local system $\xi_u$, pulled back from the underlying loop $\{m_0\} \times S^1 \subset T$, which has fibre $R$ and holonomy $a$ in positive $q$-direction. The other two submanifolds carry trivial local systems. The last ingredient needed in order to turn them into objects of the Fukaya category is a choice of almost complex structures as in Assumption \ref{th:adiscic}. Choose a one-parameter family $(J_{f,t})$ of almost complex structures on $K$ as in \eqref{eq:j-periodicity}, with the additional (generic, for dimension reasons) property that there are no non-constant $J_{f,t}$-holomorphic spheres for any $t$. We will use the same almost complex structure on $E$ for all three Lagrangian submanifolds \eqref{eq:lagrangians}:
\begin{equation} \label{eq:j-delta}
(J_\Delta)_{p,q,x,y} = i \times (-J_{f,p+1/2,x}) \times J_{f,p+1/2,y}.
\end{equation}
Projection to $T$ is $(J_\Delta,i)$-holomorphic, hence there are no non-constant $J_\Delta$-holomorphic spheres. Similarly,

\begin{lemma} \label{th:project}
There are no non-constant $J_\Delta$-holomorphic discs with boundary on any one of the submanifolds \eqref{eq:lagrangians}.
\end{lemma}

\begin{proof}
By projecting to $T$, one sees that any disc must be contained in a fibre. There, it is a map $(D,\partial D) \rightarrow K \times K$ which is holomorphic for $(-J_{f,p+1/2}) \times J_{f,p+1/2}$, and has boundary on the diagonal. By the doubling trick already mentioned in Example \ref{th:graph}, such discs correspond to $J_{f,p+1/2}$-holomorphic spheres in $K$.
\end{proof}

As a consequence, the Floer cohomology of each of our submanifolds with itself is canonically isomorphic to its ordinary cohomology. The other Floer cohomology groups are:
\begin{align}
& \mathit{HF}^*(\Delta_1,\Delta_2) \iso H^*(K;R) \oplus H^*(K;R), \label{eq:two-points} \\
& \mathit{HF}^*(\Delta_2,\Delta_{3,u}) \iso (\xi_u)_{(m_0,-2m_0)} \otimes H^*(K;R), \\
& \mathit{HF}^*(\Delta_1,\Delta_{3,u}) \iso (\xi_u)_{(m_0,0)} \otimes H^*(K;R). \label{eq:one-point-hf}
\end{align}
Dually, one can write
\begin{align}
& \mathit{HF}^*(\Delta_2,\Delta_1) \iso H^*(K;R)[-1] \oplus H^*(K;R)[-1], \label{eq:dual-two-point} \\
& \mathit{HF}^*(\Delta_{3,u},\Delta_2) \iso (\xi_u)_{(m_0,-2m_0)}^\vee \otimes H^*(K;R)[-1], \\
& \mathit{HF}^*(\Delta_{3,u},\Delta_1) \iso (\xi_u)_{(m_0,0)}^\vee \otimes H^*(K;R)[-1]. \label{eq:dual-point}
\end{align}
The proofs of these isomorphisms are straightforward in the Morse-Bott formalism from Section \ref{subsec:clean}. For instance, consider \eqref{eq:two-points}. One takes 
\begin{equation} \label{eq:constant-family}
J_{\Delta_1,\Delta_2,t} = J_\Delta
\end{equation}
to be the constant family, in which case the same argument as in Lemma \ref{th:project} shows that there are no non-constant holomorphic strips, immediately reducing the situation to ordinary Morse theory on $\Delta_1 \cap \Delta_2$ (to be precise, one has to check that the local coefficient system $o_C$ on each component $C \subset \Delta_1 \cap \Delta_2$ is trivial; but that is clear since locally near $C$, the geometry splits as a product of base and fibre). The remaining isomorphisms are proved in exactly the same way.

\begin{remark} \label{th:independent-2}
Generalizing \eqref{eq:j-delta}, consider almost complex structures on $E$ of the form
\begin{equation} \label{eq:j-delta-theta}
(J_{\Delta,\theta})_{p,q,x,y} = i \times (-J_{f,\theta(p,q),x}) \times J_{f,\theta(p,q),y},
\end{equation}
where $\theta: \bR \times S^1 \rightarrow \bR$ satisfies $\theta(p+1,q) = \theta(p,q)+1$. If we take any one of the Lagrangian submanifolds introduced above and equip it with some $J_\theta$, it becomes an object of $\mathit{Fuk}(E)$. Different choices of $\theta$ lead to canonically quasi-isomorphic objects: one sees this by constructing Piunikhin-Salamon-Schwarz (PSS) elements in Floer cohomology \cite[Lemma 8.11]{seidel03b}. The key ingredient is the fact that the space of functions $\theta$ parametrizing the almost complex structure \eqref{eq:j-delta-theta} is connected.

For concreteness, consider the pair $(\Delta_1,\Delta_2)$ and, instead of \eqref{eq:constant-family}, equip it with the constant family of almost complex structures 
\begin{equation} \label{eq:constant-j-theta}
J_{\Delta_1,\Delta_2,\theta, t} = J_{\Delta,\theta}.
\end{equation}
This is leads to a Floer cohomology group which we temporarily denote by $\mathit{HF}^*(\Delta_1,\Delta_2)_\theta$. There are no nontrivial holomorphic strips for \eqref{eq:constant-j-theta}, hence the Morse-Bott approach yields an isomorphism parallel to \eqref{eq:two-points}:
\begin{equation}
\mathit{HF}^*(\Delta_1,\Delta_2)_\theta \iso H^*(K;R) \oplus H^*(K;R). \label{eq:two-points-theta}
\end{equation}
On the other hand, the previous PSS argument yields a canonical isomorphism $\mathit{HF}^*(\Delta_1,\Delta_2)_\theta \iso \mathit{HF}^*(\Delta_1,\Delta_2)$. Moreover, a parametrized moduli space argument shows that this isomorphism, \eqref{eq:two-points}, and \eqref{eq:two-points-theta} form a commutative diagram. In a little less precise language, one can summarize this by saying that the isomorphism \eqref{eq:two-points-theta} is independent of $\theta$. The same applies to the other Floer cohomology groups computed above.
\end{remark}

Some of the products on Floer cohomology are also elementary, meaning that they involve no actual count of nontrivial holomorphic maps. For instance, consider
\begin{equation} \label{eq:trivial-products}
\begin{aligned}
& \mathit{HF}^*(\Delta_2,\Delta_2) \otimes \mathit{HF}^*(\Delta_1,\Delta_2) \longrightarrow \mathit{HF}^*(\Delta_1,\Delta_2), \\
& \mathit{HF}^*(\Delta_1,\Delta_1) \otimes \mathit{HF}^*(\Delta_1,\Delta_2) \longrightarrow \mathit{HF}^*(\Delta_1,\Delta_2).
\end{aligned}
\end{equation}
Each of these turns out to be the action of $H^*(\Delta_k;R)$ on $H^*(\Delta_1 \cap \Delta_2;R)$ by restriction and cup-product. The proof again uses a constant family of almost complex structures equal to $J_\Delta$, a suitable choice of Morse functions on $\Delta_k$, and the arguments from Lemma \ref{th:project}; we omit the details. Note that by the cyclic symmetry of the product, this also determines
\begin{align}
& \mathit{HF}^*(\Delta_1,\Delta_1) \otimes \mathit{HF}^*(\Delta_2,\Delta_1) \longrightarrow \mathit{HF}^*(\Delta_2,\Delta_1), \\
& \mathit{HF}^*(\Delta_2,\Delta_1) \otimes \mathit{HF}^*(\Delta_2,\Delta_2) \longrightarrow \mathit{HF}^*(\Delta_2,\Delta_1), \\
& \mathit{HF}^*(\Delta_2,\Delta_1) \otimes \mathit{HF}^*(\Delta_1,\Delta_2) \longrightarrow \mathit{HF}^*(\Delta_1,\Delta_1), \\
& \mathit{HF}^*(\Delta_1,\Delta_2) \otimes \mathit{HF}^*(\Delta_2,\Delta_1) \longrightarrow \mathit{HF}^*(\Delta_2,\Delta_2).
\end{align}

We conclude this preliminary discussion by introducing low-degree generators analogous to those in Section \ref{subsec:twotorus}, namely
\begin{equation}
\begin{aligned} 
& \mathit{HF}^0(\Delta_1,\Delta_2) = R \cdot [w_1] \oplus R \cdot [w_2], \\
& \mathit{HF}^1(\Delta_2,\Delta_1) = R \cdot [w_3] \oplus R \cdot [w_4], \\
& \mathit{HF}^0(\Delta_1,\Delta_{3,u}) = R \cdot [z_{1,u}], \\
& \mathit{HF}^1(\Delta_{3,u},\Delta_1) = R \cdot [y_{1,u}], \\
& \mathit{HF}^0(\Delta_2,\Delta_{3,u}) = R \cdot [z_{2,u}], \\
& \mathit{HF}^1(\Delta_{3,u},\Delta_2) = R \cdot [y_{2,u}]
\end{aligned}
\end{equation}
(the notation $[w_1]$ indicates the cohomology class for some underlying choice of cocycle $w_1$, and we've inserted the dots to avoid confusion with polynomial rings). As before, $[w_1]$ and $[-w_2]$ correspond to the class $1 \in H^0(K;R)$ under \eqref{eq:two-points}, for the components lying over $(\half,0)$ and $(0,0)$, respectively; $[z_{1,u}]$ and $[z_{2,u}]$ are defined using \eqref{eq:xi-trivializations}; and the other generators are fixed in such a way that the products
\begin{equation} \label{eq:delta-dual-generators}
\begin{aligned}
& [w_4] \cdot [w_1] \in \mathit{HF}^1(\Delta_1,\Delta_1), \\
& [w_1] \cdot [w_4] \in \mathit{HF}^1(\Delta_2,\Delta_2), \\
& -[w_2] \cdot [w_3] \in \mathit{HF}^1(\Delta_2,\Delta_2), \\
& -[w_3] \cdot [w_2] \in \mathit{HF}^1(\Delta_1,\Delta_1), \\
\end{aligned}
\qquad
\begin{aligned}
& [y_{1,u}] \cdot [z_{1,u}] \in \mathit{HF}^1(\Delta_1,\Delta_1), \\
& [z_{1,u}] \cdot [y_{1,u}] \in \mathit{HF}^1(\Delta_{3,u},\Delta_{3,u}), \\
& [y_{2,u}] \cdot [z_{2,u}] \in \mathit{HF}^1(\Delta_2,\Delta_2), \\
& [z_{2,u}] \cdot [y_{2,u}] \in \mathit{HF}^1(\Delta_{3,u},\Delta_{3,u})
\end{aligned}
\end{equation}
all yield the generator of $H^1(S^1 \times K;R) \iso H^1(S^1;R)$ obtained by orienting the underlying loops in $T^2$ as in Section \ref{subsec:twotorus}. Moreover, from our computation of \eqref{eq:trivial-products} it follows that $[w_1] \cdot [w_3]$, $[w_3] \cdot [w_1]$, $[w_2] \cdot [w_4]$ and $[w_4] \cdot [w_2]$ vanish. Hence,

\begin{lemma} \label{th:q-cohomologically}
The subspace of $\bigoplus_{i,j=1}^2 \mathit{HF}^*(\Delta_i,\Delta_j)$ consisting of elements of degree $\leq 1$ is a subalgebra, and in fact isomorphic to the algebra $Q$ from Definition \ref{th:q-algebra}. \qed
\end{lemma}

\subsection{Counting triangles\label{subsec:flat-triangles}}
In parallel with our original discussion of the two-torus, we will also need to determine parts of the $A_\infty$-structure which do involve counting holomorphic curves. First of all, we need the counterpart of \eqref{eq:theta-product-1}, which computes the product
\begin{equation} \label{eq:delta-product}
\begin{aligned}
& \mathit{HF}^0(\Delta_2, \Delta_{3,u}) \otimes \mathit{HF}^0(\Delta_1,\Delta_2) \longrightarrow \mathit{HF}^0(\Delta_1, \Delta_{3,u}), \\
& [z_{2,u}] \cdot [w_1] = \vartheta_{2,1}(u) [z_{1,u}], \\
& [z_{2,u}] \cdot [w_2] = -\vartheta_{2,2}(u)[z_{1,u}].
\end{aligned}
\end{equation}
One can use associativity and \eqref{eq:delta-dual-generators} to derive two more products from this, namely
\begin{equation} \label{eq:delta-product-2}
\begin{aligned}
& \mathit{HF}^0(\Delta_1,\Delta_2) \otimes \mathit{HF}^1(\Delta_{3,u},\Delta_1) \longrightarrow \mathit{HF}^1(\Delta_{3,u},\Delta_2), \\
& [w_1] \cdot [y_{1,u}] = \vartheta_{2,1}(u) [y_{2,u}], \\
& [w_2] \cdot [y_{1,u}] = -\vartheta_{2,2}(u) [y_{2,u}],
\end{aligned}
\end{equation}
and
\begin{equation} \label{eq:delta-product-3}
\begin{aligned}
& \mathit{HF}^1(\Delta_{3,u},\Delta_1) \otimes \mathit{HF}^0(\Delta_2,\Delta_{3,u}) \longrightarrow \mathit{HF}^1(\Delta_2,\Delta_1), \\
& [y_{1,u}] \cdot [z_{2,u}] = \vartheta_{2,2}(u)[w_3] + \vartheta_{2,1}(u)[w_4],
\end{aligned}
\end{equation}
which are the analogues of \eqref{eq:cyclic-mu2} and \eqref{eq:more-products}, respectively. Next consider
\begin{equation}
x = \vartheta_{2,2}(u)[w_1] + \vartheta_{2,1}(u)[w_2] \in \mathit{HF}^0(\Delta_1,\Delta_2).
\end{equation}
We know from \eqref{eq:delta-product} that $[z_{2,u}] \cdot x = 0$, and from \eqref{eq:delta-product-2} that $x \cdot [y_{1,u}] = 0$. Therefore one can form the Massey product $\langle [z_{2,u}], x, [y_{1,u}] \rangle$ (see \cite[Remark 1.2]{seidel04} for the sign conventions in effect here). Generally speaking, such a product takes values in the quotient of $\mathit{HF}^0(\Delta_{3,u},\Delta_{3,u})$ by the two subspaces $[z_{2,u}] \cdot \mathit{HF}^0(\Delta_{3,u},\Delta_2)$ and $\mathit{HF}^{-1}(\Delta_1,\Delta_{3,u}) \cdot [y_{1,u}]$, but both vanish in our case, leading to a strictly well-defined Massey product, which we will show to be the following multiple of the identity class $[e_{3,u}]$:
\begin{equation} \label{eq:delta-massey}
\langle [z_{2,u}], x, [y_{1,u}] \rangle = u(\vartheta_{2,2}'(u)\vartheta_{2,1}(u) - \vartheta_{2,1}'(u)\vartheta_{2,2}(u)) [e_{3,u}].
\end{equation}


To simplify the computation, one can arrange things so that that different homotopy classes of holomorphic triangles can be counted separately. Let's introduce a new formal variable $\epsilon$, which means that we use a version $R_\epsilon$ of \eqref{eq:novikov-field} where the coefficients $c_k$ are allowed to lie in $\bC[\epsilon]$. Mark a point $\ast \in T^2$ such that the fibre $\pi^{-1}(\ast)$ is disjoint from the Lagrangian submanifolds under consideration. We can then construct a version of the Fukaya category relative to that fibre, where the $\epsilon^r$ term in $\mu^d$ comes from pseudo-holomorphic maps which have intersection number $r$ with $\pi^{-1}(\ast)$. In order for this to be always $\geq 0$, the almost complex structures have to be such that $\pi^{-1}(\ast)$ is an almost complex submanifold, and the inhomogeneous terms should vanish quadratically near that fibre; both assumptions are unproblematic as far as transversality is concerned. The outcome is an $A_\infty$-category over $R_\epsilon$, whose specialization to $\epsilon = 1$ recovers the relevant part of the Fukaya category. Note that the choices we have used to define Floer cohomology groups already satisfy those assumptions; hence, \eqref{eq:two-points}--\eqref{eq:dual-point} remain valid for the larger coefficient field. The same applies to more generally to the almost complex structures \eqref{eq:j-delta-theta}.
\begin{figure}
\begin{centering}
\includegraphics{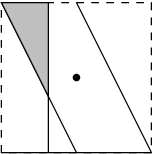}
\caption{\label{fig:abcd-more}}
\end{centering}
\end{figure}

Suppose for concreteness that $-1/2 < m_0 < 1/2$ (the remaining case can be dealt with in a similar way), and choose $\ast = (1/2,1/2)$. To define the product \eqref{eq:delta-product} one should choose a family of almost complex structures $(J_z)$, where the parameter $z$ is in the three-punctured disc, as well as an inhomogeneous term. Suppose that we set the inhomogeneous term to zero, and use a family almost complex structures in the class \eqref{eq:j-delta-theta}, which is locally constant outside a compact subset of the parameter space $z$. In addition, we want our family to have the following property: if we trivialize the part of $\pi$ lying over the triangle in Figure \ref{fig:abcd-more}, then in that trivialization $J_z = i \times (-J_{f,\theta_\ast}) \times J_{f,\theta_\ast}$, where $\theta_\ast$ is constant (independent of $z$ and of the point $(p,q)$ in the triangle).

Let's consider only the $\epsilon^0$ term of (the $R_\epsilon$-linear counterpart of) the second line in \eqref{eq:delta-product}. Because $\pi$ is pseudo-holomorphic, all contributions come from pseudo-holomorphic maps which project to the triangle from Figure \ref{fig:abcd-more}. Those correspond bijectively to $(-J_{f,\theta_\ast}) \times J_{f,\theta_\ast}$-holomorphic maps from a three-punctured disc to $K \times K$, with boundary on the diagonal. But by the same argument as in Lemma \ref{th:project}, such maps are necessarily constant. This determines the moduli space, and simultaneously shows that it is regular, yielding a contribution of $-1 = -\hbar^0\epsilon^0$ to the product (the sign is a consequence of the convention used when defining $w_2$). This contribution then remains the same under small perturbations of the auxiliary choices , so the potential lack of regularity of the higher $\epsilon^k$ spaces is not an issue. 

On the face of it, this argument would seem to fail in general, since it relied on the fact that the triangle in $T$ was embedded in order to construct the desired almost complex structure. However, one can reduce the computation for any power $\epsilon^r$ to the same kind of situation, by passing to a sufficiently large finite cover of $E$ (which depends on $r$) and using the additional freedom to change the almost complex structure there, breaking its symmetry under the covering group. Formally, this means we are using an $\epsilon$-enhanced version of the pullback from Addendum \ref{th:pullback-functor}, but remaining on the cohomology level. This allows us to easily compute the following products in the $\epsilon$-enhanced framework:
\begin{equation} \label{eq:epsilon-product}
\begin{aligned}
& [z_{2,u}] \cdot [w_1] = \epsilon^{-1/4} \vartheta_{2,1}(u)_{\hbar \mapsto \hbar\epsilon} \, [z_{1,u}], \\
& [z_{2,u}] \cdot [w_2] = -\vartheta_{2,2}(u)_{\hbar \mapsto \hbar\epsilon}\, [z_{1,u}],
\end{aligned}
\end{equation}
where $\hbar \mapsto \hbar\epsilon$ is a substitution of variables, and dividing by $\epsilon^{1/4}$ in the first line keeps the powers of $\epsilon$ integral. On the other hand, setting $\epsilon=1$ recovers our previous situation by definition, which concludes our proof of \eqref{eq:delta-product}. The strategy for \eqref{eq:delta-massey} is similar. One first passes to the $\epsilon$-analogue $\langle [z_{2,u}], \vartheta_{2,2}(u)_{\hbar \mapsto \hbar\epsilon} [w_1] + \epsilon^{-1/4}\vartheta_{2,1}(u)_{\hbar \mapsto \hbar\epsilon}[w_2], [y_{1,u}] \rangle$, which is well-defined as a consequence of \eqref{eq:epsilon-product}. For any power $\epsilon^r$, only finitely many homotopy classes can contribute, and after passing to a suitable finite cover one again has only constant maps in fibre direction, which means that the answer is the same as for the two-torus itself, where one can derive it from \eqref{eq:mu3}.

Suppose now that $u \notin \{\pm \hbar^{k/2} \, : \, k \in \bZ\}$. The computations above imply that the following is an exact triangle in $H^0(\mathit{Fuk}(E)^{\mathit{tw}})$:
\begin{equation}
\label{eq:m-triangle}
\xymatrix{\Delta_1 \ar[rrrr]^-{\frac{\vartheta_{2,2}(u)[w_1] + \vartheta_{2,1}(u)[w_2]}{\vartheta_{4,3}'(1)(\vartheta_{4,1}(u) - \vartheta_{4,3}(u))}} &&&& \Delta_2 \ar[dll]^{\;\;\;\;\;\;\; ([z_{2,u}], [z_{2,u^{-1}}])} \\
&& \Delta_{3,u} \oplus \Delta_{3,u^{-1}} \ar[ull]_{[1]}^{([y_{1,u}], -[y_{1,u^{-1}}])\;\;\;\;\;\;\;} &&
}
\end{equation}
To spell this out a little more, one first proceeds as in Lemma \ref{th:exact-triangle} to show that $\Delta_{3,u} \oplus \Delta_{3,u^{-1}}$ is a direct summand of the mapping cone of the horizontal map in \eqref{eq:m-triangle}. Temporarily denote that cone by $\tilde{C}$. There is a spectral sequence converging to $H^*(\mathit{hom}_{\mathit{Fuk}(E)^{\mathit{tw}}}(\tilde{C},\tilde{C}))$, whose starting page is
\begin{equation} \label{eq:cone-spectral-sequence}
E_1^{pq} = \begin{cases} \mathit{HF}^q(\Delta_2,\Delta_1) & p = -1, \\
\mathit{HF}^q(\Delta_1,\Delta_1) \oplus \mathit{HF}^q(\Delta_2,\Delta_2) & p = 0, \\
\mathit{HF}^q(\Delta_1,\Delta_2) & p = 1, \\
0 & \text{otherwise.}
\end{cases}
\end{equation}
The differential $d_1: E_1^{pq} \rightarrow E_1^{p+1,q}$ is given by multiplying by the morphism used to form the cone, with suitable signs. In particular, the subspace of elements in \eqref{eq:cone-spectral-sequence} of total degree $p+q = 0$ is four-dimensional; and the differentials $E_1^{-1,1} \rightarrow E_1^{0,1}$, $E_1^{0,0} \rightarrow E_1^{1,0}$ are both nonzero. This shows that $H^0(\mathit{hom}_{\mathit{Fuk}(E)^{\mathit{tw}}}(\tilde{C},\tilde{C}))$ is of dimension $\leq 2$, which implies that $\tilde{C}$ must be quasi-isomorphic to $\Delta_{3,u} \oplus \Delta_{3,u^{-1}}$.


\begin{lemma} \label{th:degree1}
For a suitable choice of auxiliary data, the differential on $\bigoplus_{i,j=1}^2 \mathit{CF}^*(\Delta_i,\Delta_j)$ vanishes, and moreover, the subspace of elements of degree $\leq 1$ is an $A_\infty$-subalgebra.
\end{lemma}

\begin{proof}
Choose a perfect Morse function $h_K$ on $K$. Choose also a perfect Morse function $h_{S^1}$ on the circle $S^1 = \bR/\bZ$, with minimum at $p_0 = 1/4$ and maximum at $p_1 = 3/4$. When defining $\mathit{CF}^*(\Delta_1,\Delta_1)$ in the Morse-Bott formalism from Section \ref{subsec:clean}, take the Morse function on $\Delta_1 \iso S^1 \times K$ given by
\begin{equation}
h_{\Delta_1,\Delta_1}(p,x) = h_{S^1}(p) + \mathrm{constant} \cdot h_K(x),
\end{equation}
where the constant is small and positive (and correspondingly, we choose the product Riemannian metric). The differential is obviously trivial. Moreover, we have that
\begin{equation} \label{eq:barrier}
\begin{aligned}
& \partial_p h_{\Delta_1,\Delta_1} > 0 && \text{along $\{1/2\} \times K$,} \\
& \partial_p h_{\Delta_1,\Delta_1} < 0 && \text{along $\{0\} \times K$.}
\end{aligned}
\end{equation}
When defining the higher order $A_\infty$-structure on $\mathit{CF}^*(\Delta_1,\Delta_1) = \mathit{CM}^*(h_{\Delta_1,\Delta_1})$, we proceed as in Remark \ref{th:single-l}, but take care that all the auxiliary families of Morse functions appearing in the process still satisfy \eqref{eq:barrier} (this is an open condition, hence does not stand in the way of transversality arguments). As a result, if $a_1,\dots,a_d$ are generators corresponding to critical points, at least one of which lies in $\{3/4\} \times K$, then $\mu^d_{\mathit{Fuk}(E)}(a_d,\dots,a_1)$ must be a linear combination of critical points also lying in $\{3/4\} \times K$. In particular, take $a$ to be the unique generator corresponding to a critical point of index $1$, which is the minimum of $h_K$ placed in $\{3/4\} \times K$. Then $\mu^d_{\mathit{Fuk}(E)}(a,\dots,a)$ is a linear combination of critical points of index $2$, which moreover lie in $\{3/4\} \times K$. But there are no such points, hence we have proved that these particular $A_\infty$-products must vanish. Of course, all these considerations can be applied to $\Delta_2$ as well.

Similarly, when defining $\mathit{CF}^*(\Delta_1,\Delta_2)$ and $\mathit{CF}^*(\Delta_2,\Delta_1)$, we choose minimal Morse functions, so that $\mu^1 = 0$. Because of this minimality property, the product $\mu^2$ is determined entirely by Lemma \ref{th:q-cohomologically}, which shows in particular that the product of any two elements of degree $1$ vanishes. Consider a higher product of elements of degree $\leq 1$,
\begin{equation} \label{eq:01-product}
\mu^d_{\mathit{Fuk}(E)}: \mathit{CF}^{\leq 1}(\Delta_{i_{d-1}},\Delta_{i_d}) \otimes \cdots \otimes
\mathit{CF}^{\leq 1}(\Delta_{i_0},\Delta_{i_1}) \longrightarrow \mathit{CF}^{\leq 2}(\Delta_{i_0},\Delta_{i_d})
\end{equation}
for some $d > 2$ and $i_0,\dots,i_d \in \{1,2\}$, and where the $i_k$ are not all equal (since that case has been dealt with before). If one of the inputs has degree $0$, the output automatically has degree $\leq 1$. The only remaining case is when all the inputs have degree $1$, which forces $(i_0,\dots,i_d) = (2,\dots,2,1,\dots,1)$, but then the output would lie in $\mathit{CF}^2(\Delta_2,\Delta_1)$, which vanishes by minimality and \eqref{eq:dual-two-point}.
\end{proof}

We have now obtained an embedding $G^1: Q \rightarrow \mathit{Fuk}(E)$ which is compatible with the multiplication $\mu^2$, and which therefore extends to an $A_\infty$-functor $G: Q_{\tilde{p}} \rightarrow \mathit{Fuk}(E)$ for some a priori unknown polynomial $\tilde{p}$.

\begin{lemma}
$\tilde{p} = p$ is the unit torus polynomial.
\end{lemma}

\begin{proof}
Fix some $u \notin \{\pm h^{k/2} \, : \, k \in \bZ\}$, and consider the morphism $v \in e_2 Q e_1 = \mathit{hom}_{Q_{\tilde{p}}}(X_1,X_2)$ given by the same formula as the horizontal arrow in \eqref{eq:m-triangle}. Denote by $C_v$ the cone of that morphism in $Q_{\tilde{p}}^{\mathit{tw}}$. Because $A_\infty$-functors preserve exact triangles, we have a commutative diagram in the category $H^0(\mathit{Fuk}(E)^{\mathit{tw}})$,
\begin{equation} \label{fukaya-m-triangle}
\xymatrix{
\cdots \ar[r] & G(X_2) \ar@{=}[d] \ar[r] & G^{\mathit{tw}}(C_v) \ar[r] \ar@{-->}[d]^{\iso} & G(X_1)[1] \ar@{=}[d] \ar[r] & \cdots \\ \cdots \ar[r] & \Delta_2 \ar[r] & \Delta_{3,u} \oplus \Delta_{3,u^{-1}} \ar[r] & \Delta_1[1] \ar[r] & \cdots
}
\end{equation}
where the top row is the obvious exact triangle, the bottom one is \eqref{eq:m-triangle}, and the dotted isomorphism is the only new ingredient. It follows from our previous analysis of \eqref{eq:cone-spectral-sequence} that $G$ induces an isomorphism
\begin{equation} \label{eq:g-isomorphism}
H^0(\mathit{hom}_{Q_{\tilde{p}}^{\mathit{tw}}}(C_v,C_v)) \longrightarrow \mathit{HF}^0(\Delta_{3,u},\Delta_{3,u}) \oplus \mathit{HF}^0(\Delta_{3,u^{-1}},\Delta_{3,u^{-1}}) = R^2.
\end{equation}
Consider the endomorphism $\tilde{t}$ of $\Delta_{3,u} \oplus \Delta_{3,u^{-1}}$ given by the same linear combination of identity elements as in \eqref{eq:geometric-t-tilde}. Using \eqref{eq:delta-product-3} one sees that this satisfies the analogue of \eqref{eq:dual-triple-product}. Hence, its preimage under \eqref{eq:g-isomorphism} satisfies the criteria from Lemma \ref{th:exactly-complete-triangle}, which shows that $\tilde{p}(v_1,v_2) = p(v_1,v_2)$, exactly as in Lemma \ref{th:fukaya-torus}.
\end{proof}

We conclude this discussion by looking at the maps \eqref{eq:functoriality} on Hochschild cohomology induced by $G$, and how they relate to the open-closed string map.

\begin{lemma} \label{th:g-push}
The map
\begin{equation} \label{eq:g-push}
H(G_*): \mathit{HH}^*(Q_p,Q_p) \longrightarrow \mathit{HH}^*(Q_p,G^*\mathit{Fuk}(E))
\end{equation}
is an isomorphism in degrees $\leq 1$.
\end{lemma}

\begin{proof}
It is convenient to replace $Q_p$ by its quasi-isomorphic image in $\mathit{Fuk}(E)$ described in Lemma \ref{th:degree1}. Denoting that by $\tilde{Q}$, what we then have to look at is the effect of the inclusion $\tilde{Q} \hookrightarrow \mathit{Fuk}(E)$. Any Hochschild cochain $g \in \mathit{CC}^{\leq 1}(\tilde{Q},\mathit{Fuk}(E))$ necessarily takes values in the subspace of morphisms of degree $\leq 1$, which is precisely $\tilde{Q}$, so we have $\mathit{CC}^{\leq 1}(\tilde{Q},\tilde{Q}) = \mathit{CC}^{\leq 1}(\tilde{Q},\mathit{Fuk}(E))$. This fails in degree $2$, but at least we have an injective map of cochains there, which is precisely what's needed to prove the desired statement.
\end{proof}

Recall that $\mathit{HH}^1(Q_p,Q_p)$ is two-dimensional, with generators $[g_1]$, $[g_2]$ which were (partially) described in Addendum \ref{th:quasi-generators}.

\begin{lemma} \label{th:g-pull}
The composition
\begin{equation} \label{eq:g-pull}
\mathit{QH}^*(E) \longrightarrow \mathit{HH}^*(\mathit{Fuk}(E),\mathit{Fuk}(E))
\xrightarrow{H(G^*)} \mathit{HH}^*(Q_p,G^*\mathit{Fuk}(E))
\end{equation}
sends $[dp] \in H^1(E;\bR)$ to $H(G_*)([g_1] + [g_2])$, and $[dq]$ to $H(G_*)(-2[g_2])$.
\end{lemma}

\begin{proof}
By construction, we have a commutative diagram
\begin{equation} \label{eq:open-closed-g}
\xymatrix{
\mathit{HH}^*(Q_p,Q_p) \ar[r]^{G_*} \ar[d] & \ar[d] \mathit{HH}^*(Q_p,\mathit{Fuk}(E)) & \ar[l]_{G^*}
\mathit{HH}^*(\mathit{Fuk}(E),\mathit{Fuk}(E)) \ar[dl] \\
e_1Q_pe_1 \oplus
e_2Q_pe_2
\ar[r]^-{H(G)} &
\mathit{HF}^*(\Delta_1,\Delta_1)
\oplus
\mathit{HF}^*(\Delta_2,\Delta_2)
&
}
\end{equation}
The vertical arrow on the left is an isomorphism in degree $1$ (by Addendum \ref{th:quasi-generators}), and hence so is the one in the middle (using Lemma \ref{th:g-push}). On the other hand, composition of the open-closed string map with the diagonal arrow in \eqref{eq:open-closed-g} just yields the ordinary restriction map $\mathit{QH}^*(E) \rightarrow H^*(\Delta_1;R) \oplus H^*(\Delta_2;R)$. The rest is diagram-chasing.
\end{proof}

We now consider the analogue of Corollary \ref{th:t2-family}. Let $\SS = \mathit{Spec}(\RR)$ be the affine curve associated to the unit torus polynomial $p$, and $\theta$ its standard one-form. The image of
\begin{equation} \label{eq:explicit-current-2}
\theta \otimes [dq] \in H^0(\SS,\Omega^1_\SS) \otimes \mathit{QH}^1(E)
\end{equation}
under the open-closed string map is a deformation field, which we denote by $[\gamma]$, for the constant family $\Fuk(E)$ of Fukaya categories over $\SS$.

\begin{corollary} \label{th:m-family}
There is a perfect family of modules $\DD_3$ which follows $[\gamma]$, and whose fibre at a point $(s_1,s_2) \in \SS$ is isomorphic to $\Delta_{3,u}$, where $u \in R^\times/\hbar^{\bZ}$ satisfies \eqref{eq:theta-parametrization}.
\end{corollary}

\begin{proof}
Take the family from Corollary \ref{th:t2-family} and map it to $\mathit{Fuk}(E)$ using $G$. As a consequence of the general discussion of functoriality in Section \ref{subsec:functor}, the image family indeed follows \eqref{eq:explicit-current-2} (the equality of Hochschild cohomology classes required in Assumption \ref{th:matching-deformations} comes from Lemma \ref{th:g-pull}). By construction, the object of the family at any point is a direct summand of a mapping cone. The triangle \eqref{eq:m-triangle} identifies that mapping cone with $\Delta_{3,u} \oplus \Delta_{3,u^{-1}}$, and one can follow the same computation as in the case of the two-torus to show that the summand picked out by the projection is indeed $\Delta_{3,u}$.
\end{proof}

\subsection{More Lagrangian submanifolds}
$E$ admits a (graded) symplectic automorphism $F$ which is trivial on the base $T$, and equals $\mathit{id} \times f$ in each fibre (this makes sense since it commutes with $f \times f$). By applying that automorphism to our given Lagrangian submanifolds, we get another collection
\begin{equation} \label{eq:lagrangians-2}
\begin{aligned}
& \Gamma_1 = F(\Delta_1) = \{q = 0, \;\; y = f(x)\}, \\
& \Gamma_2 = F(\Delta_2) = \{q = -2p, \;\; y = f(x)\}, \\
& \Gamma_{3,u} = F(\Delta_{3,u}) = \{p = m_0, \;\; y = f(x)\}, \\
\end{aligned}
\end{equation}
These come with induced gradings and {\em Spin} structures. We equip each of these Lagrangian submanifolds with the almost complex structure $J_{\Gamma,p,q,x,y} = i \times (-J_{f,p+1,x}) \times J_{f,p,y}$. Since that is the image under $F$ of $i \times (-J_{f,p+1,x}) \times J_{f,p+1,y}$, the previous computations of Floer cohomology and its product structure carry over to \eqref{eq:lagrangians-2}. We will also need to know how \eqref{eq:lagrangians-2} and \eqref{eq:lagrangians} interact. Unsurprisingly, the answers involve the fixed point Floer cohomology of $f$:
\begin{align}
& \mathit{HF}^*(\Gamma_1,\Delta_1) \iso H^*(S^1;R) \otimes \mathit{HF}^*(f), \label{eq:circle-gd} \\
& \mathit{HF}^*(\Gamma_1,\Delta_2) \iso \mathit{HF}^*(f) \oplus \mathit{HF}^*(f), \\
& \mathit{HF}^*(\Gamma_1,\Delta_{3,u}) \iso (\xi_u)_{(m_0,0)} \otimes \mathit{HF}^*(f), \label{eq:point-gd} \\
& \mathit{HF}^*(\Gamma_2,\Delta_1) \iso \mathit{HF}^*(f)[-1] \oplus \mathit{HF}^*(f)[-1], \\
& \mathit{HF}^*(\Gamma_2,\Delta_2) \iso H^*(S^1;R) \otimes \mathit{HF}^*(f), \label{eq:circle-gd-2} \\
& \mathit{HF}^*(\Gamma_2,\Delta_3) \iso (\xi_u)_{(m_0,-2m_0)} \otimes \mathit{HF}^*(f), \\
& \mathit{HF}^*(\Gamma_{3,u},\Delta_2) \iso (\xi_u)_{(m_0,-2m_0)}^\vee \otimes \mathit{HF}^*(f)[-1], \\
& \mathit{HF}^*(\Gamma_{3,u},\Delta_1) \iso (\xi_u)_{(m_0,0)}^\vee \otimes \mathit{HF}^*(f)[-1], \\
& \mathit{HF}^*(\Gamma_{3,u},\Delta_{3,u}) \iso H^*(S^1;R) \otimes \mathit{HF}^*(f). \label{eq:infty-infty}
\end{align}
For each of these, we adopt a variant of the strategy in Example \ref{th:graph}, which means that we use the family of almost complex structures on $E$ given by
\begin{equation} \label{eq:j-gamma-delta-2}
(J_{\Gamma,\Delta,t})_{p,q,x,y} = i \times (-J_{f,p+1-t/2,x}) \times J_{f,p+t/2,y}.
\end{equation}
In all cases listed above, pseudo-holomorphic strips must be contained in a fibre. We take such a strip $(u_x,u_y): \bR \times [0,1] \rightarrow \pi^{-1}(p,q) \iso K \times K$ and transform it to a map $u$ as in \eqref{eq:roll-up}, which then satisfies $u(s,t-1) = f(u(s,t))$ and $\partial_s u + J_{f,p+t}(u)\partial_t u = 0$. This makes the isomorphisms above obvious, with \eqref{eq:circle-gd} and \eqref{eq:circle-gd-2} requiring a little thought (the case of \eqref{eq:infty-infty} is much simpler, since $\pi$ has trivial monodromy in $q$-direction). Let's consider briefly the first of the two. The intersection $\Gamma_1 \cap \Delta_1$ consists of a circle $C_x$ for each fixed point $x$ of $f$. The Maslov index of $C_x$ equals the Conley-Zehnder index of $x$. The local coefficient system $o_{C_x}$ has fibre $o_x$, and its monodromy is given by the natural action of $Df_x$ on $o_x$. It is a nontrivial observation, but one which is well-known as part of the mechanism underlying \eqref{eq:self-conjugation}, that this action is trivial. If we then choose $h_{\Gamma_1,\Delta_1}$ to be the same Morse function $h_{S^1}$ on each circle, we get an isomorphism of graded vector spaces
\begin{equation} \label{eq:cone-product}
\mathit{CF}^*(\Gamma_1,\Delta_1) \iso \mathit{CM}^*(h_{S^1}) \otimes \mathit{CF}^*(f).
\end{equation}
Using the previous observation about pseudo-holomorphic strips, it is not hard to see that this is compatible with the Floer differential. The other case \eqref{eq:circle-gd-2} is parallel.

\begin{remark}
Another way to see where the potential difficulty in \eqref{eq:circle-gd} lies is to consider for a moment a more general family of Lagrangian submanifolds fibred over the same base circle, namely
\begin{equation}
\Gamma_1^m = \{q = 0, \;\; y = f^m(x)\}
\end{equation}
for some $m \geq 1$. The intersection points of $\Gamma_1^m \cap \Delta_1$ in each fibre $\pi^{-1}(p,0)$ correspond to fixed points of $f^m$ (which we assume to be nondegenerate). However, this correspondence depends on $p \in \bR$, rather than only on its image in $\bR/\bZ$: as we move around the circle, there is nontrivial monodromy which acts by $f$ on the set of these points. Moreover, even for points which are fixed by $f^p$ for some $p|m$, the induced action on $o_x$ (where $x$ is considered as an $m$-periodic point) can be nontrivial; this is the same phenomenon as the ``bad orbits'' in Symplectic Field Theory. Finally, while the moduli spaces of holomorphic strips fibre over $S^1 \times \{0\}$, that fibration can also be nontrivial. In fact, what one gets is a chain homotopy
\begin{equation}
\mathit{CF}^*(\Gamma_1^m,\Delta_1) \htp \mathit{Cone}(\mathit{id} - c_{f,f^m}: \mathit{CF}^*(f^m) \longrightarrow \mathit{CF}^*(f^m)),
\end{equation}
where $c_{f,f^m}$ is the chain map underlying the generator of the $\bZ/m$-action on $\mathit{HF}^*(f^m)$.
\end{remark}

We will also need to know two related products, namely
\begin{align}
& \mathit{HF}^*(\Delta_1,\Delta_{3,u}) \otimes \mathit{HF}^*(\Gamma_1,\Delta_1) \longrightarrow \mathit{HF}^*(\Gamma_1,\Delta_{3,u}), \label{eq:d-product-1} \\
& \mathit{HF}^*(\Gamma_1,\Delta_{3,u}) \otimes \mathit{HF}^*(\Gamma_{3,u},\Gamma_1) \longrightarrow \mathit{HF}^*(\Gamma_{3,u},\Delta_{3,u}). \label{eq:d-product-2}
\end{align}
After using \eqref{eq:one-point-hf}, the analogue of \eqref{eq:dual-point} for the $\Gamma$ Lagrangian submanifolds, as well as \eqref{eq:circle-gd}, \eqref{eq:point-gd}, \eqref{eq:infty-infty}, and cancelling the $\xi_u$ factors, these maps can be written as
\begin{align}
& H^*(K;R) \otimes H^*(S^1;R) \otimes \mathit{HF}^*(f) \longrightarrow \mathit{HF}^*(f), \label{eq:s1-product-1} \\
& \mathit{HF}^*(f)  \otimes H^*(K;R)[-1] \longrightarrow H^*(S^1;R) \otimes \mathit{HF}^*(f). \label{eq:s1-product-2}
\end{align}

\begin{lemma} \label{th:2-products}
The first map \eqref{eq:s1-product-1} vanishes on the $H^1(S^1;R)$ summand, and on the $H^0(S^1;R)$ summand it reproduces the quantum cap module structure of $\mathit{HF}^*(f)$. The second map \eqref{eq:s1-product-2} takes values in the $H^1(S^1;R)$ summand, and again reproduces the quantum cap structure.
\end{lemma}

\begin{proof}
To keep the notation simple, we consider only the first product \eqref{eq:s1-product-1} and the case $u = 1$ (so $m_0 = 0$ and $\xi_u$ is trivial). Define $\mathit{CF}^*(\Gamma_1,\Delta_1)$ as in \eqref{eq:cone-product}, taking care that $h_{S^1}$ has a single minimum at $(0,0)$ (and a single maximum elsewhere), and using the family of almost complex structures $J_{\Gamma,\Delta}$ from \eqref{eq:j-gamma-delta-2}. Next, the intersection $\Delta_1 \cap \Delta_{3,u} \iso K$ is the diagonal in the fibre at $(0,0)$. We choose a Morse function $h_{\Delta_1,\Delta_{3,u}} = h_K$ on $K$ in order to define $\mathit{CF}^*(\Delta_1,\Delta_{3,u})$, and use the constant family of almost complex structures $J_\Delta$. Finally, the intersection $\Gamma_1 \cap \Delta_{3,u}$ is transverse, and we use the same family $J_{\Gamma,\Delta}$ for it as before. To form the quantum cap product \eqref{eq:quantum-cap}, we use the same Morse function $h_K$, as well as a two-parameter family of almost complex structures of the form
\begin{equation}
J_{\mathrm{cap},s,t} = J_{f,\psi(s,t)},
\end{equation}
where $\psi: \bR^2 \rightarrow \bR$ is a function satisfying $\psi(s,t+1) = \psi(s,t)+1$ for all $(s,t)$, $\psi(s,t) = t$ for $|s| \gg 0$, and $\psi(s,t) = 0$ for $(s,t)$ close to $(0,1/2)$. This leaves enough freedom to achieve the required transversality properties.

The pearly trees that can, in principle, contribute to the product are shown in Figure \ref{fig:pearly-sequence}. Consider for a moment the simplest such tree, which has just one trivalent vertex. The Riemann surface associated to that vertex can be written as $S_v = (\bR \times [0,1]) \setminus \{(0,1)\}$. We choose the perturbation datum on $S_v$ to have trivial inhomogeneous term, and the following family $J_v$ of almost complex structures:
\begin{equation}
(J_{v,s,t})_{p,q,x,y} = i \times (-J_{f,p+\psi(s/2,1-t/2),x}) \times J_{f,p+\psi(s/2,t/2),y}.
\end{equation}
Importantly, near $(s,t) = (0,1)$ this reduces to $J_\Delta$. Solutions of the associated equation \eqref{eq:generalized-u2-equation} are all contained in the fibre over $(0,0)$. Moreover, in analogy with \eqref{eq:roll-up}, they correspond bijectively to maps $u_v: \bR^2 \rightarrow K$ solving the pseudo-holomorphic part of \eqref{eq:cap-floer}. Choose the families of Morse functions on the two semi-infinite edges of our pearly tree to be constant equal to $h_{\Gamma_1,\Delta_1}$ and $h_{\Delta_1,\Delta_{3,u}}$, respectively. The associated gradient (half-)flow line of $h_{\Gamma_1,\Delta_1}$ must necessarily be constant, whereas the other one yields the Morse-theoretic part of \eqref{eq:cap-floer}. Regularity is easy to check.

It remains to exclude contributions from more complicated pearly trees. The maps associated to the two-valent vertices in the upper branch of Figure \ref{fig:pearly-sequence} are Floer differentials for the pair $(\Delta_1,\Delta_{3,u})$, but as we have seen before there are none, since they would correspond to non-constant $J_{f,1/2}$-holomorphic spheres in $K$. On the lower branch we equip all the finite length edges with the same constant family of Morse functions. But then, all the associated gradient flow lines are necessarily constant, which means that the length of the edge is a free parameter. After a necessary but easy regularity consideration, it follows that this cannot occur in zero-dimensional moduli spaces.
\end{proof}

\begin{figure}
\begin{center}
\begin{picture}(0,0)%
\includegraphics{product1.pstex}%
\end{picture}%
\setlength{\unitlength}{3947sp}%
\begingroup\makeatletter\ifx\SetFigFont\undefined%
\gdef\SetFigFont#1#2#3#4#5{%
  \reset@font\fontsize{#1}{#2pt}%
  \fontfamily{#3}\fontseries{#4}\fontshape{#5}%
  \selectfont}%
\fi\endgroup%
\begin{picture}(3252,2311)(1111,-6050)
\put(1966,-5161){\makebox(0,0)[lb]{\smash{{\SetFigFont{10}{12.0}{\rmdefault}{\mddefault}{\updefault}{\color[rgb]{0,0,0}$\Delta_1$}%
}}}}
\put(3376,-4136){\makebox(0,0)[lb]{\smash{{\SetFigFont{10}{12.0}{\rmdefault}{\mddefault}{\updefault}{\color[rgb]{0,0,0}$h_{\Delta_1,\Delta_{3,u}}$}%
}}}}
\put(1576,-5986){\makebox(0,0)[lb]{\smash{{\SetFigFont{10}{12.0}{\rmdefault}{\mddefault}{\updefault}{\color[rgb]{0,0,0}$\Gamma_1$}%
}}}}
\put(3376,-5986){\makebox(0,0)[lb]{\smash{{\SetFigFont{10}{12.0}{\rmdefault}{\mddefault}{\updefault}{\color[rgb]{0,0,0}$\Gamma_1$}%
}}}}
\put(2326,-5311){\makebox(0,0)[lb]{\smash{{\SetFigFont{10}{12.0}{\rmdefault}{\mddefault}{\updefault}{\color[rgb]{0,0,0}$h_{\Gamma_1,\Delta_1}$}%
}}}}
\put(3976,-5311){\makebox(0,0)[lb]{\smash{{\SetFigFont{10}{12.0}{\rmdefault}{\mddefault}{\updefault}{\color[rgb]{0,0,0}$h_{\Gamma_1,\Delta_1}$}%
}}}}
\put(2776,-3886){\makebox(0,0)[lb]{\smash{{\SetFigFont{10}{12.0}{\rmdefault}{\mddefault}{\updefault}{\color[rgb]{0,0,0}$\Delta_{3,u}$}%
}}}}
\put(2776,-4786){\makebox(0,0)[lb]{\smash{{\SetFigFont{10}{12.0}{\rmdefault}{\mddefault}{\updefault}{\color[rgb]{0,0,0}$\Delta_1$}%
}}}}
\put(3376,-5086){\makebox(0,0)[lb]{\smash{{\SetFigFont{10}{12.0}{\rmdefault}{\mddefault}{\updefault}{\color[rgb]{0,0,0}$\Delta_1$}%
}}}}
\put(1126,-5161){\makebox(0,0)[lb]{\smash{{\SetFigFont{10}{12.0}{\rmdefault}{\mddefault}{\updefault}{\color[rgb]{0,0,0}$\Delta_{3,u}$}%
}}}}
\put(1201,-4486){\makebox(0,0)[lb]{\smash{{\SetFigFont{10}{12.0}{\rmdefault}{\mddefault}{\updefault}{\color[rgb]{0,0,0}$h_{\Delta_1,\Delta_{3,u}}$}%
}}}}
\end{picture}%
\caption{\label{fig:pearly-sequence}}
\end{center}%
\end{figure}%

\subsection{More families\label{subsec:overview}}
We now return to the situation from Corollary \ref{th:m-family}. By exactly the same argument (or otherwise by using the functoriality under $F$), one has a family $\GG_3$ with fibres $\Gamma_{3,u}$, and which otherwise has the same properties as $\DD_3$.

\begin{lemma} \label{th:constant-families}
The constant families $\DD_1 = \RR \otimes_R \Delta_1$, $\GG_1 = \RR \otimes_R \Gamma_1$ also follow $[\gamma]$. Moreover, one can choose relative connections on them in such a way that the induced connection on
\begin{equation}
H^*(\mathit{hom}_{\Fuk(E)}(\GG_1,\DD_1)) \iso \RR \otimes_R \mathit{HF}^*(\Gamma_1,\Delta_1)
\end{equation}
is trivial.
\end{lemma}

\begin{proof}
The class $[dq]$ is dual to the hypersurface $\{q = 1/2\}$, which is disjoint from both $\Delta_1$ and $\Gamma_1$. As an instance of Example \ref{th:totally-disjoint}, it follows that $[g]$ vanishes on the subcategory with these two objects. If we then choose trivial relative connections, the result is obviously true.
\end{proof}

Without changing the notation, we will now apply the Yoneda embedding and consider $\DD_1$ and $\GG_1$ as objects of $\Fuk(E)^{\mathit{perf}}$. Lemma \ref{th:constant-families} still holds in this context.

\begin{corollary} \label{th:construct-family}
Fix a one-dimensional subspace $B_0 \subset \mathit{HF}^{d-1}(f)$. Then, for suitable choices of relative connections, there is a line bundle $\BB \subset H^0(\mathit{hom}_{\Fuk(E)^{\mathit{perf}}}(\GG_3,\DD_3[d]))$ invariant under the induced connection, whose restriction to any fibre agrees with $B = H^1(S^1;R) \otimes B_0$ under the isomorphism \eqref{eq:infty-infty}.
\end{corollary}

\begin{proof}
Consider the double product (composition of two ordinary products in the Fukaya category; the ordering is irrelevant by associativity)
\begin{equation} \label{eq:left-right-product}
\begin{aligned}
& H^0(\mathit{hom}_{\Fuk(E)^{\mathit{perf}}}(\DD_1,\DD_3)) \otimes H^*(\mathit{hom}_{\Fuk(E)^{\mathit{perf}}}(\GG_1,\DD_1)) \otimes H^1(\mathit{hom}_{\Fuk(E)^{\mathit{perf}}}(\GG_3,\GG_1)) \\ & \qquad \qquad \longrightarrow H^*(\mathit{hom}_{\Fuk(E)^{\mathit{perf}}}(\GG_3,\DD_3))[1].
\end{aligned}
\end{equation}
We choose relative connections on $\GG_1$ and $\DD_1$ as in Lemma \ref{th:constant-families}. Consider the subbundle $\RR \otimes H^0(S^1;R) \otimes B_0$ of $H^*(\mathit{hom}_{\Fuk(E)^{\mathit{perf}}}(\GG_1,\DD_1)) \iso \RR \otimes H^*(S^1) \otimes \mathit{HF}^*(f)$, which is of course preserved by the connection. The leftmost and rightmost factors on the LHS of \eqref{eq:left-right-product} are line bundles. Hence, the image of our subbundle under \eqref{eq:left-right-product} yields a subbundle $\BB \subset H^d(\mathit{hom}_{\Fuk(E)^{\mathit{perf}}}(\GG_3,\DD_3))$ which, because of the compatibility of the product with the connections \eqref{eq:multiplicative-connections}, is itself preserved by the connection. At any point of $\SS$, \eqref{eq:left-right-product} can be written as a map
\begin{equation}
H^*(S^1;R) \otimes \mathit{HF}^*(f) \longrightarrow H^*(S^1;R) \otimes \mathit{HF}^*(f)[1].
\end{equation}
From our computation of \eqref{eq:d-product-1} and \eqref{eq:d-product-2}, we know that this is the identity on $\mathit{HF}^*(f)$ times the cup-product with a nonzero class in $H^1(S^1;R)$. This shows that $\BB$ has the desired property.
\end{proof}

This allows one to apply parallel transport at least to a certain part of $\mathit{HF}^d(\Gamma_{3,u},\Delta_{3,u})$ (probably, the same holds for the entire Floer group, but we will not consider this point here). The next issue is uniqueness, which can be dealt with by using Proposition \ref{th:uniqueness-2} based on the following observation:

\begin{lemma} \label{th:assumption-satisfied}
Any one-dimensional subspace $B \subset \mathit{HF}^0(\Gamma_{3,u}, \Delta_{3,u}[d]) = \mathit{HF}^d(\Gamma_{3,u},\Delta_{3,u})$ satisfies Assumption \ref{th:augmented-plus}.
\end{lemma}

\begin{proof}
Assumption \ref{th:augmented} for each object is obvious, since $\mathit{HF}^*(\Delta_{3,u},\Delta_{3,u}) \iso H^*(\Delta_{3,u};R)$ as a ring, and the same for $\Gamma_{3,u}$. The products in \eqref{eq:swap-sides} are part of the Floer product structure
\begin{equation} \label{eq:2-products-2}
\begin{aligned}
& \mathit{HF}^*(\Gamma_{3,u},\Delta_{3,u}) \otimes \mathit{HF}^*(\Gamma_{3,u},\Gamma_{3,u}) \longrightarrow
\mathit{HF}^*(\Gamma_{3,u},\Delta_{3,u}), \\
& \mathit{HF}^*(\Delta_{3,u},\Delta_{3,u}) \otimes \mathit{HF}^*(\Gamma_{3,u},\Delta_{3,u}) \longrightarrow
\mathit{HF}^*(\Gamma_{3,u},\Delta_{3,u}).
\end{aligned}
\end{equation}
One can determine this explicitly in the manner of Lemma \ref{th:2-products}, but we prefer to take a shortcut which bypasses computation. Namely, as part of the open-closed string map we have a map $\mathit{QH}^*(E) \rightarrow \mathit{HF}^*(\Delta_{3,u},\Delta_{3,u})$, which in this case agrees with the ordinary restriction map (in particular is surjective; compare Examples \ref{th:take-a-single} and \ref{th:quartic}), and the same for $\Gamma_{3,u}$. In fact, by assumption on $f$ our two Lagrangian submanifolds are diffeomorphic, and the restriction maps are the same. One combines this restriction map with \eqref{eq:2-products-2} to yield both a left and a right action of $\mathit{QH}^*(E)$ on $\mathit{HF}^*(\Gamma_{3,u},\Delta_{3,u})$. One can prove geometrically as in \cite[Figure 1]{seidel-solomon10} (or alternatively, derive from the fact that this is part of a map landing in Hochschild cohomology) that these two actions coincide up to Koszul signs. This leads directly to the required property.
\end{proof}

\begin{addendum} \label{th:sneaky}
The reader may have noticed that, in view of Assumption \ref{th:augmented-plus} as originally stated, we only needed to prove the required properties for the degree $0$ parts of $\mathit{HF}^*(\Delta_{3,u},\Delta_{3,u})$ and $\mathit{HF}^*(\Gamma_{3,u},\Gamma_{3,u})$, which is much easier. The real point of the argument above, which will become relevant only later, is that it still yields the desired result if we reduce the grading of the Fukaya category to $\bZ/2$.
\end{addendum}

\subsection{A double covering trick\label{subsec:covering-trick}}
Let $z: \tilde{E} \rightarrow E$ be the double cover associated to $(1,0) \in H^1(T;\bZ/2) \iso H^1(E;\bZ/2)$. Concretely,
\begin{equation} \label{eq:symplectic-mapping-torus-2}
\tilde{E} = \bR \times \bR \times K \times K\;\; / \;\; (p,q,x,y) \sim (p,q-1,x,y) \sim (p-2,q,f^2(x),f^2(y)), \\
\end{equation}
with the symplectic form $\omega_{\tilde{E}}$ pulled back from $E$. This is the mapping torus of $f^2 \times f^2$, except that the area of the base $T$ has been multiplied by $2$. Fukaya category computations for $\tilde{E}$ largely follow those for $E$, so we will only summarize the results. We have Lagrangian submanifolds $\tilde\Delta_1,\tilde\Delta_2,\tilde\Delta_{3,u}$ (fibrewise equal to the diagonal) and $\tilde\Gamma_1,\tilde\Gamma_2,\tilde\Gamma_{3,u}$ (fibrewise equal to the graph of $f^2$) defined analogously to \eqref{eq:lagrangians}, \eqref{eq:lagrangians-2}. To clarify, $\tilde\Delta_2$ is now fibered over the path $\{q=-p\}$ in $\tilde{T}$, hence does not project to $\Delta_2$ (and the same holds for $\tilde\Gamma_2$). On the other hand, for $u = \hbar^{m_0}a$ we still take $\Delta_{3,u}$ to be fibered over $\{p = m_0\}$ (and correspondingly for $\tilde{\Gamma}_{3,u}$). 

As in \eqref{eq:infty-infty} there are canonical isomorphisms
\begin{equation} \label{eq:infty-infty-lifted}
\mathit{HF}^*(\tilde\Gamma_{3,u},\tilde\Delta_{3,u}) \iso H^*(S^1;R) \otimes \mathit{HF}^*(f^2).
\end{equation}
Recall from Section \ref{subsec:pushdown} that $z$ gives rise to a functor $Z$, defined on a full subcategory $\tilde{F} \subset \mathit{Fuk}(\tilde{E})$ (that contains all the Lagrangian submanifolds occurring in our discussion), and which lands in $\mathit{Fuk}(E)$. In particular, $Z(\tilde{\Delta}_{3,u}) = \Delta_{3,u}$, whereas $Z(\tilde{\Gamma}_{3,u})$ is the analogue of $\Gamma_{3,u}$ defined using the graph of $f^2$ in each fibre. Our functor gives an isomorphism
\begin{equation} \label{eq:pushforward-iso}
\mathit{HF}^*(Z(\tilde{\Gamma}_{3,u}), Z(\tilde\Delta_{3,u})) \iso \mathit{HF}^*(\tilde{\Gamma}_{3,u},\tilde{\Delta}_{3,u}) \iso H^*(S^1;R) \otimes \mathit{HF}^*(f^2).
\end{equation}
Note that $Z(\tilde{\Gamma}_{3,u})$ and $Z(\tilde\Delta_{3,u})$ only depend on the class of $u$ in $R^\times/\hbar^\bZ$. However:

\begin{lemma} \label{th:c-change}
Passing from $u$ to $\hbar u$ changes the second isomorphism in \eqref{eq:pushforward-iso} by composition with the involution $C_{f,f^2}$.
\end{lemma}

\begin{proof}
Consider the diagram of isomorphisms
\begin{equation}
\xymatrix{
& 
\mathit{HF}^*(Z(\tilde{\Gamma}_{3,u}),Z(\tilde{\Delta}_{3,u}))
& \\
\ar[ur]^-{Z}
\mathit{HF}^*(\tilde{\Gamma}_{3,u},\tilde{\Delta}_{3,u}) \ar[rr] \ar[d]
&&
\ar[ul]_-{Z}
\mathit{HF}^*(\tilde{\Gamma}_{3,\hbar u},\tilde{\Delta}_{3,\hbar u}) \ar[d]
\\
H^*(S^1;R) \otimes \mathit{HF}^*(\Gamma,\Delta) \ar[rr] \ar[d]
&&
H^*(S^1;R) \otimes \mathit{HF}^*(\Gamma,\Delta) \ar[d]
\\
H^*(S^1;R) \otimes \mathit{HF}^*(f^2) \ar[rr]^-{C_{f,f^2}}
&&
H^*(S^1;R) \otimes \mathit{HF}^*(f^2)
}
\end{equation}
The top horizontal arrow is the action of the covering transformation for $z: \tilde{E} \rightarrow E$,
\begin{equation} \label{eq:covering-transformation}
(p,q,x,y) \longmapsto (p-1,q,f(x),f(y)). 
\end{equation}
The commutativity of the top triangle follows from the definition of $Z$. The top $\downarrow$s, on the left and right, are isomorphisms \eqref{eq:pushforward-iso}. The middle horizontal arrow is the identity on $H^*(S^1;R)$, combined with the action of $f \times f$ on Lagrangian Floer cohomology in $K \times K$. The commutativity of the square in the middle of the triangle then follows by comparing \eqref{eq:covering-transformation} and \eqref{eq:pushforward-iso}. The bottom $\downarrow$s, on the left and right, are the isomorphisms between Lagrangian Floer cohomology and fixed point Floer cohomology from Example \ref{th:graph}. Inspection of that isomorphism shows that the bottom square in the diagram commutes.
\end{proof}

To take into account the difference in the areas of the base $T$, we take the square unit polynomial $p$ and make a substitution $\hbar \mapsto \hbar^2$. This yields a new polynomial $\tilde{p}$ and associated algebraic curve $\tilde{\SS} = \mathit{Spec}(\tilde{\RR})$, with its one-form $\tilde{\theta}$. In fact, we had already considered these in Addendum \ref{th:double-cover}, where it was pointed out that (after removing finitely many points) $\tilde{\SS}$ is an \'etale double cover of $\SS$, and $\tilde{\theta}$ the pullback of $\theta$. We consider the parametrization of the set of points of $\tilde{\SS}$ by $u \in R^\times/\hbar^{2\bZ}$ which under the covering map induces \eqref{eq:theta-parametrization}. With these slight modifications, the previous argument goes through, yielding perfect families $\tilde{\DD}_3$ and $\tilde{\GG}_3$ over $\tilde{F} \subset \mathit{Fuk}(\tilde{E})$ which follow the image of $\tilde{\theta} \otimes [dq]$ under the open-closed string map, and whose fibres at any point $u$ are isomorphic to $\tilde\Delta_{3,u}$ and $\tilde{\Gamma}_{3,u}$, respectively. We now use \eqref{eq:open-closed-pushforward}, as well as the discussion of functoriality from Section \ref{subsec:functor}, to push these families down to $E$. The outcome are perfect families $Z(\tilde\DD_3)$ and $Z(\tilde{\GG}_3)$ over $\mathit{Fuk}(E)$ which follow $\theta \otimes [dq]$, and whose fibres at $u$ are isomorphic to $Z(\tilde{\Delta}_{3,u})$ and $Z(\tilde\Gamma_{3,u})$.

\begin{assumption} \label{th:both-signs}
For some $d \in \bZ$, $C_{f,f^2}: \mathit{HF}^{d-1}(f^2) \rightarrow \mathit{HF}^{d-1}(f^2)$ is not $\pm \mathit{Id}$.
\end{assumption}

Supposing from now on that this is the case, we can choose a one-dimensional subspace $B_0 \subset \mathit{HF}^{d-1}(f^2)$ which is not preserved by $C_{f,f^2}$. Let's temporarily go back to $\tilde{E}$. The analogue of Lemma \ref{th:construct-family} says that there is a line bundle $\tilde\BB \subset H^0(\mathit{hom}_{\Fuk(\tilde{E})^{\mathit{perf}}}(\tilde\GG_3,\tilde\DD_3[d]))$ invariant under the induced connection on that space, whose restriction to any fibre agrees with $B = H^1(S^1;R) \otimes B_0$ under the isomorphism \eqref{eq:infty-infty-lifted}. Applying $Z$ to this, and using the compatibility of induced connections with functors shown in \eqref{eq:strict-functoriality-of-connections}, we find that the image line bundle $Z(\tilde\BB)$ is still invariant under the induced connection.

\begin{lemma} \label{th:non-iso}
Take two points $\tilde{s}_\pm \in \tilde{\SS}$ corresponding to $u$ and $\hbar u$ (for any $u$ such that both make sense, which means excluding the finitely many branch points). Then, the triples
\begin{equation}
(Z(\tilde\GG_3)_{\tilde{s}_\pm}, Z(\tilde\DD_3[d])_{\tilde{s}_\pm}, Z(\tilde\BB)_{\tilde{s}_\pm})
\end{equation}
are not mutually isomorphic in $H^0(\mathit{Fuk}(E)^{\mathit{perf}})$.
\end{lemma}

\begin{proof}
At $\tilde{s}_+$, the relevant triple consists of the objects $z(\tilde{\Gamma}_{3,u})$ and $z(\tilde{\Delta}_{3,u})[d]$ together with the subspace of $\mathit{HF}^d(z(\tilde{\Gamma}_{3,u}),z(\tilde{\Delta}_{3,u}))$ corresponding to $B$ under the isomorphism \eqref{eq:infty-infty-lifted}. The same holds at $\tilde{s}_-$ but where the isomorphism is twisted by $C_{f,f^2}$, as a consequence of Lemma \ref{th:c-change}. Hence, our statement reduces to the following:

\begin{claim*}
There do not exist invertible elements
\begin{equation}
\begin{aligned}
& \gamma \in \mathit{HF}^0(z(\tilde{\Gamma}_{3,u}),z(\tilde{\Gamma}_{3,u})) \iso H^0(z(\tilde{\Gamma}_{3,u});R) \iso H^0(S^1 \times K;R), \\
& \delta \in \mathit{HF}^0(z(\tilde{\Delta}_{3,u}),z(\tilde{\Delta}_{3,u})) \iso H^0(z(\tilde{\Delta}_{3,u});R) \iso H^0(S^1 \times K;R),
\end{aligned}
\end{equation}
satisfying $\delta B \gamma = C_{f,f^2}(B)$.
\end{claim*}

But that is obvious because both $\mathit{HF}^0$ groups only contain multiples of the identity.
\end{proof}

\begin{addendum} \label{th:not-iso-even-without-grading}
The Claim above, and therefore Lemma \ref{th:non-iso}, continues to hold even if we allow $\delta$ and $\gamma$ to have additional terms of higher even degree. This is because the subspace $B$ itself is concentrated in a single degree.
\end{addendum}

Both points $\tilde{s}_\pm \in \tilde{\SS}$ map to the same point $s \in \SS$. This, together with the analogue of Lemma \ref{th:assumption-satisfied}, triggers Lemma \ref{th:contradiction}, which shows that:

\begin{corollary} \label{th:not-periodic}
If Assumption \ref{th:both-signs} is satisfied, the image of $[dq]$ in $\mathit{HH}^1(\mathit{Fuk}(E),\mathit{Fuk}(E))$ is not a periodic element (for the elliptic curve with one-form obtained as the closure of $\SS$ and $\theta$). \qed
\end{corollary}

Let's have a brief ``straight man'' conterpart of the previous discussion, concerning the case where the symplectic automorphism is the identity, giving rise to the trivial mapping torus $E^{\mathit{triv}} = T \times K^- \times K$. Arguing as in Example \ref{th:product-mirror}, one finds that there are quasi-equivalences
\begin{equation}
\begin{aligned}
\mathit{Fuk}(E^{\mathit{triv}})^{\mathit{perf}} & \iso (\mathit{Fuk}(T) \otimes \mathit{Fuk}(K^-) \otimes \mathit{Fuk}(K))^{\mathit{perf}} \\
& \iso (D^b\mathit{Coh}(Y_p) \otimes D^b\mathit{Coh}(X) \otimes D^b\mathit{Coh}(X))^{\mathit{perf}} \\
& \iso D^b\mathit{Coh}(Y_p \times X \times X),
\end{aligned}
\end{equation}
where $Y_p$ and $X$ are the mirrors of $T$ and $K$, respectively (the fact that one of the copies of $K$ has reversed sign of the symplectic form does not affect the statement, since $\omega_K$ and $-\omega_K$ are related by an involution, as one can see by taking $K$ a real quartic). One can construct a family of bimodules exactly as in Section \ref{subsec:universal-bimodule}, and use that to derive the following analogue of Corollary \ref{th:periodic-lattice}:
\begin{equation} \label{eq:lattice-2}
m_1[g_1] + m_2[g_2] \in \mathit{Per}(D^b\mathit{Coh}(Y_p \times X \times X),\bar{\SS},\bar\theta) \text{ for $m_1 \in \bZ$, $m_2 \in m_1 + 2\bZ$},
\end{equation}
where $[g_1]$, $[g_2]$ are the classes pulled back from $\mathit{HH}^*(Y_p,Y_p) \iso \mathit{HH}^*(Q_p,Q_p)$. Under mirror symmetry, the generators $[g_1]+[g_2]$ and $2[g_2]$ of the lattice in \eqref{eq:lattice-2} correspond to $[dp]$ and $[-dq]$ (compare Lemma \ref{th:g-pull}), hence:

\begin{corollary} \label{th:yes-periodic}
Any element in the image of $H^1(E^{\mathit{triv}};\bZ) \iso \bZ^2 \rightarrow \mathit{HH}^1(\mathit{Fuk}(E^{\mathit{triv}}),\mathit{Fuk}(E^{\mathit{triv}}))$ is periodic (for the same elliptic curve as in Corollary \ref{th:not-periodic}). \qed
\end{corollary}

As a consequence, we see that if $f$ satisfies Assumption \ref{th:both-signs}, then $E$ is not symplectically isomorphic to $E^{\mathit{triv}}$. Of course, this is by no means the most direct argument available (see the Introduction), but it has the advantage of belonging to the general framework of Fukaya categories.

\subsection{An algebraic viewpoint\label{subsec:algebraic-model}}
Let $A$ be a proper $A_\infty$-category over $R$, together with a functor $G: A \rightarrow A$. The naive {\em mapping torus category} $A^{\mathit{torus}}$ is defined as follows. Objects are of the form $X(d)$, where $X$ is an object of $A$ strictly fixed by $G$, meaning that $G(X) = X$, and $d \in \bZ$ an integer. The definition of the morphism space comes from \eqref{eq:t-subcomplex}:
\begin{equation} \label{eq:t-category}
\mathit{hom}_{A^{\mathit{torus}}}(X_0(d_0),X_1(d_1)) = \mathit{hom}_A(X_0,X_1) \otimes F \oplus
\mathit{hom}_A(X_0,X_1) \otimes F[-1],
\end{equation}
where $F$ is as in \eqref{eq:tate}, and the tensor product is over $R$. It may be more intuitive to (arbitrarily) choose a basis and write $\mathit{hom}_A(X_0,X_1) = C(X_0,X_1) \otimes_\bC R$. Then, elements of $\mathit{hom}_A(X_0,X_1) \otimes F$ can be thought of as series $a(t) = c_0 \hbar^{m_0} t^{n_0} + \cdots$, with the same convergence condition as in \eqref{eq:tate}, but coefficients $c_k \in C(X_0,X_1)$.

\begin{remark}
Because of the definition as a tensor product, we have the additional condition that for any $a(t)$, the coefficients $c_k$ which occur may span only a finite-dimensional subspace of $C(X_0,X_1)$. This is somewhat unnatural in terms of the topological nature of the ring $F$. However, if $A$ is strictly proper (has finite-dimensional morphism spaces), this point is obviously irrelevant, and of course any proper $A_\infty$-category is quasi-isomorphic to a strictly proper one.
\end{remark}

Elements of \eqref{eq:t-category} can be written as pairs $(a(t),b(t))$, where $|b(t)| = |a(t)| - 1$. The differential is
\begin{equation} \label{eq:mu1-torus}
\mu^1_{A^{\mathit{torus}}}(a(t), b(t)) = \big(\mu^1_A(a(t)), \mu^1_A(b(t)) + (-1)^{|a|-1} a(t) + (-1)^{|a|} t^{d_1-d_0} G^1(a(\hbar t)) \big).
\end{equation}

\begin{example} \label{th:multicones}
Consider a single object $X$ fixed by $G$. There is an obvious long exact sequence
\begin{equation} \label{eq:t-sequence}
\cdots H(\mathit{hom}_{A^{\mathit{torus}}}(X(d),X(d))) \rightarrow H(\mathit{hom}_A(X,X)) \otimes F \xrightarrow{\mathit{id} - H(G^1) \otimes T} H(\mathit{hom}_A(X,X)) \otimes F \cdots
\end{equation}
where $T$ is as in \eqref{eq:translation-action}. If we restrict the second map in \eqref{eq:t-sequence} to series in $t$ with vanishing constant term (in $t$), it is actually an isomorphism. Hence, we have the simpler long exact sequence
\begin{equation}
\cdots H(\mathit{hom}_{A^{\mathit{torus}}}(X(d),X(d))) \rightarrow H(\mathit{hom}_A(X,X)) \xrightarrow{\mathit{id}-H(G^1)} H(\mathit{hom}_A(X,X)) \cdots
\end{equation}
\end{example}

The composition of $(a_k(t),b_k(t)) \in \mathit{hom}_{A^{\mathit{torus}}}(X_{k-1}(d_{k-1}),X_k(d_k))$ ($k = 1,2$) is given by
\begin{equation}
\begin{aligned}
& \mu^2_{A^{\mathit{torus}}}((a_2(t),b_2(t)), (a_1(t),b_1(t))) = \big(\mu^2_A(a_2(t),a_1(t)),
(-1)^{|a_2|-1} \mu^2_A(a_2(t),b_1(t)) \\ & \qquad + \mu^2_A(b_2(t),t^{d_1-d_0} G^1(a_1(\hbar t))) +
(-1)^{|a_2|+|a_1|} t^{d_2-d_0} G^2(a_2(\hbar t),a_1(\hbar t)) \big); \\
\end{aligned}
\end{equation}
and similarly for the higher order structure maps.

\begin{lemma} \label{th:t-model}
Suppose that $X \in \mathit{Ob}\, A$ is fixed by $G$. Let $Q_p$ be the $A_\infty$-category associated to the unit torus polynomial. Then there is an $A_\infty$-functor $Q_p \rightarrow A^{\mathit{torus}}$ which maps the two objects of $Q_p$ to $X(0)$ and $X(2)$, respectively.
\end{lemma}

\begin{proof}
After replacing our original category and functor by quasi-isomorphic ones, one can assume that both are strictly unital (with the functor still acting in the same way on objects). Think of $R$ itself as an $A_\infty$-category with a single object $Z$. The embedding $R \rightarrow A$ mapping $Z$ to $X$ induces one $R^{\mathit{torus}} \rightarrow A^{\mathit{torus}}$.
Consider the full subcategory of $R^{\mathit{torus}}$ with objects $Z(0)$, $Z(2)$. This is a dg model (actually the one mentioned in Section \ref{subsec:nonarchimedean}) for the full subcategory of the derived category of modules over $F \rtimes \bZ$ with objects $F(0)$, $F(2)$. Lemma \ref{th:nonarch} then completes the proof.
\end{proof}

The proposed correspondence with geometry goes as follows. If $A$ is the Fukaya category of some compact manifold, and $G$ is given by the action of a (graded) symplectic automorphism, then $A^{\mathit{torus}}$ should conjecturally be quasi-isomorphic to a full subcategory of the Fukaya category of the associated symplectic mapping torus. Objects $X(d)$ correspond to Lagrangian submanifolds in the mapping torus obtained by taking an invariant Lagrangian submanifold in the fibre and moving it along a line in the base which goes through $(0,0)$ and has slope $-d$. Looking back to our previous discussion, $\Delta_1$ and $\Delta_2$ from \eqref{eq:lagrangians} as well as $\Gamma_1$ and $\Gamma_2$ from \eqref{eq:lagrangians-2} are of this type. Example \ref{th:multicones} shows that the algebraic model correctly computes the self-Floer cohomology. Lemma \ref{th:t-model} would be the algebraic counterpart of the Fukaya category computations in Section \ref{subsec:flat-triangles}.

\begin{remark}
The framework introduced above is naive, since it asks for strictly fixed objects. To make it more flexible, one could consider $A_\infty$-modules over $A \otimes F$ which are equivariant with respect to $G \otimes T$ (this would also allow one to include objects corresponding to Lagrangian submanifolds such as $\Delta_{3,u}$ and $\Gamma_{3,u}$). 
\end{remark}

\section{Blowing up\label{sec:blowup}}

The topic of this section, namely the behaviour of Fukaya categories under blowups, is of interest from many perspectives, among which our intended application plays only a minor role. Interested readers are referred to \cite{smith10}, from which we have stolen as much as we could (concretely, Sections \ref{subsec:toy}--\ref{subsec:correspondence} follow \cite[Section 4.5]{smith10} closely). On a technical level, we will freely use and combine a wide range of results, notably: the full-fledged construction of Fukaya categories \cite{fooo}, and split-generation results in that context \cite{abouzaid-fukaya-oh-ohta-ono11}; degeneration techniques \cite{ionel-parker04,li-ruan98}; Lagrangian correspondences \cite{wehrheim-woodward06,mww}; and the $h$-principle \cite{gromov86}. Necessarily, the exposition can't be self-contained to any extent.

Generally speaking, the passage to $\bZ/2$-graded Fukaya categories and the introduction of bounding cochains allows us to include many more objects than in the approach from Section \ref{sec:automorphisms} (see Remark \ref{th:compare} for a precise statement of the relationship). There will be a temporary departure from this framework in Section \ref{subsec:toy}, when we consider the toy model of blowing up a point in flat space (which happens to be monotone, allowing us to retreat to a simpler version of Floer theory).

\subsection{Fukaya categories}
Fix a closed symplectic manifold $M$. Let $R_{\geq 0} \subset R$ be the subalgebra of formal series involving only nonnegative powers of $\hbar$. This comes with a homomorphism $R_{\geq 0} \rightarrow \bC$ extracting the constant term, and we write $R_{>0}$ for its kernel. For any $\lambda \in R_{>0}$ there is an associated Fukaya category $\mathit{Fuk}(M)_{\lambda}$, which is a proper $\bZ/2$-graded $A_\infty$-category. We will give an impressionistic sketch of the construction, which is due to \cite{fooo} (see \cite{fukaya-oh06,fukaya11} for more thorough expository accounts).

One first associates to $M$ a {\em filtered curved $A_\infty$-category} $\mathit{FO}(M)$. Objects of $\mathit{FO}(M)$ are Lagrangian submanifolds $L \subset M$ equipped with a {\em Spin} structure and a local coefficient system $\xi$ with structure group $\mathit{GL}_r(\bC)$, for some $r$. The morphism space between any two objects is a finitely generated free $\bZ/2$-graded module over $R_{\geq 0}$. If we consider a single object, then
\begin{equation} \label{eq:fuk-endo}
\mathit{hom}_{\mathit{FO}(M)}(L,L) \otimes_{R_{\geq 0}} \bC
\end{equation}
is an $A_\infty$-algebra (the curvature vanishes since it has no $\hbar^0$ term) quasi-isomorphic to the standard one underlying the cohomology with local coefficients $H^*(L; \mathit{Hom}(\xi,\xi))$. For simplicity we will assume that $\mathit{FO}(M)$ is strictly unital.

\begin{remark}
The finite-dimensionality of morphism spaces is convenient for expository reasons, since it allows one to worry less about convergence and completeness (in the $\hbar$-adic topology). A Morse theory model as in Section \ref{subsec:clean} naturally yields finite-dimensional morphism spaces. On the other hand, one can start with an infinite-dimensional space of cochains (like the singular cochains used in \cite{fooo}) and then obtain finite-dimensional models a posteriori by applying a version of the Homological Perturbation Lemma \cite[Theorem W]{fooo}. Strict units are not an a priori feature of either approach, but can be added by first introducing a homotopy unit through additional moduli spaces \cite[Section 7.3]{fooo}, and then constructing a strict unit from that \cite[Section 3.3]{fooo}. We should point out that from a more abstract viewpoint, unitality is not really the crucial ingredient (see Remark \ref{th:general-hh} below).
\end{remark}

Objects of $\mathit{Fuk}(M)_{\lambda}$ are {\em weakly unobstructed} Lagrangian submanifolds. By this we mean objects of $\mathit{FO}(M)$ together with a {\em bounding cochain} $\alpha \in \mathit{hom}_{\mathit{FO}(M)}^1(L,L)$, which vanishes if we tensor with $\bC$, and which satisfies the following inhomogeneous Maurer-Cartan equation:
\begin{equation} \label{eq:inhomogeneous-mc}
\mu^0_{\mathit{FO}(M)} + \mu^1_{\mathit{FO}(M)}(\alpha) + \mu^2_{\mathit{FO(M)}}(\alpha,\alpha) + \cdots = \lambda e_L \in \mathit{hom}^0_{\mathit{FO}(M)}(L,L)
\end{equation}
(in the terminology of \cite[Section 3.6]{fooo}, this would be a ``weak bounding cochain''). The morphism spaces, also called Floer cochain groups following the traditional terminology, are defined by
\begin{equation} \label{eq:cf-corrected}
\mathit{CF}^*(L_0,L_1) = \mathit{hom}_{\mathit{Fuk}(M)_\lambda}(L_0,L_1) = \mathit{hom}_{\mathit{FO}(M)}(L_0,L_1) \otimes_{R_{\geq 0}} R.
\end{equation}
The $A_\infty$-structure is obtained by deforming that on $\mathit{FO}(M)$, as in the construction of twisted complexes. For instance, the differential on \eqref{eq:cf-corrected} is
\begin{equation} \label{eq:deformed-mu1}
\begin{aligned}
\mu^1_{\mathit{Fuk}(M)_{\lambda}}(a) & = \mu^1_{\mathit{FO}(M)}(a) + \mu^2_{\mathit{FO}(M)}(\alpha_1,a) + \mu^2_{\mathit{FO}(M)}(a,\alpha_0) + \mu^3_{\mathit{FO}(M)}(\alpha_1,\alpha_1,a) \\ & \qquad + \mu^3_{\mathit{FO}(M)}(\alpha_1,a,\alpha_0) + \mu^3_{\mathit{FO}(M)}(a,\alpha_0,\alpha_0) + \cdots
\end{aligned}
\end{equation}

\begin{remark} \label{th:compare}
Consider the situation where $c_1(M) = 0$ and $\lambda = 0$. Let $L$ be a Lagrangian submanifold satisfying Assumption \ref{th:adiscic}. Then the $A_\infty$-structure on $\mathit{hom}_{\mathit{FO}(M)}(L,L)$ is a trivial deformation of that on \eqref{eq:fuk-endo} (the technical details of this are subtle, but our applications only really involve the simpler case when $J_L$ has no holomorphic spheres or discs). If we restrict  \eqref{eq:inhomogeneous-mc} to $\alpha$ which have degree $1$, it reduces to a version of the Maurer-Cartan equation governing the deformation theory of the local coefficient system $\xi$ \cite{goldman-millson88}. Hence, any deformation of $\xi$ to a local coefficient system with structure group \eqref{eq:gl-subgroup} gives rise to a solution (unique up to gauge equivalence), hence produces an object of $\mathit{Fuk}(M)_0$. These observations have the following noteworthy consequence. Let $\mathit{Fuk}(M)$ be the Fukaya category according to the more restricted definition used in Section \ref{sec:automorphisms}. Then, after reducing the grading of $\mathit{Fuk}(M)$ to $\bZ/2$, there is a cohomologically full and faithful $A_\infty$-functor $\mathit{Fuk}(M) \longrightarrow \mathit{Fuk}(M)_0$.
\end{remark}

Even though $\mathit{FO}(M)$ has curvature, its Hochschild cohomology is still well-defined as usual, and we have a canonical open-closed string map
\begin{equation} \label{eq:pre-open-closed}
H^*(M;R_{\geq 0}) \longrightarrow \mathit{HH}^*(\mathit{FO}(M),\mathit{FO}(M)),
\end{equation}
From this one derives, in a way similar to \eqref{eq:twisted-gamma}, maps
\begin{equation} \label{eq:lambda-open-closed}
\mathit{QH}^*(M) \longrightarrow \mathit{HH}^*(\mathit{Fuk}(M)_{\lambda}, \mathit{Fuk}(M)_{\lambda}).
\end{equation}
Concretely, fix some class in $H^*(M;R_{\geq 0})$, and let $g_{\mathit{FO}} \in \mathit{CC}^*(\mathit{FO}(M),\mathit{FO}(M))$ be a Hochschild cocycle representing its image under \eqref{eq:pre-open-closed} (this cochain is defined by a generalization of the procedure sketched in Section \ref{subsec:define-open-closed}). Then, the image of the same class under \eqref{eq:lambda-open-closed} is represented by a cochain $g$ whose constant term is
\begin{equation}
g^0 = g^0_{\mathit{FO}} + g^1_{\mathit{FO}}(\alpha) + g^2_{\mathit{FO}}(\alpha,\alpha) + \cdots \in CF^*(L,L).
\end{equation}

\begin{remark} \label{th:general-hh}
For \eqref{eq:deformed-mu1} to square to zero, we do not really need the fact that the left hand side of \eqref{eq:inhomogeneous-mc} is a multiple of the unit, but only that it is strictly central. More generally, one can associate a Fukaya category to any $\lambda \in H^{\mathit{even}}(M;R^{>0})$ (called ``bulk deformations'' in \cite[Section 3.8]{fooo}; the special case of $H^0$ corresponds to the previously discussed construction). We will not pursue this further.
\end{remark}

\begin{remark} \label{th:auroux}
There are partial results about what kinds of Lagrangian submanifolds can occur as objects of $\mathit{Fuk}(M)_{\lambda}$ for different values of $\lambda$. Let's temporarily restrict to a symplectic manifold $M$ which is monotone (and monotone Lagrangian submanifolds $L$, with trivial bounding cochains $\alpha = 0$). This allows one to work over $\bC$ instead of $R$ (and therefore to take $\lambda \in \bC$). In this context, Auroux, Kontsevich and the author \cite[Theorem 6.1]{auroux07} showed that $c_1(M)-\lambda \cdot 1 \in \mathit{QH}^*(M) = H^*(M;\bC)$ maps to zero in $\mathit{HF}^*(L,L)$ under the open-closed string map. Therefore, 
\begin{equation} \label{eq:aks}
\mathit{Fuk}(M)_\lambda = 0 \quad \text{unless $\lambda$ is an eigenvalue of quantum multiplication with $c_1(M)$.} 
\end{equation}
There is a consequence of this, which is weaker but of independent interest. Take the dual open-closed string map $\mathit{HF}^*(L,L) \longrightarrow \mathit{QH}^{*+n}(M)$, which is such that the composition
\begin{equation}
\mathit{QH}^*(M) \longrightarrow \mathit{HF}^*(L,L) \longrightarrow \mathit{QH}^{*+n}(M)
\end{equation}
is small quantum multiplication with the Poincar{\'e} dual class $\mathit{PD}([L]) \in \mathit{QH}^n(M)$. The previous statement implies that
\begin{equation} \label{eq:c1-sum}
\sum_A \langle c_1(M), \mathit{PD}([L]), x \rangle_A^M = \lambda \int_M \mathit{PD}([L]) \cup x
\end{equation}
for all $x$ (here, $\langle x_1,\dots, x_n \rangle_A^M$ is our notation for the $n$-point genus zero Gromov-Witten invariant: $A \in H_2(M;\bZ)$, and the $x_i$ are cohomology classes). One can separate the left hand side of \eqref{eq:c1-sum} into pieces which each sum over classes $A$ with $c_1(M)(A) = k$ a fixed integer. Then, for degree reasons, only the $k = 1$ summand can be nontrivial. Using the divisor axiom and the fact that $c_1(M) \cup \mathit{PD}([L]) = 0$, one can therefore rewrite \eqref{eq:c1-sum} as
\begin{equation} \label{eq:phi-1}
\sum_A \langle \mathit{PD}([L]), x \rangle_A^M = \lambda \int_M \mathit{PD}([L]) \cup x.
\end{equation}
Equivalently, the two-point Gromov-Witten invariants define an endomorphism $\Phi$ of $\mathit{QH}^*(M)$ by Poincar{\'e} duality, and then \eqref{eq:phi-1} says that
\begin{equation} \label{eq:poincare-dual-ev}
\Phi(\mathit{PD}([L])) = \lambda\, \mathit{PD}([L]).
\end{equation}
Note that \eqref{eq:quantum-phi}, unlike \eqref{eq:aks}, imposes a restriction on $\lambda$ only if $L$ is nontrivial in homology. This limitation can't be removed: even in a simple example such as $M = \bC P^1 \times \bC P^1$, the map $\Phi$ is degree-decreasing and hence nilpotent, while the quantum product with $c_1(M)$ has eigenvalues $\{\pm 4,0\}$. Indeed, $\mathit{Fuk}(M)_0$ contains the antidiagonal, whose homology class is nonzero, but $\mathit{Fuk}(M)_{\pm 4}$ are also nontrivial, and contain nullhomologous Lagrangian tori.

It should be mentioned that \eqref{eq:poincare-dual-ev} is known to generalize. Namely, given a general (not necessarily monotone) $M$, consider the endomorphism $\Phi$ of $\mathit{QH}^*(M) = H^*(M;R)$ defined by
\begin{equation} \label{eq:quantum-phi}
\int_M \Phi(x_1) \cup x_2 = \sum_A \hbar^{\omega_M(A)} \langle x_1,x_2 \rangle_A^M.
\end{equation}
Then, a special case of \cite[Theorem 3.8.11]{fooo} says that if $L$ is an object of $\mathit{Fuk}(M)_\lambda$, the analogue of \eqref{eq:poincare-dual-ev} holds for \eqref{eq:quantum-phi}. However, the corresponding generalisation of \eqref{eq:aks} is unknown. Rather than trying to address that question, we'll allow arbitrary $\lambda$ and then formally cut down the resulting category.
\end{remark}

\subsection{Projections}
Let $A$ be an $A_\infty$-category ($\bZ/2$-graded and cohomologically unital), and $Q_+ \in \mathit{HH}^0(A,A)$ a Hochschild cohomology class which is idempotent with respect to the natural ring structure. We want to {\em project} the category $A$ accordingly, which will give rise to a new category $A_+$. The simplest way to go about that is as follows. $Q_+$ determines, for any $X \in \mathit{Ob}\,A$, an idempotent endomorphism $Q_+^0 \in H^0(\mathit{hom}_A(X,X))$. After lifting that to a homotopy idempotent $p_+$ (compare Section \ref{subsec:splittings}), one gets a (perfect) module $(X,p_+)^{\mathit{yon}}$. Choose one such $p_+$ for each $X$, and define $A_+$ to be the full $A_\infty$-subcategory of the module category of $A$ with objects $(X,p_+)^{\mathit{yon}}$.

Here is another approach, which turns out to be equivalent but yields additional properties. Let $Q_-$ be the complementary idempotent. Fix modules $(X,p_+)^{\mathit{yon}}$ and their complementary counterparts $(X,p_-)^{\mathit{yon}}$, and consider those as well as the standard Yoneda images $X^{\mathit{yon}}$. Let $A_{\pm}$ be the full subcategory of the module category of $A$ containing all those objects, and then pass to the quotient category $A_{\pm}/A_-$ in the sense of \cite{drinfeld02}, in which all the objects $(X,p_-)^{\mathit{yon}}$ become quasi-isomorphic to zero. This comes with a canonical quotient functor $A_{\pm} \rightarrow A_{\pm}/A_-$.

\begin{lemma} \label{th:projection-as-quotient}
The quotient functor restricts to a quasi-equivalence
\begin{equation} \label{eq:ppm}
A_+ \longrightarrow A_{\pm}/A_-.
\end{equation}
\end{lemma}

\begin{proof}
Since $X^{\mathit{yon}} \iso (X,p_+)^{\mathit{yon}} \oplus (X,p_-)^{\mathit{yon}}$ by construction, the objects $X^{\mathit{yon}}$ and $X^{\mathit{yon}}_+$ become quasi-isomorphic in the quotient, so \eqref{eq:ppm} is essentially onto. Note that for any two objects $X_0,X_1$ we have
\begin{equation}
\begin{aligned}
& H(\mathit{hom}_{A_\pm}((X_0,p_{0,+})^{\mathit{yon}},(X_1,p_{1,-})^{\mathit{yon}})) = 0, \\
& H(\mathit{hom}_{A_\pm}((X_0,p_{0,-})^{\mathit{yon}},(X_1,p_{1,+})^{\mathit{yon}})) = 0.
\end{aligned}
\end{equation}
By general nonsense \cite[Theorem 1.6.2(ii)]{drinfeld02}, this implies that \eqref{eq:ppm} is cohomologically full and faithful.
\end{proof}

So far, it may not have been evident why we have included the Yoneda images themselves in $A_\pm$. The point is that we can combine the Yoneda embedding and the quotient functor to get a canonical functor
$A \rightarrow A_\pm/A_-$. From the proof of Lemma \ref{th:projection-as-quotient} it follows that on the cohomological level, this is projection to the part of the morphism space singled out by $Q_+$:
\begin{equation}
H(\mathit{hom}_{A_{\pm}/A_-}(X_0^{\mathit{yon}},X_1^{\mathit{yon}})) \iso Q_+^0 H(\mathit{hom}_A(X_0,X_1)) =
H(\mathit{hom}_A(X_0,X_1)) Q_+^0.
\end{equation}
Here is an even simpler way to describe the resulting situation. Let $\tilde{A}$ be the category with the same objects as $A$, but where
\begin{equation}
\mathit{hom}_{\tilde{A}}(X_0,X_1) = \mathit{hom}_{A_\pm/A_-}(X_0^{\mathit{yon}},X_1^{\mathit{yon}}) \oplus \mathit{hom}_{A_\pm/A_+}(X_0^{\mathit{yon}},X_1^{\mathit{yon}}),
\end{equation}
with the obvious choice of $A_\infty$-structure. Then, the functor $A \rightarrow A_\pm/A_-$ and its analogue for the complementary quotient combine to yield a quasi-isomorphism $A \rightarrow \tilde{A}$. If we allow ourselves to replace $A$ by $\tilde{A}$, the situation is that we have a category whose morphism spaces are split into two parts compatibly with all compositions, and the projection just throws away one of those parts. In particular, there are canonical isomorphism
\begin{equation}
\mathit{HH}^*(A,A) \iso \mathit{HH}^*(\tilde{A},\tilde{A}) \iso
\mathit{HH}^*(A_+,A_+) \oplus \mathit{HH}^*(A_-,A_-).
\end{equation}

%

We will now apply this to Fukaya categories, with modified notation. Take an idempotent $q \in \mathit{QH}^0(M)$. Its image under \eqref{eq:lambda-open-closed} is an idempotent in Hochschild cohomology, and we denote the outcome of the resulting projection by $\mathit{Fuk}(M)_{\lambda,q}$. The following result is a modified version of Theorem \ref{th:split-generation-2}, and can be proved by adapting arguments from \cite{abouzaid-fukaya-oh-ohta-ono11} (for a more detailed exposition in the monotone case, see \cite[Corollary 3.7]{sheridan13}).

\begin{theorem} \label{th:split-generation-3}
Let $O \subset \mathit{Fuk}(M)_{\lambda,q}$ be a full $A_\infty$-subcategory. Suppose that there is a  linear map $\mathit{HH}^0(O,O) \rightarrow R$ such that the following diagram commutes:
\begin{equation}
\xymatrix{\mathit{QH}^0(M)q \ar[rr] \ar[dr]_{\int_M} && \mathit{HH}^0(O,O) \ar[dl] \\
& R &}
\end{equation}
Then the objects in $O$ split-generate $\mathit{Fuk}(M)^{\mathit{perf}}_{\lambda,q}$. \qed
\end{theorem}

\subsection{A toy model\label{subsec:toy}}
Consider $\bC P^r$ blown up at a point as a toric symplectic (and K{\"a}hler) manifold. We denote this by $\bar{B}^{\mathit{toy}}$. It is characterized symplectically by its moment polytope, which is
\begin{equation}
\big\{ t \in \bR^r \,:\, t_1, \dots, t_r \geq 0, \; \delta \leq t_1 + \cdots + t_r \leq \epsilon \big\} \label{eq:polytope}
\end{equation}
for some $0 < \delta < \epsilon$. The preimage of $\{t_1 + \cdots + t_r = \epsilon\}$ under the moment map is the hyperplane at infinity, denoted by $H^{\mathit{toy}}$. The preimage of $\{t_1 + \cdots + t_r = \delta\}$ is the exceptional divisor, denoted by $D^{\mathit{toy}}$. In particular,
\begin{equation} \label{eq:c1-toy-compactified}
c_1(\bar{B}^{\mathit{toy}}) = -(r-1)\mathit{PD}([D^{\mathit{toy}}]) + (r+1)\mathit{PD}([H^{\mathit{toy}}]).
\end{equation}
Write $B^{\mathit{toy}} = \bar{B}^{\mathit{toy}} \setminus H^{\mathit{toy}}$. This is non-compact, but because the divisor we remove has positive normal bundle, doing pseudo-holomorphic curve theory in its complement is unproblematic. We denote by $Z \in H_2(B^{\mathit{toy}})$ the class of a line lying in $D^{\mathit{toy}}$, and by $u = -\mathit{PD}([D^{\mathit{toy}}]) \in H^2_{\mathit{cpt}}(B^{\mathit{toy}})$ the negative Poincar{\'e} dual of the exceptional divisor. With this sign convention,
\begin{align}
& u(Z) = 1, \label{eq:negative-pd} \\
& c_1(B_{\mathit{toy}}) = (r-1) u, \label{eq:c1-toy} \\
& [\omega_{B^{\mathit{toy}}}] = 2\pi\delta \, u, \label{eq:omega-toy} \\
& u^r = -\mathit{PD}([\mathit{point}]), \label{eq:toy-gitler}
\end{align}
where in \eqref{eq:c1-toy} and \eqref{eq:omega-toy} we've mapped $u$ to $H^2(B^{\mathit{toy}})$. It follows that $B^{\mathit{toy}}$ is monotone:
\begin{equation} \label{eq:monotonicity}
c_1(B^{\mathit{toy}}) = \textstyle\frac{(r-1)}{2\pi\delta} [\omega_{B^{\mathit{toy}}}] \in H^2(B^{\mathit{toy}};\bR).
\end{equation}
We will take advantage of the resulting possible simplification and omit the formal parameter $\hbar$, which means working with quantum cohomology and Floer theory over $\bC$.

Let's first look at the small quantum product. $B^{\mathit{toy}}$ is a line bundle over $D^{\mathit{toy}}$, and the fibres are Poincar\'e dual to $u^{r-1} \in H^{2r-2}(B^{\mathit{toy}})$, as one can see from \eqref{eq:toy-gitler}. The contribution of lines lying inside the exceptional divisor is
\begin{equation}
\begin{aligned}
\textstyle \int_{B^{\mathit{toy}}} (u^{r-1} \ast_Z u) \, u^{r-1} & = \langle u^{r-1}, u, u^{r-1} \rangle_{Z}^{B^{\mathit{toy}}} \\ & = \langle u^{r-1}, u^{r-1} \rangle_{Z}^{B^{\mathit{toy}}} = 1.
\end{aligned}
\end{equation}
Comparison with \eqref{eq:toy-gitler} yields
\begin{equation} \label{eq:toy-quantum-gitler}
u^{r-1} \ast_Z u = - u.
\end{equation}
In principle, the multiples $dZ$, $d>1$, could also contribute to the quantum product. But the virtual dimension of the associated moduli space of three-pointed holomorphic spheres is $2r + (2r-2) d = (2+2d)r - 2d$, while the image of the evaluation map is contained in a subspace $(\bC P^{r-1})^3$ of smaller dimension $6r-6$. Hence the virtual fundamental cycle maps to zero under evaluation.

Suppose from now on that $\epsilon > \frac{r}{r-1}\delta$, and consider the Lagrangian torus $C^{\mathit{toy}} \subset B^{\mathit{toy}}$ which is the fibre of the moment map over the point $(\delta/(r-1),\dots,\delta/(r-1))$ in \eqref{eq:polytope}. For another description, recall that $B^{\mathit{toy}} \setminus D^{\mathit{toy}}$ is $U(1)^r$-equivariantly symplectically isomorphic to the open subset
\begin{equation} \label{eq:annulus}
\{ z \in \bC^r \,:\, \delta < \half \| z\|^2 < \epsilon \}
\end{equation}
with the standard (constant) symplectic form. In this isomorphism, $C^{\mathit{toy}}$ corresponds to the Clifford torus with all radii equal to $\sqrt{2\delta/(r-1)}$. We can use areas of discs in \eqref{eq:annulus} to show that $C^{\mathit{toy}}$ is monotone as well, meaning that the analogue of \eqref{eq:monotonicity} holds in $H^2(B^{\mathit{toy}},C^{\mathit{toy}};\bR)$.

Following \cite{cho-oh02} (see \cite{hori-vafa00} for physics motivation, \cite[\S 4]{auroux07} for an exposition, and \cite{fukaya-oh-ohta-ono10} for generalizations), it is convenient to formulate results about the Floer cohomology of $C^{\mathit{toy}}$ in terms of the superpotential
\begin{equation}
\begin{aligned}
& W: H^1(C^{\mathit{toy}};\bC^*) \longrightarrow \bC, \\
& z \longmapsto \sum_A n_A \, z^{\partial A}.
\end{aligned}
\end{equation}
Here, the sum is over $A \in H_2(B^{\mathit{toy}}, C^{\mathit{toy}})$; $n_A \in \bZ$ is the number of pseudo-holomorphic discs in class $A$ going through a generic point of $C^{\mathit{toy}}$; and $z^{\partial A} \in \bC^*$ is the pairing of the cohomology class $z$ with the boundary homology class $\partial A \in H_1(C^{\mathit{toy}})$. To fix the signs, we equip $C^{\mathit{toy}}$ with the trivial {\em Spin} structure (the one compatible with the rotation-invariant framing; or equivalently, the unique one which is invariant under the action of $\mathit{SL}_r(\bZ)$). The domain of $W$ should be thought of as the moduli space of flat $\bC^*$-bundles on $C^{\mathit{toy}}$, and the superpotential is obtained by counting discs weighted with their boundary holonomy. Equip $C^{\mathit{toy}}$ with the bundle $\xi$ corresponding to some point $z$, and consider the spectral sequence \cite{oh96} going from $H^*(C^{\mathit{toy}};\bC)$ to $\mathit{HF}^*(C^\mathit{toy}, C^\mathit{toy})$. Part of the differential of this sequence is the map
\begin{equation} \label{eq:d1-oh}
\begin{aligned}
& H^1(C^{\mathit{toy}};\bC) \longrightarrow H^0(C^{\mathit{toy}};\bC) = \bC, \\
& w \longmapsto dW_z(zw) = \sum_A n_A \, w(\partial A) z^{\partial A}.
\end{aligned}
\end{equation}
Because the spectral sequence is multiplicative \cite{buhovsky10}, there are only two possible outcomes. Either \eqref{eq:d1-oh} is nonzero, in which case the Floer cohomology vanishes. Or else it is zero, in which case the spectral sequence degenerates. From now on let's focus exclusively on the second case, which happens exactly when $z$ is a critical point of $W$. In that situation, we have an isomorphism $H^*(C^{\mathit{toy}};\bC) \iso \mathit{HF}^*(C^{\mathit{toy}},C^{\mathit{toy}})$, which is canonical up to composition with automorphisms of $H^*(C^{\mathit{toy}};\bC)$ of the form $\mathit{Id} + R$, where $R$ decreases degrees by at least $2$. In particular, the degree $0$ and $1$ parts
\begin{align}
\label{eq:h0-inclusion}
& H^0(C^{\mathit{toy}};\bC) \iso \mathit{HF}^0(C^{\mathit{toy}},C^{\mathit{toy}}), \\
\label{eq:h1-inclusion}
& H^1(C^{\mathit{toy}};\bC) \iso \mathit{HF}^1(C^{\mathit{toy}},C^{\mathit{toy}}),
\end{align}
are strictly canonical. The multiplicative nature of the spectral sequence ensures that \eqref{eq:h0-inclusion} yields the unit $e$ in Floer cohomology, and \eqref{eq:h1-inclusion} generates Floer cohomology as a ring. The relations between these generators are determined by the Hessian of $W$ at $z$:
\begin{equation} \label{eq:mult-relation}
w \cdot w = (D^2W)_z(zw,zw)\, e = \sum_A n_A w(\partial A)^2 z^{\partial A}\, e.
\end{equation}
Being a toric fibre, $C^{\mathit{toy}}$ is contractible in $B^{\mathit{toy}}$, hence the restriction map in ordinary cohomology vanishes. However, the specialization of the open-closed string map $\mathit{H}^*(B^{\mathit{toy}};\bC) \longrightarrow \mathit{HF}^*(C^{\mathit{toy}},C^{\mathit{toy}})$ has quantum corrections \cite{fukaya-oh-ohta-ono11}. In degree $2$, the outcome is as follows:
\begin{equation} \label{eq:toy-open-closed}
\begin{aligned}
& H^2(B^{\mathit{toy}};\bC) \longrightarrow \mathit{HF}^0(C^{\mathit{toy}},C^{\mathit{toy}}), \\
& v \longmapsto \sum_A n_A \, v(A) z^{\partial A} e,
\end{aligned}
\end{equation}
where we have chosen an arbitrary lift of $v$ to $H^2(B^{\mathit{toy}},C^{\mathit{toy}})$ to define the pairings $v(A)$ (vanishing of \eqref{eq:d1-oh} ensures that the choice of lift doesn't matter).

It remains to spell out the data in our particular case, again following \cite{cho-oh02}. Thinking of $C^{\mathit{toy}}$ as a torus orbit yields an identification $H^1(C^{\mathit{toy}}) \iso \bZ^r$. The numbers $n_A$ are $+1$ for unit vectors $A = (0,\dots,1,\dots,0)$ as well as for $A = (1,\dots,1)$, and vanish otherwise, yielding the Hori-Vafa mirror superpotential
\begin{equation}
W(z_1,\dots,z_r) = z_1 + \cdots + z_r + z_1 \cdots z_r.
\end{equation}
There are $(r-1)$ critical points $(z_1,\dots,z_r) = (\lambda,\dots,\lambda)$, where $\lambda^{r-1} = -1$. Each of them is nondegenerate, which in view of \eqref{eq:mult-relation} means that the Floer cohomology rings are complex Clifford algebras. More precisely, using the generators \eqref{eq:h1-inclusion} one gets a canonical isomorphism
\begin{equation} \label{eq:toy-clifford}
\bC \langle w_1,\dots,w_r \rangle/ \{w_i^2 \text{ for all $i$},\, w_iw_j+w_jw_i+2\lambda \text{ for all $i \neq j$}\} \longrightarrow \mathit{HF}^*(C^{\mathit{toy}},C^{\mathit{toy}}),
\end{equation}
where $\bC\langle w_1,\dots,w_r \rangle$ is the free $\bZ/2$-graded algebra with the $w_i$ as odd
generators. Finally, the open-closed string map \eqref{eq:toy-open-closed} is given by
\begin{equation}
u \longmapsto - z_1 \cdots z_r = \lambda \, e.
\end{equation}
One checks that this is compatible with the ring structure \eqref{eq:toy-quantum-gitler}. It also agrees with the general statement from \cite[Theorem 6.1]{auroux07} (see also Remark \ref{th:auroux}), which says that the critical value of the superpotential, in our case $(r-1)\lambda$, must be an eigenvalue of quantum multiplication with the first Chern class.

\begin{addendum} \label{th:even-r}
Suppose that $r$ is even. Then the Clifford algebra \eqref{eq:toy-clifford} can be (non-canonically) thought of as the total endomorphism algebra of a $\bZ/2$-graded category with $2^{r/2}$ objects, any two of which are isomorphic up to a shift. It is convenient to first diagonalize the underlying quadratic form by using the modified basis elements
\begin{equation}
\tilde{w}_i = w_i + \textstyle(-\frac{1}{r} + \frac{1}{r\sqrt{1-r}})(w_1 + \cdots + w_r),
\end{equation}
which satisfy $\tilde{w}_i^2 = \lambda$ and $\tilde{w}_i\tilde{w}_j + \tilde{w}_j\tilde{w}_i = 0$ ($i \neq j$). Then the identity endomorphisms of our objects can be taken to be the minimal idempotents
\begin{equation} \label{eq:minimal-idempotent}
p = \textstyle \half(1 \pm \frac{\sqrt{-1}}{\lambda} \tilde{w}_1 \tilde{w}_2) \cdots \half(1 \pm \frac{\sqrt{-1}}{\lambda} \tilde{w}_{r-1} \tilde{w}_r).
\end{equation}
\end{addendum}

\subsection{The blowup}
Let $M$ be a symplectic manifold, and $i: N \hookrightarrow M$ a symplectic submanifold (as usual, both are assumed to be connected) of codimension $2r$. Let $B$ be the result of blowing up that submanifold, with size $\delta > 0$. We will always suppose that $\delta$ is sufficiently small, making the existence and uniqueness of $B$ unproblematic. Write $D \subset B$ for the exceptional divisor, and $u = -\mathit{PD}([D]) \in H^2(B)$ for its negative Poincar{\'e} dual. The pushforward $H^*(D) \rightarrow H^{*+2}(B)$ and the pullback $H^*(M) \rightarrow H^*(B)$ are both injective, their images are disjoint, and together they span $H^*(B)$. Using Leray-Hirsch one then writes
\begin{equation} \label{eq:blowup-cohomology}
H^*(B) \iso H^*(M) \oplus H^*(D)[-2] \iso H^*(M) \oplus u H^*(N) \oplus \cdots \oplus u^{r-1} H^*(N).
\end{equation}
With respect to this decomposition,
\begin{align}
& [\omega_B] = [\omega_M] + 2\pi \delta\, u, \\
& c_1(B) = c_1(M) + (r-1)u. \label{eq:c1}
\end{align}

\begin{assumption} \label{th:few-spheres}
Both $M$ and $N$ have zero first Chern class, and $N$ has trivial (as a symplectic vector bundle) normal bundle. The (real) dimensions satisfy
\begin{align} \label{eq:dimension-ineq}
& \mathrm{dim}(M) \geq 2\, \mathrm{dim}(N), \\
& \mathrm{dim}(M) - \mathrm{dim}(N) \equiv 0 \, \mathrm{mod} \, 4.
\end{align}
$N$ admits a compatible almost complex structure for which there are no non-constant pseudo-holomorphic spheres. For $B$ we require a weaker kind of condition, namely that the Gromov-Witten invariants $\langle x_1,\dots,x_n \rangle_{A}^B$ for genus $0$ curves in any homology class $A \in H_2(M)$ should vanish, unless $A$ is a multiple of the class of a line in the fibre of the projective bundle $D \rightarrow N$.
\end{assumption}

The aim of these restrictions is to simplify (drastically) the analysis of the topology and Gromov-Witten theory of the blowup. The exception is the second part of \eqref{eq:dimension-ineq}, which is there only so that we can apply (a generalization of) Addendum \ref{th:even-r} later on.

%
%

As a consequence of the assumptions, $D$ is diffeomorphic to $\bC P^{r-1} \times N$, and the local structure near it is described by $B^{\mathit{toy}} \times N$. This first of all simplifies the structure of the cohomology ring of $B$ somewhat. Take the given ring structures on $H^*(M)$ and $H^*(N)$, and the restriction map $i^*: H^*(M) \rightarrow H^*(N)$. These define a ring structure on $H^*(M) \oplus u H^*(N)[u]$. To get the ring structure on \eqref{eq:blowup-cohomology} one imposes the relations
\begin{equation} \label{eq:gitler}
u^r v = -i_!(v) \in H^{2r+|v|}(M)
\end{equation}
for $v \in H^*(N)$. This is easily seen by arguing Poincar\'e dually in terms of intersections. Note that since $i^*i_!$ vanishes, multiplying \eqref{eq:gitler} by $u$ yields $u^{r+1} = 0$ (for a discussion of cohomology rings of blowups going beyond the case of trivial normal bundle, see \cite{gitler92, lambrechts-stanley08}).
%

\begin{lemma} \label{th:qh-blowup}
In terms of \eqref{eq:blowup-cohomology}, the (small) quantum product $x \ast y$ can be described as follows. If $x$ or $y$ lie in $H^*(M;R)$, the product agrees with the classical cup product. The same is true if $x \in u^j H^*(N;R)$, $y \in u^k H^*(N;R)$ with $j+k < r$. Finally, we have a modified version of \eqref{eq:gitler}:
\begin{equation} \label{eq:quantum-gitler}
u \ast u^{r-1}v  = - i_!(v) - \hbar^{2\pi \delta} uv.
\end{equation}
\end{lemma}

\begin{proof}[Proof of Lemma \ref{th:qh-blowup}]
Take an almost complex structure which, near $D$, is the product of the standard (toric K{\"a}hler) structure on $B^{\mathit{toy}}$ and a compatible almost complex structure on $N$ with no pseudo-holomorphic spheres in it. 

Let $A$ be a multiple of the class of a line in the fibre of $D \rightarrow N$. Suppose that $A$ is represented by a stable genus zero pseudo-holomorphic curve in $B$, whose irreducible components represent classes $A_1',\dots,A_r', A_1'',\dots,A_s''$. The notation is such that the components representing $A_i'$ are those which lie inside $D$, and therefore inside a fibre of $D \rightarrow N$. Hence, they satisfy $[\omega_M](A_i') = 0$. The remaining components are not contained in $D$, hence satisfy $A_i'' \cdot D = -u(A_i'') > 0$ and $[\omega_M](A_i'') = [\omega_B](A_i'') - 2\pi\delta\, u(A_i'') > 0$. If we assume that $s>0$, this is a contradiction to $[\omega_M](A) = 0$. Hence, when computing Gromov-Witten invariants for a class like $A$ (which are the only nonvanishing ones, by assumption), all the relevant stable pseudo-holomorphic spheres are those contained in fibres of $D \rightarrow N$. But their contribution is as in the previously discussed toy model case \eqref{eq:toy-quantum-gitler}, with the formal parameter $\hbar^{\omega_B(Z)}$ re-inserted.
\end{proof}

Multiplying \eqref{eq:quantum-gitler} by $u$, but this time with respect to the quantum product, yields the counterpart of \eqref{eq:toy-quantum-gitler}:
\begin{equation}
\overbrace{u \ast \cdots \ast u}^{r+1} = u \ast (u \ast u^{r-1}) = - \hbar^{2\pi \delta} u^2 = - \hbar^{ 2\pi \delta} u \ast u.
\end{equation}
Hence, the operator of quantum multiplication with $u$ has eigenvalues $0$ as well as $\lambda$, where
\begin{equation} \label{eq:eigenvalues}
\lambda^{r-1} = -\hbar^{2\pi \delta}.
\end{equation}
The generalized eigenspaces are
\begin{align} \label{eq:eigenspace0}
E_0 & = \mathit{ker}(u \ast u \ast \cdot) = \big\{ v + \hbar^{-2\pi \delta} u^{r-1} i^*(v) \, : \, v \in \mathit{H}^*(M;R)\big\}, \\
\label{eq:eigenspace-lambda} E_{\lambda} & = \mathit{ker}(\lambda - u \ast \cdot) \\
\notag
& = \big\{ -i_!(v) + \lambda^{r-1} u v + \lambda^{r-2} u^2 v + \cdots + \lambda u^{r-1} v
\, : \, v \in H^*(N;R) \big\}.
\end{align}
Each of these is a subalgebra. In particular, \eqref{eq:eigenspace-lambda} is isomorphic to $H^*(N;R)$ as a ring, with the unit element being the idempotent
\begin{equation} \label{eq:lambda-idempotent}
q = \frac{1}{(1-r)\lambda \hbar^{2\pi \delta}}(-i_!(1) + \lambda^{r-1} u + \lambda^{r-2} u^2 + \cdots + \lambda u^{r-1}).
\end{equation}

We also need to consider the (genus zero, with no descendants) relative Gromov-Witten invariants of the pair $(B,D)$. In general, such invariants have the form
\begin{equation} \label{eq:relative-gw}
\langle (\mu_1,w_1),\dots,(\mu_m,w_m),v_1,\dots,v_n \rangle_{A}^{(B,D)} \in \bQ
\end{equation}
where $m > 0$, $\mu_1,\dots,\mu_m > 0$ are the multiplicities of tangency with $D$ at the marked points, $w_1,\dots,w_m \in H^*(D;\bQ)$, $v_1,\dots,v_n \in H^*(B;\bQ)$, and $A \in H_2(B)$ is a class such that $A \cdot D = \mu_1 + \cdots + \mu_m$. The virtual dimension formula says that in order for \eqref{eq:relative-gw} to be nonzero, we must have
\begin{equation} \label{eq:virtual-1}
\mathrm{dim}(B) + 2c_1(A) + 2(m+n) - 6 = \sum_i |v_i| + \sum_j (|w_j| + 2\mu_j),
\end{equation}
where the $2\mu_j$ term takes into account the constraint imposed by having intersections with $D$ of multiplicity $\mu_j$. Equivalently, \eqref{eq:virtual-1} can be written as
\begin{equation} \label{eq:virtual-2}
\mathrm{dim}(B) + 2c_1(A) - 6 = \sum_i (|v_i| - 2) + \sum_j (|w_j| + 2\mu_j - 2).
\end{equation}
The divisor axiom shows that \eqref{eq:relative-gw} vanishes if $|v_i| < 2$ for any $i$, so we exclude that situation from now on, which means that the right hand side of \eqref{eq:virtual-2} is nonnegative. In our case, because of \eqref{eq:c1} and Assumption \ref{th:few-spheres}, we have 
\begin{equation}
\mathrm{dim}(B) + 2c_1(A) - 6 = (\mathrm{dim}(N) - 2) + (2r-2)(1 - A \cdot D) - 2 \leq (2r-2)(2 - A \cdot D) - 2. 
\end{equation}
By comparing this with \eqref{eq:virtual-2}, one sees that the only possibly nontrivial invariants \eqref{eq:relative-gw} have $A \cdot D = 1$, hence $m = 1$ and $\mu_1 = 1$.

\begin{lemma} \label{th:relative-vanish}
$\sum_A \hbar^{\omega_B(A)} \langle (1,w) \rangle_{A}^{(B,D)} = 0$ for any $w$.
\end{lemma}

\begin{proof}
The proof follows the strategy of \cite{maulik-pandharipande06} (see also \cite{hu-li-ruan08} for similar arguments in a symplectic setting), which is to compare the relative invariants of $(B,D)$ with the absolute invariants of $B$ through degeneration to the normal cone and deformation. A symplectic cut \cite{lerman95} allows us to write $B$ as the symplectic sum of two pieces $(B^{\mathit{left}},D^{\mathit{left}})$ and $(B^{\mathit{right}},D^{\mathit{right}})$. The first of these is symplectically deformation equivalent to $(B,D)$ itself, but with the cohomology class of the symplectic form changed to \begin{equation} \label{eq:moving-symplectic-form}
[\omega_{B^{\mathit{left}}}] = [\omega_B] + 2\pi (\epsilon - \delta) u.
\end{equation}
The other piece is $B^{\mathit{right}} = \bar{B}^{\mathit{toy}} \times N$ containing $D^{\mathit{right}} = H^{\mathit{toy}} \times N$. 

To keep the notation compact, we will mostly omit the homology classes of pseudo-holomorphic curves. In that case, the convention is that we sum over all classes with $\hbar$ weights given by symplectic areas. Choose a basis $\{w_i\}$ of $H^*(D^{\mathit{right}};\bQ)$ and the corresponding Poincar\'e dual basis $\{w_i^*\}$ , so that $\sum_i w_i \otimes w_i^*$ represents the Poincar\'e dual of the diagonal. In view of the restrictions on relative Gromov-Witten invariants observed above, the symplectic sum formula \cite{ionel-parker04,li-ruan98} takes on the form
\begin{equation} \label{eq:symplectic-sum}
\begin{aligned}
& \langle uw \rangle^B - \sum_{A^{\mathit{right}} \cdot D^{\mathit{right}} = 0} \hbar^{\omega_{B^{\mathit{right}}}(A^{\mathit{right}})} \langle uw \rangle_{A^{\mathit{right}}}^{B^{\mathit{right}}} \\ & = 
\sum_i \langle (1,w_i) \rangle^{(B^{\mathit{left}},D^{\mathit{left}})} \langle (1,w_i^*), uw \rangle^{(B^{\mathit{right}},D^{\mathit{right}})} \\
& + \sum_{i_1,i_2} \langle (1,w_{i_1}) \rangle^{(B^{\mathit{left}},D^{\mathit{left}})} \langle (1,w_{i_2}) \rangle^{(B^{\mathit{left}},D^{\mathit{left}})}
\langle (1,w_{i_1}^*), (1,w_{i_2}^*), uw \rangle^{(B^{\mathit{right}},D^{\mathit{right}})} \\
& + \cdots
\end{aligned}
\end{equation}
Here, we think of $uw$ as being represented by (minus) the Poincar\'e dual of $w$ inside $D$. Degeneration moves such a cycle into $B^{\mathit{right}} \setminus D^{\mathit{right}}$, and therefore the simplest term corresponds to curves in $B^{\mathit{right}}$ which do not intersect $D^{\mathit{right}}$. The other terms measure configurations consisting of a curve in $B^{\mathit{right}}$ with an arbitrary number of ``tails'' in $B^{\mathit{left}}$ (see Figure \ref{fig:sum}). Assumption \ref{th:few-spheres} ensures that the left hand side of \eqref{eq:symplectic-sum} vanishes, since it implies that only curves lying inside $D$ contribute to $\langle uw \rangle^B$, and the contribution there is the same as to $\langle uw \rangle^{B^{\mathit{right}}}$. 

It is unproblematic to assume that $\epsilon = (r+1)/(r-1) \delta$, which slightly simplifies the situation since $\bar{B}^{\mathit{toy}}$ becomes monotone. The lowest energy contribution to the relative invariants of $(B^{\mathit{right}},D^{\mathit{right}})$ comes from rational curves lying in the $\bar{B}^{\mathit{toy}}$ fibres, and in fact those which are lines in the ruling of that manifold.
Denoting the class of those lines by $A^{\mathit{right}}$, one finds that
\begin{equation}
\langle (1,w_i^*), uw_j \rangle^{(B^{\mathit{right}},D^{\mathit{right}})}_{A^{\mathit{right}}} = \delta_{ij}.
\end{equation}
Using that one rewrites \eqref{eq:symplectic-sum} as
\begin{equation}
0 = \langle (1,w) \rangle^{(B^{\mathit{left}},D^{\mathit{left}})} + \left\{\parbox{20em}{series whose terms are monomials of degree $>0$ in the $\langle (1,w_i) \rangle^{(B^{\mathit{left}},D^{\mathit{left}})}$, with coefficients containing only $>0$ powers of $\hbar$}\right\}.
\end{equation}
If we assume that $(1,\cdot)^{(B^{\mathit{left}},D^{\mathit{left}})}$ is nonzero, this immediately leads to a contradiction by looking at the lowest possible power of $\hbar$ which occurs. Deformation invariance shows that the relative Gromov-Witten invariants of $(B,D)$ and $(B^{\mathit{left}},D^{\mathit{left}})$ agree. Of course, the symplectic areas change, but because of \eqref{eq:moving-symplectic-form} they change in the same way for all the homology classes of curves involved in $\langle (1,w) \rangle^{(B,D)}$ (in fact, the last step is not even needed for our applications, since the relative invariants of $(B^{\mathit{left}},D^{\mathit{left}})$ is what will really be relevant).
\end{proof}

\begin{remark}
Readers familiar with the general symplectic sum formula from \cite{ionel-parker04} will recall that in the general formulation given there, there is a middle term (the $S$-matrix). However, it is known that this is trivial in genus $0$ \cite[Proposition 14.10]{ionel-parker04}, which is the only case considered here. In fact, an alternative proof of Lemma \ref{th:relative-vanish} could be given using the equality between relative and absolute genus $0$ Gromov-Witten invariants from \cite[Proposition 14.9]{ionel-parker04}.
\end{remark}
\begin{figure}
\begin{center}
\begin{picture}(0,0)%
\includegraphics{sum.pstex}%
\end{picture}%
\setlength{\unitlength}{2881sp}%
\begingroup\makeatletter\ifx\SetFigFont\undefined%
\gdef\SetFigFont#1#2#3#4#5{%
  \reset@font\fontsize{#1}{#2pt}%
  \fontfamily{#3}\fontseries{#4}\fontshape{#5}%
  \selectfont}%
\fi\endgroup%
\begin{picture}(3426,3115)(1014,-2014)
\put(3451,-1036){\makebox(0,0)[lb]{\smash{{\SetFigFont{10}{12}{\rmdefault}{\mddefault}{\updefault}{\color[rgb]{0,0,0}$B^{\mathit{right}}$}%
}}}}
\put(2701,389){\makebox(0,0)[lb]{\smash{{\SetFigFont{10}{12}{\rmdefault}{\mddefault}{\updefault}{\color[rgb]{0,0,0}$B^{\mathit{left}}$}%
}}}}
\end{picture}%
\caption{\label{fig:sum}}
\end{center}
\end{figure}

From now on, we will assume for technical convenience that $N$ is {\em Spin} (the reason is the same as in Example \ref{th:graph}). Using the local model $B^{\mathit{toy}} \times N \subset B$, we now introduce a Lagrangian correspondence \begin{equation}
C = C^{\mathit{toy}} \times \Delta_N \subset B^{\mathit{toy}} \times N \times N^- \subset B \times N^-.
\end{equation}

\begin{lemma} \label{th:hf-correspondence}
Take $\lambda$ as in \eqref{eq:eigenvalues}. By a suitable choice of local coefficient system and bounding cocycle, $C$ can be made into an object of $\mathit{Fuk}(B \times N^-)_{(r-1)\lambda}$, in such a way that the following holds. We have a canonical ring isomorphism
\begin{equation} \label{eq:c-ring}
\begin{aligned}
& \mathit{HF}^*(C,C) \iso \mathit{HF}^*(C^{\mathit{toy}},C^{\mathit{toy}}) \otimes H^*(N;R) \\
& \quad \iso R\langle w_1,\dots,w_r \rangle/ \{w_i^2 \text{ for all $i$},\, w_iw_j+w_jw_i+2\lambda \text{ for all $i \neq j$}\} \otimes H^*(N;R).
\end{aligned}
\end{equation}
Moreover, specializing the open-closed string map yields a ring homomorphism
\begin{equation} \label{eq:double-open-closed}
\mathit{QH}^*(B) \otimes \mathit{QH}^*(N) \longrightarrow \mathit{HF}^*(C,C),
\end{equation}
which can be partially described as follows. On the factor $\mathit{QH}^*(N) = H^*(N;R)$ we just have the obvious inclusion; and elements of the form $uv \in u H^*(N;R) \subset \mathit{QH}^*(B)$ are mapped to $\lambda v$.
\end{lemma}

\begin{proof}
Recall that, up to a suitable notion of homotopy \cite[Chapter 4]{fooo}, the curved $A_\infty$-algebra
\begin{equation} \label{eq:endo-curved}
\mathit{hom}_{\mathit{FO}(B \times N^-)}(C,C)
\end{equation}
is independent of auxiliary choices (like almost complex structures) made in defining $\mathit{FO}(B \times N^-)$. This is proved by using parametrized moduli spaces \cite[Section 4.6]{fooo}, and the same argument allows one to degenerate $B \times N^-$ to $(B^{\mathit{left}} \times N^-) \cup (B^{\mathit{right}} \times N^-)$, where $C$ goes to $C^{\mathit{toy}} \times \Delta_N \subset (B^{\mathit{right}} \times N^-) = \bar{B}^{\mathit{toy}} \times N \times N^-$. In parallel with \eqref{eq:symplectic-sum}, The resulting curved $A_\infty$-structure consists of that inherited from $C^{\mathit{toy}} \subset B^{\mathit{toy}}$ and correction terms involving relative Gromov-Witten invariants of $(B^{\mathit{left}},D^{\mathit{left}})$. To be precise, the $A_\infty$-structure involves actual cycles representing the invariants from Lemma \ref{th:relative-vanish}. However, again up to homotopy, the specific choice of cycles is irrelevant, so one can take them to be empty. The structure of \eqref{eq:endo-curved} up to homotopy determines the possible bounding cycles, as well as the endomorphism algebras of the resulting objects of the actual (unobstructed) Fukaya category \cite[Theorem 4.1.3]{fooo}. With that in mind, \eqref{eq:c-ring} follows from the computations in the model case \eqref{eq:toy-clifford} (one uses the same flat $\bC^*$-bundles).

The same argument applies to the first order infinitesimal ``bulk'' deformations induced by cycles in $B$, as long as those cycles can be moved entirely into $B^{\mathit{right}} \times N$ when degenerating, which is true for all the elements considered in the statement of the Lemma. The resulting computation is simple, and we leave it to the reader.
\end{proof}

\begin{addendum} \label{th:idempotent-goes-to-unit}
Consider again just the part of \eqref{eq:double-open-closed} concerning $\mathit{QH}^*(B)$. Using the ring structure we deduce that
\begin{equation}
i_!(1) = -\overbrace{u \ast \cdots \ast u}^r - \hbar^{2\pi \delta} u \longmapsto - \lambda^r - \hbar^{2\pi \delta} \lambda = 0.
\end{equation}
Furthermore, the image of the idempotent $q$ from \eqref{eq:lambda-idempotent} is the identity element in $\mathit{HF}^0(C,C)$.
\end{addendum}

\subsection{A correspondence functor\label{subsec:correspondence}}
Pick some $\lambda$ as in \eqref{eq:eigenvalues}, and suppose that $C$ has been made into an object of $\mathit{Fuk}(B \times N^-)_{(r-1)\lambda}$ following Lemma \ref{th:hf-correspondence}. Take the idempotent $q \in \mathit{QH}^0(B)$ which defines the projection to the corresponding eigenspace \eqref{eq:lambda-idempotent}, and let $p \in \mathit{HF}^*(C,C)$ be one of the idempotents from \eqref{eq:minimal-idempotent}.

\begin{lemma} \label{th:ww}
The formal direct summand of $C$ associated to $p$ gives rise to a cohomologically full and faitful functor
\begin{equation} \label{eq:c-functor}
\mathit{Fuk}(N)_0^{\mathit{perf}} \longrightarrow \mathit{Fuk}(B)_{(r-1)\lambda,q}^{\mathit{perf}}.
\end{equation}
For brevity, denote these two categories by $A_N$ and $A_B$. The induced maps on Hochschild cohomology and open-closed string maps fit into a commutative diagram
\begin{equation} \label{eq:open-closed-correspondence}
\xymatrix{
\mathit{QH}^*(N) \ar[rr] \ar[d] && p\mathit{HF}^*(C,C)p \ar[d] && \ar[ll] \ar[d] \mathit{QH}^*(B)q \\
\mathit{HH}^*(A_N,A_N) \ar[rr] && \mathit{HH}^*(A_N,A_B) && \ar[ll] \mathit{HH}^*(A_B,A_B)
}
\end{equation}
where the maps in the top row come from \eqref{eq:double-open-closed}.
\end{lemma}

\begin{proof}
Parts of this statement come from the general theory of Fukaya categories and Lagrangian correspondences, and do not need specific proofs. Given that $q$ maps to the identity in $\mathit{HF}^*(C,C)$, the correspondence $C$ gives rise to a functor taking values in a module category:
\begin{equation} \label{eq:formal-functor}
\mathit{Fuk}(N)_0 \longrightarrow \mathit{Fuk}(B)_{\lambda,q}^{\mathit{mod}},
\end{equation}
and this fits into a commutative diagram analogous to \eqref{eq:open-closed-correspondence}. As pointed out in \cite{smith10}, any decomposition of $C$ into formal direct summands, such as the one provided by $p$, yields a corresponding decomposition of this functor. The two additional facts that need to be proved are first of all that the functor takes values in the subcategory $\mathit{Fuk}(B)_{\lambda,q}^{\mathit{perf}}$, and that it is full and faithful. There are general results which ensure that the first property holds under suitable assumptions \cite{wehrheim-woodward09, lekili-lipyanskiy10}, and which also make it easy to determine the action of the functor on Floer cohomology groups. However, they do not apply exactly to the situation under discussion, and we will instead argue as in Lemma \ref{th:hf-correspondence}.

It is convenient to introduce a quilted version $\mathit{Fuk}^\sharp(B)_{(r-1)\lambda}$ of the Fukaya category, as in \cite{mww} but tailored to our specific application. As objects, this admits objects of $\mathit{Fuk}(B)_{(r-1)\lambda}$ as well as generalized Lagrangian submanifolds of a specific kind, namely pairs consisting of our fixed Lagrangian correspondence $C$ and an object of $\mathit{Fuk}(N)_0$. Morphisms are defined by (chain complexes underlying) quilted Floer cohomology \cite{ww}. We can again use $q$ to project to a piece of the quilted category, and denote the result by $\mathit{Fuk}^\sharp(B)_{(r-1)\lambda,q}$. By construction, \eqref{eq:formal-functor} can be factored as follows:
\begin{equation} \label{eq:quilted-functor}
\mathit{Fuk}(N)_0 \longrightarrow \mathit{Fuk}^\sharp(B)_{(r-1)\lambda,q} \longrightarrow \mathit{Fuk}(B)_{(r-1)\lambda,q}^{\mathit{mod}}.
\end{equation}
The second arrow is a Yoneda-type functor, which is full and faithful when restricted to the subcategory $\mathit{Fuk}(B)_{(r-1)\lambda,q} \subset \mathit{Fuk}^\sharp(B)_{(r-1)\lambda,q}$. What we want to show is that the first arrow takes any Lagrangian submanifold to an object that's quasi-isomorphic to one in $\mathit{Fuk}(B)_{(r-1)\lambda,q}$.

Let's temporarily turn to a simpler geometric situation, which is $C^{\mathit{toy}} \times \Delta_N \subset B^{\mathit{toy}} \times N \times N^-$. We have an analogue of \eqref{eq:quilted-functor}, and the first part of it maps any object $L_0$ of $\mathit{Fuk}(N)_0$ to the generalized Lagrangian submanifold $(C^{\mathit{toy}} \times \Delta_N, L_0)$ in $\mathit{Fuk}^\sharp(B^{\mathit{toy}} \times N)_{(r-1)\lambda,q}$. In particular, if $L_1$ is any object of $\mathit{Fuk}(B^{\mathit{toy}} \times N)_{(r-1)\lambda}$, then the quilted Floer cohomology can be expressed in terms of ordinary Floer cohomology as
\begin{equation}
H(\mathit{hom}_{\mathit{Fuk}^\sharp(B^{\mathit{toy}} \times N)_{(r-1)\lambda}}(L_1, (C^{\mathit{toy}} \times \Delta_N, L_0))) \iso \mathit{HF}^*(L_1, C^{\mathit{toy}} \times L_0).
\end{equation}
This is fairly straightforward, involving only re-folding strips and changing their widths, as indicated in Figure \ref{fig:re-quilt} (but not the deeper analytic degeneration arguments from \cite{wehrheim-woodward09}). The same thing holds for morphisms in the other direction, and these isomorphisms are compatible with products. From this, one easily concludes that $(C^{\mathit{toy}} \times \Delta_N,L_0)$ is quasi-isomorphic to $C^{\mathit{toy}} \times L_0$.

If one now looks at the original picture \eqref{eq:quilted-functor} and applies the degeneration argument from Lemma \ref{th:hf-correspondence}, the outcome is that the computations carried out inside $B^{\mathit{toy}} \times N$ could in principle be deformed by contributions lying in $B^{\mathit{left}}$, but that these in fact vanish, ensuring that the argument above remains valid. If instead of $C$ one now uses its direct summand defined by $p$, the resulting functor no longer lands in the actual Fukaya category, but in its idempotent completion, which then allows a formal extension as in \eqref{eq:c-functor}. The proof that the resulting functor is cohomologically full and faithful uses the same kind of argument, the concrete input being that $p\mathit{HF}^*(C^{\mathit{toy}},C^{\mathit{toy}})p \iso R$.
\end{proof}

\begin{figure}
\begin{centering}
\begin{picture}(0,0)%
\includegraphics{quilt.pstex}%
\end{picture}%
\setlength{\unitlength}{3158sp}%
\begingroup\makeatletter\ifx\SetFigFont\undefined%
\gdef\SetFigFont#1#2#3#4#5{%
  \reset@font\fontsize{#1}{#2pt}%
  \fontfamily{#3}\fontseries{#4}\fontshape{#5}%
  \selectfont}%
\fi\endgroup%
\begin{picture}(6912,7911)(61,-6664)
\put(1726,-1186){\makebox(0,0)[lb]{\smash{{\SetFigFont{10}{12.0}{\rmdefault}{\mddefault}{\updefault}{\color[rgb]{0,0,0}$B^{\mathit{toy}} \times N$}%
}}}}
\put(1201,-1186){\makebox(0,0)[lb]{\smash{{\SetFigFont{10}{12.0}{\rmdefault}{\mddefault}{\updefault}{\color[rgb]{0,0,0}$N$}%
}}}}
\put(451,-136){\makebox(0,0)[lb]{\smash{{\SetFigFont{10}{12.0}{\rmdefault}{\mddefault}{\updefault}{\color[rgb]{0,0,0}$L_0$}%
}}}}
\put(2701,-136){\makebox(0,0)[lb]{\smash{{\SetFigFont{10}{12.0}{\rmdefault}{\mddefault}{\updefault}{\color[rgb]{0,0,0}$L_1$}%
}}}}
\put(5701,-6586){\makebox(0,0)[lb]{\smash{{\SetFigFont{10}{12.0}{\rmdefault}{\mddefault}{\updefault}{\color[rgb]{0,0,0}$N \times B^{\mathit{toy}}$}%
}}}}
\put( 76,-2836){\makebox(0,0)[lb]{\smash{{\SetFigFont{10}{12.0}{\rmdefault}{\mddefault}{\updefault}{\color[rgb]{0,0,0}$L_0$}%
}}}}
\put(826,-3886){\makebox(0,0)[lb]{\smash{{\SetFigFont{10}{12.0}{\rmdefault}{\mddefault}{\updefault}{\color[rgb]{0,0,0}$N$}%
}}}}
\put(3226,-2836){\makebox(0,0)[lb]{\smash{{\SetFigFont{10}{12.0}{\rmdefault}{\mddefault}{\updefault}{\color[rgb]{0,0,0}$C^{\mathit{toy}}$}%
}}}}
\put(4251,-3211){\makebox(0,0)[lb]{\smash{{\SetFigFont{10}{12.0}{\familydefault}{\mddefault}{\updefault}{\color[rgb]{0,0,0}removing the diagonal,}%
}}}}
\put(4251,-61){\makebox(0,0)[lb]{\smash{{\SetFigFont{10}{12.0}{\familydefault}{\mddefault}{\updefault}{\color[rgb]{0,0,0}splitting the right strip into factors}%
}}}}
\put(2401,-3886){\makebox(0,0)[lb]{\smash{{\SetFigFont{10}{12.0}{\rmdefault}{\mddefault}{\updefault}{\color[rgb]{0,0,0}$B^{\mathit{toy},-}$}%
}}}}
\put(1126,-1636){\makebox(0,0)[lb]{\smash{{\SetFigFont{10}{12.0}{\rmdefault}{\mddefault}{\updefault}{\color[rgb]{0,0,0}$\Delta_N$}%
}}}}
\put(2176,-1636){\makebox(0,0)[lb]{\smash{{\SetFigFont{10}{12.0}{\rmdefault}{\mddefault}{\updefault}{\color[rgb]{0,0,0}$L_1$}%
}}}}
\put(1651,-3886){\makebox(0,0)[lb]{\smash{{\SetFigFont{10}{12.0}{\rmdefault}{\mddefault}{\updefault}{\color[rgb]{0,0,0}$N$}%
}}}}
\put(4551,-3476){\makebox(0,0)[lb]{\smash{{\SetFigFont{10}{12.0}{\familydefault}{\mddefault}{\updefault}{\color[rgb]{0,0,0}shrinking the strip}%
}}}}
\put(4576,-5461){\makebox(0,0)[lb]{\smash{{\SetFigFont{10}{12.0}{\rmdefault}{\mddefault}{\updefault}{\color[rgb]{0,0,0}$L_0 \times C^{\mathit{toy}}$}%
}}}}
\put(6601,-5461){\makebox(0,0)[lb]{\smash{{\SetFigFont{10}{12.0}{\rmdefault}{\mddefault}{\updefault}{\color[rgb]{0,0,0}$L_1$}%
}}}}
\put(4051,-6136){\makebox(0,0)[lb]{\smash{{\SetFigFont{10}{12.0}{\familydefault}{\mddefault}{\updefault}{\color[rgb]{0,0,0}folding}%
}}}}
\put(2701,-5461){\makebox(0,0)[lb]{\smash{{\SetFigFont{10}{12.0}{\rmdefault}{\mddefault}{\updefault}{\color[rgb]{0,0,0}$C^{\mathit{toy}}$}%
}}}}
\put(1276,1064){\makebox(0,0)[lb]{\smash{{\SetFigFont{10}{12.0}{\rmdefault}{\mddefault}{\updefault}{\color[rgb]{0,0,0}$C^{\mathit{toy}} \times \Delta_N$}%
}}}}
\put(1876,-6586){\makebox(0,0)[lb]{\smash{{\SetFigFont{10}{12.0}{\rmdefault}{\mddefault}{\updefault}{\color[rgb]{0,0,0}$B^{\mathit{toy},-}$}%
}}}}
\put(1201,-6586){\makebox(0,0)[lb]{\smash{{\SetFigFont{10}{12.0}{\rmdefault}{\mddefault}{\updefault}{\color[rgb]{0,0,0}$N$}%
}}}}
\put(451,-5461){\makebox(0,0)[lb]{\smash{{\SetFigFont{10}{12.0}{\rmdefault}{\mddefault}{\updefault}{\color[rgb]{0,0,0}$L_0$}%
}}}}
\put(1576,-4336){\makebox(0,0)[lb]{\smash{{\SetFigFont{10}{12.0}{\rmdefault}{\mddefault}{\updefault}{\color[rgb]{0,0,0}$L_1$}%
}}}}
\end{picture}%
\caption{\label{fig:re-quilt}}
\end{centering}
\end{figure}

\begin{addendum} \label{th:blowup-generate}
Suppose that we have a full $A_\infty$-subcategory $O_N \subset \mathit{Fuk}(N)_0$ which satisfies the assumptions of Theorem \ref{th:split-generation-3}, meaning that there is a linear map
\begin{equation} \label{eq:n-int}
\mathit{HH}^*(O_N,O_N) \longrightarrow R
\end{equation}
whose composition with the open-closed string map yields $\int_N$. Let $O_B \subset \mathit{Fuk}(B)_{(r-1)\lambda,q}^{\mathit{perf}}$ be its image under the functor induced by $(C,p)$. By restricting \eqref{eq:open-closed-correspondence} we get a commutative diagram
\begin{equation} \label{eq:o-diagram}
\xymatrix{
\mathit{QH}^*(N) \ar[rr] \ar[d] && p\mathit{HF}^*(C,C)p \ar[d] && \ar[ll] \ar[d] \mathit{QH}^*(B)q \\
\mathit{HH}^*(O_N,O_N) \ar[rr]^{\iso} && \mathit{HH}^*(O_N,O_B) && \ar[ll]_{\iso} \mathit{HH}^*(O_B,O_B).
}
\end{equation}
Take $v \in \mathit{QH}^*(N;R)$, and associate to it an element $x$ of $\mathit{QH}^*(B)q = E_\lambda$, as in \eqref{eq:eigenspace-lambda} but additionally dividing by the nonzero constant $(1-r)\lambda^r \hbar^{2\pi\delta}$. Crucially, $x$ has a nonzero top degree component (with respect to the ordinary grading of cohomology) if and only if $v$ does. Lemma \ref{th:hf-correspondence} (see also Addendum \ref{th:idempotent-goes-to-unit}) shows that $v$ and $x$ have the same image in $\mathit{HF}^*(C,C)$. Hence, if we combine \eqref{eq:n-int} with the maps in \eqref{eq:o-diagram} to define a map
\begin{equation} \label{eq:b-int}
\mathit{HH}^*(O_B,O_B) \longrightarrow R,
\end{equation}
the assumptions of Theorem \ref{th:split-generation-3} are satisfied. This shows that the image of our functor split-generates, which means that the functor induces a quasi-equivalence $\mathit{Fuk}(N)_0^{\mathit{perf}} \iso \mathit{Fuk}(B)_{(r-1)\lambda,q}^{\mathit{perf}}$.
\end{addendum}

\subsection{The examples\label{subsec:final}}
Let $K \subset \bC P^3$ be a quartic surface, containing a smooth elliptic curve $T \subset K$. We equip $K$ with the Fubini-Study form rescaled in such a way that $T$ has area $1$. Take a symplectic automorphism $f$ of $K$ as in Lemma \ref{th:2-twists}, and form the symplectic mapping torus of $f \times f$ as in \eqref{eq:symplectic-mapping-torus}, calling the result $E$. We will compare this to the same process applied to the identity map, which gives $E^{\mathit{triv}} \iso T \times K^2$ (the sign of the symplectic form does not matter, see the discussion at the end of Section \ref{subsec:covering-trick}). Define $N = T \times E$ and $N^{\mathit{triv}} = T \times E^{\mathit{triv}} = T^2 \times K^2$, with the product symplectic structures as usual. Lemma \ref{th:2-twists} implies that $N$ and $N^{\mathit{triv}}$ are diffeomorphic, and that the diffeomorphism preserves the cohomology classes of symplectic forms as well as the homotopy classes of almost complex structures. In fact, inspection of the proof of that Lemma shows that the following more precise result holds:

\begin{lemma} \label{th:deform-torus}
There are closed two-forms $\gamma \in \Omega^2(N)$, $\gamma^{\mathit{triv}} \in \Omega^2(N^{\mathit{triv}})$, as well as a family of diffeomorphism $g_r: N \rightarrow N^{\mathit{triv}}$, defined for small $r>0$, such that $g_r^*(\omega_N^{\mathit{triv}} + r\gamma^{\mathit{triv}}) = \omega_N + r\gamma$ (on the level of cohomology, this implies that our diffeomorphisms map $[\omega_N^{\mathit{triv}}]$ to $[\omega_N]$, and $[\gamma^{\mathit{triv}}]$ to $[\gamma]$). \qed
\end{lemma}

In fact, $\gamma^{\mathit{triv}}$ is pulled back from projection $N^{\mathit{triv}} \rightarrow E^{\mathit{triv}}$ (and correspondingly for $\gamma$). Embed $N^{\mathit{triv}}$ symplectically into $M = K^7$ by identifying it with $T^2 \times K^2 \times \{\mathit{point}\}^3$. Denote that embedding by $i^{\mathit{triv}}$. Since $T \subset K$ represents a nonzero homology class, $[\gamma^{\mathit{triv}}]$ is in the image of $i^{\mathit{triv,*}}$. Fix a closed two-form $\delta$ on $M$ such that $i^{\mathit{triv},*}[\delta] = [\gamma^{\mathit{triv}}]$. 

\begin{lemma} \label{th:gromov}
There is a symplectic embedding $i: N \rightarrow M$ which (as a smooth embedding) is isotopic to $i^{\mathit{triv}} \circ g$. In fact, the embedding has the following sharper property. There are isotopies
\begin{equation}
\begin{aligned}
& \phi_r^{\mathit{triv}}: N^{\mathit{triv}} \longrightarrow N^{\mathit{triv}}, && 
\phi_r^{\mathit{triv},*}(\omega_N^{\mathit{triv}} + r\, i^{\mathit{triv},*}\delta) = 
\omega_N^{\mathit{triv}} + r\, \gamma^{\mathit{triv}}, \\
& \phi_r: N \longrightarrow N, &&
\phi_r^*(\omega_N + r i^*\delta) = \omega_N + r \gamma
\end{aligned}
\end{equation}
defined for small $r \geq 0$, and starting at the identity for $r = 0$; such that for any $r>0$, the two embeddings
\begin{equation} \label{eq:2-embeddings}
i \circ \phi_r, \; 
i^{\mathit{triv}} \circ \phi_r^{\mathit{triv}} \circ g_r
\;: (N,\omega_N + r\gamma) \longrightarrow (M,\omega_M + r\delta)
\end{equation}
are isotopic (through symplectic embeddings for these forms).
\end{lemma}

\begin{proof}
Consider first the map 
\begin{equation} \label{eq:first-emb}
i^{\mathit{triv}} \circ g_r: N \longrightarrow M,
\end{equation}
for some $r>0$. This satisfies $(i^{\mathit{triv}} \circ g_r)^*[\omega_M] = [\omega_N]$. Moreover, the symplectic form $(i^{\mathit{triv}} \circ g_r)^*\omega_M = g_r^*\omega_{N^{\mathit{triv}}}$ can be deformed (through symplectic forms) to $g_r^*(\omega_N^{\mathit{triv}} + r \gamma^{\mathit{triv}}) = \omega_N + r\gamma$, and from there back to $\omega_N$. Hence, the derivative of \eqref{eq:first-emb} can be deformed (through injective vector bundle map) to an embedding of symplectic vector bundles. In other words, the map \eqref{eq:first-emb} is {\em formally symplectic}. Applying the $h$-principle for symplectic immersions \cite[(3.4.2.A)]{gromov86} (see also \cite[(16.4.3)]{eliashberg-mishachev}) then yields symplectic immersion $i: N \rightarrow M$. Since $\mathrm{dim}(M) > 2\,\mathrm{dim}(N)$, one can assume (after a generic perturbation) that $i$ is actually an embedding.

The existence of $\phi_r$ and $\phi_r^{\mathit{triv}}$ follows from Moser's argument (restricting to small $r$). For any given $r>0$, we now have the two symplectic embeddings \eqref{eq:2-embeddings}. By construction, these are isotopic in the {\em formally symplectic} sense. The parametrized version of the previously used $h$-principle (see the references above, or \cite{datta} for a more specific exposition) shows that they can be deformed into each other through symplectic immersions. As before, since $\mathrm{dim}(M) > 2\,\mathrm{dim}(N) + 1$, a small perturbation will turn these immersions into embeddings.
\end{proof}

Let $B$ be the result of blowing up $M$ along $N$ (embedded through the map $i$ we have just constructed) with small parameter $\delta>0$, and $B^{\mathit{triv}}$ the same with $N^{\mathit{triv}}$. Lemma \ref{th:gromov} implies that $B$ is symplectically deformation equivalent to $B^{\mathit{triv}}$.

\begin{lemma}
Both blowups $B$ and $B^{\mathit{triv}}$ satisfy Assumption \ref{th:few-spheres}.
\end{lemma}

\begin{proof}
Let $J_K^{\mathit{int}}$ be the given (integrable) complex structure on $K$, for which $T \subset K$ is a holomorphic curve. By standard transversality methods, we can find another compatible almost complex structure $J_K$ which agrees with $J_K^{\mathit{int}}$ in a neighbourhood of $T$, and which has no non-constant $J_K$-holomorphic spheres. Equip $K^5$ with the product structure induced by $J_K$. Then $T^2 \times \{\mathit{point}\}^3$ is a complex submanifold, and if we choose the point to lie on $T$ as well, then $J_K$ is integrable near that submanifold. We can therefore carry out the blowup process following the local algebro-geometric model. Having done that, take the product with two more copies of $K$ equipped with $J_K$. The outcome is that we get an almost complex structure $J_{B^{\mathit{triv}}}$ on the blowup and a pseudo-holomorphic blowdown map $(B^{\mathit{triv}},J_{B^{\mathit{triv}}}) \rightarrow (K,J_K)^7$. Hence, all the pseudo-holomorphic spheres must be contained in the fibres of this map, which means that they lie in multiples of the homology class $Z$ of a line in the exceptional divisor. This implies the vanishing of Gromov-Witten invariants as required in Assumption \ref{th:few-spheres}. The rest of the requirements for $B^{\mathit{triv}}$ are obvious.

Because of the deformation linking it to $i^{\mathit{triv}}$, $i$ inherits the property that the normal bundle is trivial. Since $B$ is deformation equivalent to $B^{\mathit{triv}}$, it has the same Gromov-Witten invariants. The absence of pseudo-holomorphic spheres in $N$ is easy to arrange, see Section \ref{subsec:basic-mapping-torus}.
\end{proof}

Choose some $\lambda$ as in \eqref{eq:eigenvalues}. The subspace $E_\lambda \subset \mathit{QH}^*(B)$ has an intrinsic characterization as the eigenspace of $(r-1)\lambda$ for quantum multiplication by $c_1(B)$. It is also an algebra with identity element $q$ as in \eqref{eq:lambda-idempotent}. Take $\lambda^{-1} H^3(B;\bZ) = \lambda^{-1} u H^1(N;\bZ)$, project it to $E_\lambda$ by taking its quantum multiplication with $q$, and denote the outcome by $\Gamma_\lambda \subset \mathit{QH}^1(B)$. Consider the associated Fukaya category $\mathit{Fuk}(B)_{(r-1)\lambda,q}$. Let $\SS = \mathit{Spec}(\RR)$ be the curve associated to the unit torus polynomial, with its standard one-form $\theta$. We also have their closures $\bar{\SS}$ and $\bar{\theta}$.

\begin{proposition} \label{th:1}
The image of $\Gamma_{\lambda}$ under the open-closed string map is not contained in the set $\mathit{Per}(\mathit{Fuk}(B)_{(r-1)\lambda,q},\bar{\SS},\bar{\theta})$ of periodic elements.
\end{proposition}

\begin{proof}
Corollary \ref{th:not-periodic} implies that there is a class in $H^1(N;\bZ)$ whose image is not contained in $\mathit{Per}(\mathit{Fuk}(N)_0,\bar{\SS},\bar{\theta})$. More precisely, this is not quite the result as originally stated, but has the following minor differences. First of all, the grading of the Fukaya category has been reduced to $\bZ/2$, which means that we have to check whether the uniqueness results from Section \ref{subsec:unique} are applicable. However, that was already taken care of in Addenda \ref{th:sneaky} and \ref{th:not-iso-even-without-grading}. The second difference is that we are considering the product of the symplectic mapping torus with an additional copy of $T$, which correspondingly means that we have to take the product of the Lagrangian submanifolds under consideration with a fixed circle in $T$. This requires the same kind of check, but the argument from Lemma \ref{th:assumption-satisfied} goes through as before, and similarly for Addendum \ref{th:not-iso-even-without-grading}. The third difference is that the Fukaya category $\mathit{Fuk}(N)_0$ contains more objects than our original version $\mathit{Fuk}(N)$ (see Remark \ref{th:compare}), but that is clearly irrelevant for this argument.

We can now use the full and faithful functor from Lemma \ref{th:ww} to transfer this result to $B$. By our computation of \eqref{eq:double-open-closed}, any class $v \in H^1(N;\bZ) \subset \mathit{QH}^1(N)$ and its counterpart $\lambda^{-1} uv \ast q \in \Gamma_{\lambda}$ have the same image in $\mathit{HH}^*(\mathit{Fuk}(N)_0,\mathit{Fuk}(B)_{(r-1)\lambda,q})$. As explained in Section \ref{subsec:functor}, this allows one to map families that follow a given deformation field. The rest of the argument carries over without any changes.
\end{proof}

We now consider the analogous construction for $B^{\mathit{triv}}$, defining the idempotent $q^{\mathit{triv}} \in \mathit{QH}^0(B^{\mathit{triv}})$ and subgroup $\Gamma^{\mathit{triv}}_{\lambda} \subset \mathit{QH}^1(B^{\mathit{triv}})$ as before.

\begin{proposition}
The image of $\Gamma^{\mathit{triv}}_{\lambda}$ under the open-closed string map is contained in the set $\mathit{Per}(\mathit{Fuk}(B^{\mathit{triv}})_{(r-1)\lambda,q^{\mathit{triv}}},\bar{\SS},\bar\theta)$.
\end{proposition}

\begin{proof}
As before, the first step is to establish a version of Corollary \ref{th:yes-periodic} for $\mathit{Fuk}(N^{\mathit{triv}})_0$, the key additional consideration being that we are now working with a larger Fukaya category than before. However, one can use Theorem \ref{th:split-generation-3} and the argument from Example \ref{th:product-mirror} to show that $\mathit{Fuk}(N^{\mathit{triv}})_0$ is split-generated by the subcategory $\mathit{Fuk}(N^{\mathit{triv}})$ (with its grading reduced to $\bZ/2$). One takes the family of bimodules used in the proof of Corollary \ref{th:yes-periodic}, reduces its grading to $\bZ/2$ as well, and then extends it to a family of bimodules over $\mathit{Fuk}(N^{\mathit{triv}})_0$. 
When carrying over the results to $B^{\mathit{triv}}$, one uses Addendum \ref{th:blowup-generate} for split-generation, and the same computation as in Proposition \ref{th:1}.
\end{proof}

We explained the intrinsic characterization of $E_\lambda$ and $q$ above, and that also yields an intrinsic characterization of $\Gamma_\lambda$. Comparing the two Propositions above shows that, as announced in the Introduction,

\begin{corollary}
$B$ and $B^{\mathit{triv}}$ are not symplectically isomorphic. \qed
\end{corollary}

\providecommand{\bysame}{\leavevmode\hbox to3em{\hrulefill}\thinspace}
\providecommand{\MR}{\relax\ifhmode\unskip\space\fi MR }
\providecommand{\MRhref}[2]{%
  \href{http://www.ams.org/mathscinet-getitem?mr=#1}{#2}
}
\providecommand{\href}[2]{#2}

\end{document}